\documentclass[11pt,a4paper,reqno]{amsart}

\usepackage[applemac]{inputenc}
\usepackage[T1]{fontenc}
\usepackage{amsmath}
\usepackage{amsthm}
\usepackage{amsfonts,bm}
\usepackage{amssymb}
\usepackage{graphicx}
\usepackage{xcolor}
\usepackage{amsbsy}
\usepackage{mathrsfs}
\usepackage{dsfont}
\usepackage{bbm}
\usepackage{mathtools}
\usepackage{boondox-cal}
\newtheorem{theorem}{Theorem}[section]
\newtheorem{lemma}[theorem]{Lemma}

\newtheorem{proposition}[theorem]{Proposition}

\newtheorem{definition}[theorem]{Definition}
\theoremstyle{remark}
\newtheorem{remark}[theorem]{\it \bf{Remark}\/}

\numberwithin{equation}{section}
\catcode`@=11
\def\section{\@startsection{section}{1}%
  \z@{1.5\linespacing\@plus\linespacing}{.5\linespacing}%
  {\normalfont\bfseries\large\centering}}
\catcode`@=12
\newcommand{\be}{\begin{equation}}
\newcommand{\ee}{\end{equation}}
\newcommand{\bea}{\begin{eqnarray}}
\newcommand{\eea}{\end{eqnarray}}
\newcommand{\bee}{\begin{eqnarray*}}
\newcommand{\eee}{\end{eqnarray*}}

\def\Tr{{\rm Tr \;}}

\def\pa{\partial}

\def\RR{\mathbb{R}}

\def\e{\varepsilon}

\def\bar#1{{\overline #1}}
\def\fref#1{{\rm (\ref{#1})}}

\def\R2+{\RR ^2_+}

\def\pa{\partial}

\def\lim{\mathop{\rm lim}}

\def\sup{\mathop{\rm sup}}

\def\log{{\rm log}}

\def\pa{\partial}

\def\pa{\partial}

\def\ba{\begin{array}}
\def\ea{\end{array}}

\def\ep{\epsilon}

\def\lab{\label}

\begin{document}

\title[Derivation of the kinetic wave equation  in the inhomogeneous setting]{Derivation of the kinetic wave equation for quadratic dispersive problems in the inhomogeneous setting}
\author[Ioakeim Ampatzoglou]{Ioakeim Ampatzoglou}
\address{Ioakeim Ampatzoglou,  
Department of Mathematics, Baruch College, The City University of New York, USA}
\email{ioakeim.ampatzoglou@baruch.cuny.edu}
\author[Charles Collot]{Charles Collot}
\address{Charles Collot,  
Laboratoire AGM - Analyse G\'eometrie Mod\'elisation, CY Cergy Paris Universit\'e et CNRS, France}
\email{charles.collot@cyu.fr}
\author[Pierre Germain]{Pierre Germain}
\address{Pierre Germain,  
Department of Mathematics, Imperial College, United Kingdom}
\email{p.germain@imperial.ac.uk}

\begin{abstract}We examine the validity of the kinetic description of wave turbulence for a model quadratic equation. We focus on the space-inhomogeneous case, which had not been treated earlier; the space-homogeneous case is a simple variant. We determine nonlinearities for which the kinetic description holds, or might fail, up to an arbitrarily small polynomial loss of the kinetic time scale. More precisely, we focus on the convergence of the Dyson series, which is an expansion of the solution in terms of the random data.

\end{abstract}

\maketitle

\tableofcontents

\section{Introduction}

Understanding the behavior of large physical systems is a fundamental problem of mathematical physics. With the size of the system being extremely large, deterministic prediction of its behavior is practically impossible, and one resorts to an average description. Kinetic theory provides a mesoscopic framework to study the qualitative properties of large systems and obtain a statistically accurate prediction of their evolution in time.

In systems of many nonlinear interacting waves, the effective equation is the kinetic wave equation (KWE) which describes the energy dynamics of systems where many waves interact in a weakly  nonlinear way following a dispersive time reversible dynamics. All rigorous results so far have  focused on the case where the equation is set on the torus, with space-homogeneous data, resulting in a homogeneous kinetic equation in the limit.

In this paper, we derive rigorously, up to an arbitrarily small polynomial loss of the kinetic time scale, an inhomogeneous (transport) kinetic wave equation. This is achieved by considering data whose spatial correlation exhibit a two-scale structure. The inhomogeneous kinetic wave equation approximates the average Wigner transform of the solution as the number of interacting waves goes to infinity and the strength of the nonlinearity goes to zero. We also provide examples of equations for which the kinetic limit might not hold.

\subsection{The equation, the data, and the singular limit}

Recall the notation for the Fourier multiplier $p$: 
$$\widehat{p(D) f} = p(\xi) \widehat{f}(\xi).$$
We consider the following nonlinear Schr\"odinger equations for complex fields in $\mathbb R^d$ with quadratic nonlinearities\footnote{This also includes equations of the form $i \partial_t u + \omega(D) u = \lambda M ( u + \overline{u} )^2$ by a change of variables.}:
\begin{equation}
i \partial_t u + \omega(D) u = \lambda M ( M u + M \overline{u} )^2,\label{nonlinschrod}
\end{equation}
where 
\begin{itemize}
\item $\omega(\xi) = \omega_0 + \displaystyle\frac{|\xi|^2}{2}$, with $\omega_0 = 0$ or $\epsilon^{-2}$, is the dispersion relation,
\item $M = m(\epsilon D)$, where $m$ is a smooth, bounded, real valued even function,
\item  $\lambda>0$ encodes the size of nonlinear effects. 
\end{itemize}
(the scaling laws for the dispersion relation and the multiplier are natural in the limit we will be considering).

This equation derives from the Hamiltonian
$$
\mathcal{H} (u) = \int \frac{1}{2}|\sqrt{\omega(D)} u|^2 + \frac{8\lambda}{3} (\mathfrak{Re} M u)^3.
$$

As we will see, the value of $\omega$ and $m$ at zero will be key for the validity of the kinetic wave equation.  

It is a convenient model for our purposes: on the one hand, it retains all the difficulties related to the derivation of a kinetic wave equation, from a quadratic equation, in the inhomogeneous case; and on the other hand, it avoids further technicalities related to specific equations of physical interest (quasilinearity of the equations, singularity of the dispersion relations, vectorial nature of the unknown...).

The initial data will be chosen to be a random Gaussian field 
\begin{equation}
\label{hedgehog}
u(t=0,x) = u_0(x)=\int a(x,\xi) e^{i\frac{\xi}{\epsilon} \cdot x}dW(\xi)
\end{equation}
where $\hat a\in \mathcal C^\infty_0(\mathbb R^{2d})$ and $dW$ is a Wiener integral. Equivalently, $u_0$ can be characterized by its covariance
$$
\mathbb{E}  \left[ \overline{u_0(x)}u_0(x') \right] = \int\overline{a(x,\xi)} a(x',\xi) e^{-i\frac{\xi}{\epsilon}\cdot (x-x')} \,d\xi.
$$
We will come back to this definition later, suffice it to say for the time being that this Gaussian field exhibits random behavior at scale $\sim \epsilon$, with an envelope at a scale $\sim 1$. More precisely,
$$
\mbox{as $\epsilon \to 0$,} \qquad \mathbb{E}  \left[ \overline{u_0(x)}u_0(x') \right] = F\left(\frac{x+x'}{2}, \frac{x-x'}{\epsilon} \right)+O(\epsilon),
$$
where $F$ is a smooth, decaying function. It is convenient at this point to introduce the rescaled Wigner transform
$$
W^\epsilon[u] (x,v) = \frac{1}{(2\pi)^{d/2}} \epsilon^{-d} \mathbb{E} \int \overline{u (x+\frac z 2)}u(x-\frac z 2 ) e^{i\frac{v}{\epsilon}.z}\,dz.
$$
Roughly speaking, it provides a measure of the amount of energy of $u$ (in $L^2$) localized in phase space at position $x$ and frequency $v/\epsilon$. In particular, it is such that
\begin{equation}
\label{hedgehog2}
\mbox{as $\epsilon \to 0$}, \qquad W^\epsilon[u_0] (x,v) \rightarrow |a(x,v)|^2 = \rho_0(x,v).
\end{equation}
Our aim is to show that
$$
\mbox{as $\epsilon \to 0$}, \qquad W^\epsilon [u\left( t \right)] (x,v) \rightarrow \rho(t,x,v),
$$
where $\rho$ solves the kinetic wave equation
\begin{equation}
\label{KWE} \tag{KWE}
\left\{ \begin{array}{l}
\displaystyle \partial_t \rho + \frac{1}{\epsilon} v \cdot \nabla_x \rho = \frac{8\pi}{T_{kin}}\mathcal{C}[\rho(x)] \\
\rho (t=0) = \rho_0.
\end{array} \right.
\qquad \qquad \mbox{where $\displaystyle T_{kin} = \frac{1}{\lambda^2 \epsilon^2}$}
\end{equation}
The  collision operator $\mathcal{C}$ is given by
\begin{equation}\label{collision operator nonzero}
 \begin{aligned}
  \mathcal{C}[\rho](t,x,v)=&m^2\int \bigg[\delta(\Sigma_{-})\delta(\Omega_{-})m_1^2m_2^2\rho\rho_1\rho_2\left(\frac{1}{\rho}-\frac{1}{\rho_1}-\frac{1}{\rho_2}\right)\\ 
  &\hspace{2cm}+2 \delta(\Sigma_{+})\delta(\Omega_{+})m_1^2m_2^2\rho\rho_1\rho_2\left(\frac{1}{\rho}+\frac{1}{\rho_1}-\frac{1}{\rho_2}\right) \bigg]\,dv_1\,dv_2,
  \end{aligned}
  \end{equation}
  \begin{equation*}
  \begin{cases}
  \Sigma_{-}=v-v_1-v_2\\
  \Sigma_{+}=v+v_1-v_2
  \end{cases}
  \quad
  \begin{cases}
  \Omega_{-}=\omega(v) - \omega(v_1) - \omega(v_2)\\
  \Omega_{+}= \omega(v) + \omega(v_1) - \omega(v_2),
  \end{cases}
  \quad 
  \begin{cases}
  \rho=\rho(v)\\
  \rho_i=\rho(v_i)
  \end{cases}
  \quad
   \begin{cases}
  m=m(v)\\
  m_i=m(v_i),\quad i\in\{1,2\},
  \end{cases}
  \end{equation*}
This equation displays two (singular) time scales:
\begin{itemize}
\item $\epsilon$, the transport time scale, since $\frac{1}{\epsilon}$ is the group velocity for solutions of the linear Schr\"odinger equation localized at frequency $\sim \frac{1}{\epsilon}$. In other words, $\epsilon$ is the time over which such solutions travel a distance $\sim 1$, which implies that, for $t \gg \epsilon$, one expects the solution to spread and nonlinear interactions to be damped.
\item $T_{kin}$, the characteristic time scale for the mixing in frequency space occuring through the collision operator $\mathcal{C}$. Notice the dependence in $\lambda^2$ - as opposed to $\lambda$ appearing in front of the nonlinearity of~\eqref{nonlinschrod} - which is characteristic of square-root cancellations caused by randomness.
\item Of particular relevance is of course the regime where both time scales agree, $T_{kin} = \epsilon$, or in other words $\lambda = \epsilon^{-3/2}$.
\end{itemize}

Other important time scales are
\begin{itemize}
\item $\epsilon^2$, the linear time-scale. Notice that resonances only become relevant if $t \gg \epsilon^2$.
\item $\lambda^{-1}$, the nonlinear time-scale, after which nonlinear effects become relevant.
\end{itemize}

\subsection{Background}
\subsubsection{Derivation of the kinetic wave equation}
The kinetic wave equation was first introduced by Peierls \cite{Peierls} in his work on solid state physics, and independently by Hasselmann \cite{Hasselmann1, Hasselmann2} who worked on water waves. Later, Zakharov and collaborators \cite{Zakharov1,Zakharov2} revisited the topic and provided a broad framework applying to various Hamiltonian systems satisfying weak nonlinearity, high frequency, phase randomness assumptions. Nowadays, the kinetic theory of waves, known as wave turbulence theory, is fundamental to the study of nonlinear waves, having applications e.g. in plasma theory \cite{Davidson}, oceanography \cite{Janssen,guide} and crystal thermodynamics \cite{Spohn phonons}. For an introduction to this broad research field  and its applications, see e.g. Nazarenko \cite{Nazarenko}, Newell-Rumpf \cite{Newell}. 

The first rigorous result regarding derivation of the homogeneous (KWE) was obtained in the pioneering work of Lukkarinen and Spohn \cite{LS}, who were able to reach the kinetic timescale for the cubic nonlinear Schr\"odinger equation (NLS) at statistical equilibrium, leading to a linearized version of the kinetic wave equation (see also \cite{Faou}). The key idea in \cite{LS} is to employ Feynmann diagrams to obtain control of the correlations; it has inspired most of the subsequent works.

For the cubic NLS, the derivation of the homogeneous kinetic wave equation  for random data out of statistical equilibrium was first addressed in \cite{BGHS} using Strichartz estimates to control the error term.  Later, in \cite{CG1,CG2},  two of the authors of this paper, inspired by the ideas of \cite{LS} (construction of an approximate solution, control of the higher order terms via Feynmann diagramms)  estimated  the error in Bourgain spaces instead of Strichartz spaces and were able reach  the kinetic timescale up to arbitrarily small polynomial loss. At the same time, a similar result was obtained independently by Deng and Hani \cite{DengHani1}. Recently, Deng and Hani \cite{DengHani2} reached  the kinetic timescale for the cubic NLS, which provides the first full derivation of the homogeneous (KWE) for (NLS). 

In many situations of physical interest, the leading nonlinear term is quadratic: for instance, this is the case for long-wave perturbations of the acoustic type (which can exist in most media), or interaction of three-wave packets in media with a decay dispersion law. These models have extremely wide  applications, ranging from solid state physics to hydrodynamics, plasma physics etc. Recently, under the assumption of multiplicative noise,  Staffilani and Tran \cite{StafTran} reached the kinetic timescale for the Zakharov-Kuznetsov (ZK) equation. In the absence of noise, the result of \cite{StafTran} is conditional.

Regarding the inhomogeneous (KWE) and its connection to nonlinear waves, Spohn \cite{Spohn phonons} discusses the emergence of a kinetic wave equation, which he calls phonon  Boltzmann equation. However, to the best our knowledge, there are no rigorous results justifying a derivation of an inhomogeneous kinetic wave equation from dispersive dynamics.

\subsubsection{Derivation of related kinetic models} The kinetic wave equation is to phonons, or
linear waves, what the Boltzmann equation is to classical particles. The Boltzmann
equation was rigorously derived for hard spheres in the foundational work of Lanford \cite{lanford}, who used particle hierarchies in the Boltzmann-Grad limit \cite{Grad 1, Grad 2}. Later, King \cite{king} derived the equation for short range potentials.
This program was recently put in full rigor by Gallagher-Saint-Raymond-Texier \cite{GSRT}. Short range potentials were also discussed in \cite{pulvirenti-simonella}. A few articles deal with
the derivation of kinetic models for higher order interactions \cite{IP1,IP2}, mixtures \cite{IPM} and quantum particles \cite{BCEP1,BCEP2,BCEP3}. The derivation of the quantum Boltzmann equation is closely related to
the derivation of the kinetic wave equation, but possibly more challenging, since dispersive equations can be thought of as an
intermediary step between a quantum mechanical model with a large number of particles, and
kinetic theory.

Another direction of research focuses on linear dispersive models with random potential, from which
one can derive the linear Boltzmann equation for short times \cite{Spohn2}, and the heat equation for longer times \cite{erdos1,erdos2}.

Finally, \cite{faou-germain hani,BGHS2} investigate the possibility of deriving Hamiltonian models for NLS with deterministic
data in the infinite volume, or big box, limit.

\subsection{Statement of the main result} \label{subsec:mainthm} We now state the main result of this paper, regarding the well-posedness of equation \eqref{nonlinschrod} and its approximation by the corresponding kinetic wave equation.

\begin{theorem}\label{main theorem}
Let $a\in \mathcal C^\infty_0(\mathbb R^{2d})$ and $\nu>0$. Consider Equation \eqref{nonlinschrod} with initial data \eqref{hedgehog} and
\begin{itemize}
\item either $\omega_0= \epsilon^{-2}$
\item or $\omega_0=0$ and $m(0)=0$.
\end{itemize}
Then there exist $\epsilon^*>0$ and $\kappa>0$ such that for any $0<\epsilon<\epsilon^*$ and for any $0<T<\min\{\epsilon,\epsilon^\nu T_{kin}\}$, there exists a set $E$ of probability $\mathbb{P}(E)>1-\epsilon^{\kappa}$, such that on $E$, there exists a unique solution $u$ to \eqref{nonlinschrod} in $[0,T]$. 

Moreover, the solution $u$ is approximated by the solution $\rho$ of the corresponding kinetic wave equation in the following sense:

For any $t\in[0,T]$ and $\xi\in\mathbb{R}^d$, there holds:
\begin{equation*}
\int_{\mathbb{R}^d}|\widehat{\rho}(t,\xi,v) - \widehat{W}_E^\epsilon[u](t,\xi,v)|\,dv\lesssim \epsilon^{\nu}\left(\frac{T}{T_{kin}}\right),
\end{equation*}
where
$$W^\epsilon_E[u] (x,v)  = \frac{1}{(2\pi)^d} \mathbb{E} \left[\mathds{1}_E\int_{\mathbb{R}^d} \overline{u \left(x+\frac{\epsilon y}{2}\right)} u\left(x-\frac{\epsilon y}{2} \right) e^{iv\cdot y}\,dy\right], $$
$E$ is the exceptional set of existence obtained above and
 $\rho$ solves \eqref{KWE} with initial data \eqref{hedgehog2}.
\end{theorem}

\begin{remark} Our result has a clear homogeneous counterpart for the Fourier modes of the solution if the equation \eqref{nonlinschrod} is set on the torus instead of $\mathbb{R}^d$.
\end{remark}

\begin{remark}
What ranges of $\epsilon$ and $\lambda$ are relevant in the previous theorem? First, the approximation is accurate in the limit $\epsilon \to 0$; second, in order to approach the kinetic time scale $T_{kin}$ up to a small power of $\epsilon$, the above theorem requires $\epsilon < T_{kin}$, or in other words $\lambda > \epsilon^{-3/2}$. Physically, this means that the kinetic time scale should be smaller than the kinetic time scale; otherwise, dispersive decay prevents nonlinear interaction from having a sizable effect. \end{remark}

\subsection{Strategy of the proof} The proof is based on building a sufficiently good approximation of the solution  and representing it as a Dyson's series.  The iterative  scheme  we adopt to approximate our solution  is given by:
$$
u^0 = e^{-it \omega(D)} u_0,
$$
\begin{equation}
\label{defun}
\left\{ 
\begin{array}{l}
i \partial_t u^n + \omega(D) u^n = \lambda \sum_{j+k = n-1} M(Mu^j + M\overline{u^j}) (Mu^k + M\overline{u^k}) \\
u^n(t=0) = 0,
\end{array},\quad n\geq 1
\right.
\end{equation}
Formally, the Dyson series representation of the solution is given by $\displaystyle u = \sum_{n=0}^\infty u^n$, but the question of convergence is delicate and will be studied carefully in the rest of the paper. 
To efficiently achieve that, we will represent the Dyson series  by binary Feynmann graphs as will be discussed in Section 5.

The solution $u$ is written as the sum of the approximate solution (truncated Dyson series) and the error term:
$$
u = u^{app} + u^{err}, \qquad \mbox{where} \qquad u^{app} = \chi\left(\frac{t}{T}\right)\sum_{n=0}^N u^n,
$$
where $\chi$ is a $C_0^\infty$ cut-off function such that  $\chi=1$ for $|t|<1$ and $\chi=0$ for $|t|>2$. 
The error $u^{err}$ satisfies the equation for $|t|\leq 2T$:
\begin{equation}\label{error equation}
i \partial_t u^{err} +\omega(D) u^{err} =  \lambda \left[ \mathfrak{L}_N (u^{err})  +  \mathfrak{B}(u^{err})+ E_N \right],
\end{equation}
where the linearized operator $\mathfrak{L}_N$ around $u^{app}$ is given by
$$\mathfrak{L}_N(w)=8 M \mathfrak{Re} M u^{app}\mathfrak{Re} M w,$$
the bilinear operator $\mathfrak{B}$ is given by
$$\mathfrak{B}(w)=4M(\mathfrak{Re} M w)^2,$$
 and the error term $E_N$ by
$$
E_N = 4\sum_{\substack{j+k \geq N \\ j,k = 0,\dots,N}} M [\mathfrak{Re} Mu^j \mathfrak{Re} M u^k].
$$
The terms on the right hand side of \eqref{error equation} are estimated in Proposition \ref{linearization proposition}, Proposition \ref{bilinear proposition} and Proposition \ref{error control proposition} respectively. Comparison to the kinetic wave equation is discussed in Section 4 and convergence to it will be proved combining the results obtained there with Proposition \ref{propexpansion}.

\noindent

\subsection{Failure of convergence on the kinetic time scale for $m(0)=1$ and $\omega_0=0$}

We believe that the kinetic wave equation might fail to describe solutions to 
\be \label{id:usualNLS}
i\pa_t u+\Delta u=(u+\bar u)^2
\ee
on the time scale $T_{kin}$, due to a low frequency inflation. Note that the kinetic equation \eqref{collision operator nonzero} is not even well defined, as the mass of the unit ball for the measure $\delta(\Sigma_+)\delta(\Omega_+)dv_1dv_2=\delta(v+v_1-v_2)\delta (2v.(v-v_2))$ diverges as $v\rightarrow 0$. This issue was already raised by Spohn, see Section 6 in \cite{Spohn phonons} for a discussion, where an hypothesis for the non-vanishing of $\omega(0)$ that is analogue to the present one in Theorem \ref{main theorem} is assumed. Hence, our convergence result of Subsection \ref{subsec:mainthm} would be sharp in the sense that at the origin in Fourier, either a cancellation of nonlinear effects  $m(0)=0$, or a lack of resonance due to a non-zero dispersion relation $\omega_0= \epsilon^{-2}$, $c_0>0$, would be needed to ensure the validity of the kinetic description.

We recall (see Section \ref{sec:diagrams}) that the Dyson series \eqref{defun} can be represented as a sum over Feynman interaction diagrams, and that their $L^2$ norm can be represented as a sum over paired graphs:
\be \label{id:dysonbelt}
u^n=\sum_{G\in \mathcal G_n} u_G, \qquad \mathbb E \, \| u^n(t) \|_{L^2(\mathbb R^d)}^2=\sum_{G'\in \mathcal G_n^p} \mathcal F_t(G') \quad \mbox{for all }t\in \mathbb R.
\ee
Our second result is that the second series above is not absolutely convergent on the kinetic time scale. This itself does not imply the divergence of $\mathbb E \| u^n(t) \|_{L^2(\mathbb R^d)}^2$ as cancellations could occur, see Remark \ref{re:divergence} and Subsection \ref{subsec:difficulties} for a discussion.

\begin{proposition} \label{prop:belt}
For all $d\geq 2$, there exists a Schwartz function $a\in \mathcal S(\mathbb R^{2d})$ such that, for any $\kappa>0$, the following holds true for initial data of the form \eqref{hedgehog} in the range:
\be \label{id:constrainttimebelt}
\epsilon^{2-\kappa}\leq t\leq \epsilon^{1+\kappa}.
\ee
There exists $n^*(d,\kappa)$, such that for all $n\geq n^*$, there exists a paired graph $G^*\in \mathcal G_{2n}^p$ as defined in Subsection \ref{pairinggraphs} for equation \eqref{id:usualNLS}, two constants $C,C'>0$ and $\epsilon_0>0$ such that for all $0<\epsilon\leq \epsilon_0$:
\be \label{id:estimationbelt}
C  (\lambda t)^{4n} \ep^{2d}t^{-d}\leq \mathcal F_t(G^*)  \leq C'(\lambda t)^{4n}  \ep^{2d}t^{-d}.
\ee
\end{proposition}

\begin{remark}

The kinetic equation can a priori only be reached provided that its time scale $T_{kin}$ is shorter than the transport time scale $\epsilon$ and that the regime is weakly nonlinear $\epsilon^2\ll \lambda^{-1}$. The sum of the absolute values of the terms in the second series in \eqref{id:dysonbelt} diverges at a time before $T_{kin}$, since the nonlinear time scale $\lambda^{-1}$ at which the estimate \eqref{id:estimationbelt} becomes singular is shorter than $T_{kin}$.

\end{remark}

\begin{remark} \label{re:divergence}

We believe that the first series in \eqref{id:dysonbelt} does not either converge on the kinetic time scale, that is, $\mathbb E \| u_G(t) \|_{L^2(\mathbb R^d)}^2$ diverges as \eqref{id:estimationbelt} for some $G\in \mathcal G_n$. In \cite{CG2} the last two authors were able to show such result, for a similar counter-example graph for a cubic nonlinearity for a different time scale. The proof showed no cancellation occurred from other pairings for the same interaction diagram $G$. We believe the same strategy could be applied here. This would not imply the actual divergence of $u^n$, but would indicate that cancellations with another interaction diagram $G'$ are required.  Such cancellations were shown to exist by Deng-Hani \cite{DengHani2} for the (NLS) on the torus, see Subsection \ref{subsec:difficulties} for a further discussion on whether their strategy is applicable in our case.

\end{remark}

\subsection{Difficulties: the belt and the inhomogeneous setting} \label{subsec:difficulties}

The main thrust of this paper is to provide a derivation of the \textit{inhomogeneous} kinetic wave equation up to the kinetic time scale, with a loss of an arbitrarily small power, while previous rigorous works all address the homogeneous problem.

A first difficulty is linked to the use of the Wigner transform, which leads to technical complications compared to Fourier series, which suffice for the homogeneous problem.

A second difficulty is linked to the range of available time-scales: with the scaling defined above, only time scales less than $\epsilon$ are of interest for the inhomogeneous problem set on $\mathbb{R}^d$: past this time scale, waves will have dispersed, since the data is localized at frequency $O(\epsilon^{-1})$, corresponding to a group velocity $O(\epsilon^{-1})$. 

Over such small time scales, the \textit{belt} family of diagrams, which first appeared  in~\cite{DengHani1, CG2} in the context of cubic problems (nonlinear Schr\"odinger equation - NLS), becomes a possible obstruction to the convergence of the Dyson series. For (NLS), it was shown in the aforementioned works that the belt diagrams would lead to a failure of convergence if only self-correlations of diagrams are considered.  But a deeper analysis in~\cite{DengHani2} shows that surprising cancellations between diagrams occur for (NLS), at least for time scales close to 1.

Considering quadratic problems leads to new perspectives on the belt diagrams. We chose the most simple dispersion relation, namely $\omega(\xi) = \omega_0 + \frac{|\xi|^2}{2}$, which can be obtained by Taylor expanding any smooth dispersion relation at a point; note that a linear term in $\xi$ can be removed by using translation invariance in space. As for the nonlinearity, $(\mathfrak{Re} u)^2$ has the advantage of being Hamiltonian, and containing the three types of interactions: $u \cdot u \rightarrow u$, $\overline{u} \cdot u \rightarrow u$, and $\overline{u} \cdot \overline{u} \rightarrow u$.

In case $\omega_0 =0$, a direct analog of the cubic belt example exists for the interaction $u \cdot \overline{u} \rightarrow u$. The underlying kinetic equation presents a singular kernel, which may be the sign that this belt diagram represents a true physical instability, and is not canceled by other diagrams, as was the case for (NLS). Still in contrast with (NLS), these belt diagrams can be dampened, and convergence of the Dyson series restored, if the structure of the nonlinear term is appropriate, namely if it provides a cancellation at output frequency 0. Under this condition, it is possible to rely on the machinery developed in~\cite{LS,CG1}.

In the case $\omega_0 = \epsilon^{-2}$, the belt example ceases to be an obstacle to the convergence of the Dyson series. This made us hopeful that convergence could be proved - which was indeed the case, but a completely new argument is needed. Namely, none of the tools used to understand the combinatorics of Feynman graphs, and to derive bounds for them, seemed to apply. In contrast  to~\cite{LS,CG1}, we introduce a more intrinsic point of view by not assuming a given ordering of intermediate times in the graph. We should mention that the works \cite{DengHani2, DengHani1} do not assume ordering of the intermediate times either.

\subsection{Application to some physical examples} This paper focuses on model equations to simplify the exposition, and identifies stable and unstable regimes in the weakly turbulent regime (weak nonlinearity, scale separation, and data with decorrelated phases) as stated in Theorem~\ref{main theorem} and Proposition~\ref{pr:formulagraph}.

Quadratic interactions occuring in our model problem are of three types: $u \cdot u \rightarrow u$, $u \cdot \overline{u} \rightarrow u$, and $\overline u \cdot \overline u \rightarrow u$, with obvious notations. As for the dispersion relation, it is of the type $\omega(\xi) = \omega_0 + \frac{|\xi|^2}{2}$ (a further requirement is that $\omega_0$ be either $0$, or comparable to $\epsilon^{-2}$, but we will gloss over this precise scaling in the following).

Our results can be summarized as follows:
\begin{itemize}
\item Interactions of the type $u \cdot u \rightarrow u$ and $\overline u \cdot \overline u \rightarrow u$ are stable on the kinetic time scale. This means that the Dyson series converges on the kinetic time scale (up to an arbitrarily small power), and that the average behavior is described by the kinetic wave equation.
\item For interactions of the type $u \cdot \overline{u} \rightarrow u$, stability on the kinetic time scale holds if either $\omega_0 \neq 0$, or the quadratic nonlinearity exhibits a cancellation at zero frequency.
\item Finally, for interactions of the type $u \cdot \overline{u} \rightarrow u$, if $\omega_0 = 0$ and the quadratic nonlinearity does not contain a cancellation, the series fails to converge on the kinetic time scale.
 \end{itemize}

It is natural to conjecture that these three bullet points remain true for quadratic nonlinear dispersive equations with a scalar unknown function, and a dispersion relation $\omega(\xi)$ with $\omega(0) = \omega_0$. We review below some classical examples.

\medskip

\noindent
The \textit{Kadomtsev-Petiashvili equation} is given by
$$
\partial_t u + \partial_x^{-1} \partial_y^2 u + \partial_x^3 u + u\partial_x u = 0.
$$
Since it only contains interactions $u \cdot u \rightarrow u$, it should be stable in the weakly turbulent regime. For the closely related Zakharov-Kuznetsov model, the kinetic time scale was indeed reached in \cite{StafTran} for the homogeneous problem with random forcing.

\medskip
 
\noindent
The \textit{beta-plane equation}
$$
\partial_t \omega + u \cdot \nabla \omega = \partial_x \Delta^{-1} \omega, \qquad u = \nabla^\perp \Delta^{-1} \omega
$$
modeling planetary flows, falls into the same category: only $u \cdot u \rightarrow u$ interactions occur.

\medskip

\noindent \textit{The elastic beam equation}
$$
\partial_t^2 u + \omega(D)^2 u + u^2 = 0
$$
becomes, after setting $v = \partial_t u - i \sqrt{\omega(D)} u$,
$$
\partial_t v + i \omega(D) v = \left( \frac{\overline{v} - v}{2 \omega(D)} \right)^2.
$$
This is equation~\eqref{nonlinschrod}, except for the Fourier multipliers $\frac{1}{\omega(D)}$. If $\omega_0 =0$, this Fourier multiplier makes the zero frequency even more singular, and thus the kinetic description is unlikely to be valid. If $\omega_0 = \epsilon^{-2}$, the Fourier multiplier is not singular at zero frequency, and our result applies to validate the kinetic description.

The asymptotic behavior of the kinetic wave equation for this model set in the lattice was recently considered in~\cite{RST}.

\medskip

\noindent
The \textit{(generalized) nonlinear Klein-Gordon equation}
$$
\partial_t^2 u + \omega(D) u + u^2 = 0,
$$
becomes, after setting $v = \partial_t u - i \sqrt{\omega(D)} u$,
$$
\partial_t v + i \sqrt{\omega(D)} v = \left( \frac{\overline{v} - v}{2\sqrt{\omega(D)}} \right)^2.
$$
As discussed above, the stability condition is $\omega_0 \neq 0$; but quadratic resonances should also exist, which is not the case if $\omega(D) = \omega_0 - \Delta$. In connection with the kinetic limit, this equation was considered by Spohn on the lattice~\cite{Spohn phonons}, where quadratic resonances do exist.

\medskip

\noindent
\textit{Water waves equations} have a more intricate structure. In a proper set of coordinates, the unknown becomes a scalar function $u$, which satisfies the following equation
$$
i \partial_t u + |D|^\alpha u = T_{m_{++}}(u,u) + T_{m_{+-}}(u,\overline u) + T_{m_{--}}(\overline u,\overline u).
$$
Here, $\alpha=\frac{1}{2}$ for gravity waves, and $\frac{3}{2}$ for capillar waves, $T_m$ stands for the pseudo-product operator with symbol $m(\xi,\eta)$, and cubic and higher-order terms were omitted,. We refer to~\cite{GMS1,GMS2} for exact formulas and more precise definitions. In the light of our discussion above, the condition for stability becomes the vanishing of $m_{--}$ if the output frequency is zero - and one checks that it is sastisfied!

\bigskip

This brief discussion only addressed some equations with a scalar unknown, excluding most examples from plasma physics and fluid mechanics, for which  some of our ideas probably also apply.

\subsection*{Acknowledgements}
While working on this project, IA was supported by the NSF grants DMS-2418020, DMS-2206618, and the Simons collaborative grant on weak turbulence. CC was supported by the ERC-2014-CoG 646650 SingWave. PG was supported by the NSF grant DMS-1501019, by the Simons collaborative grant on weak turbulence, and by the Center for Stability, Instability and Turbulence (NYUAD). 

\section{Notations} \subsection{Probability space} The underlying probability space is denoted $\Omega$, the probability measure  $\mathbb{P}$, and the expectation $\mathbb{E}$.

\subsection{Fourier transform} For $f$ a function on $\mathbb{R}^d$, we denote
$$
\widehat{f}(\xi) = \frac{1}{(2 \pi)^{d/2}} \int_{\mathbb{R}^d} f(x) e^{-i x \cdot \xi} \,dx
$$
so that
$$
f(x) = \frac{1}{(2\pi)^{d/2}} \int_{\mathbb{R}^d} \widehat{f}(\xi) e^{i x \cdot \xi}\,d\xi.
$$
With this convention, the Fourier transform is an isometry on $L^2$, and furthermore $\widehat{fg} = \frac{1}{(2 \pi)^{d/2}} \widehat{f} * \widehat{g}$.

If $F$ is a function of two variables, $F(x,v)$, we denote $\widehat{F}$ for the Fourier transform with respect to the first one:
$$
\widehat F(\xi,v) = \frac{1}{(2 \pi)^{d/2}} \int_{\mathbb{R}^d} F(x,v) e^{-i x \cdot \xi} \,dx.
$$

Given a function $f(t,x)$ on $\mathbb{R}\times\mathbb{R}^d$, we denote is space-time Fourier transform as:
$$ \widetilde{f}(\tau,\xi)=\frac{1}{(2\pi)^{\frac{d+1}{2}}}\int_{\mathbb{R}}\int_{\mathbb{R}^d} f(t,x)e^{-i(t\tau+x\cdot\xi)}\,dx\,dt.$$

\subsection{Bourgain spaces} We will use the scaled Sobolev spaces with norm
$$\|f\|_{H_\epsilon^s}=\|\langle\epsilon D\rangle^s f\|_{L^2},$$
and their associated Bourgain spaces $X_\epsilon^{s,b}$ with norm
$$\|u\|_{X_\epsilon^{s,b}}=\|e^{-it\omega(D)}u(t)\|_{H_t^b H_{\epsilon,x}^s}=\|\langle\epsilon\xi\rangle^s\langle\tau+\omega(\xi)\rangle^b\widetilde{u}(\tau,\xi)\|_{L^2(\mathbb{R}\times\mathbb{R}^d)}.$$
More details regarding Bourgain spaces are given in Appendix \ref{Xsbbasics}.

For $\epsilon>0$ and $n\in \mathbb{Z}^d$, we now define $C_{\epsilon}^n=\{x\in \mathbb R^d, \  |x-\epsilon^{-1}n|<\epsilon^{-1}/2\}$ to be the cuboid of side $\epsilon^{-1}$ and center $\epsilon^{-1}n$. For $R>0$ and an integer $l\geq 1$ we define the dyadic annulus $A_{R}^l=\{x\in \mathbb R^d, \  2^{l-1}R<|x|\leq 2^{l}R\}$, as well as $A_{R}^0=\{x\in \mathbb R^d, \ |x|\leq R\}$. Their characteristic functions are denoted $\mathbf{1}_{C_{\epsilon}^n}$ and $\mathbf{1}_{A_{R}^l}$, and enables us to define the projection operators
$$
Q_{\epsilon}^n = \mathbf{1}_{C^n_{\epsilon}}(D) \qquad \mbox{and} \qquad \mathcal A_{R}^l = \mathbf{1}_{A^l_{R}}(D).
$$
Finally, we let
$$
P_{\epsilon,N} = \mathbf{1}_{C^0_{\epsilon,2N}} - \mathbf{1}_{C^0_{\epsilon,N}}
$$
These operators are bounded on $L^p$ spaces, $1<p<\infty$ and provide decompositions of the identity:
$$
 \sum_{n \in \mathbb{Z}^d} Q_{\epsilon}^n = \operatorname{Id}.
$$

\subsection{Wigner transform and space correlation}
To derive the kinetic wave equation we use the framework of the averaged Wigner transforms. It is defined for random fields, either in Fourier or in physical space, by
$$
W[u] (x,v)= \frac{1}{(2\pi)^d} \mathbb{E} \int_{\mathbb{R}^d} \bar u \left(x+\frac z 2\right)u\left(x-\frac z 2 \right)  e^{iv\cdot z}dz = \frac{1}{(2\pi)^d} \mathbb{E} \int_{\mathbb{R}^d} e^{i\xi \cdot x} \overline{\widehat u \left(v-\frac \xi 2\right)} \widehat u \left(v+\frac \xi 2 \right) \, d\xi.
$$
With this normalization, there holds
$$
\int_{\mathbb{R}^d} W[u] (x,v) \, dv = \mathbb{E} |u(x)|^2, \qquad \int_{\mathbb{R}^d}  W[u] (x,v) \, dx = \mathbb{E}  | \widehat{u}(v)|^2.
$$
The problem we consider enjoys a separation of scale between the fluctuations and the envelope, whose typical space scales are respectively $\epsilon$ and $1$. This leads us to defining the rescaled Wigner transform
\begin{align*}
W^\epsilon[u] (x,v) & = \epsilon^{-d} W[u] \left(x,\frac{v}{\epsilon}\right)\\
& = \frac{1}{(2\pi)^d} \mathbb{E} \int_{\mathbb{R}^d} \overline{u \left(x+\frac{\epsilon y}{2}\right)} u\left(x-\frac{\epsilon y}{2} \right) e^{iv\cdot y}\,dy \\
\end{align*}

The advantage of this definition is that $W^\epsilon[u](x,v)$ has $L^\infty$ norm $\sim 1$, and concentrates most of its mass in the region $|x| + |v| \lesssim 1$, for the ansatz~\eqref{hedgehog}.

Note that
$$
\widehat W^{\epsilon}[u](\xi,v)= \frac{1}{(2\pi)^{d/2}} \epsilon^{-d} \mathbb{E} \Bigl(\overline{\widehat u \left(\frac v \epsilon-\frac \xi 2\right)} \widehat u \left(\frac v \epsilon+\frac \xi 2\right)  \Bigr)
$$
or equivalently
$$
 \quad \mbox{or} \quad  \mathbb{E}\left[ \overline{\widehat{u}(\xi)} \widehat{u}(\xi') \right] = (2\pi)^{d/2} \epsilon^d \widehat{W}^\epsilon[u] \left( \xi' - \xi,\frac{\epsilon}{2} (\xi + \xi') \right).
$$
The space correlation is encoded by the correlation function:
$$
\mathbb E (\overline{u (x)}u(x'))=Q^\ep \left(\frac{x+x'}{2},\frac{x-x'}{\ep}\right), \ \ Q^{\ep}(x,y)=\mathbb E \left(\overline{u(x+\frac{\ep y}{2})} u(x-\frac{\ep y}{2})\right),
$$
so that one has the relation (where $\mathcal{F}_y$ stands for the Fourier transformation with respect to $y$):
$$
W^\ep[u](x,v)=\frac{1}{(2\pi)^{d/2}} \mathcal{F}_y^* Q^\ep(x,y).
$$

\section{The initial data}

\subsection{The general ansatz}

We consider initial data of the form \eqref{hedgehog} where $a:\mathbb R^{2d} \rightarrow \mathbb C$ and $W$ is a complex Wiener process. We will assume throughout that 
$$
\widehat{a}(\eta,\xi) \in \mathcal{C}^\infty_0.
$$
Making use of the Ito formula,
\begin{align*}
& \mathbb{E} \left[ \int f(x) \,dW(x) \overline{\int g(x) \,dW(x)} \right] = \int f(x) \overline{g(x)}\,dx \\
& \mathbb{E} \left[ \int f(x) \,dW(x) \int g(x) \,dW(x) \right] = 0,
\end{align*}
we find that the pointwise correlation is given, in physical space, by
\begin{align*}
& \mathbb{E} \left[ \overline{u_0(x)}u_0(x') \right] = \int \overline{a(x,\xi)} a(x',\xi) e^{-i\frac{\xi}{\epsilon} \cdot(x-x')} \,d\xi \\
& \mathbb{E} \left[ u_0(x) u_0(x') \right] = 0,
\end{align*}
and in Fourier space by
\begin{align*}
\mathbb{E} \left[ \overline{\widehat{u_0}(\xi)} \widehat{u_0}(\xi') \right] & = \int \overline{\widehat{a}(\xi - \frac{\eta}{\epsilon},\eta)}\,\widehat{a}(\xi' - \frac{\eta}{\epsilon},\eta)\,d\eta \\
\mathbb{E} \left[ {\widehat{u_0}(\xi)} \widehat{u_0}(\xi') \right] & = 0.
\end{align*}
Since $\widehat{a} \in \mathcal{C}_0^\infty$, note that the above is zero unless $|\xi - \xi'| \lesssim 1$ and $|\xi|, |\xi'| \lesssim \frac{1}{\epsilon}$.

\bigskip

\noindent \underline{Correlation function} The initial correlation function is $Q_0^\epsilon$, defined by
$$
\mathbb{E}[ \overline{u_0(x)}u_0(x') ] =Q^\ep_0 \left(\frac{x+x'}{2}, \frac{x-x'}{\epsilon} \right).
$$
It can be expanded as
\begin{align*}
 Q^\ep_0 \left(x,y\right)  =& \int \overline{a(x+\frac{\epsilon y}{2},\xi)} a(x-\frac{\epsilon y}{2},\xi) e^{-i \xi \cdot y} \,d\xi  \\
= & \int |a(x,\xi)|^2 e^{-i \xi \cdot y} \, d\xi + O(\epsilon) \\
=&  (2\pi)^{d/2}  \mathcal F \left(|a|^2(x,\cdot) \right)\left(y \right)+O(\epsilon),
\end{align*}
where the implicit constant in $O$ is $\lesssim \| a\|_{L^{\infty}_xW^{1,\infty}_\xi}+\sup_{x,x'} \frac{\| \bar a(x,\cdot)-a(x',\cdot)\|_{W^{1,1}(\mathbb R^d)}}{|x-x'|}$. 

\bigskip

\noindent \underline{Wigner transform} Turning to the rescaled Wigner transform,
\begin{align*}
W_0^\epsilon(x,v) = W^\epsilon[u_0](x,v) & = \frac{1}{(2\pi)^d} \iint \overline{a (x+\frac{\epsilon y}{2},\xi)}a(x-\frac{\epsilon y}{2},\xi) e^{i(v-\xi).y}\,dy \,d\xi \\
& = \frac{1}{(2\pi)^d} \iint |a (x,\xi)|^2 e^{i(v-\xi).y}\,dy \,d\xi +O(\epsilon)\\
&=|a(x,v)|^2+O(\epsilon)
\end{align*}
where the implicit constant in $O$ is $\lesssim \left( \| a \|_{L^{\infty}_xW^{s,\infty}_\xi} + \sup_{x,x'} \frac{\| a(x,\cdot)-a(x',\cdot)\|_{W^{s,1}(\mathbb R^d)}}{|x-x'|} \right)$ for $s>d$. 

Taking the Fourier transform in the first variable,
\be \label{id:formuleWa}
\widehat{W_0^\epsilon}(\xi,v) = \frac{1}{(2\pi)^{d/2}} \int \widehat{a} \left( \eta - \frac{\xi}{2},v-\epsilon \eta \right) \widehat{a} \left( \eta + \frac{\xi}{2},v-\epsilon \eta \right) \,d\eta.
\ee
We learn from this formula (and the fact that $\widehat{a} \in \mathcal{C}_0^\infty$) that there exists a compact set $K$ such that $\operatorname{Supp}(\widehat{W_0^\epsilon}) \subset K$ for all $\epsilon$, and that, uniformly in $\epsilon$, for any $\alpha$ and $\beta$,
\begin{equation}
\label{boundsderW}
| \partial_\xi^\alpha \partial_v^\beta \widehat{W_0^\epsilon}(\xi,v) | \lesssim_{\alpha,\beta} 1.
\end{equation}

Finally, note that
\be \label{id:wignerfourier}
 \mathbb{E}\left[ \overline{\widehat{u}_0(\xi)} \widehat{u}_0 (\xi') \right] = (2\pi)^{d/2} \epsilon^d \widehat{W}^\epsilon_0 \left( \xi' - \xi,\frac{\epsilon}{2} (\xi + \xi') \right).
\ee

\subsection{The envelope ansatz}

Let
$$
u_0(x)=A(x)h_\epsilon (x),
$$
where $A:\mathbb R^d \rightarrow \mathbb C$ and $h_\epsilon$ is a stationary Gaussian field:
$$
h_\epsilon (x)= \int_{\xi \in \mathbb R^d} H(\xi)e^{i\frac{\xi}{\epsilon}. x} dW(\xi).
$$
This is obviously a particular case of the general ansatz, for which
$$
a(x,\xi) = A(x) H(\xi)
$$
The correlation for the translation invariant field $h$ reads
$$
\mathbb{E}[ \overline{h_\epsilon (x)}h_\epsilon(x') ] =  \int |H(\xi)|^2 e^{-i\frac{\xi}{\epsilon}.(x-x')} \,d\xi= ()2\pi)^d \widehat{|H|^2} \left(\frac{x-x'}{\epsilon} \right),
$$
so that for the initial condition there holds:
\begin{align*}
& \mathbb{E} [\overline{u_0(x)}u_0(x') ] \\
&= (2\pi)^d \left|A\left(\frac{x+x'}{2}\right)\right|^2 \widehat{|H|^2} \left(\frac{x-x'}{\epsilon} \right)+2\pi \left(\overline{A(x)}A(x')-\left|A\left(\frac{x+x'}{2}\right)\right|^2\right) \widehat{|H|^2} \left(\frac{x-x'}{\epsilon} \right) \\
& = (2\pi)^d \left|A\left(\frac{x+x'}{2}\right)\right|^2 \widehat{|H|^2} \left(\frac{x-x'}{\epsilon} \right)+O(\epsilon)
\end{align*}
where the implicit constant in $O$ is $\lesssim \| A\|_{W^{1,\infty}} \| \frac{\mathcal F^{-1}(|H|^2)(y)}{\langle y \rangle} \|_{L^{\infty}}$. 
Thus,
$$
Q_0^\epsilon(x,y)= (2\pi)^d \left|A(x)\right|^2 \widehat{|H|^2}(y) + O(\epsilon).
$$
Finally, 
\bee
W_0^\epsilon (x,v) = |A(x)|^2|H(v)|^2+O( \epsilon)
\eee
where the implicit constant in $O$ is $\lesssim \| A \|_{W^{1,\infty}} \| y \mathcal F^{-1}(|H|^2)\|_{L^1}$.

\section{Proof of Theorem \ref{main theorem}}
Using the results obtained in the rest of the paper, we are able to prove our main result namely Theorem \ref{main theorem}.

\bigskip

\noindent
\textbf{Proof of the first part of Theorem \ref{main theorem}}
Recall equation \eqref{error equation} for $u^{err}$. For the existence part, we aim to apply Banach's fixed point theorem in $B_{X_\epsilon^{s,b}}(0,\rho)$, where $s>\frac{d}{2}-1$ and $\rho>0$ to be fixed to the mapping
\begin{align*}
\Phi:u\to\chi\left(t\right)\int_0^t  e^{i(t-s)\omega(D)}\lambda \mathfrak{L}_N(u) \,ds & + \chi\left( t \right)\int_0^t \chi\left( s \right) e^{i(t-s)\omega(D)}\lambda \mathfrak{B}(u) \,ds \\
& + \chi\left(t \right)\int_0^t  e^{i(t-s) \omega(D)}\left(\chi\left( \frac{s}{T} \right) \lambda E_N \right)\,ds
\end{align*}
(the precise choice of cutoff functions of the form $\chi(t)$ or $\chi \left(\frac{t}{T} \right)$ is merely technical, and has to do with the exact definition of the Bourgain space over which the contraction argument applies).

By propositions \ref{propexpansion} and \ref{error control proposition}, for any large $L>0$, the error term can be made smaller than $\epsilon^L$  in $X_\epsilon^{s,b}$, after excluding a set of size $< \frac{1}{2}\epsilon^{\kappa}$, by choosing $N$ sufficiently large. This leads to choosing $\rho=2\epsilon^L$. Moreover, by Proposition \ref{linearization proposition} the linear operator $\mathfrak{L}$ has an operator norm less than one, if one excludes a set of size $< \frac{1}{2}\epsilon^{-\kappa}$ and chooses $b$ sufficiently close to $1/2$. By Proposition \ref{bilinear proposition}, the bilinear term $\mathfrak{B}$ acts as a contraction on $B_{X_\epsilon^{s,b}}(0,\rho)$. Therefore, the contraction mapping principle gives a fixed point $u^{err}$ of $\Phi$,  which satisfies the bound $\|u^{err}\|_{X_\epsilon^{s,b}}\lesssim\epsilon^L$.

\bigskip

\noindent
\textbf{Proof of  the second part of Theorem \ref{main theorem}} Let $E$ the exceptional set obtained in the first part of the proof.
Forgetting for a moment about the set $E$, by Proposition \ref{comparing prop} it suffices to control 

\begin{align*}
&(2\pi)^{-d/2}\epsilon^{-d}\int_{\mathbb{R}^d} (h.o.t.)\,dv\\
&=\int_{\mathbb{R}^d}\left[\sum_{\substack{i+j\geq 4\\i,j\leq N}}\mathbb{E}\left[\overline{\widehat{u}^i(\xi^{-})}\widehat{u}^j(\xi^+)\right]+\sum_{i=0}^N \mathbb{E}\left[\overline{\widehat{u}^i(\xi^-)}\widehat{u}^{err}(\xi^+)+\overline{\widehat{u}^{err}(\xi^-)}\widehat{u}^{i}(\xi^+)\right]+\mathbb{E}\left[\overline{\widehat{u}^{err}(\xi^-)}\widehat{u}^{err}(\xi^+)\right]\right]\,dv,
\end{align*}
uniformly in time, where we use the notation $\xi^-=\frac{v}{\epsilon}-\frac{\xi}{2}$, $\xi^+=\frac{v}{\epsilon}+\frac{\xi}{2}$. By the Cauchy-Schwarz inequality, we obtain
\begin{align*}
h.o.t.&\lesssim\mathbb{E}\left[\mathds{1}_E\left(\sum_{\substack{i+j\geq 4\\i,j\leq N}}\|u^i(t)\|_{L^2}\|u^j(t)\|_{L^2}+\|u^{err}(t)\|_{L^2}\sum_{i=0}^N\|u^i(t)\|_{L^2}+\|u^{err}\|_{L^2}^2\right)\right]\\
&\lesssim\epsilon^{-\kappa}\left(\frac{t}{T_{kin}}\right)^2,
\end{align*}
after using estimate \eqref{bd:estimationunXsb} from Proposition \ref{propexpansion} and the bound for $u^{err}$ in $X^{s,b}_\epsilon$ (hence in $L^\infty_tL^2_x)$. This concludes the proof of the main theorem, except that we need to take into account the characteristic
function $\mathds{1}_E$ in the main term. But one can check that the main term
 enjoys better integrability properties:
this is achieved by raising it to a high power, and taking the expectation. Therefore,  using H\"older's inequality, the error
resulting from $\mathds{1}_E$ is at most $O(\epsilon^{c \kappa})$.

 \section{Comparison to the kinetic wave equation}\label{sec:comparison}
 
 The aim of this section is to provide a heuristic derivation of the kinetic wave equation, by comparing the first terms in the expansion of the kinetic equation on the one hand, and in the expansion of the correlation (Wigner transform) of the solution of the Hamiltonian problem on the other. Without loss of generality, we present the derivation for the case  $m(0)=0$ and $\omega_0=0$. This heuristic derivation will ultimately be justified by a control of the remainder in the expansions, which is the main achievement of the present article.
 
In order to slightly simplify notations, we will work under the  standing assumption that
$$
\epsilon^2 < t < \epsilon,
$$
which is the relevant time scale for the phenomena we want to observe.
 
 \subsection{Expanding the kinetic equation} 
We consider the kinetic equation
  \begin{equation}
  \begin{cases}
  \partial_t\rho+\frac{1}{\epsilon}v\cdot\nabla_x\rho=\displaystyle\frac{8\pi}{T_{kin}}\mathcal{C}[\rho]\\
  \rho(t=0)=W_0^\epsilon
  \end{cases},
  \quad T_{kin}=\frac{1}{\lambda^2\epsilon^2}
  \end{equation}
  where
  \begin{align*}
  \mathcal{C}[\rho](t,x,v)=&m^2\int \bigg[\delta(\Sigma_{-})\delta(\Omega_{-})m^2_1m_2^2\rho\rho_1\rho_2\left(\frac{1}{\rho}-\frac{1}{\rho_1}-\frac{1}{\rho_2}\right)\\
  & +2 \delta(\Sigma_{+})\delta(\Omega_{+})m_1^2m_2^2\rho\rho_1\rho_2\left(\frac{1}{\rho}+\frac{1}{\rho_1}-\frac{1}{\rho_2}\right) \bigg]\,dv_1\,dv_2,
  \end{align*}
  and
  \begin{equation*}
  \begin{cases}
  \Sigma_{-}=v-v_1-v_2\\
  \Sigma_{+}=v+v_1-v_2
  \end{cases}
  \quad
  \begin{cases}
  \Omega_{-}=|v|^2-|v_1|^2-|v_2|^2\\
  \Omega_{+}=|v|^2+|v_1|^2-|v_2|^2
  \end{cases}
  \quad
  \begin{cases}
  \rho_i=\rho(t,x,v_i)\\
  m_i=m(v_i)
  \end{cases}
  \end{equation*}
Define $r(t,x,v)=\rho(t,x+\frac{t}{\epsilon}v,v)$. Then $r$ satisfies 
\begin{align*}
  \partial_t r&=\frac{8\pi}{T_{kin}}m^2\int\delta(\Sigma_-)\delta(\Omega_-)m^2_1m^2_2\left(r_1r_2-rr_1-rr_2\right)\,dv_1\,dv_2\\
 & \hspace{2cm}+\frac{16\pi}{T_{kin}}m^2\int\delta(\Sigma_{+})\delta(\Omega_{+})m^2_1m^2_2\left(r_1r_2+rr_1-rr_2\right)\,dv_1\,dv_2
  \end{align*}
  where $r_i=r(t,x+\frac{t}{\epsilon}(v-v_i),v_i)$, $i\in\{0,1,2\}$. Taking the spatial Fourier transform, we have
  \begin{align*}
  \partial_t \widehat{r}&=\frac{8\pi}{T_{kin}}m^2(v)\int\delta(\xi-\eta_1-\eta_2)\delta(\Sigma_{-})\delta(\Omega_{-})m^2(v_1)m^2(v_2)\\
  &\bigg(e^{i\frac{t}{\epsilon}\alpha_0}\widehat{r}(t,\eta_1,v_1)\widehat{r}(t,\eta_2,v_2)-e^{i\frac{t}{\epsilon}\alpha_1}\widehat{r}(t,\eta_1,v)\widehat{r}(\eta_2,v_2)-e^{i\frac{t}{\epsilon}\alpha_2}\widehat{r}(t,\eta_1,v_1)\widehat{r}(\eta_2,v)\bigg)\,dv_{1,2}\,d\eta_{1,2}\\
  &+\frac{16\pi}{T_{kin}}m^2(v)\int\delta(\xi-\eta_1-\eta_2)\delta(\Sigma_{+})\delta(\Omega_{+})m^2(v_1)m^2(v_2)\\
  &\bigg(e^{i\frac{t}{\epsilon}\alpha_0}\widehat{r}(t,\eta_1,v_1)\widehat{r}(t,\eta_2,v_2)+e^{i\frac{t}{\epsilon}\alpha_1}\widehat{r}(t,\eta_1,v)\widehat{r}(\eta_2,v_2)-e^{i\frac{t}{\epsilon}\alpha_2}\widehat{r}(t,\eta_1,v_1)\widehat{r}(t,\eta_2,v)\bigg)\,dv_{1,2}\,d\eta_{1,2}
  \end{align*}
  where
  \begin{align*}
  \alpha_0&=v\cdot\xi-v_1\cdot\eta_1-v_2\cdot\eta_2\\
  \alpha_1&=v\cdot\xi-v\cdot\eta_1-v_2\cdot\eta_2\\
  \alpha_2&=v\cdot\xi-v_1\cdot\eta_1-v\cdot\eta_2
  \end{align*}
  Integrating the above, and using that $\widehat{r}(t,\xi,v)=e^{it\xi\cdot\frac{v}{\epsilon}}\widehat{\rho}(t,\xi,v)$, $r(t=0)=\rho(t=0)=W_0^\epsilon$, we obtain
  \begin{align}
 \widehat{\rho}(t,\xi,v) - e^{-it\xi\cdot\frac{v}{\epsilon}}  &\widehat{\rho}(t,\xi,v) \widehat{\rho_0}(\xi,v)=\frac{4(2\pi)^{1-\frac{d}{2}}e^{-it\xi\cdot\frac{v}{\epsilon}}}{T_{kin}}m^2(v)\int\delta(\xi-\eta_1-\eta_2)\delta(\Sigma_{-})\delta(\Omega_{-})m^2(v_1)m^2(v_2)\nonumber\\
  & \quad \bigg[\left(\int_0^te^{i\frac{\tau}{\epsilon}\alpha_0}\,d\tau\right)\widehat{W}_0^\epsilon(\eta_1,v_1)\widehat{W}_0^\epsilon(\eta_2,v_2)-\left(\int_0^te^{i\frac{\tau}{\epsilon}\alpha_1}\,d\tau\right)\widehat{W}_0^\epsilon(\eta_1,v)\widehat{W}_0^\epsilon(\eta_2,v_2)\nonumber\\
&\hspace{2cm}-\left(\int_0^te^{i\frac{\tau}{\epsilon}\alpha_2}\,d\tau\right)\widehat{W}_0^\epsilon(\eta_1,v_1)\widehat{W}_0^\epsilon(\eta_2,v)\bigg]\,dv_{1,2}\,d\eta_{1,2}\nonumber\\
  &+\frac{8(2\pi)^{1-\frac{d}{2}}e^{-it\xi\cdot\frac{v}{\epsilon}}}{T_{kin}}m^2(v)\int\delta(\xi-\eta_1-\eta_2)\delta(\Sigma_{+})\delta(\Omega_{+})m^2(v_1)m^2(v_2)\nonumber\\
  &\quad \bigg[\left(\int_0^te^{i\frac{\tau}{\epsilon}\alpha_0}\,d\tau\right)\widehat{W}_0^\epsilon(\eta_1,v_1)\widehat{W}_0^\epsilon(\eta_2,v_2)+\left(\int_0^te^{i\frac{\tau}{\epsilon}\alpha_1}\,d\tau\right)\widehat{W}_0^\epsilon(\eta_1,v)\widehat{W}_0^\epsilon(\eta_2,v_2)\nonumber\\
&\hspace{2cm}-\left(\int_0^te^{i\frac{\tau}{\epsilon}\alpha_2}\,d\tau\right)\widehat{W}_0^\epsilon(\eta_1,v_1)\widehat{W}_0^\epsilon(\eta_2,v)\bigg]\,dv_{1,2}\,d\eta_{1,2}\nonumber\\
& +O\left(\frac{t}{T_{kin}}\right)^2 \label{kinetic  expansion}
\end{align}  

 \subsection{Expanding the solution $u$}
We consider the quadratic dispersive equation \eqref{nonlinschrod} 
 $$\begin{cases}
 i\partial_tu+\frac{\Delta}{2} u=M\lambda(Mu+M\bar{u})^2,\\
 u(t=0)=u_0,
 \end{cases}
 $$
 with initial data $u_0$ is given by \eqref{hedgehog}.
 We write the solution as 
\begin{equation}\label{expansion of the solution}
 u=u^{app}+u^{err}=\sum_{n=0}^N u^n+u^{err},
 \end{equation} where 
 $$
u^0 = e^{-it \frac{\Delta}{2}} u_0,
$$
\begin{equation*}
\left\{ 
\begin{array}{l}
i \partial_t u^n + \frac{\Delta}{2} u^n = \lambda \sum_{j+k = n-1} M(Mu^j + M\overline{u^j}) (Mu^k + M\overline{u^k}) \\
u^n(t=0) = 0.
\end{array},\quad n\geq 1
\right.
\end{equation*}
 Throughout this section we will focus on the first three iterates. Taking the Fourier transform, and using the identity $\widehat{\overline{v}}(\xi)=\overline{\widehat{v}(-\xi)}$ we obtain expressions with respect to the initial data for
 the linear term
 \begin{equation}\label{linear solution}
 \widehat{u}^0(\xi)=e^{-it\frac{|\xi|^2}{2}}\widehat{u}_0(\xi),
 \end{equation}
 the bilinear term
\begin{align}
\widehat{u}^1(\xi)=\frac{-i\lambda m(\epsilon\xi)}{(2\pi)^{d/2}}e^{-it\frac{|\xi|^2}{2}}\int_0^t\int_{\xi=\xi_1+\xi_2}m(\epsilon\xi_1)m(\epsilon\xi_2)\bigg(e^{is_1\Omega_{0,-1,-2}}\widehat{u}_0(\xi_1)\widehat{u}_0(\xi_2)+e^{is_1\Omega_{0,1,2}}\overline{\widehat{u}_0(-\xi_1)}\overline{\widehat{u}_0(-\xi_2)}&\nonumber\\
+2e^{is_1\Omega_{0,1,-2}}\overline{\widehat{u}_0(-\xi_1)}\widehat{u}_0(\xi_2)\bigg)\,d\xi_{1,2}\,ds_1&\label{first order approximation},
\end{align}
and, finally, the trilinear term
\begin{align}
\widehat{u}^2(\xi)&=-\frac{2\lambda^2m(\epsilon\xi)}{(2\pi)^d}e^{-it\frac{|\xi|^2}{2}}\int_0^t\int_0^{s_1}\int_{\substack{\xi=\xi_1+\xi_2\\ \xi_2=\xi_1'+\xi_2'}}m(\epsilon\xi_1)m^2(\epsilon\xi_2)m(\epsilon\xi_1')m(\epsilon\xi_2') e^{is_1\Omega_{0,-1,-2}}\widehat{u}_0(\xi_1)\nonumber\\
&\bigg(e^{is_0\Omega_{2,-1',-2'}}\widehat{u}_0(\xi_1')\widehat{u}_0(\xi_2')+e^{is_0\Omega_{2,1',2'}}\overline{\widehat{u}_0(-\xi_1')}\overline{\widehat{u}_0(-\xi_2')}+2e^{is_0\Omega_{2,1',-2'}}\overline{\widehat{u}_0(-\xi_1')}\widehat{u}_0(\xi_2')\bigg)\,d\xi_{1,2}'\,d\xi_1\,\,ds_{0,1}\nonumber\\
&-\frac{2\lambda^2m(\epsilon\xi)}{(2\pi)^d}e^{-it\frac{|\xi|^2}{2}}\int_0^t\int_0^{s_1}\int_{\substack{\xi=\xi_1+\xi_2\\ \xi_2=\xi_1'+\xi_2'}}m(\epsilon\xi_1)m^2(\epsilon\xi_2)m(\epsilon\xi_1')m(\epsilon\xi_2')e^{is_1\Omega_{0,1,-2}}\overline{\widehat{u}_0(-\xi_1)}\nonumber\\
&\bigg(e^{is_0\Omega_{2,-1',-2'}}\widehat{u}_0(\xi_1')\widehat{u}_0(\xi_2')+e^{is_0\Omega_{2,1',2'}}\overline{\widehat{u}_0(-\xi_1')}\overline{\widehat{u}_0(-\xi_2')}+2e^{is_0\Omega_{2,1',-2'}}\overline{\widehat{u}_0(-\xi_1')}\widehat{u}_0(\xi_2')\bigg))\,d\xi_{1,2}'\,d\xi_1\,\,ds_{0,1}\nonumber\\
&+\frac{2\lambda^2m(\epsilon\xi)}{(2\pi)^d}e^{-it\frac{|\xi|^2}{2}}\int_0^t\int_0^{s_1}\int_{\substack{\xi=\xi_1+\xi_2\\ \xi_2=\xi_1'+\xi_2'}}m(\epsilon\xi_1)m^2(\epsilon\xi_2)m(\epsilon\xi_1')m(\epsilon\xi_2')e^{is_1\Omega_{0,1,2}}\overline{\widehat{u}_0(-\xi_1)}\nonumber\\
&\bigg(e^{is_0\Omega_{-2,-1',-2'}}\widehat{u}_0(\xi_1')\widehat{u}_0(\xi_2') +e^{is_0\Omega_{-2,1',2'}}\overline{\widehat{u}_0(-\xi_1')}\overline{\widehat{u}_0(-\xi_2')}+2e^{is_0\Omega_{-2,1',-2'}}\overline{\widehat{u}_0(-\xi_1')}\widehat{u}_0(\xi_2')\bigg))\,d\xi_{1,2}'\,d\xi_1\,\,ds_{0,1}\nonumber\\
&+\frac{2\lambda^2m(\epsilon\xi)}{(2\pi)^d}e^{-it\frac{|\xi|^2}{2}}\int_0^t\int_0^{s_1}\int_{\substack{\xi=\xi_1+\xi_2\\ \xi_2=\xi_1'+\xi_2'}}m(\epsilon\xi_1)m^2(\epsilon\xi_2)m(\epsilon\xi_1')m(\epsilon\xi_2')e^{is_1\Omega_{0,-1,2}}\widehat{u}_0(\xi_1)\nonumber\\
&\bigg(e^{is_0\Omega_{-2,-1',-2'}}\widehat{u}_0(\xi_1')\widehat{u}_0(\xi_2')+e^{is_0\Omega_{-2,1',2'}}\overline{\widehat{u}_0(-\xi_1')}\overline{\widehat{u}_0(-\xi_2')}+2e^{is_0\Omega_{-2,1',-2'}}\overline{\widehat{u}_0(-\xi_1')}\widehat{u}_0(\xi_2')\bigg))\,d\xi_{1,2}'\,d\xi_1\,\,ds_{0,1}.\label{second order approximation}
\end{align}
Above, we denote $\Omega_{0,-1,-2}=\frac{1}{2}(|\xi|^2-|\xi_1|^2-|\xi_2|^2)$, $\Omega_{0,1,-2}=\frac{1}{2}(|\xi|^2+|\xi_1|^2-|\xi_2|^2)$, etc...

 \subsection{Expanding the correlations}
Recall that the rescaled Wigner transform is given by
 \begin{equation}\label{Wigner formula 1}
 \widehat{W}^\epsilon[u](\xi,v)=(2\pi)^{-d/2}\epsilon^{-d}\mathbb{E}\left[\overline{\widehat{u}\left(\xi^-\right)}\widehat{u}\left(\xi^+ \right)\right],
 \end{equation}
 with
 $$
 \begin{cases}
 \xi^-=\frac{v}{\epsilon}-\frac{\xi}{2} \\ \xi^+=\frac{v}{\epsilon} + \frac{\xi}{2}
 \end{cases}
 \qquad \mbox{or} \qquad
 \begin{cases}
 \xi=\xi^+-\xi^- \\ v=\frac{\epsilon}{2}(\xi^++\xi^-).
 \end{cases}
 $$
 We will also make use of the formula
 \begin{equation}\label{Wigner formula 2}
 \mathbb{E}\left[ \overline{\widehat{u}(\xi_1)}\widehat{u}(\xi_2)\right]=(2\pi)^{d/2}\epsilon^d \widehat{W}^\epsilon\left[\xi_2-\xi_1,\frac{\epsilon}{2}\left(\xi_1+\xi_2\right)\right].
 \end{equation}
Inserting the expansion~\eqref{expansion of the solution} in the definition of $\widehat{W}^\epsilon[u]$,
 \begin{align}
 (2\pi)^{d/2}\epsilon^d\widehat{W}^\epsilon[u](\xi,v)&=\mathbb{E}\left[ \overline{\widehat{u}^0(\xi^-)}\widehat{u}^0(\xi^+)\right]\label{term 00}\\
 &\hspace{1cm}+\mathbb{E}\left[ \overline{\widehat{u}^1(\xi^-)}\widehat{u}^0(\xi^+)\right]+\mathbb{E}\left[ \overline{\widehat{u}^0(\xi^-)}\widehat{u}^1(\xi^+)\right]\label{term 01}\\
 &\hspace{1cm}+\mathbb{E}\left[ \overline{\widehat{u}^1(\xi^-)}\widehat{u}^1(\xi^+)\right]\label{term  11}\\
 &\hspace{1cm}+\mathbb{E}\left[ \overline{\widehat{u}^2(\xi^-)}\widehat{u}^0(\xi^+)\right]+\mathbb{E}\left[ \overline{\widehat{u}^0(\xi^-)}\widehat{u}^2(\xi^+)\right]\label{term 02}\\
 &\hspace{1cm}+ h.o.t.
 \end{align}
 where
 \begin{align}
  h.o.t&= \sum_{i+j\geq 4}\mathbb{E}\left[\overline{\widehat{u}^i(\xi^{-})}\widehat{u}^j(\xi^+)\right]+\sum_{i=0}^N \mathbb{E}\left[\overline{\widehat{u}^i(\xi^-)}\widehat{u}^{err}(\xi^+)+\overline{\widehat{u}^{err}(\xi^-)}\widehat{u}^{i}(\xi^+)\right]+\mathbb{E}\left[\overline{\widehat{u}^{err}(\xi^-)}\widehat{u}^{err}(\xi^+)\right]\label{hot},
 \end{align}
and to obtain \eqref{hot} we used the fact that there is cancellation for $i+j=3$ due to Wick's formula.
 \noindent
 \textbf{The linear-linear term \eqref{term 00}.} By \eqref{linear solution}, we have
 \begin{equation}\label{00 EV}
 \mathbb{E}\left[ \overline{\widehat{u}^0(\xi^-)}\widehat{u}^0(\xi^+)\right]=e^{-i\frac{t}{2}(|\xi^+|^2-|\xi^-|^2)}\mathbb{E}\left[ \overline{\widehat{u}_0(\xi^-)}\widehat{u}_0(\xi^+)\right]=e^{-it\xi\cdot\frac{v}{\epsilon}} (2\pi)^{d/2}\epsilon^d \widehat{W}_0^\epsilon(\xi,v)
 \end{equation}
 
 \bigskip
 
 \noindent
 \textbf{The linear-bilinear term \eqref{term 01}}. It vanishes by Wick's formula.
  
\bigskip
 
\noindent 
\textbf{The bilinear-bilinear term \eqref{term  11}}. It will be convenient to write $u^1$ under the form
\begin{align*} 
  \widehat{u}^1(t,\xi^+)&=-\frac{i\lambda m(\epsilon\xi^+)}{(2\pi)^{d/2}}e^{-it\frac{|\xi^+|^2}{2}}\int_0^t\int_{\xi^+=\xi_1'+\xi_2'}m(\epsilon\xi_1')m(\epsilon\xi_2')\\
  &\bigg(e^{is_1'\Omega_{+,-1',-2'}}\widehat{u}_0(\xi_1')\widehat{u}_0(\xi_2')+e^{is_1'\Omega_{+,1',2'}}\overline{\widehat{u}_0(-\xi_1')}\text{ }\overline{\widehat{u}_0(-\xi_2')}+2e^{is_1'\Omega_{+,1',-2'}}\overline{\widehat{u}_0(-\xi_1')}\widehat{u}_0(\xi_2')\bigg)\,d\xi_{1,2}'\,ds_1',
  \end{align*}
  where $\Omega_{+,-1',-2'}=\frac{1}{2}(|\xi^+|^2-|\xi_1'|^2-|\xi_2'|^2)$, etc... and
  \begin{align*} 
 \overline{ \widehat{u}^1(t,\xi^-)}&=\frac{i\lambda m(\epsilon\xi^-)}{(2\pi)^{d/2}}e^{it\frac{|\xi^-|^2}{2}}\int_0^t\int_{\xi^-=\xi_1+\xi_2}m(\epsilon\xi_1)m(\epsilon\xi_2)\bigg(e^{-is_1\Omega_{-,-1,-2}}\overline{\widehat{u}_0(\xi_1)}\overline{\widehat{u}_0(\xi_2)}+e^{-is_1\Omega_{-,1,2}}\widehat{u}_0(-\xi_1)\widehat{u}_0(-\xi_2)\\
  &\hspace{2cm}+2e^{-is_1\Omega_{-,1,-2}}\widehat{u}_0(-\xi_1)\overline{\widehat{u}_0(\xi_2)}\bigg)\,d\xi_{1,2}\,ds_1.
\end{align*}  
Using these formulas, we obtain
\begin{align}
  &(2\pi)^{d}\lambda^{-2}e^{it\xi\cdot\frac{v}{\epsilon}}\mathbb{E}\left[\overline{\widehat{u}^1(t,\xi^-)}\widehat{u}^1(t,\xi^+)\right]\nonumber\\
  &=m(\epsilon\xi^+)m(\epsilon\xi^-)\int_{\substack{\xi^+=\xi_1'+\xi_2'\\\xi^-=\xi_1+\xi_2}}m(\epsilon\xi_1)m(\epsilon\xi_2)m(\epsilon\xi_1')m(\epsilon\xi_2')\int_0^t\int_0^t e^{i\left(s_1'\Omega_{+,-1',-2'}-s_1\Omega_{-,-1,-2}\right)}\,ds_1\,ds_1'\nonumber\\
  &\hspace{6cm}\mathbb{E}\left[\overline{\widehat{u}_0(\xi_1)}\overline{\widehat{u}_0(\xi_2)}\widehat{u}_0(\xi_1')\widehat{u}_0(\xi_2')\right]\,d\xi_{1,2}\,d\xi'_{1,2}\nonumber\\
  &+m(\epsilon\xi^+)m(\epsilon\xi^-)\int_{\substack{\xi^+=\xi_1'+\xi_2'\\\xi^-=\xi_1+\xi_2}}m(\epsilon\xi_1)m(\epsilon\xi_2)m(\epsilon\xi_1')m(\epsilon\xi_2')\int_0^t\int_{0}^t e^{i\left(s_1'\Omega_{+,1',2'}-s_1\Omega_{-,1,2}\right)}\,ds_1 \,ds_1'\nonumber\\
  &\hspace{6cm}\mathbb{E}\left[\widehat{u}_0(-\xi_1)\widehat{u}_0(-\xi_2)\overline{\widehat{u}_0(-\xi_1')}\overline{\widehat{u}_0(-\xi_2')}\right]\,d\xi_{1,2}\,d\xi'_{1,2}\nonumber\\
  &+4m(\epsilon\xi^+)m(\epsilon\xi^-)\int_{\substack{\xi^+=\xi_1'+\xi_2'\\\xi^-=\xi_1+\xi_2}}m(\epsilon\xi_1)m(\epsilon\xi_2)m(\epsilon\xi_1')m(\epsilon\xi_2')\int_0^t\int_0^t e^{i\left(s_1'\Omega_{+,1',-2'}-s_1\Omega_{-,1,-2}\right)}\,ds_1\,ds_1'\nonumber\\
  &\hspace{6cm}\mathbb{E}\left[\widehat{u}_0(-\xi_1)\overline{\widehat{u}_0(\xi_2)}\overline{\widehat{u}_0(-\xi_1')}\widehat{u}_0(\xi_2')\right]\,d\xi_{1,2}\,d\xi'_{1,2}.\nonumber
\end{align}
By Wick's formula, this is
\begin{align}
&=2(2\pi)^{d}\epsilon^{2d}m(\epsilon\xi^+)m(\epsilon\xi^-)\int_{\substack{\xi^+=\xi_1'+\xi_2'\\\xi^-=\xi_1+\xi_2}}m(\epsilon\xi_1)m(\epsilon\xi_2)m(\epsilon\xi_1')m(\epsilon\xi_2')\int_0^t\int_0^t e^{i\left(s_1'\Omega_{+,-1',-2'}-s_1\Omega_{-,-1,-2}\right)}\,ds_1\,ds_1'\nonumber\\
&\hspace{4cm}\widehat{W}_0^\epsilon\left(\xi_1'-\xi_1,\frac{\epsilon}{2}\left(\xi_1'+\xi_1\right)\right)\widehat{W}_0^\epsilon\left(\xi_2'-\xi_2,\frac{\epsilon}{2}\left(\xi_2'+\xi_2\right)\right)\,d\xi_{1,2}\,d\xi'_{1,2}\label{11-1}\\
  &+2(2\pi)^{d}\epsilon^{2d}m(\epsilon\xi^+)m(\epsilon\xi^-)\int_{\substack{\xi^+=\xi_1'+\xi_2'\\\xi^-=\xi_1+\xi_2}}m(\epsilon\xi_1)m(\epsilon\xi_2)m(\epsilon\xi_1')m(\epsilon\xi_2')\int_0^t\int_0^t e^{i\left(s_1'\Omega_{+,1',2'}-s_1\Omega_{-,1,2}\right)}\,ds_1\,ds_1'\nonumber\\
  &\hspace{4cm}\widehat{W}_0^\epsilon\left(\xi_1'-\xi_1,\frac{\epsilon}{2}\left(-\xi_1-\xi_1'\right)\right)\widehat{W}_0^\epsilon\left(\xi_2'-\xi_2,\frac{\epsilon}{2}\left(-\xi_2-\xi_2'\right)\right)\,d\xi_{1,2}\,d\xi'_{1,2}\label{11-2}\\
  &+4(2\pi)^{d}\epsilon^{2d}m(\epsilon\xi^+)m(\epsilon\xi^-)\int_{\substack{\xi^+=\xi_1'+\xi_2'\\\xi^-=\xi_1+\xi_2}}m(\epsilon\xi_1)m(\epsilon\xi_2)m(\epsilon\xi_1')m(\epsilon\xi_2')\int_0^t\int_0^t e^{i\left(s_1'\Omega_{+',1',-2'}-s_1\Omega_{-,1,-2}\right)}\,ds_1\,ds_1'\nonumber\\
  &\hspace{4cm}\widehat{W}_0^\epsilon\left(\xi_1'-\xi_1,\frac{\epsilon}{2}\left(-\xi_1-\xi_1'\right)\right)\widehat{W}_0^\epsilon\left(\xi_2'-\xi_2,\frac{\epsilon}{2}\left(\xi_2'+\xi_2\right)\right)\,d\xi_{1,2}\,d\xi'_{1,2}  \label{11-3}\\
   &+4(2\pi)^{d}\epsilon^{2d}m(\epsilon\xi^+)m(\epsilon\xi^-)\int_{\substack{\xi^+=\xi_1'+\xi_2'\\\xi^-=\xi_1+\xi_2}}m(\epsilon\xi_1)m(\epsilon\xi_2)m(\epsilon\xi_1')m(\epsilon\xi_2')\int_0^t\int_0^t e^{i\left(s_1'\Omega_{+',1',-2'}-s_1\Omega_{-,1,-2}\right)}\,ds_1\,ds_1'\nonumber\\
   &\hspace{4cm}\widehat{W}_0^\epsilon\left(-\xi_1-\xi_2,\frac{\epsilon}{2}\left(\xi_2-\xi_1\right)\right)\widehat{W}_0^\epsilon\left(\xi_1'+\xi_2',\frac{\epsilon}{2}\left(\xi_2'-\xi_1'\right)\right)\,d\xi_{1,2}\,d\xi'_{1,2}  \label{11-4}
  \end{align}

\noindent
\underline{Term \eqref{11-1}} We perform the change of variables 
  $$\begin{cases}
  \eta_1=\xi_1'-\xi_1\\
  v_1=\frac{\epsilon}{2}(\xi_1'+\xi_1)
  \end{cases}\quad
  \begin{cases}
  \eta_2=\xi_2'-\xi_2\\
  v_2=\frac{\epsilon}{2}(\xi_2'+\xi_2),
  \end{cases}$$
  which is of Jacobian $\epsilon^{d}$, when restricted to the domain on integration. By our choice of data, $|\eta_i|,|v_i|=O(1)$ for $i\in\{1,2\}$. Moreover,  $$\xi=\xi^+-\xi^-=\xi_1'+\xi_2'-\xi_1-\xi_2=\eta_1+\eta_2=O(1)$$ 
  and 
 $$v=\frac{\epsilon}{2}(\xi^++\xi^-)=\frac{\epsilon}{2}(\xi_1'+\xi_2'+\xi_1+\xi_2)=v_1+v_2=O(1)$$ 
This change of variables leads to the expression 
  \begin{align*}\eqref{11-1}&=2(2\pi)^d\epsilon^dm\left(v+\frac{\epsilon}{2}\xi\right)m\left(v-\frac{\epsilon}{2}\xi\right)\int_{\substack{\xi=\eta_1+\eta_2\\v=v_1+v_2}}\prod_{i\in\{1,2\}}m\left(v_i\pm\frac{\epsilon}{2}\eta_i\right)\\
  &\int_0^t\int_0^t e^{i\left(s_1'\Omega_{+,-1',-2'}-s_1\Omega_{-,-1,-2}\right)}\,ds_1\,ds_1'\widehat{W}_0^\epsilon\left(\eta_1,v_1\right)\widehat{W}_0^\epsilon\left(\eta_2,v_2\right)\,d\eta_{1,2}\,dv_{1,2}\\
  &=2(2\pi)^d\epsilon^dm^2(v)\int_{\substack{\xi=\eta_1+\eta_2\\v=v_1+v_2}}m^2(v_1)m^2(v_2)\\
  &\int_0^t\int_0^t e^{i\left(s_1'\Omega_{+,-1',-2'}-s_1\Omega_{-,-1,-2}\right)}\,ds_1\,ds_1'\widehat{W}_0^\epsilon\left(\eta_1,v_1\right)\widehat{W}_0^\epsilon\left(\eta_2,v_2\right)\,d\eta_{1,2}\,dv_{1,2}+ O(t^2\epsilon^{d+2}),
  \end{align*}
where the resonance moduli expressed in the new variables are
  \begin{align*}
  \Omega_{+,-1',-2'}
  &=\frac{1}{2\epsilon^2}(|v|^2-|v_1|^2-|v_2|^2)+\frac{1}{2\epsilon}(v\cdot\xi-v_1\cdot\eta_1-v_2\cdot\eta_2)+\frac{1}{8}(|\xi|^2-|\eta_1|^2-|\eta_2|^2),\\
  \end{align*}
  and
  \begin{align*}
  \Omega_{-,-1,-2}
  &=\frac{1}{2\epsilon^2}(|v|^2-|v_1|^2-|v_2|^2)-\frac{1}{2\epsilon}(v\cdot\xi-v_1\cdot\eta_1-v_2\cdot\eta_2)+\frac{1}{8}(|\xi|^2-|\eta_1|^2-|\eta_2|^2).\\
  \end{align*}
  We have
  \begin{equation*}
  \begin{cases}
  \Omega_{+,-1',-2'}+\Omega_{-,-1,-2}= \frac{\Omega_{0,-1,-2}}{\epsilon^2}+\gamma_0\\
  \Omega_{+,-1',-2'}-\Omega_{-,-1,-2}= \frac{\alpha_0}{\epsilon}
  \end{cases}
  \quad
  \begin{cases}
  \Omega_{0,-1,-2}=|v|^2-|v_1|^2-|v_2|^2\\
  \alpha_0=v\cdot\xi-v_1\cdot\eta_1-\eta_2\cdot\eta_2\\
  \gamma_0=O(1)
  \end{cases}
  \end{equation*}
Changing variables $\tau=\frac{s_1+s_1'}{2}$, $\sigma=\frac{s_1-s_1'}{2}$, we have
  \begin{align*}
 \int_0^t\int_0^t e^{i(s_1'\Omega_{+,-1',-2'}-s_1\Omega_{-,-1,-2})}\,ds_1\,ds_1'
 &=2\int_0^t e^{i\tau\frac{\alpha_0}{\epsilon}}\int_{-\theta}^\theta e^{i\sigma(\frac{\Omega_{0,-1,-2}}{\epsilon^2}+\gamma_0)}\,d\sigma\,d\tau,\quad\theta=\min\{\tau,t-\tau\}\\
 &=4\int_0^t e^{i\tau\frac{\alpha_0}{\epsilon}}\frac{\sin(\theta\left(\frac{\Omega_{0,-1,-2}}{\epsilon^2}+\gamma_0\right))}{\frac{\Omega_{0,-1,-2}}{\epsilon^2}+\gamma_0}\,d\tau\\
  \end{align*}
  
We will now rely on the 
\begin{lemma}[Dirichlet kernel] \label{lemmaDirichlet}
Let $f \in \mathcal{C}^\infty_0$ be such that $\left\| \partial_x^k f \right\|_{\infty} \lesssim 1$ for any $k \in \mathbb{N}$. Then, for any $M \in \mathbb{N}$,
$$
\int \frac{\sin(\lambda x)}{x} f(x) \,dx = \pi f(0) + O \left(\lambda^{-M} \right)
$$
\end{lemma}

\begin{proof} For a cutoff function $\chi$, decompose
\begin{align*}
\int \frac{\sin(\lambda x)}{x} f(x) \,dx & =  \int \frac{\sin(\lambda x)}{x} f(x) \left[1 - \chi(\sqrt{\lambda} x)\right] \,dx + \int \frac{\sin(\lambda x)}{x} f(0) \chi(\sqrt{\lambda} x)\,dx \\
& + \int  \frac{\sin(\lambda x)}{x} \sum_{n=1}^N \frac{f(n)(0)}{n!} x^n\chi(\sqrt{\lambda} x)\,dx + \int \frac{\sin(\lambda x)}{x} \left[ f(x) - \sum_{n=0}^N \frac{f(n)(0)}{n!} x^n \right] \chi(\sqrt{\lambda} x)\,dx \\
& = I + II + III + IV.
\end{align*}
Using integration by parts, one sees that $I$ and $III$ decay faster than any power of $\lambda$. A direct estimate gives $|IV| \lesssim \lambda^{-N/2}$. Finally, the leading contribution is given by $II$, and the constant is provided by the identity (Dirichlet integral) $\int \frac{\sin x}{x}\,dx = \pi$.
\end{proof}
  
This lemma can be expressed as the formula 
\begin{equation}
\label{chardonneret}
\frac{\sin(\lambda x)}{x} = \pi \delta + O((1+\lambda)^{-N})
\end{equation}
(which is understood by duality with a smooth, rapidly decaying function, whose derivatives are pointwise $O(1)$). Coming back to the expression involving resonance moduli, and denoting $Z =   \Omega_{0,-1,-2}+\epsilon^2\gamma_0$, we obtain
\begin{align*}
& \int_0^t\int_0^t e^{i\left(s_1'\Omega_{+,-1',-2'}-s_1\Omega_{-,-1,-2}\right)}\,ds_1\,ds_1' = 4 \epsilon^2 \int_0^t e^{i\tau\frac{\alpha_0}{\epsilon}}\frac{\sin(\frac{\theta}{\epsilon^2} Z)}{Z}\,d\tau \\
& \qquad \qquad \qquad \qquad = 4 \pi \epsilon^2 \delta(Z)  \int_0^t e^{i\tau\frac{\alpha_0}{\epsilon}} \,d\tau + O \left( \epsilon^2 \int_0^t (1 + \frac{\theta}{\epsilon^2} )^{-N} \,d\tau \right)
\end{align*}
(notice that, since $t<\epsilon$, the function $e^{i\tau\frac{\alpha_0}{\epsilon}}$ has all its derivatives $\lesssim 1$, which makes the application of~\eqref{chardonneret} legitimate - this is not true close to a critical point of $\Omega_{+,-1',-2'}$, but we shall gloss over this technical point).

Since $Z = \Omega_{0,-1,-2} + O(\epsilon)$, the above is
$$
\dots = 4 \pi \epsilon^2 \delta(\Omega_{0,-1,-2})  \int_0^t e^{i\tau\frac{\alpha_0}{\epsilon}} \,d\tau + O(\epsilon^4),
$$
which finally leads to
\begin{align}
  \eqref{11-1}&=4(2\pi)^{d+1}\epsilon^{d+2}m^2(v)\int\delta(\xi-\eta_1-\eta_2)\delta(\Sigma_{0,-1,-2})\delta(\Omega_{0,-1,-2})m^2(v_1)m^2(v_2)\nonumber\\
  & \left(\int_0^t e^{i\tau\frac{\alpha_0}{\epsilon}}\,d\tau\right)\widehat{W}_0^\epsilon\left(\eta_1,v_1\right)\widehat{W}_0^\epsilon\left(\eta_2,v_2\right)\,d\eta_{1,2}\,dv_{1,2}+O(t^2\epsilon^{d+2}+\epsilon^{d+4}),\label{11 EV 1}
  \end{align}
  where $\alpha_0=v\cdot\xi-v_1\cdot\eta_1-v_2\cdot\eta_2$, $\Sigma_{0,-1,-2}=v-v_1-v_2$ and $\Omega_{0,-1,-2}=|v|^2-|v_1|^2-|v_2|^2$
 
 \bigskip

\noindent  \underline{Term \eqref{11-2}} We perform the change of variables
  $$\begin{cases}
  \eta_1=\xi_1'-\xi_1\\
  v_1=\frac{\epsilon}{2}(-\xi_1-\xi_1')
  \end{cases}\quad
  \begin{cases}
  \eta_2=\xi_2'-\xi_2\\
  v_2=\frac{\epsilon}{2}(-\xi_2-\xi_2'),
  \end{cases}$$
which gives the relations
\begin{align*}
& \xi=\xi^+-\xi^-=\xi_1'+\xi_2'-\xi_1-\xi_2=\eta_1+\eta_2 \\
& v=\frac{\epsilon}{2}(\xi^++\xi^-)=\frac{\epsilon}{2}(\xi_1'+\xi_2'+\xi_1+\xi_2)=-v_1-v_2.
\end{align*}
With these new integration variables,
  \begin{align*}
  \eqref{11-2}&=2(2\pi)^d\epsilon^dm\left(v+\frac{\epsilon}{2}\xi\right)m\left(v-\frac{\epsilon}{2}\xi\right)\int_{\substack{\xi=\eta_1+\eta_2\\v=-v_1-v_2}}\prod_{i\in\{1,2\}}m\left(-v_i\pm\frac{\epsilon}{2}\eta_i\right)\\
  &\int_0^t\int_0^t e^{i(s_1'\Omega_{+,1',2'}-s_1\Omega_{-,+1,+2})}\,ds_1\,ds_1'\widehat{W}_0^\epsilon\left(\eta_1,v_1\right)\widehat{W}_0^\epsilon\left(\eta_2,v_2\right)\,d\eta_{1,2}\,dv_{1,2}\\
  &=2(2\pi)^d\epsilon^dm^2(v)\int_{\substack{\xi=\eta_1+\eta_2\\v=-v_1-v_2}}m^2(v_1)m^2(v_2)\\
  &\int_0^t\int_0^t e^{i(s_1'\Omega_{+,1',2'}-s_1\Omega_{-,+1,+2})}\,ds_1\,ds_1'\widehat{W}_0^\epsilon\left(\eta_1,v_1\right)\widehat{W}_0^\epsilon\left(\eta_2,v_2\right)\,d\eta_{1,2}\,dv_{1,2}+O(t^2\epsilon^{d+2})
  \end{align*}
The resonance moduli  above are given by
 \begin{align*}
& \Omega_{+,1',2'} =\frac{1}{2\epsilon^2}(|v|^2+|v_1|^2+|v_2|^2)+\frac{1}{2\epsilon}(v\cdot\xi+v_1\cdot\eta_1+v_2\cdot\eta_2)+\frac{1}{8}(|\xi|^2+|\eta_1|^2+|\eta_2|^2)\\
& \Omega_{-,1,2} = \frac{1}{2\epsilon^2}(|v|^2+|v_1|^2+|v_2|^2)-\frac{1}{2\epsilon}(v\cdot\xi+v_1\cdot\eta_1+v_2\cdot\eta_2)+\frac{1}{8}(|\xi|^2+|\eta_1|^2+|\eta_2|^2)
  \end{align*}
By the same argument as for~\eqref{11-1}, this term will give no contribution besides $O(t^2\epsilon^{d+2}+\epsilon^{d+4})$, since it contains a factor $\delta(\Omega_{0,1,2})$, and $\Omega_{0,1,2}$ only vanishes at a point.

\bigskip
\noindent  \underline{Term \eqref{11-3}} We perform the change of variables
  $$\begin{cases}
  \eta_1=\xi_1'-\xi_1\\
  v_1=\frac{\epsilon}{2}(-\xi_1-\xi_1')
  \end{cases}\quad
  \begin{cases}
  \eta_2=\xi_2'-\xi_2\\
  v_2=\frac{\epsilon}{2}(\xi_2'+\xi_2),
  \end{cases}$$
 which gives the relations 
 \begin{align*}
& \xi=\xi^+-\xi^-=\xi_1'+\xi_2'-\xi_1-\xi_2=\eta_1+\eta_2 \\
 & v=\frac{\epsilon}{2}(\xi^++\xi^-)=\frac{\epsilon}{2}(\xi_1'+\xi_2'+\xi_1+\xi_2)=-v_1+v_2.
 \end{align*}
Then we can write
  \begin{align*}
  \eqref{11-3}&=4(2\pi)^d\epsilon^dm\left(v+\frac{\epsilon}{2}\xi\right)m\left(v-\frac{\epsilon}{2}\xi\right)\int_{\substack{\xi=\eta_1+\eta_2\\v=-v_1+v_2}}\prod_{i\in\{1,2\}}m\left(\left(-1\right)^i v_i\pm\frac{\epsilon}{2}\eta_i\right)\\
  &\int_0^t \int_0^t e^{i(s_1'\Omega_{+,1',-2'}-s_1\Omega_{-,1,-2})}\,ds_1\,ds_1'\widehat{W}_0^\epsilon\left(\eta_1,v_1\right)\widehat{W}_0^\epsilon\left(\eta_2,v_2\right)\,d\eta_{1,2}\,dv_{1,2}\\
  &=4(2\pi)^d\epsilon^dm^2(v)\int_{\substack{\xi=\eta_1+\eta_2\\v=-v_1+v_2}}m^2(v_1)m^2(v_2)\\
  &\int_0^t \int_0^t e^{i(s_1'\Omega_{+,1',-2'}-s_1\Omega_{-,1,-2})}\,ds_1\,ds_1'\widehat{W}_0^\epsilon\left(\eta_1,v_1\right)\widehat{W}_0^\epsilon\left(\eta_2,v_2\right)\,d\eta_{1,2}\,dv_{1,2}+O(t^2\epsilon^{d+2}).
  \end{align*}
Since
  \begin{equation*}
  \begin{cases}
  \xi_1'=-\frac{v_1}{\epsilon}+\frac{\eta_1}{2}\\
 \xi_2'=\frac{v_2}{\epsilon}+\frac{\eta_2}{2} 
  \end{cases}
  \quad
  \begin{cases}
  \xi_1=-\frac{\eta_1}{2}-\frac{v_1}{\epsilon}\\
  \xi_2=\frac{v_2}{\epsilon}-\frac{\eta_2}{2},
  \end{cases}
  \end{equation*}
  the corresponding resonance moduli  are
  \begin{align*}
&  \Omega_{+,1',-2'}=\frac{1}{2\epsilon^2}(|v|^2+|v_1|^2-|v_2|^2)+\frac{1}{2\epsilon}(v\cdot\xi-v_1\cdot\eta_1-v_2\cdot\eta_2)+\frac{1}{8}(|\xi|^2+|\eta_1|^2-|\eta_2|^2)\\
&  \Omega_{-,1,-2}=\frac{1}{2\epsilon^2}(|v|^2+|v_1|^2-|v_2|^2)-\frac{1}{2\epsilon}(v\cdot\xi-v_1\cdot\eta_1-v_2\cdot\eta_2)+\frac{1}{8}(|\xi|^2+|\eta_2|^2-|\eta_1|^2).
  \end{align*}
By a similar argument to that used for~\eqref{11-1}, we obtain
  \begin{align}
  \eqref{11-3}&=8(2\pi)^{d+1}\epsilon^{d+2}m^2(v)\int\delta(\xi-\eta_1-\eta_2)\delta(\Sigma_{0,1,-2})\delta(\Omega_{0,1,-2})m^2(v_1)m^2(v_2)\nonumber\\
  &\left(\int_0^t e^{i\tau\frac{\alpha_0}{\epsilon}}\,d\tau\right)\widehat{W}_0^\epsilon\left(\eta_1,v_1\right)\widehat{W}_0^\epsilon\left(\eta_2,v_2\right)\,d\eta_{1,2}\,dv_{1,2}+O(t^2\epsilon^{d+2}+ \epsilon^{d+4}).\label{11 EV 2}
  \end{align}
 where $\alpha_0=v\cdot\xi-v_1\cdot\eta_1-v_2\cdot\eta_2$, $\Sigma_{0,1,-2}=v+v_1-v_2$ and $\Omega_{0,1,-2}=|v|^2+|v_1|^2-|v_2|^2$.
 
 \bigskip 
 \noindent

  \underline{Term \eqref{11-4}} This term is degenerate and we cannot take advantage of any oscillations. However, as we will see, it will become negligible in the limit. It can be equivalently written as
  \begin{align*}\eqref{11-4}&=4(2\pi)^{d}\epsilon^{2d}m(\epsilon\xi^+)m(\epsilon\xi^-)\int\int m(\epsilon\xi_1)m\left(\epsilon\left(\xi^--\xi_1\right)\right)m(\epsilon\xi_1')m\left(\epsilon\left(\xi^+-\xi_1'\right)\right)\\
  &\int_0^t\int_0^te^{i\left(s_1'\Omega_{+,1',-2'}-s_1\Omega_{-,1,-2}\right)}\,ds_1\,ds_1'\widehat{W}_0^\epsilon\left(-\xi^-,\frac{\epsilon}{2}\left(\xi^--2\xi_1\right)\right)\widehat{W}_0^\epsilon\left(\xi^+,\frac{\epsilon}{2}\left(\xi^+-2\xi_1'\right)\right)\,d\xi_{1}\,d\xi'_{1}
  \end{align*}
 so performing the change of variables
 \begin{equation*}
 \begin{cases}
 v_1=\frac{\epsilon}{2}\left(\xi^--2\xi_1\right)\\
 v_2=\frac{\epsilon}{2}\left(\xi^+-2\xi_1'\right)
 \end{cases}
 \end{equation*}
 which is of Jacobian $\epsilon^{2d}$,  we take
 \begin{align*}
 \eqref{11-4}&=4(2\pi)^{d}m(\epsilon\xi^+)m(\epsilon\xi^-)\int\int m\left(\frac{\epsilon\xi^-}{2}-v_1\right)m\left(\frac{\epsilon\xi^-}{2}+v_1\right)m\left(\frac{\epsilon\xi^+}{2}-v_2\right)m\left(\frac{\epsilon\xi^+}{2}+v_2\right)\\
 &\hspace{2cm}\int_0^t\int_0^te^{i\left(s_1'\Omega_{+,1',-2'}-s_1\Omega_{-,1,-2}\right)}\,ds_1\,ds_1'
 \widehat{W}_0^\epsilon\left(-\xi^-,v_1\right)\widehat{W}_0^\epsilon\left(\xi^+,v_2\right)\,dv_1\,dv_2
 \end{align*}
 Notice that by \eqref{boundsderW}, the above expression is non zero only when $|\xi^-|,|\xi^+|=O(1)$ or equivalently when $|v|=O(\epsilon)$. Therefore, when \eqref{11-4} is integrated against $v$ it will produce a term of order $O\left(t^2\epsilon^{d+2}\right)$.

\bigskip 

\noindent
\textbf{The trilinear-linear term \eqref{term 02}} By definition of $\widehat{u}^2$,
\begin{align}
 &(2\pi)^d\lambda^{-2}e^{it\xi\cdot\frac{v}{\epsilon}}\mathbb{E}[\overline{\widehat{u}^0(\xi^-)}\widehat{u}^2(\xi^+)]\nonumber\\
 =&-4m(\epsilon\xi^+)\int_0^t\int_0^{s_1}\int_{\substack{\xi^+=\xi_1+\xi_2\\ \xi_2=\xi_3+\xi_4}}m(\epsilon\xi_1)m^2(\epsilon\xi_2)m(\epsilon\xi_3)m(\epsilon\xi_4) e^{i(s_1\Omega_{+,-1,-2}+s_0\Omega_{2,3,-4})}\nonumber\\
 &\hspace{6cm}\mathbb{E}\left[\overline{\widehat{u}_0(\xi^-)}\widehat{u}_0(\xi_1)\overline{\widehat{u}_0(-\xi_3)}\widehat{u}_0(\xi_4)\right]\,d\xi_{1,3,4}\,ds_0\,ds_1\nonumber\\
 &+4m(\epsilon\xi^+)\int_0^t\int_0^{s_1}\int_{\substack{\xi^+=\xi_1+\xi_2\\ \xi_2=\xi_3+\xi_4}}m(\epsilon\xi_1)m^2(\epsilon\xi_2)m(\epsilon\xi_3)m(\epsilon\xi_4)e^{i(s_1\Omega_{+,-1,2}+s_0\Omega_{-2,3,-4})}\nonumber\\
 &\hspace{6cm}\mathbb{E}\left[\overline{\widehat{u}_0(\xi^-)}\widehat{u}_0(\xi_1)\overline{\widehat{u}_0(-\xi_3)}\widehat{u}_0(\xi_4)\right]\,d\xi_{1,3,4}\,ds_0\,ds_1\nonumber\\
 &-2m(\epsilon\xi^+)\int_0^t\int_0^{s_1}\int_{\substack{\xi^+=\xi_1+\xi_2\\ \xi_2=\xi_3+\xi_4}}m(\epsilon\xi_1)m^2(\epsilon\xi_2)m(\epsilon\xi_3)m(\epsilon\xi_4)e^{i(s_1\Omega_{+,1,-2}+s_0\Omega_{2,-3,-4})}\nonumber\\
 &\hspace{6cm}\mathbb{E}\left[\overline{\widehat{u}_0(\xi^-)}\overline{\widehat{u}_0(-\xi_1)}\widehat{u}_0(\xi_3)\widehat{u}_0(\xi_4)\right]\,d\xi_{1,3,4}\,ds_0\,ds_1\nonumber\\
 &+2m(\epsilon\xi^+)\int_0^t\int_0^{s_1}\int_{\substack{\xi^+=\xi_1+\xi_2\\ \xi_2=\xi_3+\xi_4}}m(\epsilon\xi_1)m^2(\epsilon\xi_2)m(\epsilon\xi_3)m(\epsilon\xi_4)e^{is_1(\Omega_{+,1,2}+s_0\Omega_{-2,-3,-4})}\nonumber\\
 &\hspace{6cm}\mathbb{E}\left[\overline{\widehat{u}_0(\xi^-)}\overline{\widehat{u}_0(-\xi_1)}\widehat{u}_0(\xi_3)\widehat{u}_0(\xi_4)\right]\,d\xi_{1,3,4}\,ds_0\,ds_1. \nonumber
 \end{align}
 By Wick's formula and symmetry, this is
 \begin{align}
\dots  =&-4(2\pi)^d\epsilon^{2d}m(\epsilon\xi^+)\int_0^t\int_0^{s_1}\int_{\substack{\xi^+=\xi_1+\xi_2\\ \xi_2=\xi_3+\xi_4}}m(\epsilon\xi_1)m^2(\epsilon\xi_2)m(\epsilon\xi_3)m(\epsilon\xi_4)e^{i(s_1\Omega_{+,-1,-2}+s_0\Omega_{2,3,-4})}\nonumber\\
&\hspace{3cm}\widehat{W}_0^\epsilon\left(\xi_4-\xi^-,\frac{\epsilon}{2}\left(\xi_4+\xi^-\right)\right)\widehat{W}_0^\epsilon\left(\xi_1+\xi_3,\frac{\epsilon}{2}\left(\xi_1-\xi_3\right)\right)\,d\xi_{1,3,4}\,ds_0\,ds_1\label{02-1}\\
&-4(2\pi)^d\epsilon^{2d}m(\epsilon\xi^+)\int_0^t\int_0^{s_1}\int_{\substack{\xi^+=\xi_1+\xi_2\\ \xi_2=\xi_3+\xi_4}}m(\epsilon\xi_1)m^2(\epsilon\xi_2)m(\epsilon\xi_3)m(\epsilon\xi_4)e^{i(s_1\Omega_{+,-1,-2}+s_0\Omega_{2,3,-4})}\nonumber\\
&\hspace{3cm}\widehat{W}_0^\epsilon\left(\xi_1-\xi^-,\frac{\epsilon}{2}\left(\xi_1+\xi^-\right)\right)\widehat{W}_0^\epsilon\left(\xi_3+\xi_4,\frac{\epsilon}{2}\left(\xi_4-\xi_3\right)\right)\,d\xi_{1,3,4}\,ds_0\,ds_1\label{02-1 deg}\\
 &+4(2\pi)^d\epsilon^{2d}m(\epsilon\xi^+)\int_0^t\int_0^{s_1}\int_{\substack{\xi^+=\xi_1+\xi_2\\ \xi_2=\xi_3+\xi_4}}m(\epsilon\xi_1)m^2(\epsilon\xi_2)m(\epsilon\xi_3)m(\epsilon\xi_4)e^{i(s_1\Omega_{+,-1,2}+s_0\Omega_{-2,3,-4})}\nonumber\\
 &\hspace{3cm}\widehat{W}_0^\epsilon\left(\xi_4-\xi^-,\frac{\epsilon}{2}\left(\xi_4+\xi^-\right)\right)\widehat{W}_0^\epsilon\left(\xi_1+\xi_3,\frac{\epsilon}{2}\left(\xi_1-\xi_3\right)\right)\,d\xi_{1,3,4}\,ds_0\,ds_1\label{02-2}\\
 &+4(2\pi)^d\epsilon^{2d}m(\epsilon\xi^+)\int_0^t\int_0^{s_1}\int_{\substack{\xi^+=\xi_1+\xi_2\\ \xi_2=\xi_3+\xi_4}}m(\epsilon\xi_1)m^2(\epsilon\xi_2)m(\epsilon\xi_3)m(\epsilon\xi_4)e^{i(s_1\Omega_{+,-1,2}+s_0\Omega_{-2,3,-4})}\nonumber\\
 &\hspace{3cm}\widehat{W}_0^\epsilon\left(\xi_1-\xi^-,\frac{\epsilon}{2}\left(\xi_1+\xi^-\right)\right)\widehat{W}_0^\epsilon\left(\xi_3+\xi_4,\frac{\epsilon}{2}\left(\xi_4-\xi_3\right)\right)\,d\xi_{1,3,4}\,ds_0\,ds_1\label{02-2 deg}\\
 &-4(2\pi)^d\epsilon^{2d}m(\epsilon\xi^+)\int_0^t\int_0^{s_1}\int_{\substack{\xi^+=\xi_1+\xi_2\\ \xi_2=\xi_3+\xi_4}}m(\epsilon\xi_1)m^2(\epsilon\xi_2)m(\epsilon\xi_3)m(\epsilon\xi_4)e^{i(s_1\Omega_{+,1,-2}+s_0\Omega_{2,-3,-4})}\nonumber\\
 &\hspace{3cm}\widehat{W}_0^\epsilon\left(\xi_4-\xi^-,\frac{\epsilon}{2}\left(\xi_4+\xi^-\right)\right)\widehat{W}_0^\epsilon\left(\xi_3+\xi_1,\frac{\epsilon}{2}\left(\xi_3-\xi_1\right)\right)\,d\xi_{1,3,4}\,ds_0\,ds_1\label{02-3}\\
 &+4(2\pi)^d\epsilon^{2d}m(\epsilon\xi^+)\int_0^t\int_0^{s_1}\int_{\substack{\xi^+=\xi_1+\xi_2\\ \xi_2=\xi_3+\xi_4}}m(\epsilon\xi_1)m^2(\epsilon\xi_2)m(\epsilon\xi_3)m(\epsilon\xi_4)e^{i(s_1\Omega_{+,1,2}+s_0\Omega_{-2,-3,-4})}\nonumber\\
 &\hspace{3cm}\widehat{W}_0^\epsilon\left(\xi_4-\xi^-,\frac{\epsilon}{2}\left(\xi_4+\xi^-\right)\right)\widehat{W}_0^\epsilon\left(\xi_3+\xi_1,\frac{\epsilon}{2}\left(\xi_3-\xi_1\right)\right)\,d\xi_{1,3,4}\,ds_0\,ds_1\label{02-4}
 \end{align}
 
 \noindent
  \underline{Term \eqref{02-1}} We perform the change of variables
  \begin{equation}\label{change of variables 1}
  \begin{cases}
   \eta_1=\xi_4-\xi^{-}\\
  \eta_2=\xi_1+\xi_3\\
  v_2=\frac{\epsilon}{2}(\xi_1-\xi_3)\\
  \end{cases}
  \end{equation}
  which is of Jacobian $\epsilon^{d}$. In these new variables,
\begin{align*}
& \xi=\xi^+-\xi^-=\xi_1+\xi_2-\xi^-=\xi_1+\xi_3+\xi_4-\xi^-=\eta_1+\eta_2 \\
& \frac{\epsilon}{2}(\xi_4+\xi^-)=\frac{\epsilon}{2}(\eta_1+2\xi^-)=\frac{\epsilon}{2}(\eta_1+\frac{2v}{\epsilon}-\xi)=v-\frac{\epsilon}{2}\eta_2.
\end{align*}
Therefore,
  \begin{align*}
  \eqref{02-1}&=-4(2\pi)^d\epsilon^dm\left(v+\frac{\epsilon}{2}\xi\right)\int_0^t\int_0^{s_1}\int_{\xi=\eta_1+\eta_2}m\left(v_2+\frac{\epsilon}{2}\eta_2\right)m\left(-v_2+\frac{\epsilon}{2}\eta_2\right)\\
  &m^2\left(v-v_2+\frac{\epsilon}{2}\eta_1\right)m\left(v+\frac{\epsilon}{2}\left(\eta_1-\eta_2\right)\right)e^{i(s_1\Omega_{+,-1,-2}+s_0\Omega_{2,3,-4})}\widehat{W}_0^\epsilon\left(\eta_1,v-\frac{\epsilon}{2}\eta_2\right)\widehat{W}_0^\epsilon\left(\eta_2,v_2\right)\,d\eta_{1,2}\,dv_2\,ds_0\,ds_1
  \end{align*}
The resonance moduli expressed in the new variables are
\begin{align*}
\Omega_{+,-1,-2}
&=\frac{1}{2\epsilon^2}\left(|v|^2-|v_2|^2-|v-v_2|^2\right)+\frac{1}{2\epsilon}\left(v\cdot\xi-v_2\cdot\eta_2-(v-v_2)\cdot\eta_1\right)+\frac{1}{8}\left(|\xi|^2-|\eta_2|^2-|\eta_1|^2\right)\\
\end{align*}
and
\begin{align*}
\Omega_{2,3,-4}
&=\frac{1}{2\epsilon^2}\left(|v-v_2|^2+|v_2|^2-|v|^2\right)+\frac{1}{2\epsilon}\left((v-v_2)\cdot\eta_1-v_2\cdot\eta_2-v\cdot(\eta_1-\eta_2)\right)+\frac{1}{8}\left(|\eta_1|^2+|\eta_2|^2-|\eta_1-\eta_2|^2\right)\\
\end{align*}
Their sum and difference are
\begin{equation*}
\begin{cases}
\Omega_{+,-1,-2}+\Omega_{2,3,-4}=\frac{\alpha_1}{\epsilon}+\gamma_1\\
\Omega_{+,-1,-2}-\Omega_{2,3,-4}=\frac{\Omega_1}{\epsilon^2}+\frac{1}{\epsilon}\widetilde{\beta}_1+\widetilde{\gamma}_1
\end{cases}
\quad
\begin{cases}
\Omega_1=|v|^2-|v_2|^2-|v-v_2|^2\\
\alpha_1=v\cdot\xi-v\cdot\eta_1-v_2\cdot\eta_2\\
\widetilde{\beta}_1, \gamma_1, \widetilde{\gamma}_1=O(1).
\end{cases}
\end{equation*}
Therefore, performing the change of variables $\tau=\frac{s_0+s_1}{2}$, $\sigma=\frac{s_1-s_0}{2}$, we obtain
\begin{align*}
\int_0^t\int_0^{s_1} e^{i(s_1\Omega_{+,-1,-2}+s_0\Omega_{2,3,-4})}\,ds_0\,ds_1&=2\int_0^t\int_0^\theta e^{i\tau(\frac{\alpha_1}{\epsilon}+\gamma_1)}e^{i\sigma(\frac{\Omega_1}{\epsilon^2}+\frac{\widetilde{\beta}_1}{\epsilon}+\widetilde{\gamma}_1})\,d\sigma \,d\tau,\quad\theta=\min\{\tau,t-\tau\}
\end{align*}
At this point, we will resort here to the following lemma.

\begin{lemma} For a compactly supported function $f$ such that $\| \partial_x^k f \| \lesssim_k 1$, for any $\lambda>0$, and for any $N \in \mathbb{N}$,
$$
\int \int_0^\lambda e^{i\sigma x} \,d\sigma\, f(x) \,dx = 2\pi ( \mathbb{P}_+f)(0) + O(\lambda^{-N}),
$$
where $\mathbb{P}_+$ is the projector on positive frequencies, in other words the Fourier multiplier with symbol ${\mathbf{1}}_{[0,\infty)}$.
\end{lemma}

\begin{proof} Since the Fourier transform of $\int_0^\lambda e^{i\sigma x} \,d\sigma$ is $\sqrt{2\pi} \mathbf{1}_{[0,\lambda]}$, and by self-adjointness of $\mathbb{P}_+$,
$$
\int \int_0^\lambda e^{i\sigma x} \,d\sigma \, f(x) \,dx = \int \mathbb{P}_+ \left[ \int_{-\lambda}^\lambda e^{i\sigma x} \,d\sigma \right] f(x) \,dx = \int \int_{-\lambda}^\lambda e^{i\sigma x} \,d\sigma \, \mathbb{P}_+ f(x) \,dx.
$$
The desired conclusion now follows by Lemma~\ref{lemmaDirichlet}.
\end{proof}

The conclusion of this lemma can be written somewhat formally as 
$$
\int_0^\lambda e^{i\sigma x} \,d\sigma = 2\pi \delta_+(x) + O(\lambda^{-N}),
$$
where $\delta_+$ is the distribution defined by $\langle \delta_+ , f \rangle = (\mathbb{P}_+ f)(0)$.
For the expression that we are trying to approximate, this implies that
\begin{align*}
\int_0^t\int_0^{s_1} e^{i(s_1\Omega_{+,-1,-2}+s_0\Omega_{2,3,-4})}\,ds_0\,ds_1
& = 2 \int_0^t e^{i\tau(\frac{\alpha_1}{\epsilon}+\gamma_1)} \epsilon^2 \left[2 \pi \delta_+(\Omega_1 + \epsilon \widetilde{\beta}_1 + \epsilon^2 \widetilde{\gamma}_1) + O((1 + \frac{\theta}{\epsilon^2})^{-N}) \right] \,d\tau \\
& = 4 \pi \epsilon^2 \int_0^t e^{i\tau \frac{\alpha_1}{\epsilon}} \delta_+(\Omega_1) \,d\sigma + O(\epsilon^4).
\end{align*}
Therefore,
 \begin{align*}
  \eqref{02-1}&=-8(2\pi)^{d+1}\epsilon^{d+2}m^2\left(v\right)\int_{\xi=\eta_1+\eta_2}\int_0^t m^2(v_2)m^2(v-v_2)e^{i\tau\frac{\alpha_1}{\epsilon}}\,d\tau \, \delta_+(\Omega_1)\widehat{W}_0^\epsilon\left(\eta_1,v \right)\widehat{W}_0^\epsilon\left(\eta_2,v_2\right)\,d\eta_{1,2}\,dv_2 \\
  &+ O(t^2\epsilon^{d+2}+\epsilon^{d+4})
 \end{align*}
Setting $v_1 = v-v_2$, or equivalently adding $\delta(\Sigma_{0,-1,-2}) = \delta(v-v_1-v_2)$ to the above integrand, this expression can be written as
 \begin{align*}
\eqref{02-1}&=-8(2\pi)^{d+1}m^2(v)\epsilon^{d+2} \int \delta(\xi + \eta_1 + \eta_2) \delta(\Sigma_{0,-1,-2}) \delta_+(\Omega_{0,-1,-2})m^2(v_1)m^2(v_2)\\
&\hspace{3cm} \int_0^t e^{i\tau\frac{\alpha_1}{\epsilon}}\,d\tau \,\widehat{W}_0^\epsilon\left(\eta_1,v \right)\widehat{W}_0^\epsilon\left(\eta_2,v_2\right)\,d\eta_{1,2}\,dv_{1,2} + O(t^2\epsilon^{d+2}+\epsilon^{d+4})
\end{align*}

The other term in \eqref{term 02} will give, by a similar calculation
projection to positive frequencies. So adding those two, we obtain the same term without any projection. This gives a combined contribution of
 \begin{align*}
& -8(2\pi)^{d+1}\epsilon^{d+2} m^2(v)\int \delta(\xi - \eta_1 -\eta_2) \delta(\Sigma_{0,-1,-2}) \delta(\Omega_{0,-1,-2})m^2(v_1)m^2(v_2)\\
&\hspace{3cm}\int_0^t e^{i\tau\frac{\alpha_1}{\epsilon}}\,d\tau \,\widehat{W}_0^\epsilon\left(\eta_1,v \right)\widehat{W}_0^\epsilon\left(\eta_2,v_2\right)\,d\eta_{1,2}\,dv_{1,2}
 + O(t^2\epsilon^{d+2}+\epsilon^{d+4}).
\end{align*}

\bigskip
\noindent

\underline{Term \eqref{02-1 deg}}  As we will see this terms is degenerate and will vanish in the limit. For the term  \eqref{02-1 deg} we perform  the change of variables  
\begin{equation}\label{change of variables 2}
  \begin{cases}
   \eta_1=\xi_1-\xi^{-}\\
  \eta_2=\xi_3+\xi_4\\
  v_2=\frac{\epsilon}{2}(\xi_4-\xi_3)\\
  \end{cases}
  \end{equation}
  which is of Jacobian $\epsilon^{d}$. In these new variables,
\begin{align*}
& \xi=\xi^+-\xi^-=\xi_1+\xi_2-\xi^-=\xi_1+\xi_3+\xi_4-\xi^-=\eta_1+\eta_2 \\
& \frac{\epsilon}{2}(\xi_1+\xi^-)=\frac{\epsilon}{2}(\eta_1+2\xi^-)=\frac{\epsilon}{2}(\eta_1+\frac{2v}{\epsilon}-\xi)=v-\frac{\epsilon}{2}\eta_2.
\end{align*}
Therefore,
  \begin{align*}
  \eqref{02-1 deg}&=-4(2\pi)^d\epsilon^dm\left(v+\frac{\epsilon}{2}\xi\right)\int_0^t\int_0^{s_1}\int_{\xi=\eta_1+\eta_2}m\left(v+\frac{\epsilon}{2}\left(\eta_1-\eta_2\right)\right) m^2(\epsilon\eta_2)m\left(v_2-\frac{\epsilon}{2}\eta_2\right)m\left(v_2+\frac{\epsilon}{2}\eta_2\right)\\
  &\hspace{5cm} e^{i(s_1\Omega_{+,-1,-2}+s_0\Omega_{2,3,-4})}
  \widehat{W}_0^\epsilon\left(\eta_1,v-\frac{\epsilon}{2}\eta_2\right)\widehat{W}_0^\epsilon\left(\eta_2,v_2\right)\,d\eta_{1,2}\,dv_2\,ds_{0,1}\\
 &=O(t^2\epsilon^{d+2}),
  \end{align*}
    since $|\eta_2|=O(1)$. 

\bigskip
\noindent

\underline{Term \eqref{02-2}} 
We perform the change of variables \eqref{change of variables 2} which yields
\begin{align*}
\eqref{02-2}&=4(2\pi)^d\epsilon^dm\left(v+\frac{\epsilon}{2}\xi\right)\int_0^t\int_0^{s_1}\int_{\xi=\eta_1+\eta_2}m\left(v_2+\frac{\epsilon}{2}\eta_2\right)m\left(-v_2+\frac{\epsilon}{2}\eta_2\right)\\
  &m^2\left(v-v_2+\frac{\epsilon}{2}\eta_1\right)m\left(v+\frac{\epsilon}{2}\left(\eta_1-\eta_2\right)\right)e^{i(s_1\Omega_{+,-1,2}+s_0\Omega_{-2,3,-4})}\widehat{W}_0^\epsilon\left(\eta_1,v-\frac{\epsilon}{2}\eta_2\right)\widehat{W}_0^\epsilon\left(\eta_2,v_2\right)\,d\eta_{1,2}\,dv_2\,ds_{0,1},
  \end{align*}
where the resonance moduli are
  \begin{align*}
& \Omega_{+,-1,2}=\frac{1}{2\epsilon^2}\left(|v|^2-|v_2|^2+|v-v_2|^2\right)+\frac{1}{2\epsilon}\left(v\cdot\xi-v_2\cdot\eta_2+(v-v_2)\cdot \eta_1\right)+\frac{1}{8}\left(|\xi|^2-|\eta_2|^2+|\eta_1|^2\right)\\
& \Omega_{-2,3,-4} =\frac{1}{2\epsilon^2}\left(-|v-v_2|^2+|v_2|^2-|v|^2\right)+\frac{1}{2\epsilon}\left(-(v-v_2)\cdot\eta_1-v_2\cdot\eta_2-(\eta_1-\eta_2)\cdot v\right)\\
&\hspace{6.9cm}+\frac{1}{8}\left(|\eta_2|^2-|\eta_1|^2-|\eta_1-\eta_2|^2\right).
\end{align*}
Their sum and difference are
  \begin{equation*}
  \begin{cases}
  \Omega_{+,-1,2}+\Omega_{-2,3,-4}=\frac{\alpha_1}{\epsilon},+\gamma_2\\
  \Omega_{+,-1,2}-\Omega_{-2,3,-4}=\frac{\Omega_2}{\epsilon^2}+\frac{\widetilde{\beta}_2}{\epsilon}+\widetilde{\gamma}_2
  \end{cases}\\
  \quad
  \begin{cases}
  \Omega_2=|v|^2-|v_2|^2+|v_2-v|^2\\
  \alpha_1= v\cdot\xi - v\cdot \eta_1-  v_2\cdot\eta_2  \\
  \gamma_2,  \widetilde{\beta}_2,  \widetilde{\gamma}_2=O(1)
  \end{cases}
  \end{equation*}
 By a similar argument to the one used for~\eqref{02-1}, we obtain, after adding with the symmetric term in~\eqref{term 02}, a contribution of
   \begin{align*}
& 8(2\pi)^{d+1}\epsilon^{d+2}m^2(v)\int\delta(\xi-\eta_1-\eta_2)\delta(\Sigma_{0,1,-2})\delta(\Omega_{0,1,-2})m^2(v_1)m^2(v_2)\\
&\hspace{3cm}\int_0^t \left(e^{i\tau\frac{\alpha_1}{\epsilon}}\,d\tau\right)\widehat{W}_0^\epsilon\left(\eta_1,v\right)\widehat{W}_0^\epsilon\left(\eta_2,v_2\right)\,d\eta_{1,2}\,dv_{1,2}+O(t^2\epsilon^{d+2}+\epsilon^{d+4})
 \end{align*}
 where $\Sigma_{0,1,-2}=v+v_1-v_2$ and $\Omega_{0,1,-2}=|v|^2+|v_1|^2-|v_2|^2$.
 
 \bigskip \noindent
 \underline{Term \eqref{02-2 deg}}. This term is degenerate and gives a contribution $O(t^2\epsilon^{d+2})$ similarly to \eqref{02-1 deg}.
 
\bigskip \noindent
\underline{Term \eqref{02-3}} 
We perform the change of variables
  \begin{equation}\label{change of variables 3}
  \begin{cases}
  \eta_1=\xi_3+\xi_1\\
     \eta_2=\xi_4-\xi^{-}\\
  v_1=\frac{\epsilon}{2}(\xi_3-\xi_1)\\
  \end{cases}
  \end{equation}
  which is of Jacobian $\epsilon^{d}$. In these new variables,
\begin{align*}
& \xi=\xi^+-\xi^-=\xi_1+\xi_2-\xi^-=\xi_1+\xi_3+\xi_4-\xi^-=\eta_1+\eta_2\\
& \frac{\epsilon}{2}(\xi_4+\xi^-)=\frac{\epsilon}{2}(\eta_2+2\xi^-)=\frac{\epsilon}{2}(\eta_2+\frac{2v}{\epsilon}-\xi)=v-\frac{\epsilon}{2}\eta_1.
\end{align*}
Therefore,
  \begin{align*}
  \eqref{02-3}&=-4(2\pi)^d\epsilon^dm\left(v+\frac{\epsilon}{2}\xi\right)\int_0^t\int_0^{s_1}\int_{\xi=\eta_1+\eta_2}m\left(v_1+\frac{\epsilon}{2}\eta_1\right)m\left(-v_1+\frac{\epsilon}{2}\eta_1\right)\\
  &m^2\left(v+v_1+\frac{\epsilon}{2}\eta_2\right)m\left(v+\frac{\epsilon}{2}\left(\eta_2-\eta_1\right)\right)e^{i(s_1\Omega_{+,1,-2}+s_0\Omega_{2,-3,-4})} \widehat{W}_0^\epsilon\left(\eta_1,v_1\right)\widehat{W}_0^\epsilon\left(\eta_2,v-\frac{\epsilon}{2}\eta_1\right) \,d\eta_{1,2}\,dv_1\,ds_{0,1}.
  \end{align*}
  The resonance moduli are given by
  \begin{align*}
&  \Omega_{+,1,-2}=\frac{1}{2\epsilon^2}\left(|v|^2+|v_1|^2-|v+v_1|^2\right)+\frac{1}{2\epsilon}\left(v\cdot\xi-v_2\cdot\eta_1-(v+v_1)\cdot\eta_2\right)+\frac{1}{8}\left(|\xi|^2+|\eta_1|^2-|\eta_2|^2\right)\\
& \Omega_{2,-3,-4}=\frac{1}{2\epsilon^2}\left(|v+v_1|^2-|v_1|^2-|v|^2\right)+\frac{1}{2\epsilon}\left((v+v_1)\cdot\eta_2-v_1\cdot\eta_1-v\cdot(\eta_2-\eta_1)\right)+\frac{1}{8}\left(|\eta_2|^2-|\eta_1|^2-|\eta_2-\eta_1|^2\right),
  \end{align*}
with sum and difference
  \begin{equation*}
  \begin{cases}
  \Omega_{+,1,-2}+\Omega_{2,-3,-4}=\frac{\alpha_2}{\epsilon} + \gamma_3\\
  \Omega_{+,1,-2}-\Omega_{2,-3,-4}=\frac{\Omega_3}{\epsilon^2}+\frac{1}{\epsilon}\widetilde{\beta_3}+\widetilde{\gamma_3}
  \end{cases}
  \quad
  \begin{cases}
  \Omega_3=|v|^2+|v_1|^2-|v+v_1|^2\\
  \alpha_2=v\cdot\xi-  v_1 \cdot \eta_1-v\cdot\eta_2\\
  \gamma_3,  \widetilde{\beta}_3,  \widetilde{\gamma}_3=O(1).
  \end{cases}
  \end{equation*}
 By a similar argument to the above, we obtain a combined contribution with the symmetric term in~\eqref{term 02} 
 \begin{align*}
&-8(2\pi)^{d+1}\epsilon^{d+2}m^2(v)\int\delta(\xi-\eta_1-\eta_2)\delta(\Sigma_{0,1,-2})\delta(\Omega_{0,1,-2})m^2(v_1)m^2(v_2)\\
&\hspace{3cm}\int_0^t \left(e^{i\tau\frac{\alpha_2}{\epsilon}}\,d\tau\right)\widehat{W}_0^\epsilon\left(\eta_1,v_1\right)\widehat{W}_0^\epsilon\left(\eta_2,v\right)\,d\eta_{1,2}\,dv_{1,2}+O(t^2\epsilon^{d+2}+\epsilon^{d+4}).
 \end{align*}
 
 \bigskip
 \noindent
 \underline{Term \eqref{02-4}} We perform the change of variables \eqref{change of variables 3}, which yields
  \begin{align*}
  \eqref{02-4}&=4(2\pi)^d\epsilon^dm\left(v+\frac{\epsilon}{2}\xi\right)\int_0^t\int_0^{s_1}\int_{\xi=\eta_1+\eta_2}m\left(v_1+\frac{\epsilon}{2}\eta_1\right)m\left(-v_1+\frac{\epsilon}{2}\eta_1\right)\\
  &m^2\left(v+v_1+\frac{\epsilon}{2}\eta_2\right)m\left(v+\frac{\epsilon}{2}\left(\eta_2-\eta_1\right)\right)e^{i(s_1\Omega_{+,1,2}+is_0\Omega_{-2,-3,-4})}\widehat{W}_0^\epsilon\left(\eta_1,v-\frac{\epsilon}{2}\eta_2\right)\widehat{W}_0^\epsilon\left(\eta_2,v_2\right)\,d\eta_{1,2}\,dv_2\,ds_0\,ds_1,
  \end{align*}
with resonance moduli
\begin{align*}
& \Omega_{+,1,2} =\frac{1}{2\epsilon^2}\left(|v|^2+|v_2|^2+|v+v_1|^2\right)+\frac{1}{2\epsilon}\left(v\cdot\xi-v_2\cdot\eta_2+(v+v_2)\cdot\eta_1\right)+\frac{1}{8}\left(|\xi|^2+|\eta_2|^2+|\eta_1|^2\right)\\
&  \Omega_{-2,-3,-4}=-\frac{1}{2\epsilon^2}\left(|v+v_2|^2+|v_2|^2+|v_1|^2\right)+\frac{1}{2\epsilon}\left(-(v+v_2)\cdot\eta_1-v_2\cdot\eta_2-v(\eta_1-\eta_2)\right)\\
&\hspace{6.9cm}-\frac{1}{8}\left(|\eta_1|^2+|\eta_2|^2+|\eta_1-\eta_2|^2\right).
  \end{align*}
  We have
  \begin{equation*}
  \begin{cases}
  \Omega_{+,1,2}+\Omega_{-2,-3,-4}=\frac{\alpha_4}{\epsilon},+\gamma_4\\
  \Omega_{+,1,2}-\Omega_{-2,-3,-4}=\frac{\Omega}{\epsilon^2}+\frac{\widetilde{\beta}_4}{\epsilon}+\widetilde{\gamma}_4
  \end{cases}\\
  \quad
  \begin{cases}
  \Omega_{0,1,2}=|v|^2+|v_2|^2+|v+v_2|^2\\
  \alpha_4=(v-v_2)\cdot\eta_2=v\cdot\xi-2 v_2\cdot\eta_2-v\cdot(\eta_1-\eta_2)\\
  \gamma_4,  \widetilde{\beta}_4,  \widetilde{\gamma}_4=O(1)
  \end{cases}
  \end{equation*}
  By the same argument as above, this term will give no contribution besides $O(\epsilon^{3}t+\epsilon^{4}t+\epsilon^{d+2}\min\{t,\epsilon^2\})$, since it will have a factor $\delta(\Omega_{0,1,2})$.
  
  \bigskip

   Combining all the above, and using the fact that $t<<\epsilon$, $T_{kin}=\frac{1}{\lambda^2\epsilon^2}$, we obtain
  \begin{align}
  \widehat{W}^\epsilon[u](\xi,v)= &  \widehat{W}^\epsilon_0(\xi,v) + \frac{4(2\pi)^{1-\frac{d}{2}}}{T_{kin}}e^{-it\xi\cdot\frac{v}{\epsilon}}m^2(v)\int\delta(\xi-\eta_1-\eta_2)\delta(\Sigma_{0,-1,-2})\delta(\Omega_{0,-1,-2})m^2(v_1)m^2(v_2)\nonumber\\
  &\times\bigg[\left(\int_0^t e^{i\tau\frac{\alpha_0}{\epsilon}}\,d\tau\right)\widehat{W}_0^\epsilon\left(\eta_1,v_1\right)\widehat{W}_0^\epsilon\left(\eta_2,v_2\right)-\left(\int_0^t e^{i\tau\frac{\alpha_1}{\epsilon}}\,d\tau\right)\widehat{W}_0^\epsilon\left(\eta_1,v\right)\widehat{W}_0^\epsilon\left(\eta_2,v_2\right)\nonumber\\
 &\hspace{3cm}-\left(\int_0^t e^{i\tau\frac{\alpha_2}{\epsilon}}\,d\tau\right)\widehat{W}_0^\epsilon\left(\eta_1,v_1\right)\widehat{W}_0^\epsilon\left(\eta_2,v\right)\bigg]\,d\eta_{1,2}\,dv_{1,2}\nonumber\\
 &+\frac{8(2\pi)^{1-\frac{d}{2}}}{T_{kin}}e^{-it\xi\cdot\frac{v}{\epsilon}}m^2(v)\int\delta(\xi-\eta_1-\eta_2)\delta(\Sigma_{0,1,-2})\delta(\Omega_{0,1,-2})m^2(v_1)m^2(v_2)\nonumber\\
  &\times\bigg[\left(\int_0^t e^{i\tau\frac{\alpha_0}{\epsilon}}\,d\tau\right)\widehat{W}_0^\epsilon\left(\eta_1,v_1\right)\widehat{W}_0^\epsilon\left(\eta_2,v_2\right)+\left(\int_0^t e^{i\tau\frac{\alpha_1}{\epsilon}}\,d\tau\right)\widehat{W}_0^\epsilon\left(\eta_1,v\right)\widehat{W}_0^\epsilon\left(\eta_2,v_2\right)\nonumber\\
 &\hspace{3cm}-\left(\int_0^t e^{i\tau\frac{\alpha_2}{\epsilon}}\,d\tau\right)\widehat{W}_0^\epsilon\left(\eta_1,v_1\right)\widehat{W}_0^\epsilon\left(\eta_2,v\right)\bigg]\,d\eta_{1,2}\,dv_{1,2}\nonumber\\
 & +\lambda^2\epsilon^{-d}\times\eqref{11-4}+O(\lambda^2\epsilon^4)+(2\pi)^{-d/2}\epsilon^{-d}\times( h.o.t.)\label{correlation expansion}
  \end{align}
  where
  \begin{equation*}
  \begin{cases}
  T_{kin}=\frac{1}{\lambda^2\epsilon^2}\\
  \Sigma_{0,-1,-2}=v-v_1-v_2\\
  \Sigma_{0,1,-2}=v+v_1-v_2\\
  \Omega_{0,-1,-2}=|v|^2-|v_1|^2-|v_2|^2\\
  \Omega_{0,1,-2}=|v|^2+|v_1|^2-|v_2|^2\\
  \alpha_0=v\cdot\xi-v_1\cdot\eta_1-v_2\cdot\eta_2\\
  \alpha_1=v\cdot\xi-v_1\cdot\eta_1-v\cdot\eta_2\\
  \alpha_2=v\cdot\xi-v\cdot\eta_1-v_2\cdot\eta_2,
  \end{cases}
  \end{equation*}
  and the higher order terms are given by \eqref{hot}.
  
\subsection{Conclusion} Gathering the above computations gives the following proposition.
\begin{proposition}\label{comparing prop} In the regime $\epsilon^2 \ll  t  \ll \min(\epsilon,T_{kin})$ there holds
	\begin{equation}\label{integral expansion}
	\int|\widehat{\rho}(t,\xi,v) - \widehat{W}^\epsilon[u](t,\xi,v)|\,dv = O\left(\frac{t}{T_{kin}}\right)^2 + O(\lambda^2\epsilon^4 ) + (2\pi)^{-d/2}\epsilon^{-d}\int_{\mathbb{R}^d}(h.o.t.)\,dv,
	\end{equation}
	where the higher order terms are given by \eqref{hot}.
\end{proposition}
\begin{proof}
Again, we prove the result for $m(0)=0$ and $\omega_0=0$. Combining \eqref{kinetic  expansion} and \eqref{correlation expansion}, we obtain that
\begin{align}
& \widehat{\rho}(t,\xi,v) = \widehat{W}^\epsilon[u](t,\xi,v) + O \left( \frac{t}{T_{kin}} \right)^2 +\lambda^2\epsilon^{-d}\times\eqref{11-4}+ O(\lambda^2 \epsilon^4)+h.o.t.
\end{align}
Since integrating \eqref{11-4} gives a contribution $O(t^2\epsilon^{d+2})$ and $t \ll \epsilon$, we obtain
\begin{equation*}
\int|\widehat{\rho}(t,\xi,v) - \widehat{W}^\epsilon[u](t,\xi,v)|\,dv = O\left(\frac{t}{T_{kin}}\right)^2+O(\lambda^2\epsilon^4 )+(2\pi)^{-d/2}\epsilon^{-d}\int_{\mathbb{R}^d}(h.o.t.)\,dv.
\end{equation*}
\end{proof}

The aim of the rest of the present paper is to estimate the higher order terms and show they are smaller than the leading term.

\section{Graph analysis for the diagrammatic expansion of the solution} \label{sec:diagrams}

Proceeding as in~\cite{CG1} (based on \cite{LS}), we perform a diagrammatic expansion and write $u^n$ as a sum over Feynman graphs. There are numerous differences between the framework developed in that paper, and the one needed in the present manuscript. First, the equation here is quadratic, instead of cubic, resulting in a binary instead of a ternary tree; second, waves of arbitrary parities might interact; third, the problem being set on the whole space $\mathbb R^d$, certain sums are replaced by integrals and new "slow" variables $\eta$ appear. Most importantly, we handle the time constraints in a completely novel way, in order to deal with dispersion relations which are nonzero at the origin $\omega(\xi)=\epsilon^{-2}+\frac{|\xi|^2}{2}$, resulting in the introduction of new and different tools for graph analysis.

\subsection{Main result} 
The main result from this graphical expansion is the following: the expectation in probability of Lebesgue, Sobolev and Bourgain norms for the approximating series $\sum u^n$ can be computed as a sum of oscillatory integrals in large dimensions. In this sum, each oscillatory integral is completely described by an associated graph. Moreover, the oscillatory phases in each oscillatory integral can be divided between those of \emph{degree zero}, those of \emph{degree one and linear}, and those of \emph{degree one and quadratic}, according to their dependance on \emph{interaction free variables}. This distinction will be useful later on.

For the expectation of the $L^2$ norm, the outcome of this analysis is the following. All objects mentioned in the Proposition below are defined rigorously afterwards in the rest of this section.

\begin{proposition} \label{pr:formulagraph}

For each $n\geq 0$, the following holds true. There exists a finite set $\mathcal G^p_n$ of \emph{paired graphs of depth $n$} and, for each $t\geq 0$, a function $\mathcal F_t:\mathcal G^p_n\mapsto \mathbb C$ such that:
\begin{equation}
\label{sittelle}
\mathbb E \| u^n(t) \|_{L^2}^2 =\sum_{G\in \mathcal G^p_n} \mathcal F_t(G).
\end{equation}
For each $G\in \mathcal G^p_n$, there holds the formula:
\begin{align}
\label{chouettehulotte} \mathcal F_t(G)&= (2\pi)^{\frac d2} \lambda^{2n}\ep^{d(n+1)}  \int_{\underline \eta \in \mathbb R^{d(n+1)}_0}  \int_{\underline \xi^f \in \mathbb R^{d(n+1)}}  \int_{\underline s\in \mathbb R_+^{2n}} \Delta_t(\underline s) d\underline{\xi^f} \, d \underline{\eta}  \, d\underline{s}  \\
& \nonumber \qquad \qquad \qquad \qquad \qquad  M_G(\underline \xi) \prod_{\{i,j\}\in P} \widehat{W_0^\epsilon} (\eta_{i,j},\frac{\ep}{2}(\sigma_{0,i}\xi_{0,i}+\sigma_{0,j} \xi_{0,j}))   \prod_{v\in \mathcal V_i} e^{-i s_v \sum_{\tilde v\in \mathcal p^+(v)}\Omega_{\tilde v}}
\end{align}
where we wrote $\eta_{i,j}$ instead of $\eta_{\{i,j\}}$ to simplify notations\footnote{This abuse of notation will be made throughout the paper} and where we used the notations:
\begin{itemize}
\item $\underline \eta =(\eta_{i,j})_{\{i,j\}\in P}\in (\mathbb R^d)^{n+1}$ are the \emph{slow free variables}.
\item $\mathbb R^{d(n+1)}_0=\{\underline \eta \in \mathbb R^{d(n+1)}, \ \sum \eta_{i,j}=0\}$.
\item $\mathcal V_i=\{v_1,...,v_{2n}\}$ gathers the \emph{interaction vertices}, ordered according to the \emph{integration order}.
\item $\underline{s} = (s_v)_{v\in \mathcal V_i} \in \mathbb{R}_+^{2n}$ gathers \emph{intermediate time slices}.
\item $P=P(G)$ is a \emph{pairing}, a partition of $\{1,2n+2\}$ into pairs $\{i,j\}$ uniquely determined by $G$.
\item $\underline{\xi^f} =(\xi^f_1,...\xi^f_{n+1})\in (\mathbb R^d)^{n+1}$ are the \emph{interaction free variables}.
\item $\Delta_t$ is the indicatrix function of a set: the set of intermediate time slices $\underline s$ satisfying the \emph{time constraints} of the graph.
\item $M_G(\underline \xi)$ encodes the effects of the Fourier multiplier $m$:
$$
M_G(\underline \xi)=  \prod_{i=1}^{2n+2}m(\epsilon \xi_{0,i}) \prod_{k=1}^{2n} m^2(\epsilon \tilde \xi_k)
$$
\item $\mathcal p^+(v)\subset \mathcal V_i$ is the set containing $v$ and the vertices up on the right of $v$ in the graph. It is such that for $1\leq k<k'\leq 2n$, $v_k\notin \mathcal p^+(v_{k'})$.
\end{itemize}

There exists two disjoint sets of \emph{degree zero vertices} $\mathcal V^0$ and \emph{degree one vertices} $\mathcal V^1$ such that $\mathcal V_i=\mathcal V^0\cup \mathcal V^1$ and $\# \mathcal V^0=\# \mathcal V^1=n$. The set $\mathcal V^1$ can be labeled by indices $1\leq k_1<...<k_{n}\leq 2 n$, in other words $\mathcal V^1=\{v_{k_1},...,v_{k_n}\}$. This set can be further partitioned into \emph{linear} and \emph{quadratic} vertices $\mathcal V^1=\mathcal V^1_l \cup \mathcal V^1_q$ with $\mathcal V^1_l \cap \mathcal V^1_q=\emptyset$. The frequency associated to the left edge below $v_{k_i}$ is an interaction free frequency, denoted $\xi_i^f$.

\begin{itemize}
	\item[(i)] For each $1\leq k \leq 2n$, $\tilde \xi_k$ is the frequency on top of $v_k$, given by:
$$
\tilde \xi_k = \sum_{1\leq j \leq n, \ k_j\geq k} \tilde c_{k,j} \xi^f_j + \sum_{ \{i',j'\}\in P, i'<j' }\tilde c_{k,i',j'} \eta_{ i',j' }\quad \mbox{with} \quad \tilde c_{k,j},\tilde c_{k,i',j'} \in \{-1,0,1\}.
$$

\item[(ii)] For every $\{i,j\}\in P$, there holds that $\sigma_{0,i}\in \{\pm 1\}$ and:
$$
\xi_{0,i}=  \sum_{j=0}^n \bar c_{i,j} \xi^f_j + \sum_{ \{i',j'\}\in P, i'<j' }\bar c_{i,i',j'} \eta_{ i',j' }\quad \mbox{with} \quad \bar c_{i,j},\bar c_{i,i',j'} \in \{-1,0,1\} .
$$
Moreover, the map $((\xi^f_i)_{1\leq i\leq n+1}, (\eta_{i,j})_{\{i,j\}\in P})\mapsto ( \sigma_{0,i}\xi_{0,i}+\sigma_{0,j}\xi_{0,j}, \eta_{i,j})_{\{i,j\}\in P, \ i<j}$ is a bijection onto $\mathbb R^{d(n+1)}\times \mathbb R^{d(n+1)}_0$.

\item[(iii)] Assume $1\leq k \leq 2n$ is such that $k=k_i$ for some $1\leq i\leq n$, so that $v_{k}\in \mathcal V^1$. Then there exist two signs $\sigma_k,\tilde \sigma_k\in \{\pm 1\}^2$ such that, if $v\in \mathcal V^1_l$:
\be \label{id:formulaxik1}
\Omega_{v_{k}} =\sigma_k \tilde \xi_k \cdot \xi^f_i \ + \ \left\{\begin{array}{l l l} \frac{1}{2}(\tilde \sigma_k +\sigma_k)|\tilde \xi_k|^2 & \mbox{if }\omega(\xi)=\frac{|\xi|^2}{2}, \\ \tilde \sigma_k \epsilon^{-2}+\frac{1}{2}(\tilde \sigma_k +\sigma_k)|\tilde \xi_k|^2 & \mbox{if }\omega(\xi)=\frac{|\xi|^2}{2}+\epsilon^{-2} ,\end{array} \right.
\ee 
and if $v\in \mathcal V^1_q$:
\be \label{id:formulaxik2}
\Omega_{v_k} =-\sigma_k \xi^f_i \cdot (\xi^f_i+\tilde \xi_k) \ + \ \left\{\begin{array}{l l l} \frac{1}{2}(\tilde \sigma_k-\sigma_k)|\tilde \xi_k|^2 & \mbox{if }\omega(\xi)=\frac{|\xi|^2}{2}, \\ (\tilde \sigma_k-2\sigma_k )\epsilon^{-2}+\frac{1}{2}(\tilde \sigma_k-\sigma_k )|\tilde \xi_k|^2 & \mbox{if }\omega(\xi)=\frac{|\xi|^2}{2}+\epsilon^{-2} .\end{array} \right.
\ee

\item[(iv)] Assume that $1\leq k \leq 2n$ is such that $k_{i-1}<k<k_{i}$ for some $1\leq i \leq n$, so that $v_k\in \mathcal V^0$. Then $\Omega_{v_k}$ is a quadratic polynomial which depends only on the variables $\{\eta_{i,j}\}_{\{ i,j\}\in P, \ i<j}$, and on the variables $(\xi^f_{j})_{j\geq i}$.
\end{itemize}

\end{proposition}

\begin{remark}

One crucial information in Proposition \ref{pr:formulagraph} is that for all $1\leq i\leq  n $:
\begin{itemize}
\item For $k=k_i$, the quantity $e^{-i s_{v_k} \sum_{\tilde v\in \mathcal p^+(v_k)}\Omega_{\tilde v}}$ does not depend on the previous free variables $\xi^f_j$ for $j<i$. Moreover, only $\Omega_{v_k}$ actually depends on $\xi^f_i$ and its dependance is explicit, given by \eqref{id:formulaxik1} and \eqref{id:formulaxik2}.
\item For $k> k_i$ either if $v_k$ is a degree zero or degree one vertex, the quantity $e^{-i s_{v_k} \sum_{\tilde v\in \mathcal p^+(v_k)}\Omega_{\tilde v}}$ does not depend on the previous free variables $\xi^f_j$ for $j\leq i$.
\end{itemize}

\end{remark}

The rest of this section presents the diagrammatic expansion for the Dyson series, and in particular defines rigorously all the objects mentioned in Proposition \ref{pr:formulagraph}, leading eventually to its proof at the end of Subsection \ref{subsec:kirchhoff}. Some elementary facts from graph analysis are given without proofs, in which case we refer to \cite{CG1} for the details.

\subsection{Graphical representation of the Dyson series} \label{subsec:graphs}

This subsection explains how $u^n$ can be represented as a sum of functions represented by graphs.

\subsubsection{Definition of an interaction graph} \label{subsubsec:graph}

An interaction graph of depth $n$ is an oriented binary planar tree $G=\{\mathcal V,v_a,v_l,v_r,p,\sigma)$ where:
\begin{itemize}
	\item $\mathcal V=\mathcal V_R\cup \mathcal V_i \cup \mathcal V_0$ is the collection of vertices. $\mathcal V_R=\{v_R\}$ contains the \emph{root vertex} (representing $\widehat{u_G}(\xi_R)$). $\mathcal V_0\neq \emptyset$ contains the \emph{initial vertices} (representing the initial datum $\widehat{u_0}(\xi_{v_0}$)). $\mathcal V_i$ contains the $n$ \emph{interaction vertices} (each representing an iteration of the nonlinearity).
	\item $v_a:\mathcal V_0\cup \mathcal V_i\rightarrow \mathcal V_i\cup \mathcal V_R$, $v_l:\mathcal V_i\rightarrow \mathcal V_0\cup \mathcal V_i$, and $v_r:\mathcal V_i\cup \mathcal V_R\rightarrow \mathcal V_0\cup \mathcal V_i$ represent the \emph{positions} of the vertices. $v_a(v)$, $v_l(v)$, and $v_r(v)$ are respectively the vertices \emph{above}, \emph{below on the left}, and \emph{below on the right} of $v$. They satisfy the following:
	\begin{itemize}
		\item[(i)] There exists a unique \emph{top vertex} $v_{\text{top}}\in \mathcal V_0\cup \mathcal V_i$ such that $v_a(v_{\text{top}})=v_R$. By convention, it is at bottom right of the root vertex: $v_{\text{top}}=v_r(v_R)$.
		\item[(i)] For all $v\in \mathcal V_i$, there holds $v_{l}(v)\neq v_{r}(v)$ and these are the only antecedents of $v$ by $v_a$, i.e. $\{\tilde v\in \mathcal V_0\cup \mathcal V_i, \ v_a(\tilde v)=v\}=\{v_l(v),v_r(v)\}$.
		\item[(ii)] For all $v_0\in \mathcal V_0$, there exists a unique $v_a(v_0)\in \mathcal V_R \cup \mathcal V_i$ such that $(v_0,v_a)\in \mathcal E$.
	\end{itemize}
	We also denote $e_a(v) = (v, v_a(v))$, $e_l(v)= (v, v_l(v)) $, and $e_r(v)= (v, v_r(v))$.
	\item $ \mathcal E\subset \mathcal V^2$ is the set of oriented edges (representing a free evolution $e^{is\Delta}$), and is equal to:
	$$
	\mathcal V=\{(v_{\text{top}},v_R)\}\cup_{v\in \mathcal V_i}\{(v_l(v),v),(v_r(v),v)\}.
	$$
	Above, $(v_{\text{top}},v_R)$ is called the \emph{root edge} and for $v\in \mathcal V_i$, $(v_l(v),v)$ and $(v_r(v),v)$ are called \textit{interaction edges}.
	\item $\sigma:\mathcal V_i\cup \mathcal V_0\rightarrow \{-1,1\}$ is the \emph{parity} (encoding if complex conjugation was taken in the iteration of the nonlinearity). For (i) above, it must satisfy that $\sigma_v=+1$. We extend it to a parity function for the edges $\sigma:\mathcal E\rightarrow \{-1,1\}$ (slightly abusing notations) as follows: if $e=(v,v')$ then $\sigma_e=\sigma_{v}$ is the parity of the vertex below. The \emph{total parity} of $G$ is defined as $\sigma_G=\prod_{v\in \mathcal V_i\cup \mathcal V_0}\sigma_{v}$.
\end{itemize}
With this definition, the graph $G$ is a connected tree with $n+1$ initial vertices. The set of interaction graphs of depth $n$ is denoted by $\mathcal G(n)$.

\subsubsection{Frequencies and Kirchhoff laws}\label{subsubsec:freq}

To each edge $e\in \mathcal E$ we associate a frequency variable $\xi_e\in \mathbb R^d$. The Kirchhoff laws of the graph specify that, at each interaction vertex, the two frequencies of the edges below add up to the frequency of the edge above, end that the output frequency of the edge on top of the graph is $\xi_R$. This is written as:
$$
\Delta_{\xi_R}(\underline \xi)= \delta(\xi_R-\xi_{e_a(v_{\text{top}})})  \Delta (\underline \xi), \quad \mbox{with} \quad \Delta (\underline \xi)= \prod_{v\in \mathcal V_i} \delta(\xi_{e_a(v)}-\xi_{e_l(v)}-\xi_{e_r(v)}).
$$
The frequency multiplier $M(\underline \xi)$ is then expressed as:
\be \label{id:fouriermultipliergeneralized}
M(\underline \xi)= m(\epsilon \xi_{e_a(v_{\text{top}})})
\prod_{v_0\in \mathcal V_0}m(\epsilon \xi_{e_a(v_0)}) \prod_{v\in \mathcal V_i \backslash \{v_{\text{top}}\}} m^2(\epsilon \xi_{e_a(v)}).
\ee

\subsubsection{Interaction time variables and time constraints} \label{subsubsec:time}

A \emph{forward path} of length $l$ is a finite collection of edges $\mathcal p=(e_1,...,e_l)$ such that for each $1\leq i\leq i+1\leq l$, writing $e_i=(v,v')$ and $e_{i+1}=(\tilde v,\tilde v')$, there holds $v'=\tilde v$. We can thus write alternatively with a slight abuse of notations $\mathcal p=(v_1,...,v_{l+1})$ where $e_i=(v_i,v_{i+1})$. We then say that $\mathcal p$ \emph{leads to} $v_{l+1}$. We remark that for each initial vertex $v_0\in \mathcal V_0$, there exists a unique forward path $\mathcal p=(v_0,v_1,...,v_R)$ ending at the root vertex.

Given any two initial vertices $v_0\neq v_0'\in \mathcal V_0$, we say that $v_0$ is at the left of $v_0'$, if, denoting by $\mathcal p=(v_0,v_1,...,\tilde v,\bar v,...,v_R)$ and $\mathcal p=(v_0',v_1',...,\tilde v',\bar v,...,v_R)$ their forward path ending at the root vertex, they intersect at $\bar v\in\mathcal V_i$ and there holds that $\tilde v$ and $\tilde v'$ are at the left and right respectively of $\bar v$, namely $\tilde v=v_l(\bar v)$ and $\tilde v'=v_r(\bar v)$. This defines a total order on the set of initial vertices, so that we order it as $\mathcal V_0=(v_{0,1},...,v_{0,n+1})$ from left to right. We adapt the notation for the frequency variables and write $\xi_{e_a(v_{0,i})}=\xi_{0,i}$ for $1\leq i \leq n+1$.

Given any two vertices $v \neq v'\in \mathcal V$, we say that $v$ is above $v'$ (or $v'$ is below $v$) and write $v> v'$ (or $v'< v$), if $v$ belongs to the unique forward path starting at $v'$ and ending at the root vertex. This defines a partial ordering for the vertices of the graph, called the \emph{time order}.

To each vertex we associate a time variable. The time variable of any initial vertex $v_0\in \mathcal V_0$ is $t_{v_0}=0$. To the root vertex we associate the total time $t_{v_R}=t$. To each interaction vertex $v\in \mathcal V_i$ we associate an interaction time variable $t_v\in \mathbb R_+$. We require that $t_v\leq t_{v'}$ whenever $v$ is below $v'$. The time constraint function is thus:
$$
\Delta_t (\underline t)= \delta(t_{v_R}-t) \prod_{v,v'\in \mathcal V_i, \ v <v' } {\bf 1}(t_v\leq t_{v'}).
$$

\subsubsection{General formula}

We describe the expansion \fref{defun} which encodes iterations of Duhamel formula \fref{defun} via diagrams. For all $n\geq 0$:
\be \label{id:uassumofinteractiongraphs}
u^n=\sum_{G\in \mathcal G(n)}u_G,
\ee
where the sum is performed over all graphs $G$ in the set of all \emph{interaction graphs of depth $n$} denoted by $\mathcal G_n$, and where for each $G\in \mathcal G_n$,
$$
u_G=u_G^++u_G^{-},
$$
where $u_G^+$ and $u_G^-$ stand for the decomposition between positive and negative times, i.e. $u_G^+(t)={\bf 1 }(t\geq 0) u_G(t)$, and are given by:
\begin{align}
\label{id:formulaungraphs} u_G^+(t,\xi_R)& = e^{-it\omega (\xi_R)}  \left(\frac{-i\lambda}{(2\pi)^{d/2}}\right)^n  (-1)^{\sigma_G} \int_{\mathbb R^{d(2n+1)}}  \int_{\mathbb R_+^{n+1}} \, d \underline{\xi} \, d \underline{t}  \, \Delta_{\xi_R}( \underline{\xi}) \Delta_{t}( \underline{t})    \\
\nonumber &\qquad \qquad\qquad \qquad\qquad \qquad   \qquad   \qquad \qquad   \qquad M(\underline \xi)\prod_{i = 1}^{n+1} \widehat u_0(\xi_{0,i},\sigma_{0,i}) \prod_{v\in \mathcal V_i} e^{-i\Omega_v t_v},
\end{align}
and
\begin{align}
\label{id:formulaungraphs-} u_G^-(t,\xi_R)& = e^{-it\omega (\xi_R)}  \left(\frac{-i\lambda}{(2\pi)^{d/2}}\right)^n  (-1)^{\sigma_G+n} \int_{\mathbb R^{d(2n+1)}}  \int_{\mathbb R_+^{n+1}} \, d \underline{\xi} \, d \underline{t}  \, \Delta_{\xi_R}( \underline{\xi}) \Delta_{-t}( \underline{t})    \\
\nonumber &\qquad \qquad\qquad \qquad\qquad \qquad   \qquad   \qquad \qquad   \qquad M(\underline \xi)\prod_{i = 1}^{n+1} \widehat u_0(\xi_{0,i},\sigma_{0,i}) \prod_{v\in \mathcal V_i} e^{i\Omega_v t_v},
\end{align}
where we used the following notations:
\begin{itemize}

\item To each graph $G\in \mathcal G_n$ is associated a parity function $\sigma =\sigma(G)$. It determines the total parity of the graph $\sigma_G\in \mathbb N$ which records how many complex conjugations are taken in the interactions in the graph. To each vertex $v$, it associates a parity $\sigma_v \in {\pm 1}$. In particular, it determines $(\sigma_{0,i})_{1 \leq i \leq n+1}\in \{\pm 1\}^{n+1}$ which records, for each initial vertex, the parity (whether $u$ or $\overline{u}$ interacts).
\item $\underline{\xi} =  (\xi_{e})_{e\in \mathcal E}\in \mathbb R^{d(2n+1)}$ gathers all the frequency variables and determines $(\xi_{0,i})_{1\leq i \leq n+1}$.
\item $\underline{t} = (t_v)_{v\in \mathcal V_i\cup \mathcal V_R}\in \mathbb R_+^{n}$ gathers all the interaction time variables for each interaction vertex $v\in \mathcal V_i$, and the time variable of the root vertex.
\item $ \Delta_{\xi_R}( \underline{\xi})$ encodes the Kirchhoff laws of the graph.
\item $ \Delta_{t}( \underline{t})$ encodes the time constraints of the graph.
\item $M(\underline \xi)$ is a product of multipliers corresponding to $M$, i.e. to which form of the nonlinearity was taken.
\item $\widehat u_0(\xi,+1)=\widehat u_0(\xi)$ and $\widehat u_0(\xi,-1)= \widehat{ \overline u_0}(\xi) = \overline{\widehat u_0}(-\xi)$.
\item  The resonance modulus corresponding to the interaction vertex $v$ is:
\be \label{id:formulaOmegav}
\Omega_v=  \sigma_{v_a(v)} \omega(\xi_{v_a(v)}) - \sigma_{v_{l}(v)} \omega(\xi_{v_{l}(v)})  - \sigma_{v_{r}(v)} \omega(\xi_{v_r(v)})  .
\ee
\end{itemize}
The formulas \eqref{id:formulaungraphs} and \eqref{id:formulaungraphs-} are very similar. This is due to the following symmetry: if $u(t)$ solves  \eqref{nonlinschrod}, then $\bar u(-t)$ is also a solution. We will thus from now on focus on positive times and consider \eqref{id:formulaungraphs}, as adaptations for negative times are straightforward.

An example, treated in the next subsubsection \ref{subsubsec:ex}, will probably be most helpful. The precise definitions of all the objects above in \eqref{id:formulaungraphs} are given in subsubsections \ref{subsubsec:graph}, \ref{subsubsec:freq} and \ref{subsubsec:time}.

\subsubsection{Basic examples}\label{subsubsec:ex}

We give as an illustration the most basic graph that represents the Fourier transform of the function $\frac{-i\lambda}{(2\pi)^{d/2}} \int _0^t dt' e^{-i t'\omega(D)}M(M e^{i t'\omega(D)}u_0)^2$ evaluated at $(t,\xi_R)$ (where $t'$ is renamed as $t_{v_1}$):
\begin{center}
    \includegraphics[scale=0.75]{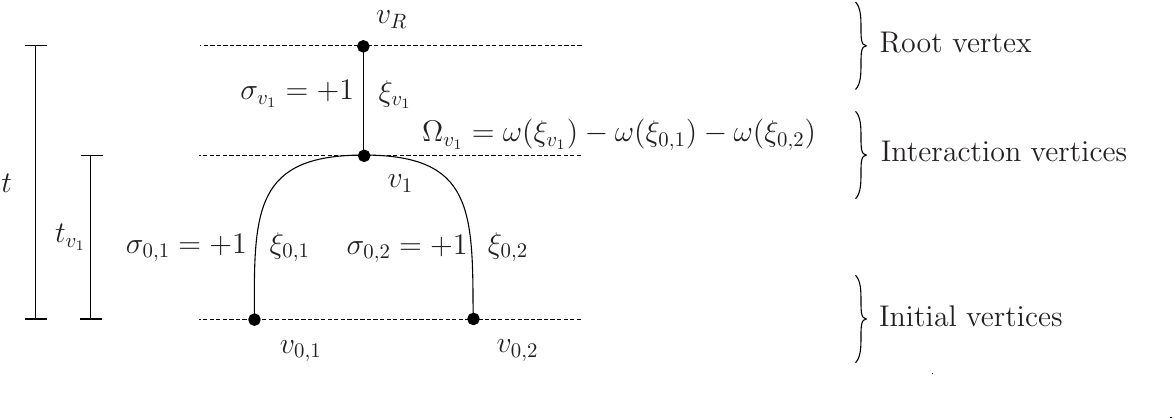}
\end{center}
It is one of the four elements in the sum $\sum_{G\in \mathcal G_1}$ in the formula \eqref{id:formulaungraphs} for $\hat u^1$. The three remaining elements, corresponding to the development $(Mu^0+M\overline{u^0})^2=(Mu^0)^2 +Mu^0M\overline{u^0}+M\overline{u^0}Mu^0+(M\overline{u^0})^2$, are represented by the graphs below:
\begin{center}
\includegraphics[width=9cm]{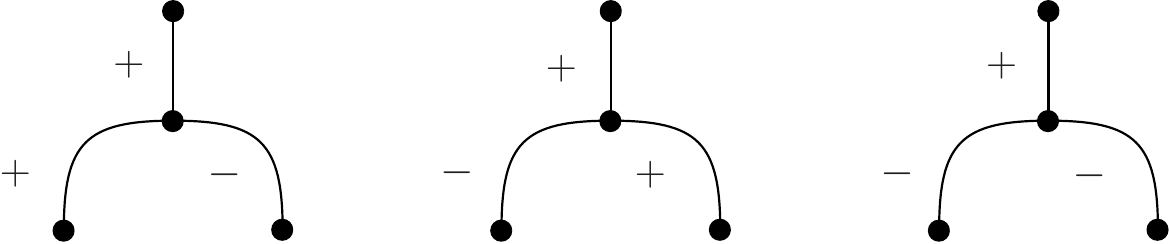}
\end{center}

\subsection{Solving time constraints} 

We present here a change of variables $\underline t\mapsto \underline s$ from time variables to \emph{time slices}, which is more suitable for understanding the interplay between the time constraint $\Delta_t$ and the oscillatory phases $e^{-it_v\Omega_v}$ in \eqref{id:formulaungraphs}.

\subsubsection{Maximal upright paths} \label{subsubsec:maximalpaths}

We study here specific paths that are used in the next subsubsection to solve the time constraints.

A path is said to be up and to the right, or \emph{upright}, if is a forward path $\mathcal p=(v_1,...,v_{\ell+1})$ whose vertices are all (except possibly the last one) at the bottom left of the vertex above them, namely $v_i=v_l(v_{i+1})$ for $i=1,...,\ell$. An upright path $\mathcal p=(v_1,...,v_{\ell+1})$ is said to be \emph{maximal} if it starts at an initial vertex $v_1\in \mathcal V_0$, and if it finishes at a vertex that is at the bottom right of the vertex above it: $v_{\ell+1}=v_r(v_a(v_{\ell+1}))$. The set of all maximal upright paths is denoted by $\mathcal P_{m}(G)$. The number of such paths is denoted by:
$$
n_{m}(G)=\# \mathcal P_{m}(G).
$$

For any $v\in \mathcal V_i$, there exists a unique maximal upright path $\mathcal p\in \mathcal P_m$ containing $v$. We denote it by $\mathcal p(v)$. By convention, we write $\mathcal p(v_R)=\{v_R\}$ (slightly abusing notations since $\{v_R\}$ is not a path). We denote the bottom and top parts of this maximal path at $v$ by:
$$
\mathcal p^+(v)=\{v'\in \mathcal p(v), \ v'\geq v \} \quad \mbox{and} \quad \mathcal p^-(v)=\{v'\in \mathcal p(v), \ v'\leq v \}.
$$

For any maximal path $\mathcal p=(v_1,...,v_{\ell+1}) \in \mathcal P_m$, we say that the vertex above the last vertex of the path, $v=v_a(v_{\ell+1})\in \mathcal V_i\cup \mathcal V_R$, is the \emph{junction vertex} of $\mathcal p$ and denote it by $v=v_j(\mathcal p)$. The set of all vertices that are junction vertices is denoted by $\mathcal V^j$. Note that $v_R\in \mathcal V^j$ for $n\geq 1$. Given a junction vertex $v\in \mathcal V^j$, we denote by $\mathcal p_j(v)$ the maximal upright path such that $v=v_j(\mathcal p_j(v))$.

 We say that a vertex $v\in \mathcal V_i\cup \mathcal V_R$ is \emph{constraining} $\mathcal p\in \mathcal P_m$ if it belongs to the upright path leading to $v_j(\mathcal p)$ which is equivalent to $v\in \mathcal p^-(v_j(\mathcal p))$ . We then write $\mathcal p \triangleleft v$. By convention, $v_R \triangleright \mathcal p(v_{\text{top}})$.
 
Below is an example of a interaction graph detailing its maximal upright paths, the vertices just above them, and which vertices constrain which maximal paths.\\

\begin{center}
\includegraphics[width=15cm]{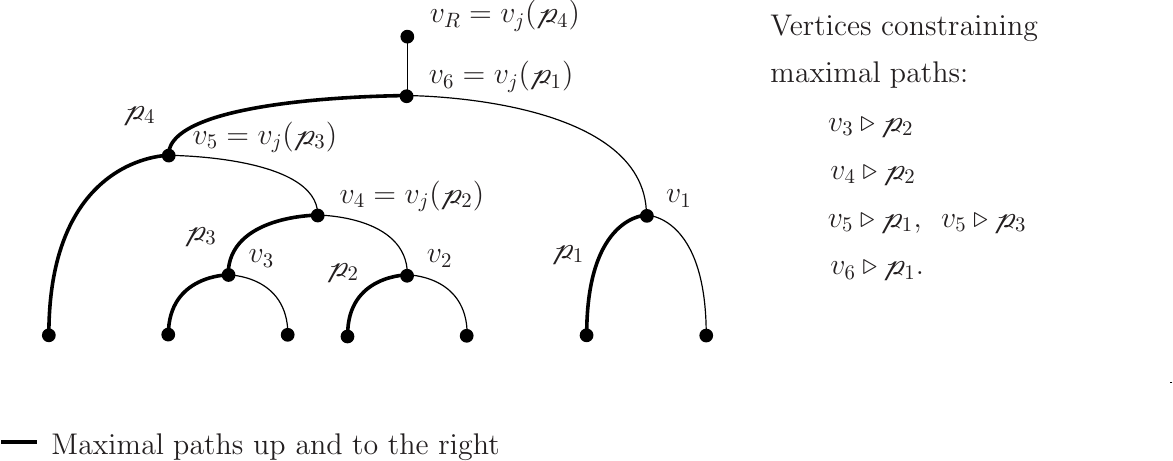}
\end{center}

\subsubsection{Solving the time constraints}

The time constraint function $\Delta_t$ is then completely determined by the \emph{maximal upright paths}.

To any edge $e=(v,v')$ that is to the left in the sense that $v=v_l(v')$, we associate a \emph{time slice} $s_e$. Time slices $s_v$ are equivalently associated to all vertices $v\in \mathcal V$ the following way:
\begin{itemize}
\item If $v\in \mathcal V_i$ then there exists a unique edge $e$ at its bottom left, which is $e=(v_l(v),v)$. We then associate to $v$ a time slice $s_v$ which is the same as that of $e$, i.e. $s_v=s_e$.
\item If $v\in \mathcal V_0$ then we set $s_v=0$.
\item If $v=v_R$ then we set $s_{v_R}=t$.
\end{itemize}
The set of all time slices of interaction vertices and of the root vertex is denoted by $\underline s=(s_v)_{v\in \mathcal V_i\cup  \mathcal V_R}$.

We impose that the time variables and the time slices of the interaction vertices satisfy the following compatibility condition. Given a vertex $v$, its time variable $t_v$ is equal to the sum of the time slices along the unique upright path leading to $v$:
$$
t_v=\sum_{\tilde v\in \mathcal p^-(v)} s_{\tilde v}, \qquad \mbox{for all }v\in \mathcal V.
$$

The time constraint function $\Delta_t (\underline t)$ imposes that $t_{v'}\leq t_{v}$ whenever $v'$ is below $v$. This is equivalent to the following condition for the time slices. Given any maximal upright path $\mathcal p \in \mathcal P_{m}$, and given its junction vertex $v_j(\mathcal p)$, then the sum of the time slices of $\mathcal p$ is less than or equal to the sum of the time slices of the upright path leading to $v_j(\mathcal p)$. This is written as:
$$
\Delta_t (\underline t)=1 \qquad \Leftrightarrow \qquad \sum_{v\in \mathcal p} s_v \leq \sum_{\tilde v \triangleright \mathcal p} s_{\tilde v}  \mbox{ for all }\mathcal p \in \mathcal P_m \ \mbox{ and } \ s_{v_R}=t.
$$
Note that, for the last maximal path $p(v_{\text{top}})$ whose initial vertex is $v_{0,1}$, the inequality above on the right means:
$$
\sum_{v\in \mathcal p(v_{\text{top}})} s_v \leq t.
$$
We can eventually define the \emph{time constraint function for time slices}, that we still denote by $\Delta_t[G](\underline s)$ with some slight abuse of notations, by:
\be \label{id:timeconstrainttimeslices}
\Delta_t (\underline s)=\delta(s_{V_R}-t) \prod_{\mathcal p \in \mathcal P_{m}} {\bf 1} \left(\sum_{v\in \mathcal p}s(v)\leq \sum_{\tilde v\triangleright \mathcal p}s(\tilde v) \right).
\ee
The oscillatory phases in the formula \eqref{id:formulaungraphs} are rewritten in terms of time slices as:
$$
e^{-it_v \Omega_v}=e^{-i \Omega_v\sum_{\tilde v\in \mathcal p^-(v)} s_{\tilde v}},
$$
so that the product of all oscillatory phases in the formula \eqref{id:formulaungraphs} is rewritten as:
\be \label{id:oscillatoryphasestimeslices}
\prod_{v\in \mathcal V_i} e^{-it_v \Omega_v}= \prod_{v\in \mathcal V_i} e^{-i s_v \sum_{\tilde v\in \mathcal p^+(v)}\Omega_{\tilde v}}.
\ee
We have now fully solved the time constraints of the graph, and can rewrite \eqref{id:formulaungraphs} as:
\begin{align}
\label{id:formulaungraphs2} \widehat{u_G^+}(t,\xi_R)& = e^{-it\omega (\xi_R)}  \left(\frac{-i\lambda}{(2\pi)^{d/2}}\right)^n  (-1)^{\sigma_G} \int_{\mathbb R^{d(2n+1)}}  \int_{\mathbb R_+^{n+1}} \, d \underline{\xi} \, d \underline{s} d s_{v_R} \, \Delta_{\xi_R}( \underline{\xi}) \Delta_{t}( \underline{s})    \\
\nonumber &\qquad \qquad \qquad\qquad \qquad   \qquad   \qquad \qquad   \qquad M(\underline \xi)\prod_{i = 1}^{n+1} \widehat u_0(\xi_{0,i},\sigma_{0,i})  \prod_{v\in \mathcal V_i} e^{-i s_v \sum_{\tilde v\in \mathcal p^+(v)} \Omega_{\tilde v}}.
\end{align}

Expressing the time constraint function $\Delta_t(\underline s)$ as a product of oscillatory integrals will be helpful later on. The following Lemma is a variant of  Lemma 4.2 in \cite{CG1}.

\begin{lemma} \label{lem:resolvantimproved}

There exists positive constants $c_G>0$, $c_v>0$ for $v\in \mathcal V_i$ and $c_{\mathcal p}>0$ for $\mathcal p\in \mathcal P_m$ such that for all $t\in \mathbb R$, $\eta>0$ and $(s_v)_{v\in \mathcal V_i}\in\mathbb R^n$:
$$
\int_{\mathbb R_+} ds_{v_R} \Delta_t (\underline s)= \frac{c_G^{ t \eta}}{(2\pi)^{n_m}} \int_{\mathbb R^{n_m}} d\underline{\alpha} e^{-i\alpha_{\mathcal p(v_{\text{top}})}t} \prod_{\mathcal p \in \mathcal P_{m}} \frac{i}{\alpha_{\mathcal p}+ic_{\mathcal p}\eta}\prod_{v \in \mathcal V_{i}} e^{s_v \left(i(\alpha_{\mathcal p(v)}-\sum_{\tilde{\mathcal p}\triangleleft v}\alpha_{\tilde{\mathcal p}})-c_v \eta\right)}
$$
where we wrote $\underline \alpha=(\alpha_{\mathcal p})_{\mathcal p \in \mathcal P_m}$.

\end{lemma}

\begin{proof}

We \emph{order} $\mathcal P_{m}(G)=(\mathcal p_1,...,\mathcal p_{n_{m}})$ from right to left with respect to the initial vertices of the paths. Namely, there exists $1=i_{n_m}<...<i_1\leq n$ such that $\mathcal p_{j}$ starts at $v_{0,i_j}\in \mathcal V_0$ (then $v_{0,i_j}$ is at the left of $v_{0,i_{k}}$ whenever $j>k$). Note then that the last maximal path $\mathcal p_{n_m}$ leads to the vertex $v_{\text{top}}$ that is just below the root vertex, so that $v_j(\mathcal p_{n_m})=v_R$ and that by convention $v_R$ is the only vertex constraining $\mathcal p_{n_m}$, i.e. $\{v\in \mathcal V, \ v\triangleright \mathcal p_{n_m}\}=\{v_R\}$.

Recalling the Fourier transformation ${\bf 1}(x\leq 0)e^{cx}=\frac{1}{2\pi} \int_{\alpha \in \mathbb R} \frac{ie^{i\alpha x}}{\alpha+ic} d\alpha$ for $c>0$, we write for each $\mathcal p \in \mathcal P_m$, for some $c_{\mathcal p}>0$ to be chosen later on:
$$
{\bf 1} \left(\sum_{v\in \mathcal p}s_v\leq \sum_{v'\triangleright \mathcal p}s_{v'} \right)= e^{c_{\mathcal p}\eta\left(\sum_{v'\triangleright \mathcal p}s_{v'}-\sum_{v\in \mathcal p}s_v \right)} \frac{1}{2\pi} \int_{\alpha_{\mathcal p} \in \mathbb R} \frac{ie^{i\alpha_{\mathcal p} \left( \sum_{v\in \mathcal p}s_v- \sum_{v'\triangleright \mathcal p}s_{v'} \right)}}{\alpha_{\mathcal p}+i c_{\mathcal p}\eta} d\alpha_{\mathcal p}.
$$
As for the last maximal path, $ \sum_{v'\triangleright \mathcal p}s(v')= s_{v_R}=t$, this leads to the formula:
$$
\int_{\mathbb R_+} ds_{v_R} \Delta_t (\underline s)= \frac{e^{c_{\mathcal p_{n_m}}\eta t}}{(2\pi)^{n_m}} \int_{\mathbb R^{n_m}} d\underline{\alpha} e^{-i\alpha_{\mathcal p(v_{\text{top}})}t} \prod_{\mathcal p \in \mathcal P_{m}} \frac{i}{\alpha_{\mathcal p}+ic_{\mathcal p}\eta} \prod_{v \in \mathcal V_{i}} e^{s_v \left(i(\alpha_{\mathcal p(v)}-\sum_{\mathcal p'\triangleleft v}\alpha_{\mathcal p'})-\eta \left(c_{\mathcal p(v)}-\sum_{\mathcal p'\triangleleft v}c_{\mathcal p'} \right) \right)}.
$$
Above, for the set of maximal paths $\mathcal P_m=(\mathcal p_1,...,\mathcal p_{n_m})$, it is always possible to choose the constants $c_{\mathcal p_1}$, $c_{\mathcal p_2}$, ... , $c_{\mathcal p_{n_m}}$ one after another to ensure $c_{\mathcal p(v)}-\sum_{\mathcal p'\triangleleft v}c_{\mathcal p'}>0$ for all $v\in \mathcal V$. This proves the Lemma upon taking $c_v=c_{\mathcal p(v)}-\sum_{\mathcal p'\triangleleft v}c_{\mathcal p'}$ for all $v\in \mathcal V$ and $c_G=e^{c_{\mathcal p_{n_m}}}$.
\end{proof}

Applying Lemma \ref{lem:resolvantimproved} to \eqref{id:formulaungraphs2} with $\eta=t^{-1}$, and then integrating along the $\underline s$ variables yields the alternative formula:
\begin{align}
\label{id:formulaungraphs3} \widehat{u_G^+}(t,\xi_R)& = e^{-it\omega (\xi_R)}  \left(\frac{-i\lambda}{(2\pi)^{d/2}}\right)^n \frac{(-1)^{\sigma_G}c_G}{(2\pi)^{n_m}}   \int_{\mathbb R^{d(2n+1)}}  \int_{\mathbb R^{n_m}} \, d \underline{\xi} \, d \underline{\alpha} \, \Delta_{\xi_R}( \underline{\xi}) e^{-i\alpha_{\mathcal p(v_{\text{top}})}t}  \\
\nonumber &\qquad  M(\underline \xi) \prod_{\mathcal p \in \mathcal P_{m}} \frac{i}{\alpha_{\mathcal p}+\frac{ic_{\mathcal p}}{t}} \prod_{i = 1}^{n+1} \widehat u_0(\xi_{0,i},\sigma_{0,i})  \prod_{v\in \mathcal V_i} \frac{i}{\alpha_{\mathcal p(v)}-\sum_{\tilde{\mathcal p}\triangleleft v}\alpha_{\tilde{\mathcal p}}-\sum_{\tilde{v}\in \mathcal p^+(v)}\Omega_{\tilde{v}}+\frac{ic_v}{t}}.
\end{align}

\subsection{Paired graphs}\label{pairinggraphs}

\subsubsection{General formula}

We will now take the expectation of the $L^2$ scalar product of two functions in the sum \eqref{id:formulaungraphs}, corresponding to two graphs $G^l\in \mathcal G_n$ and $G^r\in \mathcal G_n$. The left graph $G^l$ is described with variables with a $l$ superscript, and the right graph $G_r$ with a $r$ superscript. It will often be convenient to concatenate both kinds of variables, which we will denote without superscript for ease of notation. For instance, the set of interaction vertices is $\mathcal V_i=\mathcal V_i^l\cup \mathcal V_i^r$ and
\begin{align*}
& (\sigma_{0,i})_{1 \leq i \leq 2n+2} = (\sigma_{0,1}^l, \dots, \sigma^l_{0,n}, \sigma^r_{0,1}, \dots, \sigma_{0,n}^r),\\
& (\xi_{0,i})_{1 \leq i \leq 2n+2} = (\xi_{0,1}^l, \dots, \xi^l_{0,n}, \xi^r_{0,1}, \dots, \xi_{0,n}^r),
\end{align*}
and so on. Wick's formula and the Wigner transform identity \eqref{id:wignerfourier} imply that:
$$
\mathbb E \left( \prod_{i=0}^{2n+2} \widehat{u_0}( \xi_{0,i}, \sigma_{0,i}) \right)=\sum_{P} \prod_{\{i,j\}\in P} (2\pi)^{\frac d2}\ep^d \widehat{W_0^\epsilon} (\eta_{i,j},\frac{\ep}{2}( \sigma_{0,i}\xi_{0,i}+\sigma_{0,j} \xi_{0,j})) 
$$
where $\eta_{i,j}= \xi_{0,i}+\xi_{0,j}$, and $P$ is a \emph{pairing} of $\{ 1,...,2(n+1)\}$ that is consistent with $ \sigma$, that is, it is a partition of $\{1, \dots, 2n+1 \}$ into pairs $\{ i,j \}$, such that $\sigma_{0,i} = -\sigma_{0,j}$ for all $\{i,j\}\in P$. The sum above is performed over all possible pairings, and, by convention is equal to zero if no such pairing exists.

The formula corresponding to \eqref{id:uassumofinteractiongraphs} for the expectation of the $L^2$ norm of $u^n$ is for $t\geq 0$
\begin{equation}
\label{perrucheacollier}
\mathbb E \| u^n(t) \|_{L^2}^2 =\sum_{\tilde G,P} \mathcal F_t(\tilde G,P)
\end{equation}
where the sum is performed over all possible combinations of:
\begin{itemize}
\item  $\tilde G=(G^{l},G^r)$ is the tree $\tilde G$ which we now describe. It is composed of one left and one right sub-trees which have depth $n$, $G^l\in \mathcal G_n$ and $G^r\in \mathcal G_n$. The root vertices of the two sub-trees $v_R^l$ and $v^r_R$ are merged in a unique root vertex $v_R=v^l_R=v^r_R$. We use the convention that $\mathcal p_j(v_{\text{top}}^l) = \mathcal p_j(v_{\text{top}}^r) = v_R$, and that $e_l(v_R)$ and $e_r(v_R)$ do not belong to any upright path. Furthermore, the signs of the left sub-trees are flipped, and 
\item $P$ is a pairing of $\{ 1,...,2(n+1)\}$ that is consistent with $\sigma$. By convention, if no such pairing exists, the value of the corresponding empty sum is equal to zero.
\end{itemize}
Given a tree $\tilde G=(G^l,G^r)$ and a pairing $P$, we represent it as a \emph{paired graph} $G=(G^l,G^r,P)$. The set of all possible paired graphs $G$ is denoted by $\mathcal G^p_n$. Thus, the formula \eqref{perrucheacollier} corresponds to \eqref{sittelle}.

To do so, we add all the following to the tree $\tilde G$:
\begin{itemize}
\item \emph{a lower pairing vertex} $v_{-2,\{i,j\}} $ and \emph{an upper pairing vertex} $v_{-1,\{i,j\}} $ for each $\{i,j\}\in P$. They have no associated parities and time variables.
\item \emph{lower pairing edges} $(v_{-1,\{i,j\}},v_{-2,\{i,j\}})$ joining the two pairing vertices, and \emph{upper pairing edges} $(v_{-1,\{i,j\}},v_{0,i})$ and $(v_{-1,j},v_{0,i})$ joining the upper pairing vertex to the initial vertices $v_{0,i}$ and $v_{0,j}$, for all $\{i,j\}\in P$. To the edge $e=(v_{-2,\{i,j\}},v_{-1,\{i,j\}})$ we associate the frequency variable $\xi_e=\eta_{i,j}\in \mathbb R^d$. To the edges $e'=(v_{-1,\{i,j\}},v_{0,i})$ and $e''=(v_{-1,j},v_{0,i})$ we associate frequency variables $\xi_{e'}$ and $\xi_{e''}$ which will be forced by the Kirchhoff laws to be equal to $ \xi_{0,i}$ and $ \xi_{0,j}$ respectively. The pairing edges have no associated parities and time variables.
\item We require the output frequency is $0$. After integrating all Kirchhoff laws from bottom to top in the graph, we find that this output frequency is $\xi_{(v_{\text{top}}^l,v_R)}+\xi_{(v^r_{\text{top}},v_R)}=\sum_{\{i,j\}\in P}\eta_{i,j}$. We thus require that $\underline \eta \in \mathbb R^{d(n+1)}_0$. 

The Kirchhoff laws for frequencies are naturally extended to the paired graph:
$$
\Delta_G(\underline \xi,\underline{\eta})=\Delta_G(\underline{\xi^l},\underline{\xi^r},\underline{\eta})= \delta(\sum_{\{i,j\}\in P}\eta_{i,j})\Delta (\underline \xi^l) \Delta (\underline{\xi^r})\prod_{\{i,j\}\in P} \delta( \xi_{0,i}+\xi_{0,j}-\eta_{i,j}).
$$
 \end{itemize}

Explicitly:
\begin{align}
\nonumber \mathcal F_t(G)&= (2\pi)^{\frac d2} \lambda^{2n}\ep^{d(n+1)} \iiiint  \Delta_G (\underline{\xi},\underline{\eta})\Delta_t(\underline t^l)\Delta_{t}(\underline{t}^r)   d\underline{\xi} \, d \underline{\eta}  \, d\underline{t}^l \, d \underline{t}^r  \\
& \label{id:mathcalFGP} \qquad M(\underline \xi)\prod_{\{i,j\}\in P} \widehat{W_0^\epsilon} (\eta_{i,j},\frac{\ep}{2}(\sigma_{0,i}\xi_{0,i}+\sigma_{0,j} \xi_{0,j}))   \prod_{v\in G^l} e^{-i\Omega_v t_v}\prod_{v\in G^r} e^{-i\Omega_{v} t_{v}}
\end{align}
where $\underline{\xi}=(\underline \xi^l,\underline \xi^r)$, $M(\underline{\xi}) =M(\underline{\xi}^l)M(\underline{\xi}^r)  $, with $\underline{\xi}^l$ and $\underline t^l$ (resp. $\underline{\xi}^r$ and $\underline t^r$) being the frequency and time variables of the left subtree (resp. of the right subtree) which have been defined in the previous Subsection \ref{subsec:graphs}. The new variable $\underline{\eta} =  (\eta_{i,j})_{\{i,j\}\in P, \ i<j}$ comes from the Wigner transform identity \eqref{id:wignerfourier}.

The set of all maximal upright paths is denoted by $\mathcal P_m=\mathcal P_m^l\cup \mathcal P_m^r$ and the set of junction vertices by $\mathcal V^j=\mathcal V^{j,l}\cup \mathcal V^{j,r}$. Given $v\in G$ and $\mathcal p\in \mathcal P_m$, we say that $v$ is constraining $\mathcal p$ if either $(v,\mathcal p)\in G^l\times \mathcal P_m^l$ and $v$ is constraining $\mathcal p$ in the left subtree $G^l$, or if $(v,\mathcal p)\in G^r\times \mathcal P_m^r$ and $v$ is constraining $\mathcal p$ in the left subtree $G^r$ (recall that $v_R$ by convention belongs to both subtrees). We extend the notation and still write $v \triangleright \mathcal p$. We concatenate the time slices of both graphs: $\underline s=(s_v)_{v\in \mathcal V_i\cup \mathcal V_R}=(\underline s^l,\underline s^r)$. Injecting \eqref{id:oscillatoryphasestimeslices} in \eqref{id:mathcalFGP} yields:
\begin{align}
\nonumber \mathcal F_t(G)&= (2\pi)^{\frac d2} \lambda^{2n}\ep^{d(n+1)} \iiint  \Delta_G (\underline{\xi},\underline{\eta})\Delta_t(\underline s)  d\underline{\xi} \, d \underline{\eta}  \, d\underline{s} \\
& \label{id:mathcalFGP1} \qquad M(\underline \xi)\prod_{\{i,j\}\in P} \widehat{W_0^\epsilon} (\eta_{i,j},\frac{\ep}{2}(\sigma_{0,i}\xi_{0,i}+\sigma_{0,j} \xi_{0,j}))    \prod_{v\in \mathcal V_i} e^{-i s_v \sum_{\tilde v\in \mathcal p^+(v)} \Omega_{\tilde v}}.
\end{align}
where we $\Delta_t(\underline s)$ is still given by \eqref{id:timeconstrainttimeslices} but defined with the maximal paths of the paired graph $G$. We apply the resolvent formula of Lemma \ref{lem:resolvantimproved}, to both the left and right subtree, and concatenate the variables by writing: $\underline \alpha=(\underline \alpha^l,\underline \alpha^r)$ and the identity \eqref{id:mathcalFGP} becomes

\begin{align}
&\label{id:mathcalFGP2}\mathcal F_t(G)= \frac{(-1)^{\sigma_{G^l}+\sigma_{G^r}}c_{G^l}c_{G^r}}{(2\pi)^{n_m^l+n_{m}^r-\frac d2}} \lambda^{2n}\ep^{d(n+1)} \iiint  \,d \underline{\xi} \, d \underline{\eta} \,  d\underline{\alpha} \Delta_G(\underline{\xi},\underline{\eta}) \\
\nonumber & \qquad \qquad  \qquad \qquad \qquad \qquad \qquad e^{-i(\alpha_{\mathcal p(v^l_{\text{top}})}+\alpha_{\mathcal p(v^r_{\text{top}})})t}  M(\underline \xi)  \prod_{\{i,j\}\in P} \widehat{W_0^\epsilon} (\eta_{i,j},\frac{\ep}{2}(\sigma_{0,i} \xi_{0,i}+  \sigma_{0,j} \xi_{0,j})) \\
\nonumber &  \qquad \qquad \qquad \qquad \qquad \qquad \qquad \prod_{\mathcal p \in \mathcal P_{m}} \frac{i}{\alpha_{\mathcal p}+\frac{ic_{\mathcal p}}{t}}  \prod_{v\in \mathcal V_i} \frac{i}{\alpha_{\mathcal p(v)}-\sum_{\tilde{\mathcal p}\triangleleft v}\alpha_{\tilde{\mathcal p}}-\sum_{\tilde{v}\in \mathcal p^+(v)}\Omega_{\tilde{v}}+\frac{ic_v}{t}}
\end{align}

\subsubsection{Example} \label{subsubsec:examplepairedgraph} Below is an example of a paired graph. The pairing is $P=\{\{1,2\},\{3,5\},\{4,6\}\}$.
\begin{center}
    \includegraphics[scale=0.6]{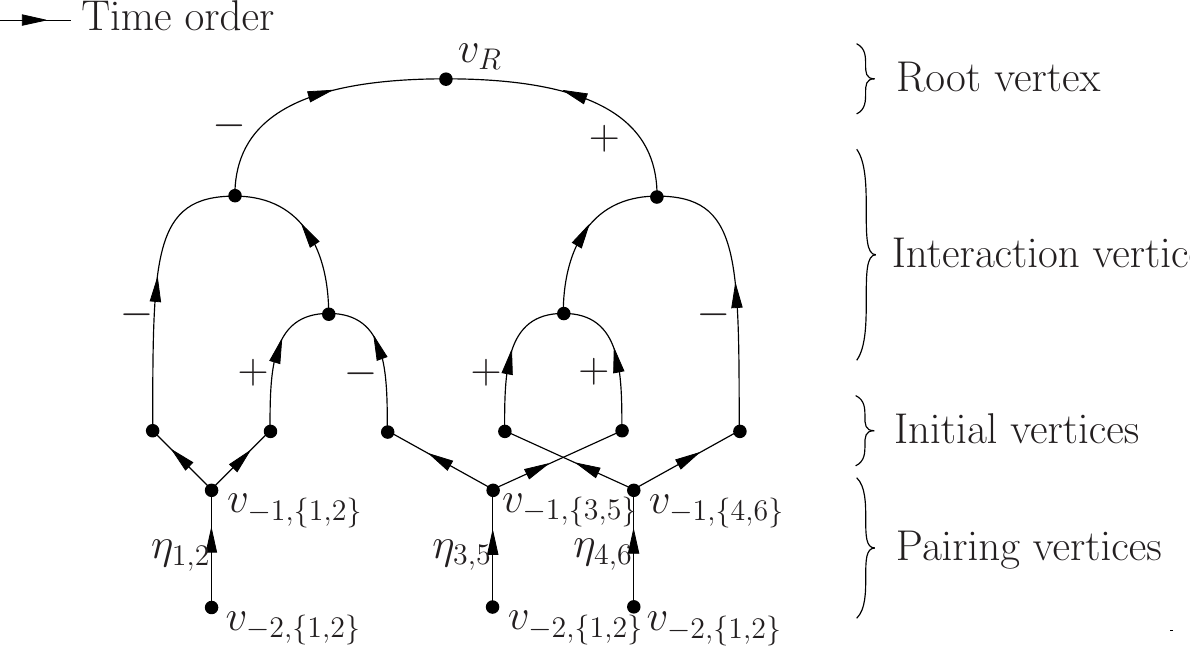}
\end{center}

\subsubsection{Time and integration orders} Given a paired graph $G=(G^l,G^r,P)\in \mathcal G_n^p$, given two $v,v'\in \mathcal V$, we say that $v$ is below $v'$ if there exists a forward path in $G$ going from $v$ to $v'$. This defines an order for $\mathcal V$, still called the \emph{time order}. It extends the time orders of $G^l$ and $G^r$.

When we will estimate integrals of the type \eqref{id:mathcalFGP}, we will consider the contribution of each oscillatory phases in the right hand side of \eqref{id:oscillatoryphasestimeslices} one after another, according to an integration order that we now describe.

An \emph{integration order} for a paired graph $G$ is an enumeration of the set of interaction vertices and of the root vertex $\mathcal V_i\cup \mathcal V_R=\{v_1,v_2,...,v_{2n+1}\}$ such that for all $1 \leq i<j\leq 2n+1$ the vertex $v_i$ cannot be above $v_j$. This property is equivalent to the fact that for all $1\leq i \leq 2n+1$, the set $\{v_1,...,v_{i-1}\}$ contains all the vertices that are below $v_i$. Moreover, $v_{2n+1}=v_R$ is always the root vertex, and $v_{2n}\in \{v_{\text{top}}^l,v_{\text{top}}^r\}$ is the top vertex of either the left or the right subtree. There always exists at least one integration order. For all paired graphs $G\in \mathcal G_n^p$, we \emph{fix once for all a unique integration order that will be used throughout the article}. The picture in the proof of Proposition \ref{pr:spanning} shows an example of an integration order.

We extend this integration order to the set of edges and frequencies. Given two edges $e,e'\in \mathcal E$, we say that $e$ is \emph{after} $e'$ for the integration order if one of the following holds true:
 \begin{itemize}
  \item $e$ is any edge and $e'$ is a pairing edge.
 \item neither $e$ nor $e'$ is a pairing edge, and, writing $e=(v,v_a(v))$ and $e'=(v',v_a(v'))$, either they are below the same vertex $v_a(v)=v_a(v')$, or the top vertex $v_a(v)$ of $e$ is after the top vertex $v_a(v')$ of $e'$ for the integration order of $G$. In the case $v_a(v)\neq v_a(v')$ we say that $v$ is strictly after $v'$.
 \end{itemize}
 We extend this terminology for frequencies and say that $\xi_{e}$ is \emph{after} $\xi_{e'}$ for the integration order whenever $e$ is after $e'$ for the integration order.

\subsection{Solving the frequency constraints} \label{subsec:kirchhoff}

We aim at understanding how to integrate over the variables $(\underline \xi,\underline \eta) $ on the support of the Kirchhoff laws function $\Delta_G$ (which encodes Kirchoff's law), in a way which is takes advantage of the oscillations of the functions $e^{it_v\Omega_v}$. Proposition \ref{pr:spanning} provides a suitable subset of the frequencies $\underline \xi$, the \emph{interaction free frequencies} $(\xi^f_i)_{1\leq i \leq n+1}$, from which all frequencies $\underline \xi$ can be recovered. Moreover, the phases $e^{it_v\Omega_v}$ have an expression that is suitable with the ordering $\xi^f_1,...,\xi^f_{n+1}$, see Lemmas \ref{lem:deg1}.

\begin{proposition} \label{pr:spanning}

For any paired graph $G\in \mathcal G^p_n$ with integration order $\{ v_1,...,v_{2n+1}\}$, there exists an associated complete integration of the frequency constraints $\Delta_G$ in the following sense. There exists a set of \emph{free edges} $\mathcal E^f=\{(v_{-2,\{i,j\}},v_{-1,\{i,j\}})\}_{\{i,j\}\in P}\cup \{e_1^f,...,e^f_n,e^f_{n+1}\}$ consisting of all pairing edges, of a sequence of \emph{interaction free edges} $\{e_1^f,...,e^f_n,\}\subset \mathcal E$ and of the root edge of the left subtree $e^f_{n+1}=(v_{\text{top}}^l,v_R)$, with corresponding \emph{slow free frequencies} $\underline{\eta}$ and \emph{interaction free frequencies} $\underline{\xi^f}=(\xi^f_i)_{1\leq i \leq n+1}$ where $\xi^f_i= \xi_{e^f_i}$ for $1\leq i \leq n+1$, such that the following properties hold true.

\begin{itemize}
\item \emph{Order compatibility with integration order}. For all $1\leq i <j\leq n+1$, $e^f_j$ is strictly after $e^f_i$ for the integration order (in other words, $v_a(e^f_j)$ is different from $v_a(e^f_i)$, and posterior for the integration order).
\item \emph{Basis property}: The family $(\underline{\xi^f},\underline{\eta}) $ is a basis for the Kirchhoff laws in the following sense: the map $(\underline \xi,\underline \eta)\to (\underline{\xi^f},\underline{\eta}) $, with domain the support of $\Delta_G$, is a linear bijection onto $ \mathbb R^{d(n+1)}\times\mathbb R^{d(n+1)}_0$.
\item \emph{Basis compatibility with integration order:} Any edge which is not a free edge, i.e. $e\notin \mathcal E^f$, is called an \emph{integrated edge} and, on the support of $\Delta_G$,
$$
\xi_{e}=\sum_{1\leq k \leq n+1} c_{e,k} \xi^f_k+\sum_{\{i',j'\}\in P} c_{e,i',j'} \eta_{i',j'}  \quad \mbox{with} \quad c_{e,k},c_{e,i',j'} \in \{-1,0,1\},
$$
with $c_{e,k}=0$ if $\xi_{e}$ appears strictly after $ \xi^f_k$ for the integration order.
\end{itemize}

\end{proposition}

\begin{proof}

This proof is very similar to that of Theorem 4.3. in \cite{CG1} (inspired by \cite{LS}). Thus we only sketch the proof, and refer to \cite{CG1} for the details. We construct iteratively the \textit{spanning tree} $G^s$, whose set of edges is the set of all integrated edges\footnote{With a slight abuse of notation since the edges of $G^s$ are unoriented} $\mathcal E\backslash \mathcal E^f$. Its edges are (for the moment) unoriented, so we write them under the form $\{v,v'\}$. The construction algorithm is as follows: first, at Step 0, add all upper pairing edges $\{\{v_{-1,\{i,j\}},v_{0,i}\} \}_{1\leq i \leq 2+2n}$ to the spanning tree under construction $G^{s,0}$. Then, at Steps 1 to $2n$ consider the interaction vertices one by one, according to the integration order: first $v_1$, then $v_2$, etc. until $v_{2n}$. 

\vspace*{0.5cm}
\begin{center}
\includegraphics[width=14cm]{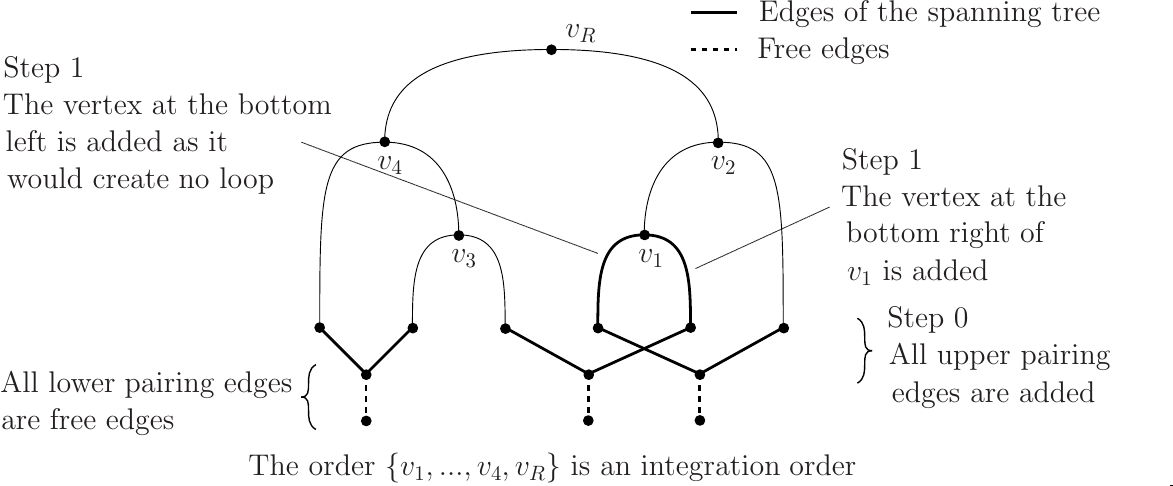}
\end{center}
\vspace*{0.5cm}

At the beginning of Step $k$, we have constructed $G^{s,k-1}$ and we reach $v_k$. We first add the edge on the right below $v_k$, which is $e_r(v_k)$. Next, we consider the edge on the left below $v_k$ which is $e_l(v_k)$: if adding it creates a loop in the spanning tree under construction, then we do not add this edge and declare it to be a free interaction edge; if adding it does not create a loop, then we add it to the spanning tree. With these additions the spanning tree under construction is renamed $G^{s,k}$ and we move on to the next vertex $v_{k+1}$ and start Step $k+1$.

At the last Step $2n+1$, we add the edge on the right below the root vertex $\{v^r_{\text{top}},v_R\}$ to the spanning tree, and we do not add the edge on its left $\{v^l_{\text{top}},v_R\}$ that we declare to be a free edge. The graph obtained at this last Step is the spanning tree $G^s$.

\vspace*{0.5cm}
\begin{center}
\includegraphics[width=14cm]{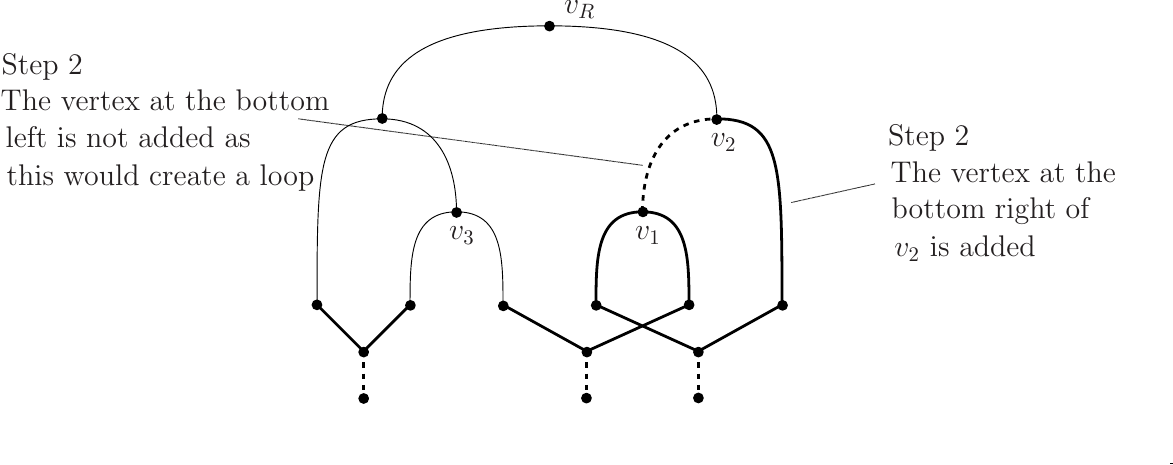}
\end{center}
\vspace*{0.5cm}

The spanning tree is indeed a tree, since it has no loop by construction. A \emph{path} in $G^s$ is a sequence $(v_1,...,v_k)$ of vertices such that $v_i\neq v_j$ for $i\neq j$, and that for each $1\leq i \leq k-1$, $\{v_i,v_{i+1}\}$ is an edge of $G^s$. Each vertex is then connected to the root vertex by a unique path. 

We define an orientation for $G^s$ as follows: an integrated edge $e=\{v,v' \}$ goes from $v$ to $v'$ if $v'$ belongs to the path from $v$ to the root vertex. This also defines a partial order: we say that $u \preceq w$ if $w$ belongs to the path from $u$ to the root vertex; in particular, $u \preceq u$. We denote by $\mathcal P(u)=\{w, \ w\preceq u\}$ the set of vertices $w$ such that $u$ belongs to the path from $w$ to the origin. 
\vspace*{0.5cm}
\begin{center} 
\includegraphics[width=15cm]{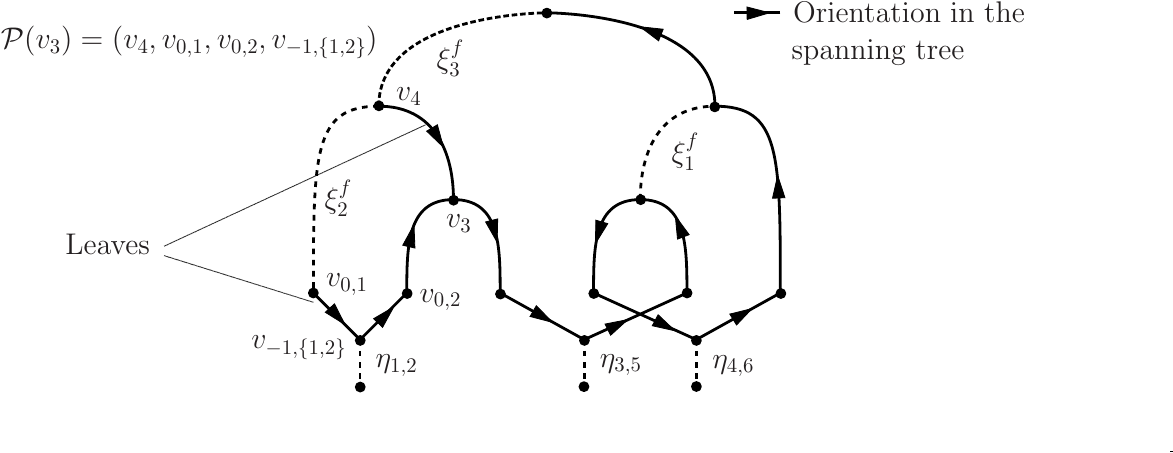}
\end{center}
\vspace*{0.5cm}
Frequencies of integrated edges are expressed in function of free frequencies. Given an (oriented) edge $e=(v,v')$, and $v$ we define the parity of the edge with respect to the vertex $v$ as
$$
\sigma_v(e)=\left\{\begin{array}{l} +1 \quad \mbox{if $v'$ is above $v$ for the time ordering}, \\-1\quad \mbox{if $v'$ is below $v$ for the time ordering}.\end{array} \right.
$$
Given a vertex $v$, $\mathcal F(v)$ denotes the set of free edges $f$ that have one extremity at $v$. Given $e=\{v,v'\}$ an integrated edge going from $v$ to $v'$, on the support of $\Delta_G$, the formula for its associated frequency is then
\be \label{id:integrationmomenta}
\xi_e= -\sigma_{v}(e) \sum_{w\in \mathcal P(v) ,\;f\in \mathcal F(w)} \sigma_w(f)\xi_f.
\ee
To finish the proof of Proposition \ref{pr:spanning}, we need to show that if $e$ is an integrated edge, then it is only a linear combination of the slow free frequencies $\underline \eta$ and of the interaction free frequencies $\xi^f_i$ appearing after $e$ for the integration order of $G$. Assume $f=\{ v',v\}$ is a free edge, with $v'$ below $v$ for the time ordering. This means that during the construction of the spanning tree, at the step where the vertex $v$ is considered, $f$ is not added as this would create a loop in the spanning tree in construction. At that step, all edges in the spanning tree are before $f$ for the integration ordering. Hence there exists a path $\widetilde p$ in the spanning tree, going from $v$ to $v'$, and all its edges are before $f$ for the time ordering. Also, there exist unique paths $p$ and $p'$ in the spanning tree, going from $v$ to the root and from $v'$ to the root respectively. These paths intersect at a vertex $v_0$. By their uniqueness, $v_0$ has to belong to $\widetilde p$. Consider now the formula above: $k_f$ can only appear in the integrated frequencies on the paths from $v$ and $v'$ to the root. Moreover, after the vertex $v_0$, the two contributions from $v$ and $v'$ in this formula cancel. Hence $k_f$ can only appear in the integrated frequencies on the path from $v$ to $v_0$, and in the integrated frequencies on the path from $v'$ to $v_0$. These belong to $\widetilde p$ hence are indeed before $k_f$ for the time ordering. This also shows that $c_{i,e} \in \{ -1,0,1 \}$.
\end{proof}

If an interaction vertex $v\in \mathcal V_i$ is such that $(v_l(v),v)$ is a free edge, then we say that $v$ is a \emph{degree one vertex}. If not, we say that $v$ is a \emph{degree zero} vertex. The sets of degree zero and degree one vertices are denoted by $\mathcal V^1$ and $\mathcal V^0$ respectively.

Let then $n_0$ and $n_1$ denote the number of degree $0$ and $1$ vertices respectively. On the one hand, the total number of interaction vertices is $2n$, so that
\begin{align*}
n_0 + n_1 = 2n;
\end{align*}
and on the other hand, the total number of interaction free variables apart from $\xi_{n+1}^f$ is $n$, so that $n_1 = n$. Therefore,
\begin{equation}
\label{perdrix}
n_0 = n_1 = n.
\end{equation}
Let $v$ be a degree one vertex. We say that it is \emph{linear} if the two vertices below it have opposite parity: $\sigma(v_l(v))\sigma(v_r(v))=-1$, and that it is \emph{quadratic} if they have the same parity $\sigma(v_l(v))\sigma(v_r(v))=+1$. The sets of degree one linear vertices and degree one quadratic vertices are denoted by $\mathcal V^1_l$ and $\mathcal V^0_q$ respectively.

\begin{lemma}[Degree one linear and quadratic vertices] \label{lem:deg1} 

Assume $v$ is a degree one vertex, with associated free frequency $\xi^f$, and denote by $\tilde \xi=\xi_{e_a(v)}$ the frequency of the edge above it.
\begin{itemize}
\item If $v$ is linear, then:
\be \label{id:formularesonancelinear}
\Omega_v = -\sigma( \xi^f)\tilde \xi \cdot \xi^f \ + \ \left\{\begin{array}{l l l}\frac{1}{2}(\sigma(\tilde \xi)+\sigma(\xi^f))|\tilde \xi|^2 & \mbox{if }\omega(\xi)=\frac{|\xi|^2}{2}, \\ \sigma(\tilde \xi)\epsilon^{-2}+\frac{1}{2}(\sigma(\tilde \xi)+\sigma(\xi^f))|\tilde \xi|^2 & \mbox{if }\omega(\xi)=\frac{|\xi|^2}{2}+\epsilon^{-2} \end{array} \right.
\ee
\item If $v$ is quadratic, then:
\be \label{id:formularesonancequadra}
\Omega_v =-\sigma(\xi^f)\xi^f \cdot (\xi^f-\tilde \xi) \ + \ \left\{\begin{array}{l l l} \frac{1}{2}(\sigma(\tilde \xi)-\sigma(\xi^f))|\tilde \xi|^2 & \mbox{if }\omega(\xi)=\frac{|\xi|^2}{2}, \\  (\sigma(\tilde \xi)-2\sigma (\xi^f))\epsilon^{-2}+\frac{1}{2}(\sigma(\tilde \xi)-\sigma(\xi^f))|\tilde \xi|^2 & \mbox{if }\omega(\xi)=\frac{|\xi|^2}{2}+\epsilon^{-2} \end{array} \right.
\ee
\item Moreover, in the two formulas above, $\tilde \xi$ only depends on the slow free variables $\underline \eta$ and on the interaction free variables $\xi^f_i$ appearing strictly after $\xi^f$ for the integration order.
\end{itemize}

\end{lemma}

\begin{proof}

For a degree one vertex, one has that $\xi_{e_r(v)}=\tilde \xi-\xi^f$ from the Kirchhoff law at $v$.

 If $v$ is linear, then $\sigma(v_l(v))=\sigma(\xi^f)=-\sigma(v_r(v))$ by definition and the formula \eqref{id:formulaOmegav} gives:
$$
\Omega_v  = \sigma(\xi^f) ( \omega(\tilde \xi-\xi^f)- \omega(\xi^f) )+  \sigma(\tilde \xi) \omega(\tilde \xi) .
$$
Plugging $\omega(\xi)=\frac{|\xi|^2}{2}$ or $\omega(\xi)=\epsilon^{-2}+\frac{|\xi|^2}{2}$ in the above formula yields \eqref{id:formularesonancelinear}.  If $v$ is quadratic, then $\sigma(v_l(v))=\sigma(\xi^f)=\sigma(v_r(v))$ by definition and the formula \eqref{id:formulaOmegav} gives:
$$
\Omega_v  =-\sigma(\xi^f)(\omega(\xi^f)+\omega(\tilde \xi-\xi^f))+\sigma(\tilde \xi)\omega(\tilde \xi).
$$
Plugging $\omega(\xi)=\frac{|\xi|^2}{2}$ or $\omega(\xi)=\epsilon^{-2}+\frac{|\xi|^2}{2}$ in the above formula yields \eqref{id:formularesonancequadra}.

\end{proof}

\begin{proof}[Proof of Proposition \ref{pr:formulagraph} ]

For all $G\in \mathcal G^p_n$, we choose fix an arbitrary integration order $\mathcal V_i\cup \mathcal V_R=\{v_1,...,v_{2n+1}\}$ and we define $\sigma_k=\sigma((v_l(v_k),v_k))$, $\tilde \sigma_k=\sigma((v_k,v_a(v_k)))$ and $\tilde \xi_k =\xi_{e_a(v_k)}$ for $1\leq k\leq 2n+1$. Then Proposition \ref{pr:formulagraph} is a direct consequence of the formulas \eqref{perrucheacollier} and \eqref{id:mathcalFGP1} which yields \eqref{chouettehulotte} upon applying Proposition \ref{pr:spanning} and Lemma \ref{lem:deg1}.

\end{proof}

Finally, let us mention that if one applies the resolvent identity of Lemma \ref{lem:resolvantimproved} to \eqref{chouettehulotte}, and we define $\sigma_G=\sigma_{G^l}+\sigma_{G^r}$, $n_m=m_m^l+n_m^r$ and $c_G=c_{G^l}c_{G^r}$ and simply write $c_k=c_{v_k}>0$ we obtain:
\begin{align}
&\label{id:mathcalFGP3} \mathcal F_t(G)= \frac{(-1)^{\sigma_{G}}c_G}{(2\pi)^{n_m-\frac d2}} \lambda^{2n}\ep^{d(n+1)} \int_{\underline \eta \in \mathbb R^{d(n+1)}_0}  \int_{\underline \xi^f \in \mathbb R^{d(n+1)}}  \int_{\underline \alpha \in \mathbb R^{n_m}}  d\underline{\xi^f} \, d \underline{\eta}  \, d\underline{\alpha} \\
\nonumber & \qquad \qquad  \qquad \qquad \qquad \qquad \qquad  e^{-i(\alpha_{\mathcal p(v^l_{\text{top}})}+\alpha_{\mathcal p(v^r_{\text{top}})})t} M(\underline \xi)  \prod_{\{i,j\}\in P} \widehat{W_0^\epsilon} (\eta_{i,j},\frac{\ep}{2}(\sigma_{0,i} \xi_{0,i}+  \sigma_{0,j} \xi_{0,j})) \\
\nonumber &  \qquad \qquad \qquad  \qquad \qquad \qquad \qquad \qquad  \prod_{\mathcal p \in \mathcal P_{m}} \frac{i}{\alpha_{\mathcal p}+\frac{ic_{\mathcal p}}{t}}  \prod_{k=1}^{2n} \frac{i}{\alpha_{\mathcal p(v_k)}-\sum_{\tilde{\mathcal p}\triangleleft v_k}\alpha_{\tilde{\mathcal p}}-\sum_{\tilde{v}\in \mathcal p^+(v_k)}\Omega_{\tilde{v}}+\frac{ic_k}{t}}
\end{align}

\section{The belt counter example}

We prove here Proposition \ref{prop:belt}. Throughout this section we study equation \ref{nonlinschrod} with the Laplace dispersion relation and $m(\xi)=1$:
\be \label{eq:usualNLS}
\left\{ \begin{array}{l}
i \partial_t u - \frac 12 \Delta u = \lambda ( u + \overline{u} )^2,\\
u(t=0)=u_0.
\end{array}\right.
\ee
Before proceeding the proof, we describe the paired graph $G^*$, and give a formula for $\mathcal F_t(G^*)$. The graph $G^*$ is made of a left subtree with unprimed variables, and of a right subtree with primed variables.
\begin{itemize}
\item The interaction vertices are $v_1,...,v_{2n},v_1',...,v_{2n}'$, and $v_{0,0},...,v_{0,2n},v_{0,0}',...,v_{0,2n}'$ are the initial vertices. There is the root vertex $v_R$.
\item The interaction and initial vertices are linked by the following edges. For $k=1,...,n$ there is an edge $(v_{2k-2},v_{2k-1})$ with parity $-1$ (with the convention that $v_0$ stands for $v_{0,0}$) and an edge $(v_{0,2k-1},v_{2k-1})$ with parity $+1$, and $(v_{2k-2}',v_{2k-1}')$ with parity $+1$ and $(v_{0,2k-1}',v_{2k-1}')$ with parity $-1$. There is $(v_{2k-1},v_{2k})$ with parity $+1$, $(v_{0,2k},v_{2k})$ with parity $-1$, $(v_{2k-1}',v_{2k}')$ with parity $-1$ and $(v_{0,2k}',v_{2k}')$ with parity $+1$.

There are two edges $(v_{0,2n},v_R)$ with parity $-1$, and  $(v_{0,2n}',v_R)$ with parity $+1$.

\item The pairing $P$ is defined as follows: $v_{0,0}$ is paired with $v_{0,0}'$ with slow free variable $\tilde \eta$, and for $k=1,...,n$, $v_{0,2k-1}$ is paired with $v_{0,2k}$ with variable $\eta_k$, and $v_{0,2k-1}'$ is paired with $v_{0,2k}'$ with variable $\eta_k'$.

\item The resonance modulus at $v_k$ is $\Omega_k$ (resp. at $v_k'$ is $\Omega_k'$). The time slice at $v_k$ is $s_{k-1}$ (resp. at $v_k'$ is $s_{k-1}'$).
\end{itemize}

The integration order we choose is $(v_1',...,v_{2n}',v_1,...,v_{2n},v_R)$. We apply Proposition \ref{pr:spanning} to determine the free variables. This corresponds to the following paired diagram:

\begin{center}
\includegraphics[width=15cm]{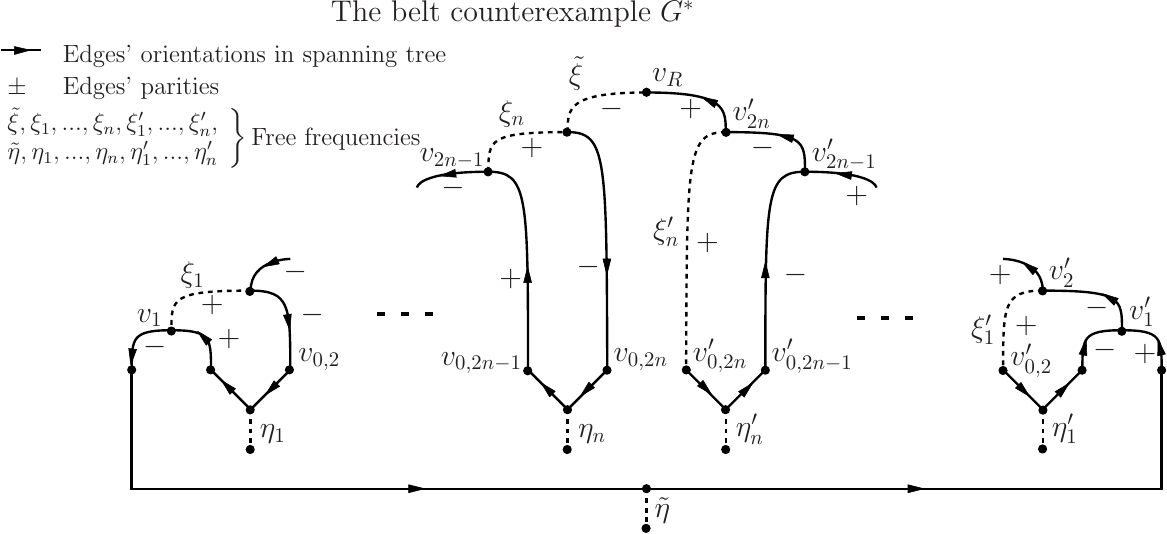}
\end{center}

Note that the left subtree is equivalent to the right subtree with reversed parity signs (up to changing the display of the edges and vertices). Hence $G^*$ is indeed a paired graph as defined in Subsection \ref{pairinggraphs}. We have chosen this representation for convenience.

Above, the free variables are indicated by dashed lines for the corresponding edges. However, in order to find a suitable formula for $\mathcal F_t(G^*)$, we change certain variables $(\tilde \xi,\xi_1',...,\xi_n',\eta_1',...,\eta_n')\mapsto (\bar \xi,\hat \xi_1,...,\hat \xi_n,\hat \eta_1,...,\hat \eta_n)$ where:
\begin{align*}
&\hat \eta_k=-\eta_k', \qquad \bar \xi=-\tilde \xi+\eta_1+...+\eta_n +\frac{\tilde \eta}{2},\\
&\hat \xi_k=\xi_k'+\tilde \xi-\eta_1-...-\eta_n-\eta'_1-...-\eta_k'-\tilde \eta.
\end{align*}
After a direct computation of Kirchhoff's laws, one finds the following values for the resonance moduli with respect to these new variables, for $k=1,...,n$:
\begin{align*}
&\Omega_{2k-1}=-\left(\bar \xi-\frac{\tilde \eta}{2}-\eta_1-...-\eta_{k+1}\right).\xi_k, \qquad \Omega_{2k}=\left(\bar \xi-\frac{\tilde \eta}{2}-\eta_1-...-\eta_{k}\right).\xi_k\\
&\Omega_{2k-1}'=\left(\bar \xi+\frac{\tilde \eta}{2}-\hat \eta_1-...-\hat \eta_{k-1}\right).\hat \xi_k, \qquad \Omega_{2k}=-\left(\bar \xi+\frac{\tilde \eta}{2}-\hat \eta_1-...-\hat \eta_{k}\right).\hat \xi_k,
\end{align*}
with the convention that $\eta_{-1}=\hat \eta_{-1}=0$. The sum of cumulated resonance moduli is thus for $k=1,...,n$:
\begin{align*}
&\sum_{i=2k-1}^{2n} \Omega_i= -\sum_{i=k}^n \xi_j.\eta_j, \qquad \sum_{i=2k}^{2n} \Omega_i=\left(\bar \xi-\frac{\tilde \eta}{2}-\eta_1-...-\eta_k\right).\xi_k -\sum_{i=k+1}^n \xi_j.\eta_j,\\
&\sum_{i=2k-1}^{2n} \Omega_i'= \sum_{i=k}^n \hat \xi_j.\hat \eta_j,\qquad \sum_{i=2k}^{2n} \Omega_i'=-\left(\bar \xi+\frac{\tilde \eta}{2}-\hat \eta_1-...-\hat \eta_k\right).\hat \xi_k -2\sum_{i=k+1}^n \hat \xi_j.\hat \eta_j.
\end{align*}
We introduce the following notation to ease the computations:
$$
\mathcal S_{n,t}=\left\{ (s_0,...,s_{2n-1},s_0',...,s_{2n-1}')\in \mathbb R_+^{n}\times \mathbb R_+^{n}, \quad \sum_{i=0}^{2n-1} s_i\leq t \mbox{ and } \sum_{i=0}^{2n-1} s_i'\leq t \right\}.
$$
One thus ends up with the formula:
\begin{align*}
& \mathcal F_t(G^*) = 2\pi \lambda^{4n}\ep^{d(2n+1)} \int_{\mathcal S_{n,t}}  \int_{\mathbb R^{d(4n +2)}} \,d \bar \xi \,d\tilde \eta   \,d \underline{\xi} \, d\underline{\hat \xi} \, d \underline{\eta} \, d \underline{\hat \eta} \, d\underline{s} \, d \underline{s}'  \delta \left(\tilde \eta +\sum_{1}^{n} \eta_i- \hat \eta_i\right) \\
& \qquad  \qquad \qquad  \qquad  \qquad \qquad  \qquad  \qquad  \hat W^\e_0 \left(\tilde \eta,\ep \bar \xi\right)\prod_{k=1}^{n} \hat W^\e_0 \left(\eta_k,\e(\bar \xi -\frac{\tilde \eta}{2}+\xi_k-\eta_1-...-\eta_{k-1}-\frac{\eta_k}{2}) \right)\\
&  \qquad  \qquad \qquad  \qquad  \qquad \qquad  \qquad  \qquad \prod_{k=1}^{n} \overline{\hat W^\e_0} \left(\hat \eta_k,\e(\bar \xi +\frac{\tilde \eta}{2}+\hat \xi_k-\hat \eta_1-...-\hat \eta_{k-1}-\frac{\hat \eta_k}{2}) \right) \\
& \qquad  \qquad \qquad  \qquad  \qquad \qquad  \qquad  \qquad \prod_{k=1}^n e^{is_{2k-2}\sum_{i=k}^n \xi_i.\eta_i}e^{-is_{2k-1}((\bar \xi-\frac{\tilde \eta }{2}-\eta_1-...-\eta_k).\xi_k-\sum_{i=k+1}^n\xi_i.\eta_i)}\\
& \qquad  \qquad \qquad  \qquad  \qquad \qquad  \qquad  \qquad \prod_{k=1}^n e^{-is_{2k-2}'\sum_{i=k}^n \hat \xi_i.\hat \eta_i}e^{-is_{2k-1}'((\bar \xi+\frac{\tilde \eta }{2}-\hat \eta_1-...-\hat \eta_k).\xi_k+\sum_{i=k+1}^n\hat \xi_i.\hat \eta_i)}\\
\end{align*}
We change variables again:
$$
\xi_k=\frac{v_k}{\e}, \qquad \hat \xi_k=\frac{\hat v_k}{\e},
$$
and define:
\begin{align}
\label{id:defomegak} &w_k=\bar \xi-\frac{\tilde \eta}{2} -\eta_1-...-\eta_{k-1}-\frac{\eta_k}{2},\qquad \omega_k=\left(\eta_k\sum_{i=0}^{2k-2} s_i-(\bar \xi-\frac{\tilde \eta}{2}-\eta_1-...-\eta_{k})s_{2k-1} \right),\\
\label{id:defomegak'} &\hat w_k=\bar \xi+\frac{\tilde \eta}{2} -\hat \eta_1-...-\hat \eta_{k-1}-\frac{\hat \eta_k}{2},\qquad \hat \omega_k=\left(\hat \eta_k\sum_{i=0}^{2k-2} s_i'-(\bar \xi+\frac{\tilde \eta}{2}-\hat \eta_1-...-\hat \eta_{k})s_{2k-1} \right).
\end{align}
This gives the following formula:
\begin{align*}
& \mathcal F_t(G^*) = 2\pi \lambda^{4n}\ep^{d} \int_{\mathcal S_{n,t}}  \int_{\mathbb R^{d(4n +2)}}  \,d \bar \xi \,d\tilde \eta  \,d \underline{v} \, d\underline{\hat v} \, d \underline{\eta} \, d \underline{\hat \eta} \, d\underline{s} \, d \underline{s}' \delta \left(\tilde \eta +\sum_{1}^{n} \eta_i- \hat \eta_i\right)  \\
&\qquad \qquad \qquad \qquad \qquad\qquad  \hat W^\e_0 \left(\tilde \eta,\ep \bar \xi\right)\prod_{k=1}^{n} \hat W^\e_0 \left(\eta_k,v_k+\e w_k \right)e^{iv_k.\frac{\omega_k}{\e}}\prod_{k=1}^{n}\overline{\hat W^\e_0} \left(\eta_k,\hat v_k+\e \hat w_k \right)e^{-i\hat v_k.\frac{\hat \omega_k}{\e}} .
\end{align*}
We introduce the inverse Fourier transform:
$$
\varphi(\eta,\zeta)=\int_{\mathbb R^d} \hat W_0^{\e}(\eta,v)e^{iv.\zeta}dv,
$$
which, after integration over the variables $v_1,...,v_n,\hat v_1,...,\hat v_n$, transforms the expression into:
\begin{align}
\nonumber & \mathcal F_t(G^*) = 2\pi \lambda^{4n}\ep^{d} \int_{\mathcal S_{n,t}}  \int_{\mathbb R^{d(2n +2)}}  \,d \bar \xi \,d\tilde \eta  \,d \underline{v} \, d\underline{\hat v} \, d \underline{\eta} \, d \underline{\hat \eta} \, d\underline{s} \, d \underline{s}' \delta \left(\tilde \eta +\sum_{1}^{n} \eta_i- \hat \eta_i\right) \\
\label{id:formulemathcalFbelt}& \qquad  \qquad \qquad \qquad \qquad \qquad\qquad  \hat W^\e_0 \left(\tilde \eta,\ep\bar \xi\right)\prod_{k=1}^{n} \varphi \left(\eta_k,\frac{\omega_k}{\e} \right)e^{-i\omega_k.w_k} \overline{\varphi} \left(\hat \eta_k,\frac{\hat \omega_k}{\e} \right)e^{i\hat \omega_k.\hat w_k}.
\end{align}
With the formula \eqref{id:formulemathcalFbelt} at hand we can prove Proposition \ref{prop:belt}.

\begin{proof}[Proof of Proposition \ref{prop:belt}]
We now choose $n^*$ as:
\be \label{id:defn*belt}
n^*= 10 \lceil \frac{d}{\kappa}\rceil .
\ee
and decompose the domain in the integral \eqref{id:formulemathcalFbelt} above in different subsets:
\begin{align*}
D &=\left\{ (\underline s,\underline{s'},\bar \xi,\tilde \eta,\underline{\eta},\underline{\eta'})\in \mathcal S_{n,t} \times \mathbb R^{d(2n +2)}, \quad |\tilde \eta|+ \sup_{k=1,...,2n} |\eta_k|+|\eta_k'|\leq \epsilon^{-\frac{\kappa}{10}} \right\},\\
D_1 &=\left\{ (\underline s,\underline{s'},\bar \xi,\tilde \eta,\underline{\eta},\underline{\eta'})\in D, \quad |\bar \xi|\leq \ep^{\frac{\kappa}{10}}t^{-\frac 12} \right\},\\
D_2 &=\left\{ (\underline s,\underline{s'},\bar \xi,\tilde \eta,\underline{\eta},\underline{\eta'})\in D, \quad |\bar \xi|> \ep^{\frac{\kappa}{10}}t^{-\frac 12} \quad  \mbox{and} \quad \sup_{k=1,...,n}  s_{2k-1}+ \sup_{k=1,...,n}  s_{2k-1}'\leq \ep^{1-\frac{\kappa}{10}}|\bar \xi|^{-1}  \right\},\\
D_3 &=\left\{ (\underline s,\underline{s'},\bar \xi,\tilde \eta,\underline{\eta},\underline{\eta'})\in D, \quad |\bar \xi|> \ep^{\frac{\kappa}{10}}t^{-\frac 12} \quad  \mbox{and}\quad  \sup_{k=1,...,n}  s_{2k-1}+ \sup_{k=1,...,n}  s_{2k-1}'> \ep^{1-\frac{\kappa}{10}}|\bar \xi|^{-1}  \right\},\\
D' &=\left\{ (\underline s,\underline{s'},\bar \xi,\tilde \eta,\underline{\eta},\underline{\eta'})\in \left( \mathcal S_{n,t} \times \mathbb R^{d(2n +2)}\right) \backslash D \right\},
\end{align*}
so that:
\be \label{id:decompositionleadinsubleadingbelt}
 \mathcal F_t(G^*) = 2\pi \lambda^{4n}\ep^{d}\left( \int_{D'}...+ \int_{D_1}...+ \int_{D_2}...+ \int_{D_3}...\right)
\ee
\textbf{Step 1} \emph{Subleading terms}. In this step we estimate the $D'$, $D_2$ and $D_3$ contributions with the sole assumption on $a$ that it is any Schwartz function. Note that with this hypothesis and \eqref{id:formuleWa}, $\hat W^\ep_0$ and $\varphi$ are Schwartz functions and that any seminorm of the Schwartz space of these functions is uniformly bounded in the range $0<\epsilon \leq 1$. In particular we bound the integrand in \eqref{id:formulemathcalFbelt} by:
\begin{align} 
\nonumber & \left| \hat W^\e_0 \left(\tilde \eta,\ep\bar \xi\right)\prod_{k=1}^{n} \varphi \left(\eta_k,\frac{\omega_k}{\e} \right)e^{-i\omega_k.w_k} \overline{\varphi} \left(\hat \eta_k,\frac{\hat \omega_k}{\e} \right)e^{i\hat \omega_k.\hat w_k} \right| \\
\label{bd:beltschwartz} & \qquad \qquad \qquad \qquad \qquad \qquad \qquad \leq \langle \tilde \eta \rangle^{-K}\langle \epsilon \bar \xi \rangle^{-K}\prod_1^n \langle \eta_k\rangle^{-K} \langle \frac{\omega_k}{\epsilon}\rangle^{-K} \langle\hat \eta_k\rangle^{-K} \langle \frac{\hat \omega_k}{\epsilon}\rangle^{-K} 
\end{align}
for any choice of a large constant $K>0$.

For the contribution of $D'$, we further decompose $D'=\tilde D' \cup  D'_1\cup...\cup D'_n \cup D''_1\cup...\cup D''_n $ where
$$
\tilde D'=D'\cap \{|\tilde \eta|> \frac 14 \epsilon^{-\frac{\kappa}{10}} \}, \quad D'_k=D'\cap \{|\eta_k|> \frac 14 \epsilon^{-\frac{\kappa}{10}} \}, \quad D''_k=D'\cap \{|\eta_k'|> \frac 14 \epsilon^{-\frac{\kappa}{10}} \}
$$
On $\tilde D'$, there holds that:
$$
\int_{|\tilde \eta | \geq \epsilon^{-\frac{\kappa}{10}}/4}  \delta \left(\tilde \eta +\sum_{1}^{n} \eta_i- \hat \eta_i\right) \langle \tilde \eta \rangle^{-K} d\tilde \eta \leq C(\kappa,K) \ep^{\frac{K\kappa}{10}}.
$$
Therefore from \eqref{id:formulemathcalFbelt} and \eqref{bd:beltschwartz}, after integration first over $\tilde \eta$, then over $\bar \xi,\eta_1,...,\eta_n,\tilde \eta_1,...,\tilde \eta_n$ which produces a $\ep^{C(d,n)}$ factor, and finally over $s_0,...,s_{2n-1},s_0',...,s_{2n-1}'$ which produces a $t^{4n}$ factor:
\begin{align*}
& \left| \int_{\tilde D'}... \right| \lesssim  \lambda^{4n}\ep^{d+\frac{K\kappa}{10}} \int_{\mathcal S_{n,t}}  \int_{\mathbb R^{d(2n+1)}} d\bar \xi \, d \underline{\eta} \, d \underline{\hat \eta} \, d\underline{s} \, d \underline{s}'  \langle \epsilon \bar \xi \rangle^{-K}\prod_1^n \langle \eta_k\rangle^{-K} \langle\hat \eta_k\rangle^{-K}  \lesssim \lambda^{4n}t^{4n}\epsilon^{K'}
\end{align*}
for any $K'>0$, up to choosing $K$ large enough. The integrals over $D_k'$ and $D_k''$ for $k=1,...,2n$ are estimated similarly, resulting in:
\be \label{id:boundsubbelt1}
 \left| \int_{D'}... \right| \lesssim  (\lambda t)^{4n}\epsilon^{K}.
\ee
To estimate the contribution from $D_2$ we define for $R>0$ the set:
$$
\mathcal S_{n,t,R}:= \left\{(\underline s,\underline{s'})\in \mathcal S_{n,t},  \quad \sup_{k=1,...,n}  s_{2k-1}+ \sup_{k=1,...,n}  s_{2k-1}'\leq \ep^{1-\frac{\kappa}{10}}R^{-1} \right\}.
$$
We then perform the following estimate, integrating first over the $\tilde \eta,\underline{\eta},\underline{\eta'}$ variables using \eqref{bd:beltschwartz}  and the definition of $D_2$, and then the constraints $|s_{2k-1}|,|s_{2k-1}|\lesssim \ep^{1-\frac{\kappa}{10}}|\xi|^{-1}$ and $|s_{2k}|,|s_{2k}|\lesssim t$ for $k=0,...,n$:
\begin{align}
\nonumber \left| \int_{D_2}... \right| & \lesssim  \lambda^{4n}\ep^{d} \int_{D_2} \,d \bar \xi \,d\tilde \eta    \, d \underline{\eta} \, d \underline{\hat \eta} \, d\underline{s} \, d \underline{s}'  \delta \left(\tilde \eta +\sum_{1}^{n} \eta_i- \hat \eta_i\right)  \langle \tilde \eta \rangle^{-K} \prod_1^n \langle \eta_k\rangle^{-K} \langle\hat \eta_k\rangle^{-K} \\
\nonumber & \lesssim  \lambda^{4n}\ep^{d} \int_{ |\bar \xi|> \ep^{\frac{\kappa}{10}}t^{-\frac 12} } \,d \bar \xi \int_{(\underline{s},\underline{s'})\in \mathcal S_{n,t,|\bar \xi|}}   \, d\underline{s} \, d \underline{s}'  \ \lesssim  \ \lambda^{4n}\ep^{d} \e^{(1-\frac{\kappa}{10})2n} t^{2n}\int_{ |\bar \xi|> \ep^{\frac{\kappa}{10}}t^{-\frac 12} } \frac{\,d \bar \xi }{|\bar \xi|^{2n}}\\
\label{id:boundsubbelt2} &  \ \lesssim \  \lambda^{4n} t^{4n} \left(\ep^{d+2n-\frac{\kappa}{10}(4n-d)}t^{-n-\frac d2} \right)  \ \lesssim  \ \lambda^{4n} t^{4n} \ep^{3d}
\end{align}
where we used \eqref{id:constrainttimebelt} and \eqref{id:defn*belt} for the last line.

We now turn to $D_3$, that we decompose as $D_3=D_{3,1}\cup ... \cup D_{3,n}\cup D_{3,1}'\cup ... \cup D_{3,n}' $ where $D_{3,k}=D_3 \cap \{ 2 s_{2k-1}>\ep^{1-\frac{\kappa}{10}}|\bar \xi|^{-1} \}$ and $D_{3,k}'=D_3 \cap \{ 2s_{2k-1}'> \ep^{1-\frac{\kappa}{10}}|\bar \xi|^{-1}\}$. On $D_{3,1}$ there holds using \eqref{id:defomegak}, and the inequalities $2s_{2k-1}'>\ep^{1-\frac{\kappa}{10}}|\bar \xi|^{-1}$, $|\eta_i|\leq \ep^{-\kappa/10}$, $2t>s_{2i}$ and \eqref{id:constrainttimebelt}:
$$
\left| \frac{\omega_1}{\ep}\right| \geq \ep^{-1}\left( |\bar \xi s_{2k-1}|- |\eta_k \sum_{i=0}^{2k-2}s_i +(\frac{\tilde \eta}{2}+\eta_1+...\eta_k)s_{2k}|\right)\geq \frac{\ep^{-\frac{\kappa}{10}}}{2}-C \ep^{\frac{9}{10}\kappa}\geq \frac{\ep^{-\frac{\kappa}{10}}}{4}
$$
for $\ep$ small enough. This implies that $\langle \frac{\omega_1}{\ep}\rangle^{-K}\lesssim \varepsilon^{\frac{\kappa K}{10}}$. We inject this bound in \eqref{bd:beltschwartz} and estimate the $D_{3,1}$ contribution as:
\begin{align*}
\nonumber \left| \int_{D_{3,1}}... \right| & \lesssim  \lambda^{4n}\ep^{\frac{\kappa K}{10}+d} \int_{D_{3,1}} \,d \bar \xi \,d\tilde \eta    \, d \underline{\eta} \, d \underline{\hat \eta} \, d\underline{s} \, d \underline{s}'  \delta \left(\tilde \eta +\sum_{1}^{n} \eta_i- \hat \eta_i\right) \ \lesssim  \ \lambda^{4n} t^{4n} \ep^{K'}
\end{align*}
for any $K'>0$, for $K$ large enough. The other contributions of $D_{3,2},...,D_{3,n},D_{3,1}',...,D_{3,n}'$ can be estimated similarly, resulting in:
\be \label{id:boundsubbelt3}
\left| \int_{D_{3}}... \right|  \lesssim  \lambda^{4n} t^{4n} \ep^{d+K'}.
\ee

\noindent \textbf{Step 2} \emph{Leading term}. We now choose $a$ of the following factorised form:
$$
a(x,v)=\chi(x)\chi'(v)
$$
for two non zero Schwartz functions $\chi$ and $\chi'$ such that:
$$
\chi(z) \geq 0, \quad \hat \chi(z) \geq 0, \quad \chi'(z) \geq 0, \quad \mbox{and} \quad \mathcal F( \chi^{'2})(z) \geq 0 \qquad \mbox{for all }z\in \mathbb R^d.
$$
The formula \eqref{id:formuleWa} and the above nonnegativity properties imply that for all $\ep>0$:
\be \label{eq:nonnegativitybelt}
\hat W^\ep_0(\eta,v)\geq 0 \quad \mbox{and} \quad \varphi(\eta,\omega)\geq 0 \qquad \mbox{for all }\eta,v,w\in \mathbb R^d.
\ee
We then perform a first order Taylor expansion estimate on $D_1$:
\be \label{eq:taylorbelt}
e^{i\omega_k . w_k}=1+O(|\omega_k.w_k|)=1+O(\ep^{\frac{\kappa}{5}}), \qquad e^{i\hat \omega_k . \hat w_k}=1+O(\ep^{\frac{\kappa}{5}})
\ee
where we used that $s_{i}\leq t$, $|\bar \xi|\leq \ep^{\kappa/10}t^{-1/2}$ and $|\eta_i|\leq \ep^{-\kappa/10}$ for the first inequality above, and where the second inequality is obtained similarly. The following identity then follows from \eqref{eq:nonnegativitybelt} and \eqref{eq:taylorbelt}:
\begin{align}
\nonumber & 2\pi \lambda^{4n}\ep^{d} \int_{D_1}...= 2\pi \left(1+O(\ep^{\frac{\kappa}{5}}) \right)\lambda^{4n}\ep^{d} \int_{D_1}  \,d \bar \xi \,d\tilde \eta  \, d \underline{\eta} \, d \underline{\hat \eta} \, d\underline{s} \, d \underline{s}' \delta \left(\tilde \eta +\sum_{1}^{n} \eta_i- \hat \eta_i\right) \\
\label{id:formulebeltstep21}& \qquad \qquad \qquad  \qquad \qquad  \qquad  \qquad \qquad \qquad \qquad \qquad\qquad  \hat W^\e_0 \left(\tilde \eta,\ep\bar \xi\right)\prod_{k=1}^{n} \varphi \left(\eta_k,\frac{\omega_k}{\e} \right) \overline{\varphi} \left(\hat \eta_k,\frac{\hat \omega_k}{\e} \right).
\end{align}
We define the following sets:
\begin{align*}
& \tilde D_1=D_1 \cap \left\{ |\bar \xi |\leq \frac{\ep}{t}\right\}, \qquad \bar D_1=D_1 \cap \left\{  |\bar \xi |\geq \frac{\ep}{t} \quad \mbox{and}\quad \sup_{1\leq k \leq n} s_{2k-1}+s'_{2k-1}\leq \frac{\ep}{|\bar \xi|}\right\}.
\end{align*}
On $\tilde D_1$ there holds from \eqref{id:constrainttimebelt} and $|\eta_i|,|\tilde \eta_i|\lesssim \ep^{-\kappa/10}$ (where $|\mathcal S_{n,t}|$ is the Lebesgue measure of $|\mathcal S_{n,t}|$):
$$
|\ep \bar \xi|\leq \ep^\kappa, \qquad \left| \frac{\omega_k}{\ep}\right|,\left| \frac{\hat \omega_k}{\ep} \right|\lesssim 1, \qquad |\mathcal S_{n,t}|\approx t^{4n},
$$
so that using the nonnegativity \eqref{eq:nonnegativitybelt} and the fact that $W= a^2+O(\ep)$ in the Schwartz space:
\begin{align}
\nonumber  c\left(\frac{\ep}{t}\right)^dt^{4n}\leq & \int_{\tilde D_1}  \,d \bar \xi \,d\tilde \eta  \, d \underline{\eta} \, d \underline{\hat \eta} \, d\underline{s} \, d \underline{s}' \delta \left(\tilde \eta +\sum_{1}^{n} \eta_i- \hat \eta_i\right)  \hat W^\e_0 \left(\tilde \eta,\ep\bar \xi\right)\prod_{k=1}^{n} \varphi \left(\eta_k,\frac{\omega_k}{\e} \right) \overline{\varphi} \left(\hat \eta_k,\frac{\hat \omega_k}{\e} \right)\\
\label{id:formulebeltstep22} & \qquad \qquad  \qquad  \qquad  \qquad  \qquad  \qquad  \qquad \leq \  \frac{1}{c}\left(\frac{\ep}{t}\right)^dt^{4n}
\end{align}
for some $c>0$. On $\bar D_1$ we change variables $(s_0,...,s_{2n-1},s_0',...,s_{2n-1}')\mapsto (\tilde s_0,...,\tilde s_{2n-1},\tilde s_0',...,\tilde s_{2n-1}')$ where:
$$
s_{2k-1}=\frac{\epsilon}{|\bar \xi|}\tilde s_{2k}, \quad s_{2k}=t\tilde s_{2k}, \quad s_{2k-1}'=\frac{\epsilon}{|\bar \xi|}\tilde s_{2k}', \quad \mbox{and}\quad s_{2k}'=t\tilde s_{2k}'.
$$
The set $\mathcal S_{n,t}$ is changed into:
$$
\tilde{\mathcal S}_{n,t}=\left\{ (\tilde s_0,...,\tilde s_{n-1},\tilde s_0',...,\tilde s_{n-1}')\in \mathbb R_+^{n}\times \mathbb R_+^{n}, \quad \sum_0^n \tilde s_{2i}+\frac{\ep}{|\bar \xi|t}\tilde s_{2i-1}\leq 1 \mbox{ and } \sum_0^n \tilde s_{2i}'+\frac{\ep}{|\bar \xi|t}\tilde s_{2i-1}'\leq t\right\}.
$$
On $\bar D_1$ there holds $|\eta_i|,|\hat \eta_i|\leq \ep^{-\kappa/10}$, hence since $t\leq \ep^{1+\kappa}$ in these new variables from \eqref{id:defomegak}:
$$
\frac{\omega_k}{\epsilon}=-\frac{\bar \xi }{|\bar \xi|}\tilde s_{2k-1}+O(\ep^{\frac 12}), \qquad \frac{\hat \omega_k}{\epsilon}=-\frac{\bar \xi }{|\bar \xi|}\tilde s_{2k-1}'+O(\ep^{\frac 12}).
$$
Hence, applying \eqref{bd:beltschwartz} and the above change of variables one finds:
\begin{align}
\nonumber 0 \leq & \int_{\bar D_1} ... \ \leq t^{2n}\int_{|\bar \xi|\geq \frac{\ep}{t}}\int_{\mathbb R^{d(2n+1)}}  \int_{\tilde{\mathcal S}_{t,n}} \,d \bar \xi \,d\tilde \eta  \, d \underline{\eta} \, d \underline{\hat \eta} \, d\underline{\tilde s} \, d \underline{\tilde s}' \delta \left(\tilde \eta +\sum_{1}^{n} \eta_i- \hat \eta_i\right) \\
\nonumber &\qquad \qquad \qquad \qquad \qquad  \qquad  \left(\frac{\ep}{|\bar \xi|} \right)^{2n}\langle \tilde \eta \rangle^{-K}\langle \epsilon \bar \xi \rangle^{-K}\prod_1^n \langle \eta_k\rangle^{-K} \langle \tilde s_{2k-1}\rangle^{-K} \langle\hat \eta_k\rangle^{-K}  \langle \tilde s_{2k-1}'\rangle^{-K}  \\
\label{id:formulebeltstep23} & \qquad \qquad   \lesssim \ t^{2n}\ep^{2n} \int_{|\bar \xi|\geq \frac{\ep}{t}}|\bar \xi|^{-2n}d\bar \xi \ \lesssim \ (\frac{\ep}{t})^d t^{4n},
\end{align}
where the lower bound is a consequence of the nonnegativity \eqref{eq:nonnegativitybelt}. Therefore, injecting \eqref{id:formulebeltstep22} and \eqref{id:formulebeltstep23} in \eqref{id:formulebeltstep21}, the contribution of the $D_1$ part in $\mathcal F_t(G^*)$ is, for some constant $c>0$:
\be \label{id:boundleadingbelt}
c\lambda^{4n}t^{4n} \ep^{2d}t^{-d} \leq 2\pi \lambda^{4n} \ep^d \int_{D_1} ... \leq \frac 1c \lambda^{4n}t^{4n} \ep^{2d}t^{-d}.
\ee
\textbf{Step 3} \emph{Conclusion}. We inject the bounds \eqref{id:boundleadingbelt}, \eqref{id:boundsubbelt1}, \eqref{id:boundsubbelt2} and \eqref{id:boundsubbelt3} in the decomposition \eqref{id:decompositionleadinsubleadingbelt}, this establishes the desired formula \eqref{id:estimationbelt} upon choosing $K'$ large enough.

\end{proof}

\section{Estimates on the expansion}

The aim of this section is to prove the following proposition.

\begin{proposition}
\label{propexpansion}
The iterates $u^n$, defined through~\eqref{defun}, satisfy the bounds
\begin{itemize}
\item For equation \eqref{nonlinschrod} with $\omega_0=\epsilon^{-2}$ or $\omega_0=0$ and $m(0)=0$, for any $\nu>0$, there exists $b>\frac{1}{2}$ such that
\begin{align}
\label{bd:estimationunL2}&\mathbb E \, \| u^n(t) \|_{L^2}^2 \lesssim \left\{ \begin{array}{l l l}  \left( \lambda t \right)^{2n} &\mbox{if }|t|\leq \epsilon^2,\\  \left( \frac{t}{T_{kin}} \right)^{n} |\log \epsilon |^{2(n+1)} &\mbox{if }|t|\geq \epsilon^2, \end{array}\right. \\
\label{bd:estimationunXsb}&\mathbb E \left\| \chi \left( \frac{t}{T} \right) u^n(t) \right\|^2_{X^{s,b}_\epsilon} \lesssim  \epsilon^{-\nu }\left( \frac{T}{T_{kin}} \right)^{n} \qquad \mbox{for }T\geq \epsilon^2.
\end{align}
\item For equation \eqref{nonlinschrod} with $\omega_0=0$ and $m(0) \neq 0$,
\begin{align*}
& \mathbb E \| u^n(t) \|_{L^2}^2 \lesssim \left( \lambda t \right)^{n}\\
& \mathbb E \left\| \chi \left( \frac{t}{T} \right) u^n(t) \right\|_{X^{s,b}_\epsilon}^2 \lesssim \left(  \lambda T \right)^{n}.
\end{align*}
\end{itemize}
\end{proposition}

\subsection{The $L^2$ estimate} \label{subsecL2}
We denote $B^m(r)$ the Euclidean ball of radius $r$ in dimension $m$, and $B^m_0(r)=B^m(r)\cap \mathbb R^m_0$. Given $v\in \mathcal V^1_l$ we define four sets which will distinguish whether $v$ is degenerate or not, and, if not, which type of degeneracy happens at $v$. In the case where $v\in \mathcal V^j$ is a junction vertex we set:
\begin{align*}
&S_v^1 = \{(\underline \alpha,\underline \eta,\underline \xi^f)\in   B^{n_m}(K\epsilon^{-K'})\times B^{d(n+1)}_0(K) \times B^{d(n+1)}(K\epsilon^{-1}), \  \Bigl|\alpha_{\mathcal p(v)}-\sum_{v\triangleright \mathcal p'}\alpha_{\mathcal p'}-\sum_{\tilde v \in \mathcal p^+(v)}\Omega_{\tilde v}\Bigr|> \delta \epsilon^{-2} \}\\
&S_v^2 =  \{(\underline \alpha,\underline \eta,\underline \xi^f)\in  B^{n_m}(\epsilon^{-K'})\times B^{d(n+1)}_0(K) \times B^{d(n+1)}(K\epsilon^{-1}), \   |\xi_{(v,v_a(v))}|>  \delta \epsilon^{-1} \}\\
&S_v^3 =   \{(\underline \alpha,\underline \eta,\underline \xi^f)\in B^{n_m}(\epsilon^{-K'})\times B^{d(n+1)}_0(K) \times B^{d(n+1)}(K\epsilon^{-1}), \  |\alpha_{\mathcal p_j(v)}|> \delta \epsilon^{-2} \},\\
&S_v^4 =   B^{n_m}(\epsilon^{-K'})\times B^{d(n+1)}_0(K) \times B^{d(n+1)}(K\epsilon^{-1}) \ \backslash \ (\cup_{i=1,2,3} S_v^i)
\end{align*}
(the constants $K,K',\delta>0$ will be fixed later). 
In the case where $v\notin \mathcal V^j$ is not a junction vertex, then we define $S_v^1$ and $S_v^2$ as above, we set $S_v^3=\emptyset$, and $S_v^4 =B^{n_m}(\epsilon^{-K'})\times B^{d(n+1)}_0(K) \times B^{d(n+1)}(K\epsilon^{-1}) \ \backslash \ (\cup_{i=1,2} S_v^i)$.

\begin{definition} \label{def:degenerate}
Let $\delta>0$. Given a set $S\subset B^{n_m}(\epsilon^{-K'})\times B^{d(n+1)}_0(K) \times B^{d(n+1)}(K\epsilon^{-1})$, we say that a degree one linear vertex $v\in \mathcal V^1$ is \emph{degenerate} on $S$ if for all $(\underline \alpha,\underline \eta,\underline \xi^f)\in S$ the following three conditions are met simultaneously:
\begin{align*}
& |\alpha_{\mathcal p(v)}-\sum_{v\triangleright \mathcal p'}\alpha_{\mathcal p'}-\sum_{\tilde v \in \mathcal p^+(v)}\Omega_{\tilde v}|\leq \delta \epsilon^{-2},\\
& |\xi_{(v,v_a(v))}|\leq \delta \epsilon^{-1},\\
& \mbox{if } v\in \mathcal  V^j \mbox{ is a junction vertex then }|\alpha_{\mathcal p_j(v)}|\leq \delta \epsilon^{-2}.
\end{align*}
Equivalently, $v$ is degenerate on $S$ if $S\subset S^4_v$. 

We say that a vertex $v\in \mathcal V_i\cup \mathcal V_R$ is \emph{nondegenerate} on $S$ if either $v\in \mathcal V_R$, or $v$ is a degree zero or a degree one quadratic vertex, or if $v$ is a degree one linear vertex such that for each $(\underline \alpha,\underline \eta,\underline \xi^f)\in S$ at least one of the three conditions above fail.
\end{definition}

We will partition the domain of integration in \eqref{chouettehulotte} according to the non-degeneracy/degeneracy of each vertex. For this aim, given a function $\beta:\mathcal V^1_l\mapsto \{1,2,3,4\}$, we define:
\be \label{diagrams:id:definition-S-beta}
S_\beta =\cap_{v\in \mathcal V_l^1} S^{\beta(v)}_v.
\ee
Note that any vertex $v\in \mathcal V^1_l$ is either degenerate (if $\beta (v)=4$) or nondegenerate (if $\beta(v)=1,2,3$) on such a set $S_\beta$. Note also that $B^{n_m}(\epsilon^{-K'})\times B^{d(n+1)}_0(K) \times B^{d(n+1)}(K\epsilon^{-2})=\cup_\beta S_\beta $. Degenerate degree one linear vertices have implications for the vertices above them, as stated below.

\begin{lemma} \label{lem:degenerate-implies-non-degenerate-above}

Assume that $\omega(\xi)=\epsilon^{-2}+\frac{|\xi|^2}{2}$ and $K,K'>0$, then for $\delta(K)>0$ small enough the following holds true. Given any set $S\subset B^{n_m}(\epsilon^{-K'}) \times B^{d(n+1)}_0(K) \times B^{d(n+1)}(K\epsilon^{-1})$ and $v\in \mathcal V^1_l$ a degenerate degree one linear vertex on $S$, then:
\begin{itemize}
\item[(i)] If $v$ is at the left of the vertex above it ($v=v_l(v_a(v))$) then at $v_a(v)$, for all $(\underline \alpha,\underline \eta,\underline \xi^f)\in S$:
\be \label{bd:degenimpliesnondegenleft}
\left|\alpha_{\mathcal p(v_a(v))}-\sum_{v_a(v)\triangleright \mathcal p'}\alpha_{\mathcal p'}-\sum_{\tilde v \in \mathcal p^+(v_a(v))}\Omega_{\tilde v} \right| \geq \frac{\epsilon^{-2}}{2}
\ee
\item[(ii)] If $v$ is at the right of the interaction vertex above it ($v=v_r(v_a(v))$ and $v_a(v)\in \mathcal V_i$) then by definition $v_a(v)\in \mathcal V^j$ is a junction vertex with $\mathcal p_j(v_a(v))=\mathcal p(v)$, and one has for all $(\underline \alpha,\underline \eta,\underline \xi^f)\in S$:
\be \label{bd:degenimpliesnondegenright}
|\alpha_{\mathcal p_j(v_a(v))}|\geq \frac{\epsilon^{-2}}{2}.
\ee
\item[(iii)] If $v\in \{v^l_{\text{top}},v^r_{\text{top}} \}$ is one of the top vertices then one has for all $(\underline \alpha,\underline \eta,\underline \xi^f)\in S$:
\be \label{bd:degenimpliesnondegenrightbis}
|\alpha_{\mathcal p(v)}|\geq \frac{\epsilon^{-2}}{2}.
\ee

\end{itemize}

\end{lemma}

\begin{remark} \label{re:martinpecheur}

The above lemma implies in particular that if $v$ is a degree one linear vertex that is degenerate on $S$, then the vertex above it, namely $v_a(v)$, is nondegenerate on $S$. In particular, given any partition function $\beta:\mathcal V^1_l\mapsto \{1,2,3,4\}$, if for some $v\in \mathcal V^1_l$ one has $v_a(v)\in \mathcal V^1_l$ then if $\beta$ requires degeneracy at both vertices, i.e. $\beta(v)=\beta(v_a(v))=4$, then $S_\beta=\emptyset$.

\end{remark}

\begin{proof}

Since $v$ is degenerate, \eqref{id:formularesonancelinear} implies that:
\be \label{id:intermediatenondegeneratesets1}
|\Omega_v |=|\epsilon^{-2}-\sigma(\tilde \xi)\sigma( \xi^f)\tilde \xi.\xi^f \ +\frac 12 (1+\sigma(\tilde \xi)\sigma(\xi^f))|\tilde \xi|^2|\geq \epsilon^{-2}-K\delta \epsilon^{-2}-\delta^2\epsilon^{-2}\geq \frac{3\epsilon^{-2}}{4}
\ee
for $\delta$ small enough. We now let $v'=v_a(v)$ and define four cases A, B, C and D. In case A one has $v=v_l(v')$ and $v\in \mathcal V^j$ is a junction vertex. In case B one has $v=v_l(v')$ and $v\notin \mathcal V^j$. Cases A and B cover (i) in the Lemma. In case C one has either $v=v_r(v')$ or $v\in \{v^l_{\text{top}},v^r_{\text{top}} \}$, and $v\in \mathcal V^j$. In case D one has either $v=v_r(v')$ or $v\in \{v^l_{\text{top}},v^r_{\text{top}} \}$, and $v\notin \mathcal V^j$. Cases C and D cover (ii) and (iii) in the Lemma.

By definition, we have in case A that $\mathcal p(v)=\mathcal p(v')$ and $\{\mathcal p', \ v\triangleright \mathcal p'\}=\{\mathcal p', \ v'\triangleright \mathcal p'\}\cup{\mathcal p_j(v)}$, in case B that $\mathcal p(v)=\mathcal p(v')$ and $\{\mathcal p', \ v\triangleright \mathcal p'\}=\{\mathcal p', \ v'\triangleright \mathcal p'\}$, in case C that $\{\mathcal p', v \triangleright \mathcal p'\}=\{\mathcal p_j(v)\}$, and in case D that $\{\mathcal p', v \triangleright \mathcal p'\}=\emptyset$.

For (i) we have therefore
$$
\alpha_{\mathcal p(v')}-\sum_{v'\triangleright \mathcal p'}\alpha_{\mathcal p'}-\sum_{\tilde v \in \mathcal p^+(v')}\Omega_{\tilde v}=\alpha_{\mathcal p(v)}-\sum_{v\triangleright \mathcal p'}\alpha_{\mathcal p'}-\sum_{\tilde v \in \mathcal p^+(v)}\Omega_{\tilde v} \ + \ \left\{ \begin{array}{l l} \Omega_v -\alpha_{\mathcal p_j(v)} \qquad \mbox{for case A}, \\ \Omega_v \qquad \mbox{for case B}. \end{array} \right.
$$
which, using \eqref{id:intermediatenondegeneratesets1} and the degeneracy of $v$, yields for $\delta>0$ small enough
$$
|\alpha_{\mathcal p(v')}-\sum_{v'\triangleright \mathcal p'}\alpha_{\mathcal p'}-\sum_{\tilde v \in \mathcal p^+(v')}\Omega_{\tilde v}|\geq \frac{3\epsilon^{-2}}{4}-2\delta \epsilon^{-2} \geq \frac{\epsilon^{-2}}{2}
$$
and proves the lemma in this case. For (ii), we have
$$
\alpha_{\mathcal p(v)}-\sum_{v\triangleright \mathcal p'}\alpha_{\mathcal p'}-\sum_{\tilde v \in \mathcal p^+(v)}\Omega_{\tilde v} =\alpha_{\mathcal p(v)} \ - \ \left\{ \begin{array}{l l} \Omega_v +\alpha_{\mathcal p_j(v)} \qquad \mbox{for case C}, \\ \Omega_v \qquad \mbox{for case D}. \end{array} \right.
$$
which, using \eqref{id:intermediatenondegeneratesets1} and the degeneracy of $v$, yields for $\delta>0$ small enough
$$
|\alpha_{\mathcal p(v)}|\geq \frac{3\epsilon^{-2}}{4}-2\delta \epsilon^{-2} \geq \frac{\epsilon^{-2}}{2}
$$
and proves the lemma in this case as well as $\mathcal p(v)=\mathcal p_j(v_a(v))$.
\end{proof}

We will study carefully degenerate degree one linear vertices by including them in larger clusters.

\begin{definition} \label{def:cluster}

Given a set $S\subset B^{n_m}(\epsilon^{-K'})\times B^{d(n+1)}_0(K) \times B^{d(n+1)}(K\epsilon^{-1})$, we say that $\mathcal C \subset \mathcal V_i\cup \mathcal V_R$ is a \emph{degenerate cluster} on $S$ if either of the three following possibilities occur:
\begin{itemize}
\item Type I: $\mathcal C=\{v,v' \}$ with $v$ being at the bottom left of $v'$, i.e. $v=v_l(v')$, and is such that $v\in \mathcal V^1_l$ is degenerate on $S$, and $v'$ is nondegenerate on $S$.
\item Type II: $\mathcal C=\{v,v' \}$ with $v$ being at the bottom right of $v'$, i.e. $v=v_r(v')$, and is such that $v\in \mathcal V^1_l$ is degenerate on $S$, and $v'$ is nondegenerate on $S$.
\item Type III: $\mathcal C=\{v,v',v'' \}$ with $v$ and $v'$ being at the bottom left and right of $v''$, i.e. $v=v_r(v'')$ and $v'=v_r(v'')$, and is such that $v,v'\in \mathcal V^1_l$ are degenerate on $S$, and $v''$ is nondegenerate on $S$.
\end{itemize}

\end{definition}

The lemma below states that, given any set $S_\beta$ in the partition of the domain of integration, one can always decompose the graph as a disjoint union of degenerate clusters and of a set of vertices that are all nondegenerate.

\begin{lemma}[Decomposition into nondegenerate vertices and degenerate clusters] \label{lem:decomposition}

For any set of the form $S_\beta$, there exists $\mathcal C_1,...,\mathcal C_{n_d(G,\beta)}$ disjoints degenerate clusters on $S_\beta$ such that:
$$
\mathcal V_i \cup \mathcal V_R = \tilde{\mathcal V} \sqcup \mathcal C_1\sqcup...\sqcup\mathcal C_{n_d(G,\beta)}
$$
where $ \tilde{\mathcal V} $ only contains non-degenerate vertices on $S_\beta$.

\end{lemma}

\begin{proof}

Let $v\in \mathcal V^1_l$ be degenerate on $S_\beta$, and let $\tilde v$ be the other vertex that is below $v_a(v)$. If $\tilde v$ is nondegenerate, we define the degenerate cluster $\mathcal C_v$ as $\mathcal C_v=\{v,v_a(v)\}$. If $\tilde v$ is degenerate, we define the degenerate cluster $\mathcal C_v=\mathcal C_{\tilde v}$ as $\mathcal C_v=\{v,\tilde v, v_a(v)\}$.

From Remark \ref{re:martinpecheur}, $\mathcal C_v$ is indeed a degenerate cluster on $S_\beta$ as $v_a(v)$ is non-degenerate on $S_\beta$. Since, in each degenerate cluster, the vertex above is nondegenerate, then the degenerate clusters that we have defined are disjoint. We order them as $\mathcal C_1,...,\mathcal C_{n_d(G,S)}$, and by definition the remaining vertices in $(\mathcal V_i\cup \mathcal V_R)\backslash (\cup_{j=1}^{n_d(G,\beta)}) \mathcal C_j$ are all nondegenerate.

\end{proof}

We now turn to the proof of Proposition \ref{propexpansion}.

\begin{proof}[Proof of \eqref{bd:estimationunL2} in Proposition \ref{propexpansion}]\text{  }\\
\textbf{Step 1} \emph{Preliminary reduction}. We only prove the result for $t>0$, as the computation for $t<0$ is the same from \eqref{id:formulaungraphs-}. From Proposition \ref{pr:formulagraph}, it is enough to prove that given any $G\in \mathcal G_n^p$ and $t\geq 0$ there holds:
\be \label{bd:FtL2}
|\mathcal F_t(G)|\leq C \left\{ \begin{array}{l l l} (\lambda t)^{2n} & \mbox{if } 0\leq t\leq \epsilon^2, \\ (\frac{t}{T_{kin}})^{n} |\log \epsilon|^{2(n+1)}&\mbox{if }\epsilon^2\leq t \end{array} \right.
\ee
where $C=C(n)>0$. We now fix $G$ and $t$. We first prove the result for $0\leq t\leq \epsilon^2$. Bounding all oscillatory phases and $M$ in \eqref{id:mathcalFGP1} by $1$, and then applying Lemma \ref{pr:spanning} we obtain:
\be \label{bd:Ftsuperrough}
|\mathcal F_t(G)| \lesssim \lambda^{2n}\ep^{d(n+1)} \iiint_{\mathbb R^{d(n+1)}\times \mathbb R^{d(n+1)}_0\times \mathbb R^{2n}_+ } d\underline{\xi^f} \, d \underline{\eta}  \, d\underline{s}   \Delta_t(\underline s)  \prod_{\{i,j\}\in P} |\widehat{W_0^\epsilon} (\eta_{i,j},\frac{\ep}{2}(\sigma_{0,i}\xi_{0,i}+\sigma_{0,j} \xi_{0,j}))|.
\ee
Note that, from \eqref{id:formuleWa}, $K_0=\text{diam}(\hat W_0^\epsilon)$ is bounded uniformly for $0<\epsilon\leq 1$. Hence, in \eqref{chouettehulotte} the product $\prod_{\{i,j\}\in P} \widehat{W_0^\epsilon} (\eta_{i,j},\frac{\ep}{2}(\sigma_{0,i} \xi_{0,i}+  \sigma_{0,j} \xi_{0,j}))$ is $0$ unless $|\eta_{i,j}|\leq K_0$ and $|\sigma_{0,i} \xi_{0,i}+  \sigma_{0,j} \xi_{0,j}|\leq 2\epsilon^{-1}K_0$ for all $\{i,j\}\in P$. Recalling that $\eta_{i,j}=\xi_{0,i}+\xi_{0,j}$ and that $\sigma_{0,i}\sigma_{0,j}=-1$, this implies that $|\xi_{0,i}|\leq 2K_0\epsilon^{-1}$ for all $i=1,...,2n+2$. Let now $\xi^f$ be a free variable, associated to an edge $(v,v')$ where $v$ is below $v'$. Then, integrating the Kirchhoff laws in $G$ from initial vertices up to $v$, we see that $\xi^f=\sum_{v_{0,i}\in \mathcal V_0, \ v_{0,i}\leq v} \xi_{0,i}$. Hence $|\xi^f|\leq 2(n+1)K_0\epsilon^{-1}=K\epsilon^{-1}$ as there are at most $n+1$ vertices below $v$. Therefore, the integrand in \eqref{chouettehulotte} and in \eqref{bd:Ftsuperrough} is zero for $(\underline \eta,\underline{\xi^f})$ outside of $B^{d(n+1)}_0(K) \times B^{d(n+1)}(K\epsilon^{-1})$. 

In \eqref{bd:Ftsuperrough}, integrating with respect to $\underline{\xi^f}$ produces a $\epsilon^{-d(n+1)}$ factor, over $\underline \eta$ a $1$ factor, and over $\underline s$ a $t^{2n}$ factor, so we eventually arrive at:
$$
|\mathcal F_t(G)|\lesssim (\lambda t)^{2n}.
$$
We next prove the result for $ t\geq \epsilon^2$. We first reduce the integral to $B^{n_m}(K'\epsilon^{-2}) \times B^{d(n+1)}_0(K) \times B^{d(n+1)}(K\epsilon^{-1})$ for some $K(n,a)$ and $K'(n,a)$ independent of $\epsilon$. We claim in this step that for any $K'>2$ large enough, \eqref{id:FtL2inter1} can be upper bounded by
\be
\label{id:FtL2inter2}  |\mathcal F_t(G)|\lesssim  \epsilon^{\frac{K'}{4}}+ \sum_{\beta} \mathcal F_{t,G,\beta} ,
\ee
where
\begin{align}
&\label{id:FtL2inter3} \mathcal F_{t,G,\beta} =\lambda^{2n}\ep^{d(n+1)} \iiint_{(\underline \alpha ,\underline \eta,\underline \xi^f) \in S_\beta} d\underline{\alpha} \,  d \underline{\eta}    \,   d\underline{\xi^f}  M(\underline \xi)   \prod_{\mathcal p \in \mathcal P_{m}} \frac{1}{|\alpha_{\mathcal p}+\frac{ic_{\mathcal p}}{t}|}  \prod_{k=1}^{2n} \frac{1}{|\Theta_k|},
\end{align}
and where we introduced for convenience the notation
\be \lab{id:def:Theta-k}
\Theta_k=\Theta_k(t,\underline{\alpha},\underline \eta, \underline{\xi^f})=\alpha_{\mathcal p(v_k)}-\sum_{\tilde{\mathcal p}\triangleleft v_k}\alpha_{\tilde{\mathcal p}}-\sum_{\tilde{v}\in \mathcal p^+(v_k)}\Omega_{\tilde{v}}+\frac{ic_k}{t}.
\ee
We now prove \eqref{id:FtL2inter2}. Using \eqref{id:mathcalFGP3}, the above discussion on the restriction for $\underline{\xi^f}$ and $\underline \eta$ to $B^{d(n+1)}_0(K) \times B^{d(n+1)}(K\epsilon^{-1})$ and putting absolute values we obtain:
\begin{align}
\label{id:FtL2inter1} &  |\mathcal F_t(G)|\lesssim \lambda^{2n}\ep^{d(n+1)} \iiint_{(\underline \alpha ,\underline \eta,\underline \xi^f) \in \mathbb R^{n_m}\times B^{d(n+1)}_0(K) \times B^{d(n+1)}(K\epsilon^{-1})} d\underline{\alpha} \,  d \underline{\eta}    \,   d\underline{\xi^f}  |...|.
\end{align}

Next, let for some $K'>0$ to be fixed later and $\mathcal p\in \mathcal P_m$:
\begin{align*}
&\bar S=\{ (\underline \alpha,\underline \eta,\underline{\xi^f})\in \mathbb R^{n_m}\times B^{d(n+1)}_0(K) \times B^{d(n+1)}(K\epsilon^{-1}), \ |\underline \alpha|\geq \epsilon^{-K'} \} , \\
&\bar S_{\mathcal p}=\{(\underline \alpha,\underline \eta,\underline{\xi^f})\in \mathbb R^{n_m}\times B^{d(n+1)}_0(K) \times B^{d(n+1)}(K\epsilon^{-1}), \ |\alpha_{\mathcal p}| \geq \epsilon^{-K'} \mbox{ and }|\alpha_{\mathcal p}|=\sup_{\mathcal p'}|\alpha_{\mathcal p'}|\}.
\end{align*}
Then $\bar S \subset \cup_{\mathcal p}\bar S_{\mathcal p}$. Fix now $\mathcal p\in \mathcal P_m$. There exists at least one $v\in \mathcal V_i$ such that $v\in \mathcal p$. We decompose $\mathcal P_m=\{\mathcal p\}\cup \mathcal P_m^1\cup \mathcal P_m^2$ where $\mathcal P_m^1=\{\mathcal p'\in \mathcal P_m, \ v\triangleright \mathcal p'\}$ and $\mathcal P_m^2=\mathcal P_m \backslash \mathcal P_m^1 \backslash \{\mathcal p\}$, and $n(v)=\# \mathcal P_m^1$. Then, for $(\underline \eta,\underline{\xi^f})\in B^{d(n+1)}_0(K) \times B^{d(n+1)}(K\epsilon^{-1})$, there holds $|\Omega_v|\lesssim C(K)\epsilon^{-2}$, and applying several times the inequality \eqref{bd:weightedintegral3} yields for any $C>0$:
\begin{align*}
&\left| \int_{|\alpha_{\mathcal p}|\geq \epsilon^{-K'}}\frac{\log^C |\alpha_{\mathcal p}|d\alpha_{\mathcal p}}{|\alpha_{\mathcal p}+\frac{c_{\mathcal p}i}{t}|} \int_{(\alpha_{\mathcal p'})_{\mathcal p'\in \mathcal P_m^1}\in \mathbb R^{n(v)}}\frac{d (\alpha_{\mathcal p'})_{\mathcal p'\in \mathcal P_m^1}}{|\alpha_{\mathcal p}-\sum_{v\triangleright \mathcal p'}\alpha_{\mathcal p'}-\sum_{v'\in \mathcal p^+(v)}\Omega_v+\frac{ic_v}{t}|} \prod_{v\in \mathcal p'\in \mathcal P_m^1} \frac{1}{|\alpha_{\mathcal p'}+\frac{ic_{\mathcal p'}}{t}|}\right|\\
&\lesssim \int_{|\alpha_{\mathcal p}|\geq \epsilon^{-K'}} \frac{\log^C |\alpha_{\mathcal p}|}{|\alpha_{\mathcal p}+\frac{c_{\mathcal p}i}{t}|} \frac{1}{|\alpha_{\mathcal p}-\sum_{v'\in \mathcal p^+(v)}\Omega_v +\frac{c_{v}i}{t}|}d\alpha_{\mathcal p} \lesssim \epsilon^{\frac{K'}{2}}.
\end{align*}
We then estimate the integral \eqref{id:FtL2inter1} restricted to the set $\bar S_{\mathcal p}$ as follows: first we integrate over the $(\alpha_{\mathcal p'})_{\mathcal p'\in \mathcal P_m^2}$ variables the $\frac{1}{|\alpha_{\mathcal p'}+\frac it|}$ terms which produces a $\log^{n_m-1-n(v)} |\alpha_{\mathcal p}|$ factor, then we integrate over the $(\alpha_{\mathcal p'})_{\mathcal p'\in \mathcal P_m^1\cup \{\mathcal p\}}$ variables and apply the inequality above producing a $\epsilon^{\frac{K'}{2}}$ factor, and finally we integrate over the $\underline \eta$ and $\underline{\xi^f}$ variables which produces a $\epsilon^{-C(K)}$ factor, yielding:
$$
\iiint_{\bar S_{\mathcal p}} ... \ \lesssim \epsilon^{\frac{K'}{2}-C(K)}\lesssim \epsilon^{\frac{K'}{4}}
$$
for $K'$ large enough depending on $K$. Hence, since $\bar S \subset \cup_{\mathcal p}\bar S_{\mathcal p}$ we get $\iiint_{\bar S}... \lesssim \epsilon^{\frac{K'}{4}}$ for any arbitrary constant $K'>0$. We thus get the inequality \eqref{id:FtL2inter1} by noticing that $\bar S^c=\prod_\beta S_\beta$.\\

We now first treat the hardest case of \eqref{nonlinschrod} with $\omega_0=\epsilon^{-2}$, and relegate the easier proof of  \eqref{nonlinschrod} with $m(0)=0$ to Step 4.

The basic idea will be the following: consider all interaction vertices following the integration order; when reaching the vertex $v_k$, if it is of degree 1, integrate over the free variable below it; and if it is a junction vertex, integrate over $\alpha_{\mathcal{p}_j(v_k)}$. While this plan can be followed literally in the absence of clusters, complications arise if they are present.

\bigskip

\noindent
\textbf{Step 2} \emph{Upper bound for \eqref{id:FtL2inter1} in absence of clusters}. We assume first that all interaction vertices are non-degenerate on $S_\beta$, in the sense of Definition \ref{def:degenerate}. We prove \eqref{bd:FtL2} by integrating over the variables $\underline{\xi^f}$ and $\underline{\alpha}$ iteratively, according to an algorithm that considers the interaction vertices and the root vertex $v_1,...,v_{2n+1}\in \mathcal V_i\cup \mathcal V_R$ one after an other, where $v_1,...,v_{2n+1}$ is the integration order on $G$.

We define for $1\leq k \leq 2n+1$ the set $\mathcal P_{m,k}=\{\mathcal{p} \in \mathcal P_m, \ v_j(\mathcal{p}) \mbox{ is after }v_k \mbox{ for the integration order}\}$ and the variables $\underline \alpha_k=(\alpha_{\mathcal p})_{\mathcal p\in \mathcal P_{m,k}}$ and $\underline{\xi^f_k}=(\xi^f_i)_{k_i\geq k}$. Importantly, note that $\Theta_k$ only depends on $\underline{\alpha_k}$, $\underline \eta$ and $\underline{\xi^f_k}$ from Proposition \ref{pr:spanning} and the definition of the integration order. Let $n_{m,k}=\# \mathcal P_{m,k}$, $n_{0,k}=\#\{v\in \mathcal V^0, \ v\mbox{ is strictly before }v_k\mbox{ for the integration order}\}$ and $n_{1,k}=\#\{v\in \mathcal V^1, \ v\mbox{ is strictly before }v_k\mbox{ for the integration order}\}$. Note that $\mathcal P_{m,1}=\mathcal P_m$, that $\mathcal P_{m,2n+1}=\{\mathcal p(v_{\text{top}}^l),\mathcal p(v_{\text{top}}^r)\}$, that $n_{0,0}=n_{1,0}=0$ and $n_{0,2n+1}=n_0=n_{1,2n+1}=n_1=n$ from \eqref{perdrix}. We define $S_{\beta,k}$ as the image of $S_\beta$ by the projection map $(\underline{\alpha},\underline \eta,\underline{\xi^f})\mapsto (\underline{\alpha_k},\underline \eta,\underline{\xi^f_k})$. We claim that for all $1\leq k \leq 2n+1$:
\begin{align}
\label{id:FtL2inter4} &\mathcal F_{t,G,\beta}  \lesssim \lambda^{2n}\ep^{d(n+1)}\epsilon^{(2-d)n_{1,k}}t^{n_{0,k}}|\log \epsilon|^{2n_{1,k}} \iiint_{(\underline \alpha_k ,\underline \eta,\underline \xi^f_k) \in S_{\beta,k}} d\underline{\alpha}^k \,  d \underline{\eta}^k    \,   d\underline{\xi^f_k}   \prod_{\mathcal p \in \mathcal P_{m,k}} \frac{1}{|\alpha_{\mathcal p}+\frac{ic_{\mathcal p}}{t}|} \prod_{\ell = k}^{2n} \frac{1}{|\Theta_\ell|},
\end{align}
which we now prove by induction. It is trivially true for $k=1$. We assume now it is true for some $1\leq k\leq 2n$. We prove it for $k+1$ by considering four cases depending on the vertex $v_{k}$.

\medskip

\noindent \underline{Case 1: $v_{k}\in \mathcal V^0$ and $v_{k}\notin \mathcal V^j$.} In this case $\mathcal P_{m,k}=\mathcal P_{m,k+1}$, $n_{0,k}=n_{0,k+1}-1$ and $n_{1,k}=n_{1,k+1}$. There is no variable to integrate over: $\underline{\alpha_{k}}=\underline{\alpha_{k+1}}$ and $\underline{\xi^f_k}=\underline{\xi^f_{k+1}}$. We first plug these equalities in the integral in the right hand side of \eqref{id:FtL2inter4} at step $k$. Then for all $(\underline \alpha_k ,\underline \eta,\underline \xi^f_k) \in S_{\beta,k}=S_{\beta,k+1}$, we simply upper bound the term $|\Theta_k|^{-1}\lesssim t$ in the integral. The right hand side of \eqref{id:FtL2inter4} at step $k$ is then bounded by that of \eqref{id:FtL2inter4} at step $k+1$.

\medskip

\noindent \underline{Case 2: $v_{k}\in \mathcal V^0$ and $v_{k}\in \mathcal V^j$.} In this case $\mathcal P_{m,k}=\mathcal P_{m,k+1}\cup\{\mathcal p_j(v_{k})\}$, $n_{0,k}=n_{0,k+1}-1$ and $n_{1,k}=n_{1,k+1}$. We will integrate over the variable $\alpha_{\mathcal p_j(v_k)}$, noting that $\underline{\alpha_{k}}=(\alpha_{\mathcal p_j(v_k)},\underline{\alpha_{k+1}})$ and $\underline{\xi^f_{k}}=\underline{\xi^{f}_{k+1}}$.

By definition of the integration order and of junction vertices, all the vertices in $\mathcal p_j(v_k)$ have already been considered by the algorithm, i.e. $\mathcal p_j(v_k)\subset \{v_1,...,v_{k-1}\}$, and for all $\ell\geq k+1$, the vertex $v_{\ell}$ is not constraining $\mathcal p_j(v_k)$ so that $\{v\in \{v_{\ell}\}_{\ell\geq k}, \ v\triangleright v_{\mathcal p_j(v_k)}\}=\{v_k\}$. Thus, in the integrand in \eqref{id:FtL2inter4} at step $k$, the terms $|\alpha_{\mathcal p_j(v_k)}+\frac{ic_{\mathcal p_j(v_k)}}{t}|^{-1}$ and
\be \label{id:clustersgain5}
\Theta_k=-\alpha_{\mathcal p_j(v_k)}+\gamma+\frac{ic_k}{t}
\ee
are the only ones depending on the variable $\alpha_{\mathcal p_j(v_k)}$, where $\gamma\in \mathbb R$ has an explicit expression but is independent of $\alpha_{\mathcal p_j(v_k)}$. In the integral \eqref{id:FtL2inter4} at step $k$, for a fixed $(\underline{\alpha_{k+1}},\underline \eta, \underline{\xi^f_{k+1}})\in S_{\beta,k+1}$ we integrate over the variable $\alpha_{\mathcal p_j(v_k)}$ using \eqref{id:clustersgain5} and \eqref{bd:weightedintegral3}, producing:
\begin{align*}
& \int_{|\alpha_{\mathcal p_j(v_k)}|\leq \epsilon^{-K'}, \ (\underline{\alpha_k},\underline \eta,\underline{\xi^{f}_{k}})\in S_{\beta,k}} \frac{d \alpha_{\mathcal p_j(v_k)}}{|\alpha_{\mathcal p_j(v_k)}+\frac{ic_{\mathcal p_j(v_k)}}{t}|} \frac{1}{|\Theta_k|}\\
 &\qquad \qquad \qquad \lesssim \int_{|\alpha_{\mathcal p_j(v_k)}|\leq \epsilon^{-K'}} \frac{d \alpha_{\mathcal p_j(v_k)}}{|\alpha_{\mathcal p_j(v_k)}+\frac{ic_{\mathcal p_j(v_k)}}{t}|} \frac{1}{|-\alpha_{\mathcal p_j(v_k)}+\gamma+\frac{ic_k}{t} |} \lesssim t,
\end{align*}
and get the inequality \eqref{id:FtL2inter4} at step $k+1$.

\medskip

\noindent \underline{Case 3: $v_{k}\in \mathcal V^1$ and $v_{k}\notin \mathcal V^j$.} In this case, $\mathcal P_{m,k}=\mathcal P_{m,k+1}$, $n_{0,k}=n_{0,k+1}$ and $n_{1,k}=n_{1,k+1}-1$. There is a free variable $\xi^f_i$ attached to $v_k$ (which is $\xi^f_{k_i}$ for $i$ such that $k_i=k$) to integrate over. We have $\underline{\xi^f_k}=(\xi^f_{i},\underline{\xi^f_{k+1}})$ and $\underline{\alpha_{k}}=\underline{\alpha_{k+1}}$. We note that by definition of the integration order, and from the construction of the free variables $\underline{\xi^f}$ (as stated Lemma \ref{lem:deg1}), for all $\ell \geq k+1$, for all $v\in \mathcal p^+(v_{\ell})$, the quantity $\Omega_v$ is independent of $\xi^f_i$. Thus, in the integrand of \eqref{id:FtL2inter3} at Step $k$, 
\be \label{id:clustersgain6}
\Theta_k=\gamma-\Omega_{v_k}+\frac{ic_k}{t}
\ee
is the only quantity which depends on $\xi^f_i$. Moreover, $\gamma$ is independent of $\xi^f_i$, and $\Omega_{v_k}$ is given by Lemma \ref{lem:deg1}. For any fixed $(\underline{\alpha_k},\underline \eta,\underline{\xi^f_{k+1}})\in S_{\beta,k+1}$, using \eqref{id:clustersgain6}, and then either \eqref{id:formularesonancequadra} and \eqref{bd:degreeonequadra} if $v_k$ is quadratic, or \eqref{id:formularesonancelinear} and \eqref{bd:degreeonelinear2} if $v_k$ is linear (because from the assumption of this Step 2, $v_k$ is then nondegenerate on $S_\beta$), we get
$$
\int_{|\xi^f_i|\leq K\epsilon^{-1}, \ (\underline{\alpha_{k}},\underline \eta,\underline{\xi^f_k})\in S_{\beta,k}} \frac{1}{|\Theta_k|}d\xi^f_i\lesssim \int_{|\xi^f_i|\leq K\epsilon^{-1}, \ (\underline{\alpha_{k}},\underline \eta,\underline{\xi^f_k})\in S_{\beta,k}} \frac{1}{|\gamma-\Omega_{v_k}(\xi^f)+\frac{ic_k}{t}|}d\xi^f_i\lesssim \epsilon^{2-d}|\log \epsilon|.
$$
In the inequality \eqref{id:FtL2inter4} at Step $k$, we integrate over $\xi^f_i$ using the above inequality, and obtain \eqref{id:FtL2inter4} at Step $k+1$.

\medskip

\noindent \underline{Case 4: $v_{k}\in \mathcal V^1$ and $v_{k}\in \mathcal V^j$.} In this case $\mathcal P_{m,k}=\mathcal P_{m,k+1}\cup\{\mathcal p_j(v_k)\}$, $n_{0,k}=n_{0,k+1}$ and $n_{1,k}=n_{1,k+1}-1$. There are two variables to integrate over: a free variable $\xi^f_i$ attached to $v_k$ and $\alpha_{\mathcal p_j(v_k)}$, and we have $\underline{\xi^f_{k}}=(\xi^f_i,\underline{\xi^f_{k+1}})$ and $\underline{\alpha_{k}}=(\alpha_{\mathcal p_j(v_k)},\underline{\alpha_k})$. By definition of the integration order, and from Proposition \ref{pr:spanning}, for all $\ell \geq k+1$, the quantity $\Theta_\ell$ depends neither on $\alpha_{\mathcal p_j(v_k)}$, nor on $\xi^f_i$. Thus, in the integrand of \eqref{id:FtL2inter4} at Step $k$, the terms $|\alpha_{\mathcal p_j(v_k)}+\frac{c_{\mathcal p_j(v_k)}i}{t}|$ and
\be \label{id:clustersgain7}
\Theta_k=\gamma-\alpha_{\mathcal p_j(v_k)}-\Omega_{v_k}+\frac{ic_k}{t}
\ee
are the only ones which depends on $\xi^f_i$ and $\alpha_{\mathcal p_j(v_k)}$. Above, $\gamma$ is independent of $\xi^f_i$ and $\alpha_{\mathcal p_j(v_k)}$, and $\Omega_{v_k}(\xi^f)$ is given by Lemma \ref{lem:deg1}. For any fixed $(\underline{\alpha_k},\underline \eta,\underline{\xi^f_{k+1}})\in S_{\beta,k+1}$, using \eqref{id:clustersgain7} and then either \eqref{id:formularesonancequadra} and \eqref{bd:degreeonequadra} if $v_k$ is quadratic, or \eqref{id:formularesonancelinear} and \eqref{bd:degreeonelinear2} if $v_k$ is linear (because from the assumption of this Step 2, $v_k$ is then nondegenerate on $S_\beta$ and in the integral below $|\alpha_{\mathcal p_j(v_k)}|\leq \delta \epsilon^{-2}$), we get
\begin{align*}
& \int_{|\alpha_{\mathcal p_j(v_k)}|\leq \delta \epsilon^{-2}}\int_{|\xi^f_i|\leq K\epsilon^{-1}, \ (\underline{\alpha_k},\underline \eta, \underline{\xi^f_{k}})\in S_{\beta,k}}\frac{1}{|\alpha_{\mathcal p_j(v_k)}+\frac{c_{\mathcal p_j(v_k)}i}{t}|} \frac{1}{|\Theta_k|} d\xi^f_i d\alpha_{\mathcal p_j(v_k)} \\
& \lesssim \int_{|\alpha_{\mathcal p_j(v_k)}|\leq \delta \epsilon^{-2}} \frac{d\alpha_{\mathcal p_j(v_k)}}{|\alpha_{\mathcal p_j(v_k)}+\frac{c_{\mathcal p_j(v_k)}i}{t}|} \int_{|\xi^f_i|\leq K\epsilon^{-1}, \ (\underline{\alpha_k},\underline \eta, \underline{\xi^f_{k}})\in S_{\beta,k}} \frac{d\xi^f_i }{|\gamma-\alpha_{\mathcal p_j(v_k)}-\Omega_{v_k}(\xi^f)+\frac{ic_k}{t}|}\\
& \lesssim \int_{|\alpha_{\mathcal p_j(v_k)}|\leq \delta \epsilon^{-2}} \frac{d\alpha_{\mathcal p_j(v_k)}}{|\alpha_{\mathcal p_j(v_k)}+\frac{c_{\mathcal p_j(v_k)}i}{t}|}\epsilon^{2-d}|\log \epsilon|\ \lesssim \ \epsilon^{2-d}|\log \epsilon|^2.
\end{align*}
For the part of the integral for which $|\alpha_{\mathcal p_j(v_k)}|> \delta \epsilon^{-2}$ we inverse the order of integration by Fubini, simply bound $|\alpha_{\mathcal p_j(v_k)}+\frac{c_{\mathcal p_j(v_k)}i}{t}|^{-1}\lesssim \epsilon^2$ and find:
\begin{align*}
& \int_{|\alpha_{\mathcal p_j(v_k)}|\geq \delta \epsilon^{-2}}\int_{|\xi^f_i|\leq K\epsilon^{-1}, \ (\underline{\alpha_k},\underline \eta, \underline{\xi^f_{k}})\in S_{\beta,k}} \frac{1}{|\alpha_{\mathcal p_j(v_k)}+\frac{c_{\mathcal p_j(v_k)}i}{t}|} \frac{1}{|\Theta_k|} d\xi^f_i d\alpha_{\mathcal p_j(v_k)} \\
&\lesssim \int_{|\xi^f_i|\leq K\epsilon^{-1}}d\xi^f_i  \int_{|\alpha_{\mathcal p_j(v_k)}|\geq \delta \epsilon^{-2}, \ (\underline{\alpha_k},\underline \eta, \underline{\xi^f_{k}})\in S_{\beta,k}}\frac{1}{|\alpha_{\mathcal p_j(v_k)}+\frac{c_{\mathcal p_j(v_k)}i}{t}|} \frac{1}{|\gamma-\alpha_{\mathcal p_j(v_k)}-\Omega_{v_k}(\xi^f)+\frac{ic_k}{t}|}  d\alpha_{\mathcal p_j(v_k)}\\
&\lesssim \int_{|\xi^f_i|\leq K\epsilon^{-1}}d\xi^f_i  \epsilon^{2} \ \lesssim \ \epsilon^{2-d}.
\end{align*}

Combining the two inequalities above yields \eqref{id:FtL2inter4} at Step $k+1$.\\
 
By induction, we obtain that \eqref{id:FtL2inter4} holds for all $1\leq k \leq 2n+1$. To prove the final estimate \eqref{bd:FtL2}, we take $k=2n+1$ in \eqref{id:FtL2inter4}, and then integrate over the $\alpha_{\mathcal{p}(v_{\text{top}}^l)},\alpha_{\mathcal{p}(v_{\text{top}}^r)}$ variables producing a $|\log \epsilon|^2$ factor, over $\xi^f_{2n+1}$ producing a $\epsilon^{-d}$ factor, and over the $\underline \eta$ variables producing a $1$ factor:
\begin{align*}
\label{id:FtL2inter4} &\mathcal F_{t,G,\beta}  \lesssim \lambda^{2n}\ep^{d(n+1)} |\log \epsilon |^{2n}\epsilon^{(2-d)n}t^{n} \iiint_{(\alpha_{p(v_{\text{top}}^l)},\alpha_{p(v_{\text{top}}^r)} ,\underline \eta,\xi^f_{n+1}) \in \times B^{2}(0,\epsilon^{-K'})\times B^{d(n+1)}(K)\times B^1(K\epsilon^{-1})}  \\
&\qquad \qquad \qquad \qquad \qquad \qquad \qquad d\alpha_{\mathcal p(v_{\text{top}}^l)} d\alpha_{\mathcal p(v_{\text{top}}^r)} d\underline \eta d\xi^f_{n+1}\frac{1}{|\alpha_{\mathcal p(v_{\text{top}}^l)}+\frac{ic_{\alpha_{\mathcal p(v_{\text{top}}^l)}}}{t}|}\frac{1}{|\alpha_{\mathcal p(v_{\text{top}}^r)}+\frac{ic_{\alpha_{\mathcal p(v_{\text{top}}^r)}}}{t}|} \\
&\qquad \qquad  \lesssim \lambda^{2n}\ep^{d(n+1)} |\log \epsilon |^{2n}\epsilon^{(2-d)n}t^{n} \epsilon^{-d} |\log \epsilon|^2 \ = \ (\frac{t}{\lambda^{-2}\epsilon^{-2}})^n |\log \epsilon |^{2(n+1)},
\end{align*}
where we used $\int_{|\alpha|\leq \epsilon^{-K'}}d\alpha |\alpha+i/t|^{-1}\lesssim |\log \epsilon|$. The inequality \eqref{bd:FtL2} is proved, concluding Step 2.\\

\noindent \textbf{Step 3} \emph{Upper bound for \eqref{id:FtL2inter1} in presence of clusters}. We now treat the general case for which there exist degenerate vertices in the sense of Definition \ref{def:degenerate}. We apply Lemma \ref{lem:decomposition} and gather them into clusters $\mathcal C_1,...,\mathcal C_{n_d}$, and recall the decomposition $\mathcal V_i \cup \mathcal V_R = \tilde{\mathcal V} \sqcup \mathcal C_1\sqcup...\sqcup\mathcal C_{n_d(G,\beta)}$. As in Step 2, we prove \eqref{bd:FtL2} by integrating over the variables $\underline{\xi^f}$ and $\underline{\alpha}$ in \eqref{id:FtL2inter4} iteratively, according to an algorithm that considers again the interaction vertices and the root vertex $v_1,...,v_{2n+1}\in \mathcal V_i\cup \mathcal V_R$ one after the other according to the integration order.

The outcome of the strategy in Step 2 can be summarised as follows: each degree $0$ vertex produces a  factor $t$, and each non-degenerate degree $1$ vertex produces a factor $\epsilon^{2-d}|\log \epsilon|^2$. Given a cluster $\mathcal C$ containing $n_{0}(\mathcal C)\in \{0,1\}$ degree zero vertex and $n_1(\mathcal C)\in \{1,2,3\}$ degree one vertices, when reaching one of its vertices during the integration algorithm, we will perform different estimates. We will prove, overall, the same estimate for this group of vertices, that is, that $\mathcal C$ produces a $t^{n_0(\mathcal C)}(\epsilon^{2-d}|\log \epsilon|)^{n_1(\mathcal C)}$ factor.

We now consider each vertex $v_1,...,v_{2n+1}\in \mathcal V_i\cup \mathcal V_R$ one after the other according to the integration order and assume we reach $v_k$. Suppose $v_k\in \tilde{\mathcal V_i}$ does not belong to a cluster. Then we proceed as in Step 2. As a result if $v_k$ is of degree zero this produces a $t$ factor, and if $v_k$ is of degree one this produces a factor $\epsilon^{2-d}|\log \epsilon|^2$. Suppose now we reach $v_k\in \mathcal C$ the first (according to the integration order) vertex of a cluster $\mathcal C$. Suppose in addition that the vertex above $v_k$ is not the root vertex, which will be treated after. By definition $v_k$ is degenerate in the sense of Definition \ref{def:degenerate}.

\medskip

\noindent
\underline{Case 1: $\mathcal C=(v_k,v_{k'})$ is a type I cluster (in the sense of Definition \ref{def:cluster}) and $v_{k'}\in \mathcal V^0$.} Then $n_0(\mathcal C)=n_1(\mathcal C)=1$. Assume first $v_k,v_{k'}\notin \mathcal V^j$. Denote by $\xi^f$ the free variable at $v_k$. At $v_k$ we simply bound  $|\Theta_k|^{-1}\lesssim t$ and integrate over $\xi^f$:
\be \label{bd:L2inter1}
\int_{|\xi^f|\leq K\epsilon^{-1}, \ (\underline{\alpha_{k}},\underline \eta,\underline{\xi^f_k})\in S_{\beta,k}} \frac{1}{|\Theta_k|}d\xi^f\lesssim \int_{|\xi^f|\leq K\epsilon^{-1}} td\xi^f\lesssim t \epsilon^{-d}.
\ee
We then pursue the algorithm and consider the next vertices $v_{k+1}$, $v_{k+2}$, ... . When the algorithm reaches $v_{k'}$, we bound $|\Theta_{k'}|^{-1}\lesssim \epsilon^2$ by applying \eqref{bd:degenimpliesnondegenleft} since $\mathcal C$ is a type I cluster. Combining the factors we got at $v_k$ and $v_{k'}$, we find that $\mathcal C$ produced a $t \epsilon^{-d} \epsilon^2\leq t^{n_0(\mathcal C)}(\epsilon^{2-d}|\log \epsilon|)^{n_1(\mathcal C)}$ factor.

If $v_k\in \mathcal V^j$, then at $v_k$ we start by integrating over $\alpha_{\mathcal p_j(v_k)}$ using \eqref{bd:weightedintegral3} and get 
\begin{align*}
& \int_{|\alpha_{\mathcal p_j(v_k)}|\leq \epsilon^{-K'}, \ (\underline{\alpha_k},\underline \eta, \underline{\xi^f_{k}})\in S_{\beta,k}}  \frac{1}{|\alpha_{\mathcal p_j(v_k)}+\frac{i c_{\mathcal p_j(v_k)}}{t}|}\frac{d\alpha_{\mathcal p(v_k)}}{|\alpha_{\mathcal p(v_k)}-\sum_{\tilde{\mathcal p}\triangleleft v_{k}}\alpha_{\tilde{\mathcal p}}-\sum_{\tilde{v}\in \mathcal p^+(v_k)}\Omega_{\tilde{v}}+\frac{ic_k}{t}|}\\
 &\lesssim \frac{1}{|\alpha_{\mathcal p(v_k)}+\alpha_{\mathcal p_j(v_k)}-\sum_{\tilde{\mathcal p}\triangleleft v_{k}}\alpha_{\tilde{\mathcal p}}-\sum_{\tilde{v}\in \mathcal p^+(v_k)}\Omega_{\tilde{v}}+\frac{ic_k}{t}|},
\end{align*}
and we are back to the previous reasoning for the case $v_k\notin \mathcal V^j$. If $v_{k'}\in \mathcal V^j$ then the analogue estimate at $v_{k'}$, integrating first over $\alpha_{\mathcal p_j(v_{k'})}$, sends back similarly to the previous reasoning for the case $v_{k'}\notin \mathcal V^j$. Hence the same bound for $\mathcal C$ holds in the cases where $v_k\in \mathcal V^j$ or $v_{k'}\in \mathcal V^j$.

\medskip

\noindent \underline{Case 2: $\mathcal C=(v_k,v_{k'})$ is a type I cluster and $v_{k'}\in \mathcal V^1$.} Then $n_0(\mathcal C)=0$ and $n_1(\mathcal C)=2$. Denote by $\xi^f$ the free variable at $v_k$, by $\xi^{'f}$ that at $v_{k'}$, and assume $v_k,v_{k'}\notin \mathcal V^j$. Since $(v_k,v_{k'})$ is the edge above $v_k$, and is at the bottom left of $v_{k'}$, we have that in the formula at $v_k$, there holds $\tilde \xi=\xi^{'f}$ in \eqref{id:formularesonancelinear}, so that this formula gives $\Theta_k=-\sigma (\xi^f)2\xi^{'f}.\xi^f+\gamma+\frac{ic_k}{t}$ where $\gamma$ is independent of $\xi^f$. When the algorithm reaches $v_k$ we integrate over $\xi^f$ using \eqref{bd:degreeonelinear1} and obtain:
\be \label{bd:clustersinter2}
\int_{|\xi^f|\leq K\epsilon^{-1}, \ (\underline{\alpha_{k}},\underline \eta,\underline{\xi^f_k})\in S_{\beta,k}} \frac{1}{|\Theta_k|}d\xi^f\lesssim\int_{|\xi^f|\leq K\epsilon^{-1}} \frac{1}{|-\sigma (\xi^f)2\xi^{'f}.\xi^f+\gamma+\frac{ic_k}{t} |}d\xi^f \lesssim \frac{\epsilon^{1-d}|\log \epsilon|}{|\xi^{'f}|}.
\ee
When later the algorithm reaches $v_{k'}$ we bound $|\Theta_{k'}|^{-1}\lesssim \epsilon^2$ by \eqref{bd:degenimpliesnondegenleft}, and integrate the $|\xi^{'f}|^{-1}$ factor produced by \eqref{bd:clustersinter2}:
$$
\int_{|\xi^{'f}|\leq K\epsilon^{-1}, \ (\underline{\alpha_{k'}},\underline \eta,\underline{\xi^f_{k'}})\in S_{\beta,k'}}\frac{1}{|\xi^{'f}|}  \frac{1}{|\Theta_{k'}|}d\xi^{'f} \lesssim \int_{|\xi^{'f}|\leq K\epsilon^{-1}} \epsilon^2\frac{1}{|\xi^{'f}|}d\xi^{'f} \lesssim \epsilon^{3-d}.
$$
Combining the factors at $v_k$ and $v_{k'}$, $\mathcal C$ produced a $\epsilon^{1-d}|\log \epsilon|\epsilon^{3-d}\leq t^{n_0(\mathcal C)}(\epsilon^{2-d}|\log \epsilon|)^{n_1(\mathcal C)}$ factor.

If $v_k\in \mathcal V^j$ (resp. $v_{k'}\in \mathcal V^j$), integrating first over $\alpha_{\mathcal p_j(v_k)}$ (resp. $\alpha_{\mathcal p_j(v_{k'})}$) using \eqref{bd:weightedintegral3} sends back to the previous case $v_k\notin \mathcal V^j$ (resp. $v_{k'}\notin \mathcal V^j$). Details for this procedure are given in the last paragraph of Case 1 and we shall omit them here and later. Hence our method for $v_k,v_{k'}\notin \mathcal V^j$ also covers the cases $v_k,v_{k'}\in \mathcal V^j$.

\medskip

\noindent \underline{Case 3: $\mathcal C=(v_k,v_{k'})$ is a type II cluster and $v_{k'}\in \mathcal V^0$.} Assume $v_k\notin \mathcal V^j$. When reaching $v_k$ we simply bound $|\Theta_k|^{-1}\lesssim t$ and integrate over $\xi^f$:
\be \label{bd:L2inter11}
\int_{|\xi^f|\leq K\epsilon^{-1}, \ (\underline{\alpha_{k}},\underline \eta,\underline{\xi^f_k})\in S_{\beta,k}} \frac{1}{|\Theta_k|}d\xi^f\lesssim t\int_{|\xi^f|\leq K\epsilon^{-1}} d\xi^f \lesssim t \epsilon^{-d}.
\ee
Next, when reaching $v_{k'}$, we apply \eqref{bd:degenimpliesnondegenright} and get $v_{k'}\in \mathcal V^j$ with $|\alpha_{\mathcal p_j(v_{k'})}|\geq \delta \epsilon^{-2}$. Writing $\Theta_k=\gamma -\alpha_{\mathcal p_j(v_{k'})}+\frac{ic_k}{t}$ where $\gamma$ is independent of $\alpha_{\mathcal p_j(v_{k'})}$, integrating over $\alpha_{\mathcal p_j(v_{k'})}$ using the previous bound we get:
\begin{align*}
&\int_{|\alpha_{\mathcal p_j(v_{k'})}|\leq \epsilon^{-K'}, \ (\underline{\alpha_{k'}},\underline \eta,\underline{\xi^{f}_{k'}})\in S_{\beta,k'}} \frac{d \alpha_{\mathcal p_j(v_{k'})}}{|\alpha_{\mathcal p_j(v_k)}+\frac{ic_{\mathcal p_j(v_k)}}{t}|} \frac{1}{|\Theta_k|} \lesssim \int_{\delta \epsilon^{-2}\leq |\alpha_{\mathcal p_j(v_{k'})}|\leq \epsilon^{-K'}}   \frac{\epsilon^2 d \alpha_{\mathcal p_j(v_{k'})}}{|\gamma -\alpha_{\mathcal p_j(v_{k'})}+\frac{ic_k}{t}|} \ \lesssim \ \epsilon^2|\log \epsilon|.
\end{align*}
The factors we got at $v_k$ and $v_{k'}$ thus give that $\mathcal C$ produced a $t\epsilon^{-d}\epsilon^2|\log \epsilon |\leq t^{n_0(\mathcal C)}(\epsilon^{2-d}|\log \epsilon|)^{n_1(\mathcal C)}$ factor. As in Case 1 and 2, the case $v_k\in \mathcal V^j$ can be dealt with the exact same way by integrating first over $\alpha_{\mathcal p_j(v_k)}$, we refer to the last paragraph of Case 1 for details.

\medskip

\noindent \underline{Case 4: $\mathcal C=(v_k,v_{k'})$ is a type II cluster and $v_{k'}\in \mathcal V^1$.} Assume $v_k\notin \mathcal V^j$. Let $\xi^f$ and $\xi^{'f}$ be the free variables at $v_k$ and $v_{k'}$ respectively. Let $\tilde \xi$ (resp. $\tilde \xi'$) denote the variable associated to the edge on top of $v_{k}$ (resp. $v_{k'}$). As $\mathcal C$ is a type II cluster, we have by Kirchhoff law at $v_{k'}$ that $\tilde \xi=\tilde \xi'-\xi^{'f}$ and that $\tilde \xi'$ depends neither on $\xi^f$ nor on $\xi^{'f}$. Hence \eqref{id:formularesonancelinear} gives $\Theta_k=\sigma(\xi^f)(\xi^{'f}-\tilde \xi') \cdot \xi^{f}+\gamma+\frac{ic_k}{t}$
 where $\gamma$ is independent of $\xi^f$. At $v_k$ we integrate over $\xi^f$ using this identity and \eqref{bd:degreeonelinear1}, giving:
\be \label{bd:clustersinter3}
\int_{|\xi^f|\leq K\epsilon^{-1}, \ (\underline{\alpha_{k}},\underline \eta,\underline{\xi^f_k})\in S_{\beta,k}} \frac{1}{|\Theta_k|}d\xi^f \lesssim \int_{|\xi^f|\leq K\epsilon^{-1}} \frac{1}{|\sigma(\xi^f)(\xi^{f'}-\tilde \xi')\cdot\xi^{f}+\gamma+\frac{ic_k}{t}|}d\xi^f\lesssim \frac{\epsilon^{1-d}|\log \epsilon|}{|\xi^{'f}-\tilde \xi'|}.
\ee
Next, at $v_{k'}$, we apply \eqref{bd:degenimpliesnondegenright} so that $v_{k'}\in \mathcal V^j$ with $|\alpha_{\mathcal p_j(v_{k'})}|\geq \delta \epsilon^{-2}$ on $S_{\beta,k'}$. Moreover, $\Theta_{k'}=-\alpha_{\mathcal p_j(v_{k'})}+\gamma'+\frac{ic_{k'}}{t}$ with $\gamma'$ independent of $\alpha_{\mathcal p_j(v_{k'})}$. We first integrate over $\alpha_{\mathcal p_j(v_{k'})}$ using these bound and equality, producing a factor $\epsilon^2 |\log \epsilon|$, and then integrate the $|\xi^{'f}-\tilde \xi'|^{-1}$ gained from \eqref{bd:clustersinter3} over $\xi^{f'}$, producing an $\epsilon^{1-d}$ factor, resulting in:
\begin{align*}
&\int_{|\alpha_{\mathcal p_j(v_{k'})}|\leq \epsilon^{-K'}, \ |\xi^{'f}|\leq K\epsilon^{-2} \ (\underline{\alpha_{k'}},\underline \eta,\underline{\xi^{f}_{k'}})\in S_{\beta,k'}} d \alpha_{\mathcal p_j(v_{k'})} d\xi^{'f} \frac{1}{|\xi^{'f}-\tilde \xi'|} \frac{1}{|\alpha_{\mathcal p_j(v_{k'})}+\frac{ic_{\mathcal p_j(v_{k'})}}{t}|} \frac{1}{|\Theta_{k'}|} \lesssim \epsilon^{3-d}|\log \epsilon|.
\end{align*}
The factors obtained at $v_k$ and $v_{k'}$ give a total factor for $\mathcal C$ of $\epsilon^{1-d}|\log \epsilon| \epsilon^{3-d}|\log \epsilon|\leq t^{n_0(\mathcal C)}(\epsilon^{2-d}|\log \epsilon|)^{n_1(\mathcal C)}$. Again, as in all previous cases, the subcase $v_k\in \mathcal V^j$ can be dealt with the exact same way by integrating first over $\alpha_{\mathcal p_j(v_k)}$, see Case 1 for details.

\medskip

\noindent \underline{Case 5: $\mathcal C=(v_k,v_{k'},v_{k''})$ is a type III cluster and $v_{k''}\in \mathcal V^0$}, so $n_0(\mathcal C)=1$ and $n_1(\mathcal C)=2$. Assume $v_k,v_{k'}\notin \mathcal V^j$. Assume firstly $v_k=v_l(v_{k''})$ is before $v_{k'}=v_r(v_{k''})$ in the integration order. Let $\xi^f$ and $\xi^{'f}$ be the free variables at $v_k$ and $v_{k'}$ respectively.

At $v_k$ we bound $|\Theta_k|^{-1}\lesssim t$ and integrate over $\xi^f$:
\be \label{bd:clustersinter10}
\int_{|\xi^f|\leq K\epsilon^{-1}, \ (\underline{\alpha_{k}},\underline \eta,\underline{\xi^f_k})\in S_{\beta,k}} \frac{1}{|\Theta_k|}d\xi^f\lesssim \int_{|\xi^f|\leq K\epsilon^{-1}} td\xi^f\lesssim t \epsilon^{-d}.
\ee
At $v_{k'}$ we first upper bound the factor associated to $v_{k''}$ as $|\Theta_{k''}|^{-1}\lesssim \epsilon^2$ by applying \eqref{bd:degenimpliesnondegenleft}, bound $|\alpha_{\mathcal p(v_{k'})}+\frac{ic_{\mathcal p(v_k)}}{t}|^{-1}\lesssim \epsilon^2$ by applying \eqref{bd:degenimpliesnondegenright} and write $\Theta_{k'}=\alpha_{\mathcal p(v_{k'})}+\gamma+\frac{c_{k'}i}{t}$ where $\gamma$ is independent of $\alpha_{\mathcal p(v_{k'})}$. Note that these three terms are the only ones depending on $\alpha_{\mathcal p(v_{k'})}$ in the right hand side of \eqref{id:FtL2inter4} at Step $k'$. After plugging these bounds, we integrate with respect to $\alpha_{\mathcal p(v_{k'})}$ producing a $|\log \epsilon|$ factor, and then over $\xi^{'f}$ producing a $\epsilon^{-d}$ factor, and obtain:
\begin{align}
\label{bd:clustersinter11} &\int_{|\alpha_{\mathcal p(v_{k'})}|\leq \epsilon^{-K'}, \ |\xi^{'f}|\leq K\epsilon^{-2} \ (\underline{\alpha_{k'}},\underline \eta,\underline{\xi^{f}_{k'}})\in S_{\beta,k'}}  \frac{1}{|\alpha_{\mathcal p(v_{k'})}+\frac{ic_{\mathcal p_j(v_{k'})}}{t}|} \frac{1}{|\Theta_{k'}|} \frac{1}{|\Theta_{k''}|}d \alpha_{\mathcal p(v_{k'})} d\xi^{'f} \lesssim \epsilon^{4-d}|\log \epsilon|.
\end{align}
When reaching $v_{k''}$, we do not do anything, resulting in a $1$ factor. Combining the factors obtained at $v_k$, $v_{k'}$ and $v_{k''}$ give a total factor for $\mathcal C$ of $t\epsilon^{-d}\epsilon^{4-d}|\log \epsilon|\leq t^{n_0(\mathcal C)}(\epsilon^{2-d}|\log \epsilon|)^{n_1(\mathcal C)}$. 

Assume secondly $v_{k'}=v_r(v_{k''})$ is before $v_{k}=v_l(v_{k''})$ in the integration order. The same reasoning applies. Indeed, when reaching $v_{k'}$ we perform the estimate \eqref{bd:clustersinter10} (replacing the $k$ notation by $k'$ in this inequality). Next when when reaching $v_{k}$ we perform the estimate \eqref{bd:clustersinter11}, integrating over $\xi^{f}$ and $\alpha_{\mathcal p(v_{k'})}$ (which is permitted as $\Theta_k=\gamma'-\alpha_{\mathcal p(v_{k'})}+\frac{ic_k}{t}$ with $\gamma'$ independent of $\alpha_{\mathcal p(v_{k'})}$ since $v_k\triangleright \mathcal p(v_{k'})$). This produces the same $t\epsilon^{4-2d}|\log \epsilon|$ factor for $\mathcal C$.

Again, the subcase $v_k,v_{k'}\in \mathcal V^j$ can be dealt with in the same way by integrating first over $\alpha_{\mathcal p_j(v_k)}$ and $\alpha_{\mathcal p_j(v_{k'})}$, see Case 1 for details.

\medskip

\noindent 
\underline{Case 6: $\mathcal C=(v_k,v_{k'},v_{k''})$ is a type III cluster and $v_{k''}\in \mathcal V^1$.} Assume $v_k,v_{k'}\notin \mathcal V^j$. Assume firstly $v_k=v_l(v_{k''})$ is before $v_{k'}=v_r(v_{k''})$ in the integration order. Let $\xi^f,\xi^{'f},\xi^{''f}$ be the free variables at $v_k,v_{k'},v_{k''}$.

At $v_k$, we note that $\xi^{''f}$ is the variable associated to the edge above $v_k$, so that \eqref{id:formularesonancelinear} gives $\Theta_k=-\sigma(\xi^f)2\xi^{''f}.\xi^f+\gamma+\frac{c_ki}{t}$ with $\gamma$ independent of $\xi^f$. We integrate over $\xi^f$ using \eqref{bd:degreeonelinear1} and get:
\be \label{bd:clustersinter12}
\int_{|\xi^f|\leq K\epsilon^{-1}, \ (\underline{\alpha_{k}},\underline \eta,\underline{\xi^f_k})\in S_{\beta,k}} \frac{1}{|\Theta_k|}d\xi^f\lesssim \int_{|\xi^f|\leq K\epsilon^{-1}} \frac{1}{|-\sigma(\xi^f)2\xi^{''f}.\xi^f+\gamma+\frac{c_ki}{t}|}d\xi^f \lesssim \frac{\epsilon^{1-d}|\log \epsilon|}{|\xi^{''f}|}.
\ee
At $v_{k'}$ we first upper bound the factor associated to $v_{k''}$ as $|\Theta_{k''}|^{-1}\lesssim \epsilon^2$ by applying \eqref{bd:degenimpliesnondegenleft}, bound $|\alpha_{\mathcal p(v_{k'})}+\frac{ic_{\mathcal p(v_k)}}{t}|^{-1}\lesssim \epsilon^2$ by applying \eqref{bd:degenimpliesnondegenleft} and write $\Theta_{k'}=\alpha_{\mathcal p(v_{k'})}+\gamma'+\frac{ic_{k'}}{t}$ where $\gamma'$ is independent of $\alpha_{\mathcal p(v_{k'})}$. We integrate with respect to $\alpha_{\mathcal p(v_{k'})}$ producing a $|\log \epsilon|$ factor, and then over $\xi^{'f}$ producing a $\epsilon^{-d}$ factor, and obtain:
\begin{align}
\label{bd:clustersinter13} &\int_{|\alpha_{\mathcal p(v_{k'})}|\leq \epsilon^{-K'}, \ |\xi^{'f}|\leq K\epsilon^{-2} \ (\underline{\alpha_{k'}},\underline \eta,\underline{\xi^{f}_{k'}})\in S_{\beta,k'}}  \frac{1}{|\alpha_{\mathcal p(v_{k'})}+\frac{ic_{\mathcal p_j(v_{k'})}}{t}|} \frac{1}{|\Theta_{k'}|} \frac{1}{|\Theta_{k''}|}d \alpha_{\mathcal p(v_{k'})} d\xi^{'f} \lesssim \epsilon^{4-d}|\log \epsilon|.
\end{align}
When reaching $v_{k''}$, integrate over the variable $\xi^{''f}$ the $|\xi^{''f}|^{-1}$ factor produced by \eqref{bd:clustersinter12}, giving $\int_{|\xi^{''f}|\leq K\epsilon^{-1}, \ (\underline{\alpha_{k''}},\underline \eta,\underline{\xi^f_{k'}})\in S_{\beta,k''}}|\xi^{''f}|^{-1}d\xi^{''f}\lesssim \epsilon^{1-d}.$ Combining the factors obtained at $v_k$, $v_{k'}$ and $v_{k''}$ give a total factor for $\mathcal C$ of $\epsilon^{1-d}|\log \epsilon|\epsilon^{4-d}|\log \epsilon|\epsilon^{1-d}\leq t^{n_0(\mathcal C)}(\epsilon^{2-d}|\log \epsilon|)^{n_1(\mathcal C)}$. 

Assume secondly $v_{k'}=v_r(v_{k''})$ is before $v_{k}=v_l(v_{k''})$ in the integration order. We reason the same way. When reaching $v_{k'}$ the variable associated to the edge on top of $v_{k'}$ is by Kirchhoff law $\tilde \xi^{''}-\xi^{''f}$ where $\tilde \xi^{''}$ is independent of $\xi^{f},\xi^{'f},\xi^{''f}$. We perform the estimate \eqref{bd:clustersinter12}, producing a $\epsilon^{1-d} |\tilde \xi^{''}-\xi^{''f}||\log\epsilon|$ factor. Next when when reaching $v_{k}$ we perform the estimate \eqref{bd:clustersinter13}, integrating over $\xi^{f}$ and $\alpha_{\mathcal p(v_{k'})}$ (which is permitted as $\Theta_k=\gamma''-\alpha_{\mathcal p(v_{k'})}+\frac{c_k i}{t}$ with $\gamma''$ independent of $\alpha_{\mathcal p(v_{k'})}$ since $v_k\triangleright \mathcal p(v_{k'})$), producing a $\epsilon^{4-d}|\log \epsilon|$ factor. At $v_{k''}$, we integrate over the variable $\xi^{''f}$ the $|\tilde \xi^{''}-\xi^{''f}|^{-1}$ factor produced at $v_{k'}$, giving a $\epsilon^{1-d}$ factor. The total factor for $\mathcal C$ is again $\epsilon^{6-3d}|\log \epsilon|^2$.

Again, the subcase $v_k,v_{k'}\in \mathcal V^j$ can be dealt with in the same way by integrating first over $\alpha_{\mathcal p_j(v_k)}$ and $\alpha_{\mathcal p_j(v_{k'})}$, see Case 1 for details.

\medskip

\noindent \underline{End of the proof}. After all interaction vertices have been considered, the last step of the algorithm considers the root vertex.

In the first case, $v_R\in \tilde{\mathcal V}$ does not belong to a cluster. We perform the same estimate as in the end of Step 2, resulting in the same $\epsilon^{-d}|\log \epsilon|^2$ factor. As each degree zero vertex (resp. degree one vertex) considered in the non-degenerate set $\tilde{\mathcal V}$ or in a cluster $\mathcal C$, produced a $t$ factor (resp. a $\epsilon^{2-d}|\log \epsilon|^2$ factor), the final estimate is, using \eqref{perdrix}:
\begin{align} \label{bd:finalfornonzero}
\mathcal F_{t,G,\beta}  \lesssim \lambda^{2n}\ep^{d(n+1)} t^{n_0} (\epsilon^{2-d}|\log \epsilon|^2)^{n_1} \epsilon^{-d}|\log \epsilon|^2 \lesssim (\frac{t}{\lambda^{-2}\epsilon^{-2}})^n |\log \epsilon |^{2(n+1)},
\end{align}
which proves \eqref{bd:FtL2}.

In the second case, $v_R\in \mathcal C$ belongs to a cluster. Then $\mathcal C=\{v_R\}\cup \mathcal C'$ with $\mathcal C'$ being either $\{v^l_{\text{top}}\}$, $\{v^r_{\text{top}}\}$ or $\{v^l_{\text{top}},v^r_{\text{top}}\}$. Let $v_k\in \mathcal C'$ and assume the algorithm reaches $v_k$ with free variable $\xi_f$. Assume $v_k\notin \mathcal V^i$. Then \eqref{bd:degenimpliesnondegenrightbis} implies that $|\alpha_{\mathcal p(v_k)}+\frac{ic_{\mathcal p(v_k)}}{t}|\lesssim \epsilon^2$. Moreover, $\Theta_k=\alpha_{\mathcal p(v_k)}+\gamma+\frac{i c_k}{t}$ with $\gamma$ independent of $\alpha_{\mathcal p(v_k)}$. We use these bound and equality and integrate with respect to $\alpha_{\mathcal p(v_k)}$:
$$
\int_{|\alpha_{\mathcal p(v_{k})}|\leq \epsilon^{-K'}, \ (\underline{\alpha_{k}},\underline \eta,\underline{\xi^{f}_{k}})\in S_{\beta,k}}  \frac{d \alpha_{\mathcal p(v_{k'})}}{|\alpha_{\mathcal p(v_{k})}+\frac{ic_{\mathcal p_j(v_{k})}}{t}|} \frac{1}{|\Theta_{k}|}  \lesssim \int_{\delta\epsilon^2\leq |\alpha_{\mathcal p(v_{k})}|\leq \epsilon^{-K'}}  \frac{\epsilon^2 d \alpha_{\mathcal p(v_{k})}}{|\alpha_{\mathcal p(v_k)}+\gamma+\frac{i}{t}|} \lesssim \epsilon^2|\log \epsilon|.
$$
We then integrate with respect to $\xi^{f}$ producing an additional $\epsilon^{-d}$ factor. If $v_k\in \mathcal V^i$, as in all previous cases, we first integrate over $\alpha_{\mathcal p_j(v_k)}$ and then we are back to estimating as in the case $v_k\notin \mathcal V^i$, see the end of Case 1 for details. Hence at $v_k$ we obtained a $ \epsilon^{2-d}|\log \epsilon|$ factor, which is the same as in Step 2 for a nondegenerate degree one vertex.

We perform this estimate for all the vertices in $\mathcal C'$, so that $\mathcal C'$ produced a total factor of $(\epsilon^{2-d}|\log \epsilon|)^{n_1(\mathcal C)}$ as in the previous cases of Step 3. Finally, when reaching $v_R$, we integrate the remaining $|\alpha_{\mathcal p(v)}+\frac{ic_{\mathcal p(v)}}{t}|d\alpha_{\mathcal p(v)}$ terms for $v\in \{v_{\text{top}}^l,v^{r}_{\text{top}}\}\backslash \mathcal C'$ (if any), giving in a $|\log \epsilon|^{2-\# \mathcal C'}$ factor, and we integrate $d\xi^f_{n+1}$ over the ball $|\xi^f_{n+1}|\leq K\epsilon^{-1}$, producing a $\epsilon^{-d}$ factor. Hence in this second case we got the same estimate as in the first case, and \eqref{bd:finalfornonzero} is obtained as well, ending the proof of \eqref{bd:FtL2}.\\

\noindent \textbf{Step 4} \emph{The case of equation \eqref{nonlinschrod} with $\omega_0=0$ and $m(0)=0$}. This case is simpler since it corresponds to Step 2 and avoids the use of clusters to deal with degenerate vertices. More precisely, in this case it suffices from \eqref{id:fouriermultipliergeneralized} to prove the bound \eqref{bd:FtL2} for expressions of the form:
\begin{align}
&\label{id:FtL2inter29} \mathcal F_{t,G,\beta} =\lambda^{2n}\ep^{d(n+1)} \iiint_{(\underline \alpha ,\underline \eta,\underline \xi^f) \in S_\beta} d\underline{\alpha} \,  d \underline{\eta}    \,   d\underline{\xi^f}   \prod_{\mathcal p \in \mathcal P_{m}} \frac{1}{|\alpha_{\mathcal p}+\frac{ic_{\mathcal p}}{t}|}  \prod_{k=1}^{2n} \frac{m(\epsilon \tilde \xi_k)}{|\Theta_k|},
\end{align}

We estimate again according to an algorithm that considers the vertices $v_1,...,v_{2n+1}$ one after another according to the integration order. When the algorithm reaches a vertex $v_k$, if $v_k$ is non-degenerate on $S_\beta$, we apply the same study as in Step 2. As a result, a degree 0 vertex (resp. degree 1) produces a $t$ factor (resp. a $\epsilon^{2-d}|\log \epsilon|^2$ factor). Assume now $v_k$ is a degree one linear vertex with free variable $\xi^f$ that is degenerate on $S_\beta$. Then the variable $\tilde \xi_k $ that is associated to the edge above $v_k$ satisfies $|\tilde \xi_k|\leq \delta \epsilon^{-1}$ from Definition \ref{def:degenerate}, so that $m(\epsilon \tilde \xi_k)=O(|\epsilon \tilde \xi_k|)$ for $\delta$ small since $m(0)=0$. Using this and \eqref{id:formularesonancelinear}, we integrate $m(\tilde \xi_k)|\Theta_k|^{-1}$ with respect to the variable $\xi^f$ applying Corollary~\ref{lagopede}, and get a factor $\epsilon^{2-d} |\log \epsilon|$. Hence at this vertex we get the usual estimate for nondegenerate vertices. The rest of the proof of \eqref{id:fouriermultipliergeneralized} is exactly the same as in Step 2.

\end{proof}

\subsection{The $X^{s,b}$ estimate} \label{proofxsb}
The proof follows the same strategy as that of the $L^2$ norm, we will simply highlight what are the necessary modifications. Recall the identities $u^n=\sum_{G\in \mathcal G_n}u_G$ and $u_G=u_G^++u_G^-$. Apply the resolvent identity of Lemma \ref{lem:resolvantimproved} with $\eta = \frac{1}{T}$ to \eqref{id:formulaungraphs2}, and then integrating along the $\underline s$ variables one obtains
\begin{align*}
\widehat{u^+_G}(t,\xi_R)& = e^{-it\omega (\xi_R)}  \left(\frac{-i\lambda}{(2\pi)^{d/2}}\right)^n \sum_{G\in \mathcal G_n}\frac{(-1)^{\sigma_G}c_G^{ \frac{t}{T}}}{(2\pi)^{n_m}}   \int_{\mathbb R^{d(2n+1)}}  \int_{\mathbb R^{n_m}} \, d \underline{\xi} \, d \underline{\alpha} \, \Delta_{\xi_R}( \underline{\xi}) e^{-i\alpha_{\mathcal p(v_{\text{top}})} t}  \\
&\qquad  M(\underline \xi) \prod_{\mathcal p \in \mathcal P_{m}} \frac{i}{\alpha_{\mathcal p}+i\frac{c_{\mathcal p}}{T}} \prod_{i = 1}^{n+1} \widehat u_0(\xi_{0,i},\sigma_{0,i})  \prod_{v\in \mathcal V_i} \frac{i}{\alpha_{\mathcal p(v)}-\sum_{\tilde{\mathcal p}\triangleleft v}\alpha_{\tilde{\mathcal p}}-\sum_{\tilde{v}\in \mathcal p^+(v)}\Omega_{\tilde{v}}+i\frac{c_v}{T}}
\end{align*}

(a similar formula holds for $u_G^-$ using \eqref{id:formulaungraphs-}). This yields the following expression for the spacetime Fourier transform of ${\bf 1 }(t\geq 0)u^n$ (notice that the $c_{G}^{\frac tT}$ factor has been absorbed in the cut-off $\chi(t/T)$ to simplify notations):
\begin{align*}
 \mathcal F\left(\chi\left(\frac t T\right){\bf 1}(t\geq 0)u^n \right)(\tau,\xi_R) & =  \left(\frac{-i\lambda}{(2\pi)^{d/2}}\right)^n \sum_{G\in \mathcal G_n}\frac{(-1)^{\sigma_G}}{(2\pi)^{n_m+\frac 12}}   \int_{\mathbb R^{d(2n+1)}}  \int_{\mathbb R^{n_m}} \, d \underline{\xi} \, d \underline{\alpha} \, \Delta_{\xi_R}( \underline{\xi})  \\
&\qquad \qquad \qquad T\hat \chi(T(\tau+\omega(\xi_R)+\alpha_{\mathcal p(v_{\text{top}})})) M(\underline \xi) \\
&\qquad \prod_{\mathcal p \in \mathcal P_{m}} \frac{i}{\alpha_{\mathcal p}+i\frac{c_{\mathcal p}}{T}} \prod_{i = 1}^{n+1} \widehat u_0(\xi_{0,i},\sigma_{0,i})  \prod_{v\in \mathcal V_i} \frac{i}{\alpha_{\mathcal p(v)}-\sum_{\tilde{\mathcal p}\triangleleft v}\alpha_{\tilde{\mathcal p}}-\sum_{\tilde{v}\in \mathcal p^+(v)}\Omega_{\tilde{v}}+i\frac{c_v}{T}}.
\end{align*}
(again, a similar formula holds for ${\bf 1}(t< 0)u^n$ from \eqref{id:formulaungraphs-}). We keep all notations from Subsection \ref{pairinggraphs}. The identity corresponding to \eqref{sittelle} is now
$$
\mathbb E \left\| \chi\left(\frac tT\right){\bf 1}(t\geq 0) u^n \right\|_{X^{s,b}_\epsilon}^{2} =\sum_{G\in \mathcal G_n^p} \mathcal F_T(G)
$$
with, given a paired graph $G\in \mathcal G_n^p$ (recalling that for such a graph $\xi_{v_{\text{top}}^l}+\xi_{v_{\text{top}}^r}=0$, and changing variables $\tau \mapsto \tau +\omega(\xi_{v_{\text{top}}^l})$):
\begin{align}
&\label{bd:mathcalFXsbestimate1} \mathcal F_T(G)= \frac{(-1)^{\sigma_{G}}}{(2\pi)^{n_m-\frac d2}} \lambda^{2n}\ep^{d(n+1)} \iiint  \,d \underline{\xi} \,  d\underline{\alpha} \, d \tau \Delta_G(\underline{\xi}) \langle \epsilon \xi_{v_{\text{top}}^l}\rangle^{2s} \langle \tau \rangle^{2b} M(\underline \xi) \\
 & \nonumber \qquad \qquad \qquad \qquad \qquad   \prod_{\{i,j\}\in P} \widehat{W_0^\epsilon} (\eta_{i,j},\frac{\ep}{2}(\sigma_{0,i} \xi_{0,i}+  \sigma_{0,j} \xi_{0,j}))  T\hat \chi(T(\tau+\alpha_{\mathcal p(v_{\text{top}}^l)}))  T\hat \chi(T(\tau+\alpha_{\mathcal p(v_{\text{top}}^r)}))  \\
&\nonumber   \qquad \qquad \qquad  \qquad \qquad \qquad \qquad \qquad  \prod_{\mathcal p \in \mathcal P_{m}} \frac{i}{\alpha_{\mathcal p}+\frac{ic_{\mathcal p}}{T}}  \prod_{v\in \mathcal V_i} \frac{i}{\alpha_{\mathcal p(v)}-\sum_{\tilde{\mathcal p}\triangleleft v}\alpha_{\tilde{\mathcal p}}-\sum_{\tilde{v}\in \mathcal p^+(v)}\Omega_{\tilde{v}}+\frac{ic_v}{T}}.
\end{align}

\begin{proof}[Proof of \eqref{bd:estimationunXsb} in Proposition \ref{propexpansion}]

We prove the desired bound for $n\geq 1$ and ${\bf 1 }(t\geq 0)u^n$. Indeed, the bound for $u^0$, the free evolution of the initial datum, is a direct computation, and the proof of the bound for ${\bf 1 }(t\leq 0)u^n$ for $n\geq 1$ is the same as that for ${\bf 1 }(t\geq 0)u^n$ using \eqref{id:formulaungraphs-}. The proof is so similar to that of \eqref{bd:estimationunL2} that we only highlight the differences.

It suffices to estimate \eqref{bd:mathcalFXsbestimate1}. We solve Kirchhoff's laws with Proposition \ref{pr:spanning} and reduce the integration over the free variables $\underline{\xi^f}$ and $\underline \eta$. We put absolute values in the integrand. Next, we upper bound $T\hat \chi (Tz)\lesssim |z+\frac{i}{T}|^{-1}$ since $\chi$ is in the Schwartz class. Arguing exactly as in the beginning of the proof of \eqref{bd:estimationunL2}, the product $  \prod_{\{i,j\}\in P} \widehat{W_0^\epsilon} (\eta_{i,j},\frac{\ep}{2}(\sigma_{0,i} \xi_{0,i}+  \sigma_{0,j} \xi_{0,j}))$ is zero unless $|\underline \xi|\leq K \epsilon^{-1}$ for some $K(a,n)>0$. In particular, $|\xi_{v_{\text{top}}^l}|\lesssim \epsilon^{-1}$ on the support of the integrand, where we simply bound $\langle \epsilon \xi_{v_{\text{top}}^l}\rangle^{2s}\lesssim 1$. This gives:
\begin{align}
\nonumber&  |\mathcal F_T(G)|\lesssim \lambda^{2n}\ep^{d(n+1)} \iiiint_{|\underline \xi^f|\leq K\epsilon^{-1}, |\underline \eta|\leq K}  \,d \underline{\xi^f} \, d \underline{\eta} \,  d\underline{\alpha}\, d\tau   \langle \tau \rangle^{2b}  \frac{1}{|\tau+\alpha_{\mathcal p(v_{\text{top}}^l)}+\frac{i}{T}|}  \frac{1}{|\tau+\alpha_{\mathcal p(v_{\text{top}}^r)}+\frac{i}{T}|} \\
\nonumber &  \qquad \qquad \qquad  \qquad \qquad M(\underline \xi) \prod_{\mathcal p \in \mathcal P_{m}} \frac{1}{|\alpha_{\mathcal p}+\frac{ic_{\mathcal p}}{T}|}  \prod_{v\in \mathcal V_i} \frac{1}{|\alpha_{\mathcal p(v)}-\sum_{\tilde{\mathcal p}\triangleleft v}\alpha_{\tilde{\mathcal p}}-\sum_{\tilde{v}\in \mathcal p^+(v)}\Omega_{\tilde{v}}+\frac{ic_v}{T}|}\\
\label{id:decompositionXsbcomputation} & =  \underbrace{ \lambda^{2n}\ep^{d(n+1)}  \iiint_{|\underline \xi^f|\leq K\epsilon^{-1}, \ |\underline \eta|\leq K, \ |(\underline \alpha,\tau)|\leq \epsilon^{-K'}}}_{=\mathcal F_{1,T}(G)} ... \ + \underbrace{ \lambda^{2n}\ep^{d(n+1)}  \iiint_{|\underline \xi^f|\leq K\epsilon^{-1}, \ |\underline \eta|\leq K, \ |(\underline \alpha,\tau)|> \epsilon^{-K'}} ...}_{=\mathcal F_{2,T}(G)} 
\end{align}

Arguing as in the first step of the proof of \eqref{bd:estimationunL2}, for any $b\in (\frac 12, \frac 34]$ the second term is of higher order and enjoys the estimate:
$$
|\mathcal F_{2,T}(G)|\lesssim \epsilon^{\frac{K'}{2}}
$$
We are left with estimating $\mathcal F_{1,T}(G)$ that we decompose as in the proof of Proposition \ref{propexpansion} as $|\mathcal F_{1,T}(G)|\lesssim \sum_\beta \mathcal F_{T,G,\beta}$ where:
$$
\mathcal F_{T,G,\beta}= \lambda^{2n}\ep^{d(n+1)}  \iiint_{(\underline \alpha,\underline \eta,\underline \xi^f)\in S_\beta, \ |\tau|\leq \epsilon^{-K'}}  ... \ 
$$

In the integrand in \eqref{id:decompositionXsbcomputation}, the only novelty when comparing with the identity \eqref{id:mathcalFGP3} for the computation of the $L^2$ norm, is the addition of the $\tau$ variable and of the $\langle \tau \rangle^{2b}  |\tau+\alpha_{\mathcal p(v_{\text{top}}^l)}+\frac{i}{T}|^{-1}  |\tau+\alpha_{\mathcal p(v_{\text{top}}^r)}+\frac{i}{T}|^{-1}$ factor. Note that this factor only involves $\tau$, $\alpha_{\mathcal p(v_{\text{top}}^l)}$ and $\alpha_{\mathcal p(v_{\text{top}}^r)}$. We will estimate the integral $\mathcal F_{T,G,\beta}$ by considering vertices one by one according to the integration order. For each vertex $v\notin \{v_R,v_{\text{top}}^l,v_{\text{top}}^r\}$ that is neither one of the top vertices nor the root vertex, we perform the exact same estimates as for the proof of \eqref{bd:estimationunL2}. We will thus only perform different estimates at $v_R,v_{\text{top}}^l,v_{\text{top}}^r$ which we now describe.\\

\noindent \textbf{Step 1} \emph{If $m(0)=0$, or $\omega_0>\epsilon^{-2}$ and $v_R$ is not in a cluster}. In this case, when reaching $v_{\text{top}}^l$ and $v_{\text{top}}^r$ we perform the same estimates as in the proof of \eqref{bd:estimationunL2} (thus, the same estimates as in the proof of \eqref{bd:estimationunL2} have been performed at all interaction vertices). When reaching the root vertex, this produces the intermediate estimate:
\begin{align*}
\mathcal F_{T,G,\beta}&\lesssim \lambda^{2n}\epsilon^{(d(n+1))} t^n \epsilon^{(2-d)n}|\log \epsilon|^{2n}t^n  \iiiint_{|\tau|\leq \epsilon^{K'}, \ |\xi^f_{n+1}|\leq K\epsilon^{-1}, \ |(\alpha_{\mathcal p(v_{\text{top}}^l)},\alpha_{\mathcal p(v_{\text{top}}^r)})|\leq \epsilon^{-K'}}d \tau \, d \xi^f_{n+1} \\
&  d \alpha_{\mathcal p(v_{\text{top}}^l)} \, d \alpha_{\mathcal p(v_{\text{top}}^r)} \langle \tau \rangle^{2b}\frac{1}{|\tau +\alpha_{\mathcal p(v_{\text{top}}^l)}+\frac iT|}\frac{1}{|\tau +\alpha_{\mathcal p(v_{\text{top}}^r)}+\frac iT|}\frac{1}{|\alpha_{\mathcal p(v_{\text{top}}^l)}+\frac{c_{\mathcal p(v_{\text{top}}^l)}i}{T}|}\frac{1}{|\alpha_{\mathcal p(v_{\text{top}}^r)}+\frac{c_{\mathcal p(v_{\text{top}}^r)}i}{T}|}.
\end{align*}
We integrate over $\alpha_{\mathcal p(v_{\text{top}}^l)}$ and $\alpha_{\mathcal p(v_{\text{top}}^r)}$ using \eqref{bd:weightedintegral3}, then over $\tau$ and finally over $\xi^f_{n+1}$ and get:
\begin{align*}
\mathcal F_{T,G,\beta} &\lesssim \lambda^{2n}\epsilon^{2n+d} t^n|\log \epsilon|^{2n}   \iint_{|\tau|\leq \epsilon^{K'}, \ |\xi^f_{n+1}|\leq K\epsilon^{-1}} d\tau \, d \xi^f_{n+1} \langle \tau \rangle^{2b}\frac{1}{|\tau+\frac iT|}\frac{1}{|\tau +\frac iT|}\\
&\lesssim \lambda^{2n}\epsilon^{2n} t^n|\log \epsilon|^{2n+d} \epsilon^{-K'(2b-1)} \ \lesssim \epsilon^{-\kappa} (\frac{t}{T_{kin}})^n
\end{align*}
for any $\kappa>0$ if $b>\frac 12$ has been chosen close enough to $\frac 12$.\\

\noindent \textbf{Step 2} \emph{If $\omega_0>\epsilon^{-2}$ and $v_R$ is in a cluster $\mathcal C$}. Let $\tilde{\mathcal C}=\mathcal C\backslash \{v_R\}$. Then either $\tilde {\mathcal C}=\{v_{\text{top}}^l\}$, $\tilde{ \mathcal C}=\{v_{\text{top}}^r\}$, or $\tilde{\mathcal C}=\{v_{\text{top}}^l,v_{\text{top}}^r\}$ and we treat all cases simultaneously.

Let $v\in \tilde{\mathcal C}$ be the first vertex in $\tilde{\mathcal C}$ for the integration order, and denote by $\xi^f$ the free variable attached to $v$. When reaching $v$, we perform the following actions.

First, if $v\in\mathcal V^j$ is a junction vertex, then we integrate over $\alpha_{\mathcal p_j(v)}$ using \eqref{bd:weightedintegral3} and obtain:
$$
\int_{|\alpha_{\mathcal p_j(v)}|\leq \epsilon^{-K'}} \frac{1}{|\alpha_{\mathcal p_j}(v)+\frac iT|}\frac{1}{|\alpha_{\mathcal p(v)}-\alpha_{\mathcal p_j(v)}-\Omega_v+\frac{i}{T}|}d \alpha_{\mathcal p_j(v)}\lesssim \frac{1}{|\alpha_{\mathcal p(v)}-\Omega_v+\frac{i}{T}|}
$$
So this produces a $|\alpha_{\mathcal p(v)}-\Omega_v+\frac{i}{T}|^{-1}$ factor. If $v\notin\mathcal V^j$, then we do nothing for this first action, and note that a $|\alpha_{\mathcal p(v)}-\Omega_v+\frac{i}{T}|^{-1}$ factor is already present in the integrand in this case.

Second, we bound $|\alpha_{\mathcal p(v)}|\lesssim \epsilon^2$ from \eqref{bd:degenimpliesnondegenrightbis}, then we integrate over $\alpha_{\mathcal p(v)}$ using \eqref{bd:weightedintegral3}, and over $\xi^f$ using the support estimate $|\xi^f|\leq K\epsilon^{-1}$, and obtain:
\begin{align*}
\iint_{|\alpha_{\mathcal p(v)}|\leq \epsilon^{-K'}, \ |\xi^f|\leq K\epsilon^{-1}} \frac{1}{|\alpha_{\mathcal p(v)}+\frac iT|} \frac{1}{|\tau+\alpha_{\mathcal p(v)}+\frac iT|}\frac{1}{|\alpha_{\mathcal p(v)}-\Omega_v+\frac iT|}d\alpha_{\mathcal p(v)} d\xi^f \\
\lesssim\int_{ |\xi^f|\leq K\epsilon^{-1}}  \frac{\epsilon^{2}}{|\tau+\Omega_v+\frac iT|}d\xi^f \ \lesssim \ \frac{\epsilon^{2-d}}{\inf_{|\tilde \tau|\leq C\epsilon^{-2}}|\tau-\tilde \tau+\frac{i}{T}|},
\end{align*}
where $C$ is a fixed constant depending only on $K$. The total factor produced at $v$ after these two actions is $\epsilon^{2-d} (\inf_{|\tilde \tau|\leq C\epsilon^{-2}}|\tau-\tilde \tau+\frac{i}{T}|)^{-1}$.

If $\tilde {\mathcal C}$ contains two vertices, then when reaching the second vertex, we perform the exact same computation as we did for $v$, producing another $\epsilon^{2-d} (\inf_{|\tilde \tau|\leq C\epsilon^{-2}}|\tau-\tilde \tau+\frac{i}{T}|)^{-1}$ factor.

We now assume that the algorithm reaches the root vertex $v_R$. In the first case, if $\tilde{\mathcal C}$ contains two vertices, then the above estimates produced a $\epsilon^{4-2d}(\inf_{|\tilde \tau|\leq C\epsilon^{-2}}|\tau-\tilde \tau+\frac{i}{T}|)^{-2}$ factor. In the second case, if $\tilde{\mathcal C}$ contains one vertex, let $v'$ be the other top vertex that does not belong to $\mathcal C$. We then estimate using \eqref{bd:weightedintegral3} that:
$$
\int_{|\alpha_{\mathcal p(v')}|\leq \epsilon^{-K'}} d\alpha_{\mathcal p(v')} \frac{1}{|\alpha_{\mathcal p(v')}+\frac iT|}\frac{1}{|\tau+\alpha_{\mathcal p(v')}+\frac iT|}\lesssim \frac{1}{|\tau+\frac iT|} \lesssim \frac{1}{\inf_{|\tilde \tau|\leq C\epsilon^{-2}}|\tau-\tilde \tau+\frac{i}{T}|}
$$
which produces an additional $(\inf_{|\tilde \tau|\leq C\epsilon^{-2}}|\tau-\tilde \tau+\frac{i}{T}|)^{-1}$ factor. Hence in both cases, this produces a $(|\tau-\tilde \tau+\frac{i}{T}|)^{-2}$ factor. The quantity $\mathcal F_{T,G,\beta}$ has been estimated at this step of the algorithm by:
\begin{align*}
\mathcal F_{T,G,\beta} &\lesssim \lambda^{2n}\epsilon^{2n+d} t^n|\log \epsilon|^{2n}  \iint_{|\tau|\leq \epsilon^{-K'}, \ |\xi^f_{n+1}|\leq K\epsilon^{-1}} d\tau \, d \xi^f_{n+1} \langle \tau \rangle^{2b} \frac{1}{\inf_{|\tilde \tau|\leq C\epsilon^{-2}}|\tau-\tilde \tau+\frac{i}{T}|^2}. \\
&\lesssim \lambda^{2n}\epsilon^{2n} t^n|\log \epsilon|^{2n}  \epsilon^{-K'(2b-1)}\ \lesssim \ \epsilon^{-\kappa} (\frac{t}{T_{kin}})^n
\end{align*}

\end{proof}

\section{Control of the linearized operator}

The aim of this section is to provide an estimate on the linearization around the approximate solution $u^{app} = \chi(t/T)\sum_{n=0}^N u^n$. Without loss of generality we present the proof for the case of equation \eqref{nonlinschrod} with $\omega_0=\epsilon^{-2}$, since the case $\omega_0=0$, $m(0)=0$ is simpler.
 The linearization operator is given by 
$$
\mathfrak{L}_N w = 4\mathfrak{Re} u^{app} \mathfrak{Re} w=2\mathfrak{Re} u^{app} (w+\bar{w}).
$$
Notice that from the diagrammatic expansion \eqref{id:uassumofinteractiongraphs} the operator $\mathfrak{L}_N$ can be decomposed as
\be \label{choucasdestours}
\mathfrak{L}_Nw=\sum_{j=0}^N \sum_{G\in \mathcal G_j} \sum_{\iota \in \{\pm 1\}^2} \mathfrak{L}_{G,\iota}w,
\ee
where for each $G\in \mathcal G_j$ and $\iota=(\sigma_1,\sigma_2)\in \{ \pm 1 \}^2$:
$$
 \mathfrak{L}_{G,\iota}w =\chi \left(\frac tT\right) u_{G,\sigma_1} w_{\sigma_2}
$$
where $u_{G,\sigma_1}=u_G$ if $\sigma_1=+1$ and $u_{G,\sigma_1}=\overline{u_G}$ if $\sigma_1=-1$ and similarly for $w_{\sigma_2}$. Moreover, each $\mathfrak{L}_{G,\iota}$ can be localized in frequency annuli of radii $\sim 2^{l}\epsilon^{-A}$ by studying $\mathfrak{L}_{G,\iota}\mathcal A_R^l$, where $l\in \{0,1,2,...\}$ and $R=\epsilon^{-A}$ for $A$ to be fixed large later on. 
We state the main result of this section:
\begin{proposition}\label{linearization proposition}
If $N\in\mathbb{N}$, $\mu>0$, $s>0$, there exists $b>1/2$ and a set $E_{N,\mu,s}$ of probability $\mathbb{P}(E_{N,\mu,s})>1-\epsilon^{\mu}$ on which the operator norm of $\mathfrak{L}_N$ can be bounded as follows:
\begin{equation*}
\left\|\chi(t)\int_0^te^{i(t-s)\frac{\Delta_\omega}{2}}\mathfrak{L}_N\,ds\right\|_{X^{s,b}_\epsilon\to X^{s,b}_\epsilon}\lesssim_{N,\mu}\epsilon^{-\mu}\sqrt{\frac{T}{T_{kin}}}.
\end{equation*}
\end{proposition}
The following lemma is the main step in the proof of Proposition~\ref{linearization proposition}. Notice that the estimate for $l\geq 1$ is much better than for $l=0$. This means that interactions between high frequencies $\sim 2^l \epsilon^{-A}$ for the remainder and low frequencies $\sim \epsilon^{-2}$ of the approximate solution are much weaker than low-low interactions. This is intimately related to the fact that we consider the equation on the whole space, as such interactions would be more delicate to estimate on the torus.

\begin{lemma}\label{linearization lemma} For $\kappa>0$ and $\epsilon>0$ small enough, there exists for all $l=0,1,...$ a set $E_{\kappa,l}$ of measure greater than $1-2^{-l}\epsilon^\kappa$ such that in this set, for any $j\in\{0,...,N\}$, $G\in \mathcal G_j$ and $\iota \in \{\pm 1\}^2$, the following operator norm estimate holds
$$
\| \mathfrak{L}_{G,\iota}\mathcal A_R^l \|_{X^{0,\frac{1}{2}}\to X^{0,-\frac{1}{2}}} \lesssim \left\{ \begin{array}{l l l}\left(\frac{T}{T_{kin}}\right)^{\frac{j+1}{2}}\epsilon^{-\kappa} &\qquad \mbox{if }l=0,\\
2^{-\frac l8}\epsilon^{\frac{A}{8}} \left(\frac{T}{T_{kin}}\right)^{\frac{j+1}{2}}\epsilon^{-\kappa} &\qquad \mbox{if }l\geq 1.
\end{array}\right.
$$

\end{lemma}
With Lemma \ref{linearization lemma} in hand, we are able to prove Proposition \ref{linearization proposition}.

\bigskip \noindent
\textit{Proof of Proposition \ref{linearization proposition} using Lemma \ref{linearization lemma}}

Using the estimate \eqref{mesangebleue0} and the identity \eqref{choucasdestours} yields:
$$
\| \chi \int_{0}^t e^{i(t-s)\frac{\Delta \omega}{2}} \mathfrak L_N ds \|_{X_\epsilon^{s,b}\to X_\epsilon^{s,b}}\lesssim \| \mathfrak L_N \|_{X_\epsilon^{s,b}\to X_\epsilon^{s,b-1}}\lesssim \sum_{j,G,\iota} \| \mathfrak L_{G,\iota} \|_{X_\epsilon^{s,b}\to X_\epsilon^{s,b-1}}
$$
so that it suffices to prove the following estimate:
\be \label{cacatoes}
 \| \mathfrak L_{G,\iota}  \|_{X_\epsilon^{s,b}\to X_\epsilon^{s,b-1}} \lesssim \epsilon^{-\mu} \sqrt{\frac{T}{T_{kin}}}
\ee

\noindent \underline{Almost locality:} We decompose the input in frequency cubes as:
$$
\mathfrak{L}_{G,\iota}w=\sum_{n,n'\in \mathbb Z^d} Q^{n'}_\epsilon   \mathfrak{L}_{G,\iota} Q^n_\epsilon  w
$$
Since $\mathfrak{L}_{G,\iota}$ corresponds to convolution in frequency with kernel localized in a ball of size $C\epsilon^{-1}$, if $|n-n'|>R$ for some $R>0$, we have that $ Q^{n'}_\epsilon   \mathfrak{L}_{G,\iota} Q^n_\epsilon  w=0$. This in turn implies that
\begin{equation}\label{compare n-1 norm}
\|\mathfrak{L}_{G,\iota}\|_{X^{s,\frac{1}{2}}_\epsilon \to X^{s,-\frac{1}{2}}_\epsilon}\sim\|\mathfrak{L}_{G,\iota}\|_{X^{0,\frac{1}{2}} \to X^{0,-\frac{1}{2}}}.
\end{equation}
Indeed, this follows since the main part of $\mathfrak{L}_{G,\iota}$ is a convolution in space frequency and the weights of $X^{0,\frac{1}{2}}$ and $X^{0,-\frac{1}{2}}$ cancel for $|n-n'|\leq R$ as by duality

\begin{align*}
\|\mathfrak{L}_{G,\iota}\|_{X^{s,\frac{1}{2}}_\epsilon \to X^{s,-\frac{1}{2}}_\epsilon} & = \| \langle \epsilon \xi\rangle^s \mathfrak{L}_{G,\iota} \langle \epsilon \xi\rangle^{-s} \|_{X^{0,\frac{1}{2}}_\epsilon \to X^{0,-\frac{1}{2}}_\epsilon} \\
& =  \sup_{\| u\|_{X^{0,\frac 12}}=\| v\|_{X^{0,-\frac 12}}=1} \ \sum_{|n-n'|\leq R}  \langle  \mathfrak{L}_{G,\iota} Q^n_\epsilon \langle \epsilon \xi\rangle^{-s} u,  Q^{n'}_\epsilon \langle  \epsilon \xi\rangle^s v \rangle \\
& =  \sup_{\| u\|_{X^{0,\frac 12}}=\| v\|_{X^{0,-\frac 12}}=1} \ \ \sum_{|n-n'|\leq R} \frac{\langle n'\rangle^s}{\langle n\rangle^s} \langle  \mathfrak{L}_{G,\iota} Q^n_\epsilon \langle n\rangle^s \langle \epsilon \xi\rangle^{-s} u,  Q^{n'}_\epsilon \langle n'\rangle^{-s} \langle  \epsilon \xi\rangle^s v \rangle \\
&\lesssim  \|  \mathfrak{L}_{G,\iota}\|_{X^{0,\frac{1}{2}}_\epsilon \to X^{0,-\frac{1}{2}}_\epsilon}
\end{align*}
using the Cauchy-Schwarz inequality and that $ \langle n\rangle^s \langle \epsilon \xi\rangle^{-s}$ is bounded on the dyadic cube $C^n_\epsilon$ on which $Q^n_\epsilon $ projects for the last inequality.

\noindent \underline{Bound from $X^{0,\frac{1}{2}}$ to $X^{0,-\frac{1}{2}}$:} Let $E=\cap_l E_{\kappa,l}$ where $E_{\kappa,l}$ is given by Lemma \ref{linearization lemma}. Then $E$ has measure greater than $1-\epsilon^\kappa$ (up to taking a smaller $\kappa$ in the lemma). On $E$, by almost orthogonality we have
 
 \begin{align}
 \|\mathfrak{L}_{G,\iota} w\|_{X^{0,-\frac{1}{2}}}&\lesssim \left[\sum_{l\geq 0}\|\mathfrak{L}_{G,\iota} \mathcal A^l_R w\|^2_{X^{0,-\frac{1}{2}}}\right]^{1/2}\nonumber\\
 &\leq \left[\sum_{l\geq 0} \|\mathfrak{L}_{G,\iota} \mathcal A^l_R \|^2_{X_\epsilon^{0,\frac{1}{2}}\to X_\epsilon^{0,-\frac{1}{2}}} \| \mathcal A^l_R w\|_{X^{0,-\frac{1}{2}}}^2\right]^{1/2}\nonumber\\
  &\leq \left[\sum_{l\geq 0} 2^{-\frac l4} \left(\frac{T}{T_{kin}}\right)^{i+1}\epsilon^{-2\kappa}  \| \mathcal A^l_R w\|_{X^{0,-\frac{1}{2}}}^2\right]^{1/2}\nonumber\\
 &\lesssim \left(\frac{T}{T_{kin}}\right)^{\frac{i+1}{2}}\epsilon^{-\kappa}\|w\|_{X^{0,\frac{1}{2}}}.\label{localization1}
 \end{align}

\noindent
 \underline{Bound from $X^{0,0}$ to $X^{0,0}$ and interpolation.} Since $X^{0,0}$ is merely $L^2_tL^2_x$ and $u_G$ is localized in a ball of radius $C\epsilon^{-1}$, the norm of the operator $\mathfrak{L}_{G,\iota}$ from $X^{0,0}$ to $X^{0,0}$ is bounded by $\|u_G^{\iota_2}{\bf 1}(\iota_3 t\geq 0)\|_{L^\infty}\lesssim \epsilon^{-d/2}\|u_G\|_{X_\epsilon^{s,b}}$. By \eqref{bd:estimationunXsb} (which was actually showed for $u_G$ for all $G\in \mathcal G_j$) and Bienaymm\'e-Tchebychev inequality, for $k$ small enough, we can find a set $E'$ with $\mathbb{P}(E')\geq 1-\epsilon^k$ on which $\|u_G\|_{X_\epsilon^{s,b}}\lesssim 1$. Hence, the operator norm from $X^{s,0}$ to $X^{s,0}$ can be bounded by $\epsilon^{-d/2}$.
 
 Interpolating between this bound and the $X^{s,\frac{1}{2}}$ to $X^{s,-\frac{1}{2}}$ bound \eqref{localization1}, we obtain a bound from $X^{s,\frac{1}{2}-\delta}$ to $X^{s,\frac{1}{2}+\delta}$ with a loss $\epsilon^{-k}$, where $k$ can be made arbitrarily small choosing $\delta$ sufficiently small. Finally, we choose $b>\frac{1}{2}$ such that $b-1<-\frac{1}{2}-\delta$ to obtain \eqref{cacatoes} as desired.

  \subsection{Estimate on the trace}  It remains to prove Lemma \ref{linearization lemma}. Pick a graph $G\in \mathcal G_j$ an integer $l$. We prove it for simplicity in the case $M=1$ and $\omega_0=\epsilon^{-2}$. We only need to prove the bound for $\mathcal L=\mathfrak L_{G,(+1,+1)}{\bf 1}(t\geq 0)A^l_R$ for $j\geq 1$ as the proof for the other operators is similar. Using space-time Fourier transformation, and including the $X^{s,b}_\epsilon$ weights in the operator, it suffices to estimate the continuity norm on $L^2_{\tau,\xi}$ of the convolution operator
\begin{align*}
\mathfrak{R}:L^2(\mathbb{R}\times\mathbb{R}^d)&\to L^2(\mathbb{R}\times\mathbb{R}^d)\\
w(\tau_0,\xi_0)&\to \int\int K(\tau_2,\tau_0,\xi_2,\xi_0) w(\tau_0,\xi_0)\,d\tau_0\,d\xi_0,
\end{align*}
with kernel
$$
K(\tau_2,\tau_0,\xi_2,\xi_0)=\lambda\langle \tau_0+\omega(\xi_0)\rangle^{-\frac{1}{2}}\langle \tau_2+\omega(\xi_2)\rangle^{-\frac{1}{2}}\int_{\xi_1\in \mathbb R^d}\mathds{1}_{A_{R}^l(\xi_0)}\widetilde{\chi(\frac tT)u_{G}^+}(\tau_2-\tau_0,\xi_1) \delta(\xi_2-\xi_0-\xi_1)\,d\xi_1 
$$
Changing variables $(\xi_0,\xi_1,\xi_2)\to(\xi_2,-\xi_1,\xi_0)$ and $(\tau_0,\tau_1,\tau_2)\to (\tau_2,-\tau_1,\tau_0)$, we compute the adjoint kernel
$$
K^*(\tau_2,\tau_0,\xi_2,\xi_0)=\lambda\langle \tau_0+\omega (\xi_0) \rangle^{-\frac{1}{2}}\langle \tau_2+\omega(\xi_2)\rangle^{-\frac{1}{2}} \int_{\xi_1\in \mathbb R^d}\mathds{1}_{A_{R}^l}(\xi_2)\widetilde{\chi(\frac tT)\overline{u_{G}^+}}(\tau_2-\tau_0,\xi_1) \delta(\xi_2-\xi_0-\xi_1)\,d\xi_1 
$$
Iterating, we obtain that the the operator $\mathfrak{M}^N=(\mathfrak{R}^*\mathfrak{R})^N$ has kernel
\begin{align*}
M^N(\tau_{4N},\tau_0,\xi_{4N},\xi_0) &=\lambda^{2N}\langle \tau_{4N}+\omega (\xi_{4N})\rangle^{-1/2}\langle \tau_{0}+\omega (\xi_{0})\rangle^{-1/2} \iint  \,d\tau_2...\,d\tau_{4N-2}\,d\xi_1...\,d\xi_{4N-1} \\
& \qquad \prod_{m=0}^{2N-1}\delta(\xi_{2m+2}-\xi_{2m+1}-\xi_{2m})\prod_{m=0}^N\mathds{1}_{A_R^l}(\xi_{4m}) \prod_{m=1}^{2N-1}\langle \tau_{2m}+\omega (\xi_{2m})\rangle^{-1} \\
&\qquad \prod_{m=1}^{N} \widetilde{\chi(\frac tT)u_{G}^+}(\tau_{4m-2}-\tau_{4m-4},\xi_{4m-3})\widetilde{\chi(\frac tT)\overline{u_{G}^+}}(\tau_{4m}-\tau_{4m-2},\xi_{4m-1})  .
\end{align*}
The trace of the operator $\mathfrak{M}^N$ is therefore:
\begin{align*}
\Tr \mathfrak{M}^N &=\lambda^{2N} \iint  \,d\tau_0...\,d\tau_{4N}\,d\xi_1...\,d\xi_{4N-1} \delta(\tau_0-\tau_{4N}) \Delta(\underline{\xi}) \prod_{m=0}^{N-1}\mathds{1}_{A_R^l}(\xi_{4m}) \prod_{m=1}^{2N}\langle \tau_{2m}+\omega (\xi_{2m})\rangle^{-1} \\
&\qquad \qquad \qquad \qquad \prod_{m=1}^{N} \widetilde{\chi(\frac tT)u_{G}^+}(\tau_{4m-2}-\tau_{4m-4},\xi_{4m-3})\widetilde{\chi(\frac tT)\overline{u_{G}^+}}(\tau_{4m}-\tau_{4m-2},\xi_{4m-1})  .
\end{align*}
where $\Delta(\underline{\xi})=\delta(\xi_{4N}-\xi_0)\delta (\tau_{4N}-\tau_0)\prod_{m=0}^{2N-1}\delta(\xi_{2m+2}-\xi_{2m+1}-\xi_{2m})$. Above, $\widetilde{\chi(\frac{\cdot}{T})u_{G}^+}(t,\xi_{4m-3})$ (resp. $\widetilde{\chi(\frac{\cdot}{T})\overline{u_{G}^+}}(t,\xi_{4m-1})$) is given by the identity \eqref{id:formulaungraphs3} with graph $G$ (resp. \eqref{id:formulaungraphs3} with graph $G$ where all parity signs are reversed, and where the factor $e^{-it\omega (\xi_{4m-3})}$ is replaced by $(-1)^je^{it\omega (\xi_{4m-1})}$). Applying time Fourier transformation, changing variables by renaming $\tau_{2m}+\omega(\xi_{2m})$ as $\tau_{2m}$, taking the expectancy following the framework of Section \ref{sec:diagrams}, we arrive at the diagrammatic formula:
\be \label{crepuscule}
 \mathbb{E}\left[\Tr\mathfrak{M}^N\right]=\sum_P \mathcal F_T(G,N,P)
\ee
where (integrating the $2\pi$ and $c_G^{t/T}$ factors in the cut-off $\chi(t/T)$ to reduce notations)
\begin{align}
\label{id:formulaexpectancytrace} \mathcal F_T(P,G,N)=& \lambda^{2jN}\epsilon^{Nd(j+1)}  \iiiint d\underline \xi d\underline \eta d\underline \tau d\underline \alpha \delta(\tau_0-\tau_{4N}) \Delta(\underline \xi,\underline \eta)\\
\nonumber & \prod_{\mathcal p\in \mathcal P_m}\frac{i}{\alpha_{\mathcal p}+\frac{ic_{\mathcal p}}{T}}\prod_{(i,j)\in P} \hat W_0^\epsilon (\eta_{i,j},\frac \epsilon 2 (\sigma_{0,i}\xi_{0,i}+\sigma_{0,j}\xi_{0,j}))\\
\nonumber & \prod_{v\in \mathcal V_i \backslash \{v^b_1,...,v^b_{2N}\}} \frac{i}{\alpha_{\mathcal p(v)}-\sum_{\tilde{\mathcal p}\triangleleft v}\alpha_{\tilde{\mathcal p}}-\sum_{\tilde v\in \mathcal p^+(v)}\Omega_{\tilde v}+\frac{ic_v}{T}}  \prod_{m=0}^{N-1} \mathds{1}_{A_R^l}(\xi_{4m}) \\
\nonumber &\prod_{m=1}^{2N} \langle \tau_{2m} \rangle^{-1} T\hat \chi( T(\tau_{2m}-\tau_{2m-2}-\alpha_{\mathcal p(v_{\text{top}}^m)}-\Omega_{m}))
\end{align}
where for $m=1,...,N$,
\be \label{id:linearformulaOmega}
\Omega_{2m-1}=\omega (\xi_{4m-2})-\omega(\xi_{4m-4})-\omega(\xi_{4m-3}), \qquad \Omega_{2m}=\omega (\xi_{4m})-\omega(\xi_{4m-2})+\omega(\xi_{4m-1}).
\ee
The above formula is associated to a graph that we now describe.

It is made by a branch of vertices $v^b_{0},...,v^b_{2N}$ that are linked by edges $(v^b_m,v^b_{m+1})$ for $m=0,...,2N-1$. For $m=1,...,N$, below the vertex $v^b_{2m-1}$ (resp. below $v^b_{2m}$) is placed a copy of the tree $G$ with top vertex $v_{\text{top}}^{2m-1}$ (resp. $G$ with reversed parity signs and top vertex $v^{2m}_{\text{top}}$), linked to $v^{b}_{2m-1}$ by an edge $(v_{\text{top}}^{2m-1},v^{b}_{2m-1})$ (resp. to $v^{b}_{2m}$ by an edge $(v_{\text{top}}^{2m},v^{b}_{2m})$). We denote by $\mathcal V^b=\{v^b_1,...,v^b_N\}$ the collection of vertices above the trees. The collection of all vertices of the trees and of the vertices $\{v^b_1,...,v^b_N\}$ is the set of all interaction vertices $\mathcal V_i$ of the graph.

There is a root vertex $v_R$, and two edges $(v^b_{0},v_R)$ and $(v^b_{4N},v_R)$.

To the edge $(v^b_{m},v^b_{m+1})$ we associate the frequency $\xi_{2m}$, and to the edge $(v^b_{4N},v_R)$ the frequency $\xi_{4N}$. To the edge $(v^m_{\text{top}},v^b_m)$ we associate the frequency $\xi_{2m-1}$. We impose Kirchhoff laws at each vertex of the graph, except at $v^b_0$ where we impose $\xi^b_0+\xi_{(v^b_0,v_R)}=0$, so that the law at $v_R$ then reads $\xi_{4N}=\xi_0$. The edges $(v^b_{m},v^b_{m+1})$ for $m=0,...,2N-1$, and $(v^b_{2N},v_R)$, have all parity $+1$. In particular, \eqref{id:linearformulaOmega} agrees with \eqref{id:formulaOmegav}.

The collection of all maximal upright paths of each of the trees $G$, or $G$ with reversed parities, is the set of maximal paths denoted as $\mathcal P_m$. The collection of all their initial vertices, is the set of initial vertices denoted as $\mathcal V_0$. $P$ is a pairing for the set of initial vertices, and pairing vertices are defined as in Subsection \ref{pairinggraphs}. The resulting graph is as follows.

\begin{center}
\includegraphics[width=12cm]{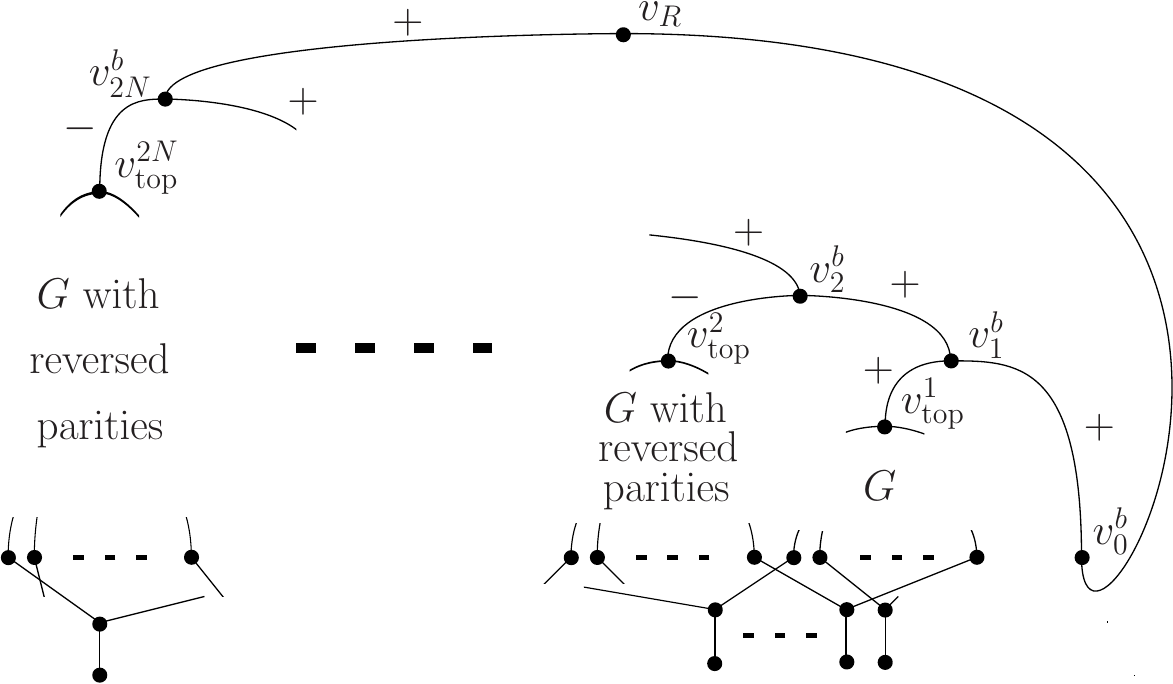}
\end{center}

To estimate \eqref{id:formulaexpectancytrace}, we use the following estimate, obtained by bounding $T|\hat \chi(Tz)|\lesssim |z+iT^{-1}|^{-1}$ as $\chi$ is in the Schwartz class and integrating over $d\tau_0d\tau_2 ... d\tau_{4N}$ iteratively using the second inequality in \eqref{bd:weightedintegral5}:
\begin{align*}
&   \int  d\underline \tau  \delta(\tau_0-\tau_{4N}) \prod_{m=1}^{2N} \langle \tau_{2m} \rangle^{-1} T\hat \chi( T(\tau_{2m}-\tau_{2m-2}-\alpha_{\mathcal p(v_{\text{top}}^m)}-\Omega_{m})) \\
  & \lesssim   \int  d\underline \tau  \delta(\tau_0-\tau_{4N}) \prod_{m=1}^{2N} \frac{1}{|\tau_{2m}+i|} \frac{1}{|\tau_{2m}-\tau_{2m-2}-\alpha_{\mathcal p(v_{\text{top}}^m)}-\Omega_{m}+\frac{i}{T}|}
\end{align*}
so that we estimate \eqref{id:formulaexpectancytrace} by:
\begin{align}
\label{id:formulaexpectancytrace2} |\mathcal F_T(P,G,N)|\lesssim & \lambda^{2jN}  \epsilon^{Nd(j+1)}  \iiiint d\underline{\tau}d\underline \xi d\underline \eta d\underline \alpha \delta(\tau_0-\tau_{4N}) \Delta(\underline \xi,\underline \eta)  \prod_{m=0}^{N-1}\mathds{1}_{A_R^l}(\xi_{4m})\\
\nonumber & \prod_{\mathcal p\in \mathcal P_m}\frac{1}{|\alpha_{\mathcal p}+\frac{ic_{\mathcal p}}{T}|}\prod_{(i,i')\in P} |\hat W_0^\epsilon (\eta_{i,i'},\frac \epsilon 2 (\sigma_{0,i}\xi_{0,i}+\sigma_{0,i'}\xi_{0,i'}))|   \\
\nonumber & \prod_{v\in \mathcal V_i \backslash \{v^b_1,...,v^b_{2N}\}} \frac{1}{|\alpha_{\mathcal p(v)}-\sum_{\tilde{\mathcal p}\triangleleft v}\alpha_{\tilde{\mathcal p}}-\sum_{\tilde v\in \mathcal p^+(v)}\Omega_{\tilde v}+\frac{ic_v}{T}|} \\
\nonumber& \qquad \prod_{m=1}^{2N} \frac{1}{|\tau_{2m}+i|} \frac{1}{|\tau_{2m}-\tau_{2m-2}-\alpha_{\mathcal p(v_{\text{top}}^m)}-\Omega_{m}+\frac{i}{T}|}
\end{align}

We define the set of junction vertices $\mathcal V^j$ to be the union of the collection of junction vertices in the graphs $G$, and $G$ with reversed signs, and of $\{v^b_1,...,v^b_{2N}\}$ ;  we say that for $m=1,...,2N$, $v^b_m \triangleright \mathcal p(v_{\text{top}}^m)$ constraints the maximal upright path leading to $v_{\text{top}}^m$.

An integration order is chosen for the interaction vertices of the graph, defined similarly as in Subsection \ref{pairinggraphs}. We choose the integration order so that vertices in the trees $G$ or $G$ with reversed parities are considered first, after what the vertices $v^b_1,v^b_2,...,v^b_{2N}$ are considered. Given this integration order, we can apply the same proof as that of Proposition \ref{pr:spanning} in order to find the free interaction frequencys. There are $N(j+1)+1$ interaction free frequencies in total.

 Notice that $\xi_0=\xi_{2N}$ is a free interaction frequency, and that $\xi_0\in A^l_R$ is in the support of the integrand of \ref{id:formulaexpectancytrace2}. Notice that there are $N(j+1)$ remaining interaction free frequencies, which thanks to Kirchhoff's law are all linear combination of the frequencies of the initial vertices in the trees $G$ or $G$ with reversed parities. Hence, on the support of the integrand of \ref{id:formulaexpectancytrace2} they are bounded by $K\epsilon^{-1}$ where $K$ depends only on $j$ and on the support of $\hat W_0$. Hence $\underline{\xi}^f\in A^l_R\times B^{dNj}(K\epsilon^{-1})$.

We say that for $m=1,...,2N$, the vertex $v^b_m$ is linear if the edge $(v_{\text{top}}^m,v^b_m)$ is a free edge, and if $\Omega_m$ is given by \eqref{id:formularesonancelinear}. Note that this coincides with the definition of linear degree one vertices of Subsection \ref{subsec:kirchhoff} for the vertices in each of the subtrees $G$, or $G$ with reversed parities.

We adapt the definition of degenerate degree one linear vertices, given in Definition \ref{def:degenerate}, to vertices in $\mathcal V^b$ as follows. Given a degree one linear vertex $v=v_m^b \in \mathcal V^b$ for some $m=1,...,2N$ we define five sets which will distinguish whether $v$ is degenerate or not, analogously to the sets $S_v^i$ for $i=1,2,3,4$ defined at the beginning of Subsection \ref{subsecL2}. We let $\tilde S=B^{2N+1+n_m}(\epsilon^{-K'})\times B^{d(n+1)}_0(K) \times A^l_R \times B^{dNj}(K\epsilon^{-1})$ and set:
\begin{align*}
&S_v^{b,1} = \{(\underline{\tau},\underline \alpha,\underline \eta,\underline \xi^f)\in \tilde S, \  |\tau_{2m}-\tau_{2m-2}-\alpha_{\mathcal p(v_{\text{top}}^m)}-\Omega_{m}|> \delta \epsilon^{-2} \}\\
&S_v^{b,2} =  \{(\underline{\tau},\underline \alpha,\underline \eta,\underline \xi^f)\in \tilde S, \  |\xi_{(v^b_m,v^b_{m+1})}|> \delta \epsilon^{-1} \}\\
&S_v^{b,3} =   \{(\underline{\tau},\underline \alpha,\underline \eta,\underline \xi^f)\in \tilde S, \   |\tau_{2m-2}|> \delta \epsilon^{-2} \},\\
&S_v^{b,4} =   \{(\underline{\tau},\underline \alpha,\underline \eta,\underline \xi^f)\in \tilde S, \   |\alpha_{\mathcal p(v_{\text{top}}^m)}|> \delta \epsilon^{-2} \},\\
&S_v^{b,5} =   \tilde S \  \backslash \ (\cup_{i=1,2,3,4} S_v^i)
\end{align*}

\begin{definition} \label{def:degenerate-b}

Let $\delta,K,K'>0$. Given a set $S\subset  B^{2N+1+n_m}(\epsilon^{-K'})\times B^{d(n+1)}_0(K) \times A^l_R \times B^{dNj}(K\epsilon^{-1})$, we say that for $m=1,...,2N$ a degree one linear vertex $v^b_m\in \mathcal V^b$ is \emph{degenerate} on $S$ if for all $(\underline{\tau},\underline \alpha,\underline \eta,\underline \xi^f)\in S$ the following three conditions are met simultaneously:
\begin{align}
\label{linear:bd:degenerate-b-1}& |\tau_{2m}-\tau_{2m-2}-\alpha_{\mathcal p(v_{\text{top}}^m)}-\Omega_{m}|\leq \delta \epsilon^{-2},\\
\label{linear:bd:degenerate-b-2}& |\xi_{(v^b_m,v^b_{m+1})}|\leq \delta \epsilon^{-1},\\
\label{linear:bd:degenerate-b-3}& |\tau_{2m-2}|\leq \delta \epsilon^{-2},\\
\label{linear:bd:degenerate-b-4}&  |\alpha_{\mathcal p(v_{\text{top}}^m)}|\leq \delta \epsilon^{-2}.
\end{align}
Equivalently, $v^b_m$ is degenerate on $S$ if $S\subset S_v^{b,5} $.

We say that for $m=1,...,2N$ a vertex $v_m^b\in \mathcal V^b$ is \emph{nondegenerate} on $S$ if either $v$ is a degree zero or a degree one quadratic vertex, or if $v$ is a degree one linear vertex such that for each $(\underline{\tau},\underline \alpha,\underline \eta,\underline \xi^f)\in S$ at least one of the four conditions above fail.
\end{definition}

We partition the domain of integration in \eqref{id:formulaexpectancytrace2} according to the non-degeneracy/degeneracy of each vertex. For this aim, given a function $\beta $ defined on $\mathcal V^1_l$ with $\beta (v)\in \{1,...,4\}$ for $v\in \mathcal V^i\backslash \mathcal V^b$ and $\beta (v)\in \{1,...,5\}$ for $v\in \mathcal V^b$, we define:
\be \label{diagrams:id:definition-S-beta-b}
S_\beta =\cap_{v\in \mathcal V_l^1} S^{\beta(v)}_v.
\ee

The result of Lemma \ref{lem:degenerate-implies-non-degenerate-above} naturally adapts for degenerate vertices in $\mathcal V^b$.

\begin{lemma} \label{lem:degenerate-implies-non-degenerate-above-b}
For all $K>0$, for $\delta(K)>0$ small enough the following holds true for any set $S\subset\mathbb R^{2N+1} \times \mathbb R^{n_m}\times \mathbb R^{d(n+1)}_0(K) \times A^l_R \times B^{d Nj}(K\epsilon^{-1})$. For $l=0$, for any $m=1,...,2N-1$, if $v^b_m$ is degenerate, then $v^b_{m+1}$ is non-degenerate and there holds:
\be \label{bd:degenimpliesnondegenleft-b}
|\tau_{2m}|\geq \frac{\epsilon^{-2}}{2}.
\ee
For any $m=1,...,2N$, if $v_{\text{top}}^m$ is degenerate, then $v^b_{m}$ is non-degenerate and there holds:
\be \label{bd:degenimpliesnondegenleft-b}
 |\alpha_{\mathcal p(v_{\text{top}}^m)}|\geq \frac{\epsilon^{-2}}{2}.
\ee
For $l\geq 1$, if $A>2$ then for all $\epsilon$ small enough, for all $m=1,...,2N$ we have that $v_{\text{top}}^m$ is always non-degenerate and
\be \label{bd:degenimpliesnondegenleft-lgeq1}
|\xi_{(v^b_m,v^b_{m+1})}|\approx 2^l\epsilon^{-A}.
\ee
\end{lemma}

\begin{proof}

Assume first that $v^b_m$ is degenerate for some $m=1,...,2N-1$. Thanks to \eqref{linear:bd:degenerate-b-2} the inequality \eqref{id:intermediatenondegeneratesets1} is valid, so that $\Omega_m\geq \frac 34 \epsilon^{-2}$. Using this, \eqref{linear:bd:degenerate-b-1}, \eqref{linear:bd:degenerate-b-3} and \eqref{linear:bd:degenerate-b-4} shows
$$
|\tau_{2m}|\geq |\Omega_m|- |\tau_{2m}-\tau_{2m-2}-\alpha_{\mathcal p(v_{\text{top}}^m)}-\Omega_{m}|- |\tau_{2m-2}| -  |\alpha_{\mathcal p(v_{\text{top}}^m)}|\geq \frac 34 \epsilon^{-2}-3\delta \epsilon^{-2}
$$
which shows \eqref{bd:degenimpliesnondegenleft-b} for $\delta$ small enough.

Second, assume that $v_{\text{top}}^m$ is degenerate for some $m=1,...,2N$. Then $\Omega_{v_{\text{top}}^m}\geq \frac 34 \epsilon^{-2}$ by \eqref{id:intermediatenondegeneratesets1}. In the case where $v_{\text{top}}^m$ is not a junction vertex, then by Definition \ref{def:degenerate} we have
$$
 |\alpha_{\mathcal p(v_{\text{top}}^m)}|\geq |\Omega_{v_{\text{top}}^m} |- |\alpha_{\mathcal p(v_{\text{top}}^m)}-\Omega_{v_{\text{top}}}| \geq  \frac 34 \epsilon^{-2}-\delta \epsilon^{-2}.
$$
In the case where $v_{\text{top}}^m$ is a junction vertex, then by Definition \ref{def:degenerate} we have
$$
 |\alpha_{\mathcal p(v_{\text{top}}^m)}|\geq |\Omega_{v_{\text{top}}^m} |- |\alpha_{\mathcal p(v_{\text{top}}^m)}-\alpha_{\mathcal p_j(v)}-\Omega_{v_{\text{top}}|}-|\alpha_{\mathcal p_j(v)}| \geq  \frac 34 \epsilon^{-2}-2\delta \epsilon^{-2}.
$$
In both cases, we obtain \eqref{bd:degenimpliesnondegenleft-b} for $\delta$ small enough.

Next, consider the case $l\geq 1$. Notice that from the Kirchhoff laws in the graph all frequencies in the trees $G$ or $G$ with reversed parities are bounded by $\epsilon^{-2}$. Hence $|\xi_{(v^m_{\textup{top}},v^b_m)}|\lesssim \epsilon^{-2}$ for all $m=1,...,2N$. by the Kirchhoff law, $\xi_{(v^b_m,v^b_{m+1})}=\xi_0^f+\sum_{\tilde m=1}^{m} \xi_{(v^b_{\tilde m},v^b_{\tilde m+1})}$. Since $|\xi_0|\approx 2^l \epsilon^{-A}$ as $\xi_0\in A^l_R$, we deduce \eqref{bd:degenimpliesnondegenleft-lgeq1}.

\end{proof}

We adapt accordingly the definition of degenerate clusters from Definition \ref{def:cluster}.

\begin{definition} \label{def:cluster-b}
Given a set $S\subset\mathbb R^{2N+1} \times \mathbb R^{n_m}\times \mathbb R^{d(n+1)}_0(K) \times \mathbb R\times B^{d Nj}(K\epsilon^{-1})$, we say that $\mathcal C \subset \mathcal V^i$ is a \emph{degenerate b-cluster} on $S$ if either of the three following possibilities occur:
\begin{itemize}
\item Type b-I: $\mathcal C=\{v_{\text{top}}^m,v^b_m \}$ for $m=1,...,2N$, and is such that $v_{\text{top}}^m$ is degenerate on $S$, and $v^b_m$ is nondegenerate on $S$.
\item Type b-II: $\mathcal C=\{v^b_{m-1},v^b_{m} \}$ for some $m=2,...,2N$, with $v^b_{m-1}$ degenerate on $S$, and $v^b_m$ is nondegenerate on $S$.
\item Type b-III: $\mathcal C=\{v_{\text{top}}^m,v^b_{m-1},v^b_m \}$ for some $m=2,...,N$, with $v_{\text{top}}^m$ and $v^b_{m-1}$ degenerate on $S$ and $v^b_m$ non-degenerate on $S$.
\end{itemize}
\end{definition}

We adapt naturally the definition \eqref{diagrams:id:definition-S-beta} of the partition sets $S_\beta$ to take into account the inequalities for degenerate vertices in $\mathcal V^b$ in Definition \ref{def:degenerate-b}. The result of Lemma \ref{lem:decomposition} then naturally extends to include degenerate b-clusters.

\begin{lemma}[Decomposition into nondegenerate vertices and degenerate clusters] \label{lem:decomposition-b}
For any set of the form $S_\beta$ given by \eqref{diagrams:id:definition-S-beta-b}, there exists $\mathcal C_1,...,\mathcal C_{n_d(G,\beta)}$ disjoints degenerate clusters or b-clusters in the sense of Definitions \ref{def:cluster} and \ref{def:cluster-b} on $S_\beta$ such that $ \mathcal V^i = \tilde{\mathcal V} \sqcup \mathcal C_1\sqcup...\sqcup\mathcal C_{n_d(G,\beta)}$ where $ \tilde{\mathcal V} $ only contains non-degenerate vertices on $S_\beta$.
\end{lemma}

\begin{proof}
Using the result of Lemmas \ref{lem:degenerate-implies-non-degenerate-above} and \ref{lem:degenerate-implies-non-degenerate-above-b}, the proof is exactly as that of Lemma \ref{lem:decomposition-b}.
\end{proof}

We now estimate in two slightly different ways \eqref{id:formulaexpectancytrace2} in the cases $l=0$ and $l\geq 1$, and obtain in the latter case a much better estimate.
\subsubsection{The case $l=0$}

In this case, $\mathds{1}_{A_R^l}(\xi_{4m})=\mathds{1}(|\xi_{4m}|\leq \epsilon^{-A})=\mathds{1}(|\xi_{0}|\leq \epsilon^{-A})$ in the integrand of \eqref{id:formulaexpectancytrace2}. Notice that \eqref{id:formulaexpectancytrace2} is very similar to \eqref{id:mathcalFGP3} that was estimated very precisely in the proof of \eqref{bd:estimationunL2}, except only for the last factor $\prod_{m=1}^{2N} |\tau_{2m}+i|^{-1} |\tau_{2m}-\tau_{2m-2}-\alpha_{\mathcal p(v_{\text{top}}^m)}-\Omega_{m}+\frac{i}{T}|^{-1}
$. The exact same strategy used in the proof of \eqref{bd:estimationunL2} can be applied here, and the contribution of these additional factors can be estimated the exact same way. We therefore only sketch the adaptation.

We first partition the domain of integration using the sets $S_\beta$ defined by \eqref{diagrams:id:definition-S-beta-b}. The same proof as that leading to \eqref{id:FtL2inter2-b} shows that
\be
\label{id:FtL2inter2-b}  |\mathcal F_T(P,G,N)|\lesssim  \epsilon^{\frac{K'}{4}}+ \sum_{\beta} \mathcal F_T(P,G,N,\beta),
\ee
where
\begin{align}
\label{id:FtL2inter3-b} \mathcal F_T(P,G,N,\beta)=& \lambda^{2jN}  \epsilon^{Nd(j+1)}  \iiiint_{(\underline{\tau},\underline \alpha ,\underline \eta,\underline \xi^f) \in S_\beta} d\underline{\tau}d\underline \xi d\underline \eta d\underline \alpha \delta(\tau_0-\tau_{4N}) \Delta(\underline \xi,\underline \eta)  \prod_{m=0}^{N-1} \mathds{1}_{A_R^l}(\xi_{4m})\\
\nonumber &\qquad  \prod_{\mathcal p\in \mathcal P_m}\frac{1}{|\alpha_{\mathcal p}+\frac{ic_{\mathcal p}}{T}|}  \prod_{v\in \mathcal V_i \backslash \{v^b_1,...,v^b_{2N}\}} \frac{1}{|\alpha_{\mathcal p(v)}-\sum_{\tilde{\mathcal p}\triangleleft v}\alpha_{\tilde{\mathcal p}}-\sum_{\tilde v\in \mathcal p^+(v)}\Omega_{\tilde v}+\frac{ic_v}{T}|} \\
\nonumber&\qquad \qquad \qquad \qquad \prod_{m=1}^{2N} \frac{1}{|\tau_{2m}+i|} \frac{1}{|\tau_{2m}-\tau_{2m-2}-\alpha_{\mathcal p(v_{\text{top}}^m)}-\Omega_{m}+\frac{i}{T}|}.
\end{align}

We then pick a set $S_\beta$ and apply Lemma \ref{lem:decomposition-b} which partition the graph into non-degenerate vertices, degenerate clusters, and degenerate b-clusters. One estimates the contribution of each interaction vertex one after another $v_1,...,v_{2N(j+1)}$, according to the integration order. We integrate over the variables $\underline{\xi^f}$, $\underline{\alpha}$ and $\underline{\tau}$ iteratively accordingly. The variables $\underline{\tau_k},\underline{\alpha_k},\underline \eta,\underline{\xi^f_k}$ and the set $S_{\beta,k}$ at the $k$-th step are defined analogously as in the proof of Proposition \ref{propexpansion}. We adapt the notation \eqref{id:def:Theta-k} for the factors in \eqref{id:FtL2inter3-b}:

$$
\Theta_k= \left\{ \begin{array}{l l} \alpha_{\mathcal p(v_k)}-\sum_{\tilde{\mathcal p}\triangleleft v_k}\alpha_{\tilde{\mathcal p}}-\sum_{\tilde{v}\in \mathcal p^+(v_k)}\Omega_{\tilde{v}}+\frac{ic_k}{t} & \qquad \mbox{if }v_k\in \mathcal V^i \backslash \mathcal V^b \\
\tau_{2m}-\tau_{2m-2}-\alpha_{\mathcal p(v_{\text{top}}^m)}-\Omega_{m}+\frac{i}{T} & \qquad \mbox{if }v_k\in v^b_m \mbox{ for some }m=1,...,2N.
\end{array}
\right.
$$

At step $k$ of the algorithm, the algorithm reaches the vertex $v_k$. If $v_k\in \mathcal V^i \backslash \mathcal V^b$ and is not in degenerate b-clusters, we perform the exact same estimates as in the proof of \eqref{bd:estimationunL2}. We recall that the outcome of the estimates of steps 2 and 3 in the proof of \eqref{bd:estimationunL2} is that each degree $0$ vertex produces a factor $T$, and each degree $1$ vertex produces a factor $\epsilon^{2-d}|\log \epsilon|^2$.

\medskip

\noindent \textbf{Estimates for a non-degenerate b-vertex.} Assume the algorithm reaches $v_k=v^b_m\in \mathcal V^b$ for some $m=1,...,2N$ which is non-degenerate. Consider first the case that $v^b_m$ is a degree one linear vertex and that \eqref{linear:bd:degenerate-b-3} fails. Let $\xi^f_i$ denote the free interaction frequency at $v_k$. Then we bound using $\Theta_k=\tau_{2m}-\tau_{2m-2}-\alpha_{\mathcal p(v_{\text{top}}^m)}-\Omega_{m}+\frac{i}{T}$ and \eqref{bd:weightedintegral3}, and then $|\tau_{2m-2}|\gtrsim \epsilon^{-2}$ by the failure of \eqref{linear:bd:degenerate-b-3}:
\begin{align}
\label{dactylo}& \int_{|\alpha_{\mathcal p(v_{\text{top}}^m)}|,|\tau_{2m-2}| \leq \epsilon^{-K'}, \ |\xi^f_i|\leq K\epsilon^{-1}, \ (\underline{\tau}_k,\underline{\alpha_{k}},\underline \eta,\underline{\xi^f_k})\in S_{\beta,k}} \frac{d\alpha_{\mathcal p(v_{\text{top}}^m)}d\tau_{2m-2} d\xi^f_i}{|\tau_{2m-2}+i||\alpha_{\mathcal p(v_{\text{top}}^m)}+\frac{i c_{\mathcal p(v_{\text{top}}^m)}}{T}|  |\Theta_k|}  \\
\nonumber &\qquad \lesssim \int_{|\tau_{2m-2}| \leq \epsilon^{-K'}, \ |\xi^f_i|\leq K\epsilon^{-1}} \frac{d\tau_{2m-2}d\xi^f_i}{|\tau_{2m-2}+i||\tau_{2m}-\tau_{2m-2}-\Omega_{m}+\frac{i}{T}|}  \\
\nonumber &\qquad \qquad \lesssim \epsilon^2 \int_{ |\xi^f_i|\leq K\epsilon^{-1}}  d\xi^f_i \int_{|\tau_{2m-2}| \leq \epsilon^{-K'}} \frac{d\tau_{2m-2}}{|\tau_{2m}-\tau_{2m-2}-\Omega_{m}+\frac{i}{T}|}\ \lesssim \ \epsilon^{2-d} |\log \epsilon|.
\end{align}
In all other cases, i.e. that $v_k$ is of degree one linear and that \eqref{linear:bd:degenerate-b-1}, or \eqref{linear:bd:degenerate-b-2} or \eqref{linear:bd:degenerate-b-4} fails, or that $v_k$ is of degree $0$ or of degree $1$ and quadratic, we start by integrating over the $d\tau_{2m-2}$ variable using $\Theta_k=\tau_{2m}-\tau_{2m-2}-\alpha_{\mathcal p(v_{\text{top}}^m)}-\Omega_{m}+\frac{i}{T}$ and \eqref{bd:weightedintegral3}:
\begin{align*}
& \int_{|\tau_{2m-2}| \leq \epsilon^{-K'}, \ (\underline{\tau}_k,\underline{\alpha_{k}},\underline \eta,\underline{\xi^f_k})\in S_{\beta,k}} \frac{d\tau_{2m-2}}{|\tau_{2m-2}+i||\alpha_{\mathcal p(v_{\text{top}}^m)}+\frac{i c_{\mathcal p(v_{\text{top}}^m)}}{T}|  |\Theta_k|}\\
&\qquad \qquad  \qquad \qquad  \qquad \qquad  \lesssim \frac{1}{|\alpha_{\mathcal p(v_{\text{top}}^m)}+\frac{i c_{\mathcal p(v_{\text{top}}^m)}}{T}|  |\tau_{2m}-\alpha_{\mathcal p(v_{\text{top}}^m)}-\Omega_{m}+\frac{i}{T}|} .
\end{align*}
Thanks to the above inequality, we are back to performing exactly the same estimates as in step 2 of the proof of \eqref{bd:estimationunL2}, and obtain a factor $T$ or a factor $\epsilon^{2-d}|\log \epsilon|^2$ for a degree $0$ vertex or a degree $1$ vertex.
Combining with \eqref{dactylo}, we obtain in all cases a factor $T$ for a degree $0$ vertex and a factor $\epsilon^{2-d}|\log \epsilon|^2$ for a degree $1$ vertex.

\medskip

\noindent \textbf{Estimates for a degenerate b-clusters.} Let $\mathcal C$ denote a degenerate b-cluster in the sense of Definition \ref{def:cluster-b}, which contains $n_0(\mathcal C)$ degree $0$ vertices and $n_1(\mathcal C)$ degree $1$ vertices. The estimate we will show below will produce a $T^{n_0(\mathcal C)}(\epsilon^{2-d}|\log \epsilon|^2)^{n_1(\mathcal C)}$ factor.

In comparison with the estimates for clusters in the proof of \eqref{bd:estimationunL2}, the sole difference is the appearance of the extra $\tau_0,...,\tau_{2N}$ variables. The same strategy as that used to estimate clusters in the proof of \eqref{bd:estimationunL2} applies, up to integrating over these extra variables. We will only give the details in the most complicated case of a type III b-cluster $\mathcal C=(v_k,v_{k'},v_{k''})$ with $v_{k''}\in \mathcal V^1$, as all other cases are simpler.

In this case we have $\mathcal C=(v^{m+1}_{\textup{top}},v^b_m,v^b_{m+1})$ for some $m=1,...,2N-1$. Assume first $v_k\in \mathcal V^j$. Let $\xi^f,\xi^{'f},\xi^{''f}$ be the free interaction frequencys at $v_k,v_{k'},v_{k''}$. The algorithm reaches first $v_k$. We have using \eqref{id:formularesonancelinear} and $\xi^{''f}=\xi_{(v^{m+1}_{\textup{top}},v^b_{m+1})}$ that $\Theta_k=-\sigma(\xi^f)2\xi^{''f}.\xi^f-\alpha_{\mathcal p_j(v_k)}+\gamma+\frac{c_{v_k}i}{T}$ where $\gamma$ depends only on $(\underline{\xi_k}^f,\underline{\alpha_k},\underline{\eta})$ but not on $\xi^f$ and $\alpha_{\mathcal p_j(v_k)}$. We integrate over $\alpha_{\mathcal p_j(v_k)}$ using \eqref{bd:weightedintegral3} and then over $\xi^f$ using \eqref{bd:degreeonelinear1} and get:
\begin{align}
\nonumber & \int_{|\alpha_{\mathcal p_j(v_k)}|\leq \epsilon^{-K'}, \ |\xi^f|\leq K\epsilon^{-1}, \ (\underline{\tau_k}\underline{\alpha_{k}},\underline \eta,\underline{\xi^f_k})\in S_{\beta,k}} \frac{d\alpha_{\mathcal p_j(v_k)} d\xi^f}{|\alpha_{\mathcal p_j(v_k)}+\frac{i c_{\mathcal p_j(v_k)}}{T}| |\Theta_k|}\\
\label{bd:clustersinter12-b}&\qquad \qquad   \lesssim \int_{|\xi^f|\leq K\epsilon^{-1}} \frac{1}{|-\sigma(\xi^f)2\xi^{''f}.\xi^f+\gamma+\frac{c_{v_k}i}{T}|}d\xi^f \ \lesssim \frac{\epsilon^{1-d}|\log \epsilon|}{|\xi^{''f}|}.
\end{align}
Then, the algorithm reaches $v_{k'}$. We first integrate successively over $\alpha_{\mathcal p_j(v_{k'})}$ and $d\tau_{2m-2}$ using \eqref{bd:weightedintegral3} and $\Theta_{k'}=|\tau_{2m}-\tau_{2m-2}-\alpha_{\mathcal p(v^m_{\textup{top}})}-\Omega_m+\frac{i}{T}|$. We then bound $|\tau_{2m}+i|^{-1}\lesssim \epsilon^2$ by \eqref{bd:degenimpliesnondegenleft-b} and $|\alpha_{\mathcal p_{v^{m+1}_{\textup{top}}}}+\frac{i c_{v^{m+1}_{\textup{top}}}}{T}|\lesssim \epsilon^2$  by \eqref{bd:degenimpliesnondegenleft-b}. We finally integrate over $\alpha_{\mathcal p(v_{\text{top}}^{m+1})}$, then $\tau_{2m}$ and then $\xi^{f'}$. This shows:
\begin{align*}
& \int_{|\alpha_{\mathcal p(v_{\text{top}}^m)}|,|\alpha_{\mathcal p(v_{\text{top}}^{m+1})}|,|\tau_{2m-2}|,|\tau_{2m}| \leq \epsilon^{-K'}, \ |\xi^{'f}|\leq K\epsilon^{-1}, \ (\underline{\tau}_{k'},\underline{\alpha_{k'}},\underline \eta,\underline{\xi^f_{k'}})\in S_{\beta,k'}}\\
& \qquad \qquad  \frac{d\alpha_{\mathcal p(v_{\text{top}}^m)}d\alpha_{\mathcal p(v_{\text{top}}^{m+1})}d\tau_{2m-2} d\tau_{2m}d\xi^{'f} }{|\tau_{2m}+i||\tau_{2m-2}+i||\alpha_{\mathcal p(v_{\text{top}}^m)}+\frac{i c_{\mathcal p(v_{\text{top}}^m)}}{T}||\alpha_{\mathcal p(v_{\text{top}}^{m+1})}+\frac{i c_{\mathcal p(v_{\text{top}}^{m+1})}}{T}|  |\Theta_{k'}||\Theta_{k''}|}  \\
\lesssim& \int_{|\alpha_{\mathcal p(v_{\text{top}}^{m+1})}|,|\tau_{2m}| \leq \epsilon^{-K'}, \ |\xi^{'f}|\leq K\epsilon^{-1}, \ (\underline{\tau}_{k'},\underline{\alpha_{k'}},\underline \eta,\underline{\xi^f_{k'}})\in S_{\beta,k'}}\\
& \qquad \qquad  \frac{ d\alpha_{\mathcal p(v_{\text{top}}^{m+1})}d\tau_{2m}d\xi^{'f} }{|\tau_{2m}+i|  |\alpha_{\mathcal p(v_{\text{top}}^{m+1})}+\frac{i c_{\mathcal p(v_{\text{top}}^{m+1})}}{T}|  |\tau_{2m}-\Omega_m+\frac iT||\tau_{2m+2}-\tau_{2m}-\alpha_{\mathcal p_{v^{m+1}_{\textup{top}}}}-\Omega_{m+1}+\frac iT|}  \\
\lesssim&\epsilon^4 \int_{|\alpha_{\mathcal p(v_{\text{top}}^{m+1})}|,|\tau_{2m}| \leq \epsilon^{-K'}, \ |\xi^{'f}|\leq K\epsilon^{-1} } \frac{ d\xi^{'f}  d\tau_{2m}d\alpha_{\mathcal p(v_{\text{top}}^{m+1})}}{  |\tau_{2m}-\Omega_m+\frac iT||\tau_{2m+2}-\tau_{2m}-\alpha_{\mathcal p_{v^{m+1}_{\textup{top}}}}-\Omega_{m+1}+\frac iT|}  \\
&\qquad \lesssim \epsilon^{4-d} |\log \epsilon|^2.
\end{align*}
When reaching $v_{k''}$, we integrate over the variable $\xi^{''f}$ the $|\xi^{''f}|^{-1}$ factor produced by \eqref{bd:clustersinter12-b}, giving $\int_{|\xi^{''f}|\leq K\epsilon^{-1}, \ (\underline{\tau_{k''}},\underline{\alpha_{k''}},\underline \eta,\underline{\xi^f_{k''}})\in S_{\beta,k''}}|\xi^{''f}|^{-1}d\xi^{''f}\lesssim \epsilon^{1-d}.$ Combining the factors obtained at $v_k$, $v_{k'}$ and $v_{k''}$ give a total factor for $\mathcal C$ of $\epsilon^{1-d}|\log \epsilon|\epsilon^{4-d}|\log \epsilon|^2\epsilon^{1-d}=T^{n_0(\mathcal C)}(\epsilon^{2-d}|\log \epsilon|)^{n_1(\mathcal C)}$. If $v_k\notin \mathcal V^j$ then the proof is the same, suffice it to notice that we do not have to integrate over $\alpha_{\mathcal p_j(v_k)}$ to start with.

\medskip

\noindent \textbf{End of the algorithm.} Combining, degree $0$ and $1$ vertices have each produced a $T$ and a $\epsilon^{2-d}|\log \epsilon|^2$ factor respectively. There are in total $N(j+1)$ degree $0$ vertices and $N(j+1)$ degree $1$ vertices. Finally, when reaching the root vertex $v_R$, we integrate over $\xi_{2N}=\xi_{0}$ (which is always a free interaction frequency), and get an extra $\epsilon^{-Ad}$ factor due to the indicatrix function $\mathds{1}_{\mathcal A_{R}^0}(\xi_{0})$. This concludes that in the case $l=0$:
\be \label{eastern}
|\mathcal F_T(P,G,N,\beta)|\lesssim \epsilon^{-Ad} (\frac{T}{T_{kin}})^{N(j+1)} |\log \epsilon|^{2N(j+1)}.
\ee

\subsubsection{The case $l\geq 1$}

We use a different and much simpler algorithm to estimate the right-hand side of \eqref{id:formulaexpectancytrace2}.

\medskip

\noindent  \textbf{Preliminary upper bounds.} First, we localize the support of the integrand in \eqref{id:formulaexpectancytrace2} and only keep certain factors. Since $\hat W_0^\epsilon$ has compact support within the same ball for all $0<\epsilon\leq 1$, we bound $\prod_{(i,i')\in P} |\hat W_0^\epsilon (\eta_{i,i'},\frac \epsilon 2 (\sigma_{0,i}\xi_{0,i}+\sigma_{0,i'}\xi_{0,i'}))|\lesssim  \prod_{k=1}^{N(j+1)} \mathbbm 1_{B^d(K\epsilon^{-2})}(\xi^f_k)\prod_{(i,i')\in P}  \mathbbm 1_{B^d(K)}(\eta_{i,i'}) $ as explained in step 1 of the proof of \eqref{bd:estimationunL2}. We bound $|\tau_{4N}-\tau_{4N-2}-\alpha_{\mathcal p(v^{2N}_{\textup{top}})}+\frac iT|^{-1}\leq T\leq 1$ and $\prod_{m=0}^{N-1}\mathbbm 1_{A^l_R}(\xi_{4m})\leq \mathbbm 1_{A^l_R}(\xi_0)$. We replace $|\tau_{4N}+i|^{-1}=|\tau_{0}+i|^{-1}$ as $\tau_{4N}=\tau_0$. This gives
\begin{align*}
 |\mathcal F_T(P,G,N)|\lesssim & \lambda^{2jN}  \epsilon^{Nd(j+1)}  \iiiint d\underline{\tau}d\underline{\xi^f} d\underline \eta d\underline \alpha \Delta(\underline \xi,\underline \eta)   \mathds{1}_{A_R^l}(\xi_{0})\prod_{k=1}^{N(j+1)} \mathbbm 1_{B^d(K\epsilon^{-2})}(\xi^f_k)\prod_{(i,i')\in P}  \mathbbm 1_{B^d(K)}(\eta_{i,i'}) \\
 & \prod_{\mathcal p\in \mathcal P_m}\frac{1}{|\alpha_{\mathcal p}+\frac{ic_{\mathcal p}}{T}|}  \prod_{v\in \mathcal V_i \backslash \{v^b_1,...,v^b_{2N}\}} \frac{1}{|\alpha_{\mathcal p(v)}-\sum_{\tilde{\mathcal p}\triangleleft v}\alpha_{\tilde{\mathcal p}}-\sum_{\tilde v\in \mathcal p^+(v)}\Omega_{\tilde v}+\frac{ic_v}{T}|}  \\
& \qquad \prod_{m=1}^{2N-1} \frac{1}{|\tau_{2m-2}+i|} \frac{1}{|\tau_{2m}-\tau_{2m-2}-\alpha_{\mathcal p(v_{\text{top}}^m)}-\Omega_{m}+\frac{i}{T}|}
\end{align*}
where now $\underline{\tau}=(\tau_0,...,\tau_{4N-2})$.

Second, we integrate over $d\underline{\alpha}$ performing rough estimates. We order the set of maximal paths upright which do not lead to one of the top vertices $v^m_{\textup{top}}$ for $m=1,...,2N$ from left to right: $\mathcal p_1,...,\mathcal p_{n_m-2N}$, by ordering their corresponding initial vertices from left to right. For each $1\leq n\leq n_m-2N$, we pick randomly a vertex $v_n\in \mathcal p_n$. We bound
$$
\prod_{v\in \mathcal V_i \backslash \{v^b_1,...,v^b_{2N}\}} \Big|\alpha_{\mathcal p(v)}-\sum_{\tilde{\mathcal p}\triangleleft v}\alpha_{\tilde{\mathcal p}}-\sum_{\tilde v\in \mathcal p^+(v)}\Omega_{\tilde v}+\frac{ic_v}{T}\Big|^{-1} \leq \prod_{n=1}^{n_m-2N} \Big|\alpha_{\mathcal p_n}-\sum_{\tilde{\mathcal p}\triangleleft v_n}\alpha_{\tilde{\mathcal p}}-\sum_{\tilde v\in \mathcal p^+(v_n)}\Omega_{\tilde v}+\frac{ic_{v_n}}{T}\Big|^{-1}.
$$
We notice that for $n>\tilde n$, the quantity $\alpha_{\mathcal p_n}-\sum_{\tilde{\mathcal p}\triangleleft v_n}\alpha_{\tilde{\mathcal p}}-\sum_{\tilde v\in \mathcal p^+(v_n)}\Omega_{\tilde v}+\frac{ic_{v_n}}{T}$ does not depend on $\alpha_{\mathcal p_{\tilde n}}$, because of the ordering we chose for these maximal paths upright. We then integrate successively over $d\alpha_{\mathcal p_1},...,d\alpha_{\mathcal p_{n_m-2N}}$, and bound $\int d\alpha_{\mathcal p_n}|\alpha_{\mathcal p_n}+\frac{ic_{\mathcal p_n}}{T}|^{-1}|\alpha_{\mathcal p_n}-\sum_{\tilde{\mathcal p}\triangleleft v_n}\alpha_{\tilde{\mathcal p}}-\sum_{\tilde v\in \mathcal p^+(v_n)}\Omega_{\tilde v}+\frac{ic_{v_n}}{T}|^{-1}\lesssim T\leq 1$ by \eqref{bd:weightedintegral3} and $T\leq \epsilon$. Then, we integrate over $d\alpha_{\mathcal p(v^m_{\textup{top}})}$ for $m=1,...,2N$ using the bound $\int |\alpha_{\mathcal \mathcal p(v^m_{\textup{top}})}+\frac{ic_{\mathcal p(v^m_{\textup{top}})}}{T}|^{-1} |\tau_{2m}-\tau_{2m-2}-\alpha_{\mathcal p(v_{\text{top}}^m)}-\Omega_{m}+\frac{i}{T}|^{-1} d\alpha_{\mathcal p(v^m_{\textup{top}})}\lesssim  |\tau_{2m}-\tau_{2m-2}-\Omega_{m}+\frac{i}{T}|^{-1}$ from \eqref{bd:weightedintegral3}. This yields
\begin{align}
\nonumber |\mathcal F_T(P,G,N)|\lesssim & \lambda^{2jN}  \epsilon^{Nd(j+1)}  \iiint d\underline{\tau}d\underline{\xi^f} d\underline \eta \Delta(\underline \xi,\underline \eta)   \mathds{1}_{A_R^l}(\xi_{0})\prod_{k=1}^{N(j+1)} \mathbbm 1_{B^d(K\epsilon^{-2})}(\xi^f_k)  \prod_{(i,i')\in P}  \mathbbm 1_{B^d(K)}(\eta_{i,i'}) \\
\label{id:formulaexpectancytrace2-l=1-1}& \qquad \qquad \qquad \qquad \qquad   \prod_{m=1}^{2N-1} \frac{1}{|\tau_{2m-2}+i|} \frac{1}{|\tau_{2m}-\tau_{2m-2}-\Omega_{m}+\frac{i}{T}|}.
\end{align}

\noindent  \textbf{The algorithm.} We estimate \eqref{id:formulaexpectancytrace2-l=1-1} by considering successively for $m=1,...,N$ the two vertices $v^b_{2m+1}$ and $v^b_{2m+2}$. We will perform estimates by integrating over $d\tau_{4m}$ and $d\tau_{4m+2}$ and certain free frequencies. The quantities $|\tau_{2\tilde m-2}+i|^{-1}$ and $|\tau_{2\tilde m}-\tau_{2\tilde m-2}-\Omega_{\tilde m}+\frac{i}{T}|^{-1}$ for $\tilde m\geq 2m+3$ will not depend on $\tau_{4m}$, $\tau_{4m+2}$ and these free frequencies, so that we will be able to iterate our algorithm.

\smallskip

\noindent \underline{Case 1 if $v^b_{2m+1}\in \mathcal V^1$}. Denote by $\xi^f_i$ its associated free interaction frequency. We integrate over $d\tau_{4m}$ using \eqref{bd:weightedintegral3}, then over $d\xi^f_i$ using \eqref{id:formularesonancelinear} and \eqref{bd:degreeonequadra} if $v^b_{2m+1}$ is linear or \eqref{id:formularesonancequadra} and \eqref{lagopede} if $v^b_{2m+1}$ is quadratic, using $|\tilde \xi|=|\xi_{(v^b_{2m+1},v^b_{2m+2})}|\approx  2^l\epsilon^{-A}$ by \eqref{bd:degenimpliesnondegenleft-lgeq1}:
\begin{align*}
& \iiint  \frac{ \mathbbm 1_{B^d(K\epsilon^{-2})}(\xi^f_i) d\tau_{4m}d\tau_{4m+2}d\xi^f_i}{|\tau_{4m}+i||\tau_{4m+2}+i| |\tau_{4m+2}-\tau_{4m}-\Omega_{2m+1}+\frac{i}{T}| |\tau_{4m+4}-\tau_{4m+2}-\Omega_{2m+2}+\frac{i}{T}|} \\
\lesssim &\int \frac{ d\tau_{4m+2}}{|\tau_{4m+2}+i| |\tau_{4m+4}-\tau_{4m+2}-\Omega_{2m+2}+\frac{i}{T}|} \int \frac{\mathbbm 1_{B^d(K\epsilon^{-2})}(\xi^f_i) d\xi^f_i}{ |\tau_{4m+2}-\Omega_{2m+1}(\xi^f_i)+\frac{i}{T}|} \\
\lesssim &\int \frac{ d\tau_{4m+2}}{|\tau_{4m+2}+i| |\tau_{4m+4}-\tau_{4m+2}-\Omega_{2m+2}+\frac{i}{T}|} \epsilon^{1-d+A}2^{-l} \log (2^l \epsilon^{-A}) \ \leq 2^{-\frac l2}\epsilon^{\frac A2} 
\end{align*}
where for the last inequality we used \eqref{bd:weightedintegral3}, $T\leq 1$ and took $A$ large enough then $\epsilon$ small enough.

\smallskip

\noindent \underline{Case 2 if $v^b_{2m+1}\notin \mathcal V^1$ and $v^b_{2m+2}\in \mathcal V^1$}. Denote by $\xi^f_i$ the associated free interaction frequency to $v^b_{2m+2}$. We integrate over $d\tau_{4m}$ using \eqref{bd:weightedintegral3} and $ |\tau_{4m+2}-\tau_{4m}-\Omega_{2m+1}+\frac{i}{T}|^{-1}\leq 1$, then over $d\tau_{4m+2}$ using \eqref{bd:weightedintegral3}, and then over $d\xi^f_i$ using \eqref{id:formularesonancelinear} and \eqref{bd:degreeonequadra} if $v^b_{2m+2}$ is linear or \eqref{id:formularesonancequadra} and \eqref{lagopede} if $v^b_{2m+2}$ is quadratic, with $|\tilde \xi|=|\xi_{(v^b_{2m+2},v^b_{2m+3})}|\approx  2^l\epsilon^{-A}$ by \eqref{bd:degenimpliesnondegenleft-lgeq1}:
\begin{align*}
& \iiint  \frac{ \mathbbm 1_{B^d(K\epsilon^{-2})}(\xi^f_i) d\tau_{4m}d\tau_{4m+2}d\xi^f_i}{|\tau_{4m}+i||\tau_{4m+2}+i| |\tau_{4m+2}-\tau_{4m}-\Omega_{2m+1}+\frac{i}{T}| |\tau_{4m+4}-\tau_{4m+2}-\Omega_{2m+2}+\frac{i}{T}|} \\
\lesssim &  \iint  \frac{ \mathbbm 1_{B^d(K\epsilon^{-2})}(\xi^f_i) d\tau_{4m+2}d\xi^f_i}{|\tau_{4m+2}+i|  |\tau_{4m+4}-\tau_{4m+2}-\Omega_{2m+2}(\xi^f_i)+\frac{i}{T}|}\\
\lesssim &  \int  \frac{ \mathbbm 1_{B^d(K\epsilon^{-2})}(\xi^f_i) d\xi^f_i}{  |\tau_{4m+4}-\Omega_{2m+2}(\xi^f_i)+\frac{i}{T}|} \ \leq \epsilon^{1-d+A}2^{-l} \log (2^l \epsilon^{-A}) \ \leq 2^{-\frac l2}\epsilon^{\frac A2} .
\end{align*}

\noindent \underline{Case 3 if $v^b_{2m+1},v^b_{2m+2} \notin \mathcal V^1$}. For $\tilde m=1,...,2N$, if $v^b_{\tilde m}\notin \mathcal V^1$, then using \eqref{id:integrationmomenta} we have $\xi_{(v^{\tilde m}_{\textup{top}},v^b_{\tilde m})}=\sum_{ \{i,j\}\in P} c_{\tilde m,i,j} \eta_{ i,j }+  \sum_{k=1}^{N(j+1)}  c_{\tilde m,k} \xi^f_k $
where $c_{\tilde m,i,j}\neq 0$ if and only if $(v^{\tilde m}_{\textup{top}},v^b_{\tilde m})$ belongs to the path going from the initial vertex $v_{0,i}$ to the root vertex. The first sum is non-zero, since $c_{\tilde m,i,j}\neq 0$ for all pairings with initial vertices below $v^{\tilde m}_{\textup{top}}$. We denote by $\sigma_{\tilde m}\eta_{\tilde m}$ one if its non-zero element with $\sigma_{\tilde m}\in \{-1,1\}$, and by $E(\tilde m)$ the set of remaining pairings $ \{i,j\}\in P$ for which $c_{\tilde m,i,j}\neq 0$, so that 
\be \label{toiture}
\xi_{(v^{\tilde m}_{\textup{top}},v^b_{\tilde m})}=\sigma_{\tilde m} \eta_{\tilde m } +\sum_{ \{i,j\}\in E(\tilde m)} c_{\tilde m,i,j} \eta_{ i,j }+  \sum_{k=1}^{N(j+1)}  c_{\tilde m,k} \xi^f_k.
\ee
We claim that for $\tilde m_2>\tilde m_1$, then $\xi_{(v^{\tilde m_2}_{\textup{top}},v^b_{\tilde m_2})}$ does not depend on $\eta_{\tilde m_1 }$. Indeed, $c_{\tilde m_1,i,j}\neq 0$ if and only if $(v^{\tilde m_1}_{\textup{top}},v^b_{\tilde m_1})$ belongs to the path going from $v_{0,i}$ to the root vertex. Since the path from $v^b_{\tilde m_1}$ to the root vertex is $(v^b_{\tilde m_1},v^b_{\tilde m_1+1},...,v^b_{2N},v_R)$, then $(v^{\tilde m'_2}_{\textup{top}},v^b_{\tilde m_2})$ does not belong to the path going from $v_{0,i}$ to the root vertex. Hence $c_{\tilde m_2,i,j}=0$. This proves the claim.

We then change variables $\eta_{2m+2}\mapsto  \eta_{2m+2}'=\sigma_{2m+2}\sigma_{2m+1} \eta_{2m+1}+\eta_{2m+2}$. We claim if $\tilde m>2m+2$ then $\Omega_{\tilde m}$ is independent of $ \eta_{2m+2}'$. Indeed, by Kirchhoff's law $\xi_{(v^b_{2m+2},v^b_{2m+3})}=\xi_{(v^b_{2m},v^b_{2m+1})}+\xi_{(v^{2m+1}_{\textup{top}},v^b_{2m+1})}+\xi_{(v^{2m+2}_{\textup{top}},v^b_{2m+2})}$, so that $\xi_{(v^b_{2m+2},v^b_{2m+3})}$ is independent of $\eta_{2m+2}' $ by \eqref{toiture}. By Kirchhoff's law again and the claim of the previous paragraph, $\xi_{(v^b_{\tilde m},v^b_{\tilde m+1})}$ is independent of $\eta_{2m+2}' $ for all $\tilde m>2m+2$. Since also $\xi_{(v^{\tilde m}_{\textup{top}},v^b_{\tilde m})}$ does not depend on $\eta_{2m+2}' $, then $\Omega_{\tilde m}$ given by \eqref{id:formularesonancelinear} or \eqref{id:formularesonancequadra} is indeed independent of $\eta_{2m+2}'$.

We then first integrate over $d\tau_{4m}$ using \eqref{bd:weightedintegral3} and $ |\tau_{4m+2}-\tau_{4m}-\Omega_{2m+1}+\frac{i}{T}|^{-1}\leq 1$, then over $d\tau_{4m+2}$ using \eqref{bd:weightedintegral3}, and finally over $\eta_{2m+2}'$ using \eqref{id:formularesonancelinear} and \eqref{bd:degreeonequadra} with $\epsilon=1$ if $v^b_{2m+2}$ is linear or \eqref{id:formularesonancequadra} and \eqref{lagopede} with $\epsilon=1$ if $v^b_{2m+2}$ is quadratic, with $|\tilde \xi|=|\xi_{(v^b_{2m+2},v^b_{2m+3})}|\approx  2^l\epsilon^{-A}$ by \eqref{bd:degenimpliesnondegenleft-lgeq1}:
\begin{align*}
& \iiint  \frac{ \mathbbm 1_{B^d(2K)}(\eta_{2m+2}') d\tau_{4m}d\tau_{4m+2}d\eta_{2m+2}'}{|\tau_{4m}+i||\tau_{4m+2}+i| |\tau_{4m+2}-\tau_{4m}-\Omega_{2m+1}+\frac{i}{T}| |\tau_{4m+4}-\tau_{4m+2}-\Omega_{2m+2}+\frac{i}{T}|} \\
\lesssim &  \iint  \frac{  \mathbbm 1_{B^d(2K)}(\eta_{2m+2}') d\tau_{4m+2}d\eta_{2m+2}'}{|\tau_{4m+2}+i|  |\tau_{4m+4}-\tau_{4m+2}-\Omega_{2m+2}(\eta_{2m+2}')+\frac{i}{T}|}\\
\lesssim &  \int  \frac{ \mathbbm 1_{B^d(2K)}(\eta_{2m+2}')  d\eta_{2m+2}'}{  |\tau_{4m+4}-\Omega_{2m+2}(\xi^f_i)+\frac{i}{T}|} \ \leq \epsilon^{A}2^{-l} \log (2^l \epsilon^{-A}) \ \leq 2^{-\frac l2}\epsilon^{\frac A2} .
\end{align*}

\noindent \underline{End of the algorithm}. In all of the three previous cases, we have obtained a $2^{-\frac l2}\epsilon^{\frac A2}$ factor. Performing this estimate for $v^b_1,...,v^b_{2N-2}$ gives a total $2^{-\frac l2 (N-1)}\epsilon^{\frac A 2 (N-1)}$ factor. When reaching $v^b_{2N-1}$ we integrate $|\tau_{4N-4}+i|^{-1}|\tau_{4N-2}+i|^{-1}|\tau_{4N-2}-\tau_{4N-4}-\Omega_{4N-2}+\frac iT|^{-1}$ over $d\tau_{4N-4}d\tau_{4N-2}$ which yields a factor $1$. We integrate over all remaining free frequencies among $\xi^f_1,...,\xi^f_{N(j+1)}$ and $\eta_{i,i'}$ for $\{i,i'\}\in P$, which yields a factor of at most $\epsilon^{-dN(j+1)}$. We finally integrate $\mathbbm 1_{A^l_R}(\xi_0)$ which yields a factor $2^l \epsilon^{-A}$. Combining, this gives
\be \label{chocolate}
 |\mathcal F_T(P,G,N)|\lesssim \lambda^{2jN}  \epsilon^{Nd(j+1)} 2^{-\frac l2 (N-1)}\epsilon^{\frac A 2 (N-1)}\epsilon^{-dN(j+1)}2^l \epsilon^{-A}\leq (\frac{T}{T_{kin}})^{N(j+1)} 2^{-\frac{Nl}{3}}\epsilon^{\frac{NA}{3}}
\ee
where we used that $T\geq \epsilon^{-2}$, took $A$ and $N$ large enough and then $\epsilon$ small enough.

\subsubsection{Conclusion}

Injecting \eqref{eastern} in \eqref{id:FtL2inter2-b} and then in \eqref{crepuscule} for $l=0$, and injecting \eqref{chocolate} in \eqref{crepuscule} for $l\geq 1$  shows that for $K'$ large enough then for $\epsilon$ small enough, for all $l=0,1,...$,
$$
\mathbb{E}\left[\Tr\mathfrak{M}^N\right] \lesssim \epsilon^{-Ad} (\frac{T}{T_{kin}})^{N(j+1)} |\log \epsilon|^{2N(j+1)} 2^{-\frac{Nl}{3}} \epsilon^{\frac{NA}{3}\delta_{l\geq 1}}.
$$

We conclude by Bienaym\'e-Tchebychev inequality, that for each $\kappa$ and $l$, there exists a set $E_l$ with measure $\mathbb{P}(E_l)>1-2^{-l}\epsilon^\kappa$ such that
$$
\Tr\mathfrak{M}^N \leq \epsilon^{-Ad} \left(\frac{T}{T_{kin}}\right)^{N(j+1)} \epsilon^{-2\kappa}2^{-\frac{Nl}{4}}\epsilon^{\frac{NA}{3}\delta_{l\geq 1}}.
$$
On this set, we have
$$
\| \mathfrak{L}_{G,\iota}\mathcal A_R^l \|_{X^{0,\frac{1}{2}}\to X^{0,-\frac{1}{2}}}\leq (\Tr\mathfrak{M}^N)^{\frac{1}{2N}}\lesssim\left(\frac{T}{T_{kin}}\right)^{\frac{j+1}{2}}\epsilon^{\frac{-Ad+2\kappa}{2N}}2^{-\frac l8}\epsilon^{\frac{A}{6}\delta_{l\geq 1}}\lesssim\left(\frac{T}{T_{kin}}\right)^{\frac{j+1}{2}}\epsilon^{-\kappa}2^{-\frac l8}\epsilon^{\frac{A}{8}\delta_{l\geq 1}}
$$
for $N$ large enough. The proof of Lemma \ref{linearization lemma} is complete and Proposition \ref{linearization proposition} follows.

\section{Control of the error}

\subsection{Bound on the error term $E_N$} 
\begin{proposition}\label{error control proposition} For any $N\in \mathbb N$, there exists $\epsilon^*>0$ such that for all $0<\epsilon\leq \epsilon^*$, for all $T\geq \epsilon^2$ and $b\in [\frac 12,1]$:
\be \label{bd:ENXsb}
\mathbb E \left\| \chi \int_0^t e^{i(t-s) \Delta}  \chi\left( \frac{s}{T} \right) E^N \,ds \right\|_{X^{s,b}_\epsilon}
\lesssim T^{1/2} \epsilon^{-1-\frac{d}{2}} \sum_{j+k \geq N}
\left( \mathbb E \left\| \chi \left(\frac{\cdot}{T} \right) u^j \right\|_{X^{s,b}_\epsilon}^2 \right)^{\frac 12}  \left( \mathbb E \left\| \chi \left(\frac{\cdot}{T} \right) u^k \right\|_{X^{s,b}_\epsilon}^2\right)^{\frac 12}.
\ee

\end{proposition}
\begin{proof}

First, notice that as $X^{s,b_1}_\epsilon$ is continuously embedded in $X^{s,b_2}_\epsilon$ for $b_2\leq b_1$, it suffices to establish \eqref{bd:ENXsb} for $b_1=1$. Notice that the Fourier support of the approximate solution is in a ball centred at the origin with radius $\lesssim \epsilon^{-1}$, making the regularity index $s$ irrelevant in our scaled Sobolev and Bourgain spaces. We write $\chi_T(t)=\chi(t/T)$ in what follows.

We apply successively~\eqref{mesangebleue0}, and the above remark on the Fourier localisation, obtaining:
\begin{align*}
& \left\| \chi_T (t)\int_0^t e^{i(t-s) \Delta} \chi_T (s)E^N \,ds \right\|_{X^{s,1}_\epsilon} \lesssim \left\|  \chi_T  E^N  \right\|_{X^{s,0}_\epsilon}\lesssim   \sum_{\substack{j+k \geq N \\ j,k = 0,\dots,N}} \left\| \chi_T  u^j u^k \right\|_{L^2_tL^2_x}.
\end{align*}
Above, applying H\"older inequality, then Bernstein inequality (using the Fourier localisation of the approximate solution), H\"older inequality again and finally \eqref{cormorant}:
\begin{align*}
&\left\| \chi_T  u^j u^k \right\|_{L^2_tL^2_x}\lesssim \left\| \chi_T  u^j \right\|_{L^\infty_tL^\infty_x} \| \chi_T  u^k \|_{L^2_tL^2_x} \\
&\qquad \qquad   \lesssim T\epsilon^{-\frac d2} \left\| \chi_T u^j \right\|_{\mathcal C(\mathbb R,L^2_x)} \| \chi_T u^k \|_{\mathcal C(\mathbb R,L^2_x)} \lesssim T\epsilon^{-\frac d2} \left\| \chi_T u^j \right\|_{X^{s,b}_\epsilon} \| \chi_T u^k \|_{X^{s,b}_\epsilon}.
\end{align*}
The Cauchy-Schwarz inequality followed by Proposition~\ref{propexpansion} gives then
\begin{align*}
\mathbb{E}  \left\| \chi_T u^j u^k \right\|_{L^2_tL^2_x} & \lesssim T\epsilon^{-\frac d2} \left( \mathbb E \left\| \chi_T u^j \right\|_{X^{s,b}_\epsilon}^2\right)^{\frac 12}  \left( \mathbb E \| \chi_T u^k \|_{X^{s,b}_\epsilon}^2\right)^{\frac 12}.
\end{align*}
Combining the three inequalities above yields \eqref{bd:ENXsb}.
\end{proof}
\subsection{The bilinear $X^{s,b}_\epsilon$ estimate}

\begin{proposition}\label{bilinear proposition} If $s > \frac{d}{2}-1$ and $b> \frac{1}{2}$,
$$
\left\| \chi(t) \int_0^t e^{i(t-s) \frac{\Delta}{2}} \chi(s) u^2 \,ds \right\|_{X^{s,b}_\epsilon} \lesssim \epsilon^{} \| u \|_{X^{s,b}_\epsilon}^2.
$$
The same estimate holds true if $u^2$ is replaced by $\overline{u}^2$ or $|u|^2$.
\end{proposition}

The proof of this proposition will rely on the following lemma, proved in~\cite{CG1}.

\begin{lemma}[Lemma 7.3, \cite{CG1}] \label{robin} If $N_1 \leq N_2 \in 2^{\mathbb{N}_0}$, for any $\kappa>0$ there exists $b_0 < \frac{1}{2}$ such that
 $$
 \|\chi(s) P_{\epsilon, N_1} u P_{\epsilon, N_2} v \|_{L^2 L^2} \lesssim N_1^{\frac{d}{2}-1 + \kappa} \epsilon^{-\frac{d}{2}+1-\kappa} \| P_{\epsilon, N_1} u \|_{X^{0,b_0}_{\epsilon}} \| P_{\epsilon, N_2} v \|_{X^{0,b_0}_{\epsilon}}
$$
The same holds if $u$ or $v$ are replaced by their complex conjugates.
\end{lemma}

Equipped with this lemma, we can now turn to the proof of the proposition.

\begin{proof}
By~\eqref{mesangebleue0}, it suffices to prove that
$$
\| \chi(s) u^2 \|_{X^{s,b-1}_\epsilon} \lesssim \| u \|_{X^{s,b}_\epsilon}^2.
$$
We will prove this bound by duality: choosing $v \in X_\epsilon^{-s,1-b}$, it reduces to estimating $\iint \chi(s) u^2 \overline v \,dx \,ds$. Applying a Littlewood-Paley decomposition in $u$ and $v$, this becomes
$$
\sum_{N_1,N_2,N_3 \in 2^{\mathbb{N}_0}} \iint \chi(s) P_{\epsilon,N_1} u P_{\epsilon,N_2} u P_{\epsilon,N_3} \overline v \,dx \,ds.
$$
Without loss of generality, we can assume that $N_2 \gtrsim N_3$. Applying the Cauchy-Schwarz inequality followed by Lemma~\ref{robin},
\begin{align*}
\left| \iint P_{\epsilon,N_1} u P_{\epsilon,N_2} u P_{\epsilon,N_3} \overline v \,dx \,ds \right| & \lesssim \| \chi(s) P_{\epsilon,N_1} u P_{\epsilon,N_2} u \|_{L^2 L^2} \| P_{\epsilon,N_3} v \|_{L^2 L^2} \\
 \lesssim & N_1^{\frac{d}{2}-1 + \kappa} \epsilon^{-\frac{d}{2}+1-\kappa} \| P_{\epsilon, N_1} u \|_{X^{0,b_0}_{\epsilon}} \| P_{\epsilon, N_2} u \|_{X^{0,b_0}_{\epsilon}}  \| P_{\epsilon, N_3} v \|_{X^{0,0}_{\epsilon}} \\
\lesssim & N_1^{\frac{d}{2}-1 + \kappa -s}  \epsilon^{-\frac{d}{2}+1-\kappa} N_2^{-s} N_3^s  \| P_{\epsilon, N_1} u \|_{X^{s,b_0}_{\epsilon}} \| P_{\epsilon, N_2} u \|_{X^{s,b_0}_{\epsilon}}  \| P_{\epsilon, N_3} v \|_{X^{-s,0}_{\epsilon}}.
\end{align*}
By almost orthogonality, it is now easy to see that
$$
\sum_{N_2 \gtrsim N_3} \left| \iint P_{\epsilon,N_1} u P_{\epsilon,N_2} u P_{\epsilon,N_3} \overline v \,dx \,ds \right| \lesssim \| u \|_{X^{s,b}_\epsilon}^2 \|v \|_{X_\epsilon^{-s,1-b}}; 
$$
Indeed, the sum over $N_1$ is just a geometric series, while the sum over $N_2$, $N_3$ can be treated by Cauchy-Schwarz.
\end{proof}

\appendix

\section{$X^{s,b}_\epsilon$ spaces}

\label{Xsbbasics}

We define $X^{s,b}_\epsilon$ spaces, and review their properties, for functions defined on $\mathbb{R}^d$. This framework can be immediately translated to the case of the torus, the only difference being additional subpolynomial losses in $\epsilon$ in some Strichartz estimates.

The $X^{s,b}$ spaces were introduced in~\cite{Bourgain}. We quickly review their properties, refering the reader to~\cite{Tao}, Section 2.6, for details. 

\bigskip

\noindent \underline{Definition} Let
$$
\| f \|_{H^s_\epsilon} = \| \langle \epsilon D \rangle^s f \|_{L^2}
$$
and
$$
\| u \|_{X^{s,b}_\epsilon} = \| e^{-it\omega(D)} u(t) \|_{L^2 H^s_\epsilon} =  \| \langle \epsilon \xi \rangle^s \langle \tau +\omega(\xi) \rangle^b \widetilde{u}(\tau,k) \|_{L^2(\mathbb{R} \times \mathbb{R}^d)}
$$

\bigskip

\noindent \underline{Time continuity} For $b > \frac{1}{2}$,
\begin{equation}
\label{cormorant}
\| u \|_{\mathcal{C} H^s_\epsilon} \lesssim \| u \|_{X^{s,b}_{\epsilon}}.
\end{equation}

\bigskip

\noindent \underline{Hyperbolic regularity} Assume that $u$ solves
$$
\left\{
\begin{array}{l}
i \partial_t u +\omega(D) u = F \\ u(t=0) = 0
\end{array}
\right.
$$
Then, denoting $\chi$ for a smooth cutoff function, supported on $B(0,2)$, and equal to $1$ on $B(0,1)$,
\begin{equation}
\label{mesangebleue0}
\| \chi(t) u \|_{X_\epsilon^{s,b-1}} \lesssim \| F \|_{X_\epsilon^{s,b}}.
\end{equation}

\bigskip 

\noindent
\underline{From group to $X^{s,b}$ estimates} Assume that, uniformly in $\tau_0 \in \mathbb{R}$,
$$
\| e^{it\tau_0} e^{-it\omega(D)} f \|_{Y} \leq C_0(\epsilon) \| \langle \epsilon D \rangle^s f \|_{L^2}
$$
Then, if $b > \frac{1}{2}$,
$$
\| u \|_{Y} \lesssim_b C_0(\epsilon) \| u \|_{X^{s,b}_\epsilon}
$$

\bigskip

\noindent
\underline{Strichartz estimates} We want to apply the previous statement to Strichartz estimates: if $d \geq 2$, for any $\kappa>0$,
\begin{align*}
& \| e^{-it \omega(D) } f \|_{L^4 L^4} \lesssim_{\kappa} \epsilon^{\frac{1}{2} - \frac{d}{4} - \kappa}  \| \langle \epsilon D \rangle^{\frac{d}{4} - \frac{1}{2} +\kappa} f \|_{L^2} 
\end{align*}
As a consequence, if $\kappa>0$, $b > \frac{1}{2}$,
\begin{equation}
\label{mesangebleue1}
\begin{split}
& \| u \|_{L^{4} L^{4}} \lesssim_{b,\kappa} \epsilon^{\frac{1}{2} - \frac{d}{4} - \kappa} \| u \|_{X_\epsilon^{\frac{d}{4} - \frac{1}{2} + \kappa,b}} .
\end{split}
\end{equation}

\bigskip

\noindent \underline{Duality} The dual of $X^{s,b}_\epsilon$ is $X^{-s,-b}_\epsilon$. Therefore, the previous inequalities imply that, if $s'<0$, $\kappa>0$, $b' <- \frac{1}{2}$,
\begin{equation}
\label{mesangebleue2}
\begin{split}
& \| \chi(t) u \|_{X_\epsilon^{-\frac{d}{4} + \frac{1}{2} - \kappa,b'}} \lesssim_{b'}  \epsilon^{\frac{1}{2} - \frac{d}{4} - \kappa} \| u \|_{L^{4/3} L^{4/3}}\qquad \mbox{if $d\geq 3$}.
\end{split}
\end{equation}
Similarly, the dual of the inequality~\eqref{cormorant} is, for any $b' < \frac{1}{2}$,
\begin{equation}
\label{mesangebleue3}
\| u \|_{X^{s,b'}_{\epsilon}} \lesssim_{b'} \| u \|_{L^1 H^s_\epsilon}.
\end{equation}

\bigskip

\noindent \underline{Interpolation} If $0 \leq \theta \leq 1$, $s = \theta s_0 + (1-\theta) s_1$ and $b = \theta b_0 + (1-\theta) b_1$,
$$
\| u \|_{X^{s,b}} \leq \| u \|_{X^{s_0,b_0}}^\theta \| u \|_{X^{s_1,b_1}}^{1-\theta}. 
$$

\section{Elementary bounds}

\begin{lemma}[Estimates for degree one vertices]

For any $0<\epsilon\leq \frac 12$, any $\tilde \xi \in \mathbb R^d$ and $0< t\leq 1$ the following estimates hold true. First,
\be \label{bd:degreeonelinear1}
\int_{|\xi^f|\lesssim \epsilon^{-1}} \frac{d\xi^f}{|\gamma+\tilde \xi.\xi^f+\frac it|} \lesssim \epsilon^{1-d}\min \left( \frac{1}{|\tilde \xi|},\frac t \epsilon\right) \left(\log \langle \tilde \xi\rangle+|\log \epsilon|\right) \qquad \mbox{for all }\gamma\in \mathbb R,
\ee
and for $0<\delta\leq 1$,
\be \label{bd:degreeonelinear2}
\int_{|\xi^f|\lesssim \epsilon^{-1}} \frac{{\bf 1}(|\tilde \xi|\geq \delta \epsilon^{-1} \mbox{ or }|\gamma+\tilde \xi.\xi^f|\geq \delta \epsilon^{-2})d\xi^f}{|\gamma+\tilde \xi.\xi^f+\frac it|} \leq C(\delta) \epsilon^{1-d} \min \Big( \epsilon |\log \epsilon|,\frac{\log \langle \tilde \xi \rangle}{|\tilde \xi|} \Big) \qquad \mbox{for all }\gamma\in \mathbb R,
\ee
and for $m\in C^1([0,\infty))$ nonnegative and bounded with $m(0)=0$:
\be \label{lagopede}
\int_{|\xi^f|\lesssim \epsilon^{-1}} \frac{m(\epsilon \tilde \xi)d\xi^f}{|\gamma+\tilde \xi.\xi^f+\frac it|} \leq C(m) \epsilon^{1-d}\min \Big( \epsilon |\log \epsilon|,\frac{\log \langle \tilde \xi \rangle}{|\tilde \xi|}  \Big)\qquad \mbox{for all }\gamma\in \mathbb R.
\ee
Second,
\be \label{bd:degreeonequadra}
\int_{|\xi^f|\lesssim \epsilon^{-1}} \frac{d\xi^f}{|\gamma+(\tilde \xi+\xi^f).\xi^f+\frac it|} \lesssim \epsilon^{1-d} \min \Big( \epsilon |\log \epsilon|,\frac{\log \langle \tilde \xi \rangle}{|\tilde \xi|}\Big) \qquad \mbox{for all }\gamma\in \mathbb R.
\ee

\end{lemma}

\begin{proof}

\noindent \underline{Proof of \eqref{bd:degreeonelinear1}}. By rotational invariance, we can assume that $\tilde \xi=(|\tilde \xi|,0,...,0)$. Integrating first over the variables $\xi_2^f,...,\xi_d^f$ and then performing the change of variables $\xi^f_1= (|\tilde \xi|t)^{-1} \zeta_1 $ one finds:
$$
\int_{|\xi^f|\lesssim \epsilon^{-1}} \frac{d\xi^f}{|\gamma+\tilde \xi.\xi^f+\frac it|} \lesssim \epsilon^{1-d}\int_{|\xi_1^f|\lesssim \epsilon^{-1}} \frac{td\xi^f_1}{|t\gamma+t|\tilde \xi|\xi_1^f+i|}=\frac{\epsilon^{d-1}}{|\tilde \xi|}\int_{|\zeta_1| \lesssim |\tilde \xi|\epsilon^{-1}t} \frac{d\zeta_1}{|t\gamma+\zeta_1+i|}.
$$
The last integral above satisfies $\int_{|\zeta_1| \lesssim A} |t\gamma+\zeta_1+i|^{-1}d\zeta_1 \lesssim A$ if $A\leq 1$ and $\int...\lesssim 1+\log (A)$ if $A\geq 1$ which proves \eqref{bd:degreeonelinear1}.

\medskip

\noindent \underline{Proof of \eqref{bd:degreeonelinear2}}. If $|\tilde \xi|\geq \delta \epsilon^{-1}$, we use \eqref{bd:degreeonelinear1} and obtain
$$
\int_{|\xi^f|\lesssim \epsilon^{-1}} \frac{{\bf 1}(|\tilde \xi|\geq \delta \epsilon^{-1} \mbox{ or }|\gamma+\tilde \xi.\xi^f|\geq \delta \epsilon^{-2})d\xi^f}{|\gamma+\tilde \xi.\xi^f+\frac it|}\leq \int_{|\xi^f|\lesssim \epsilon^{-1}} \frac{d\xi^f}{|\gamma+\tilde \xi.\xi^f+\frac it|} \leq C(\delta) \epsilon^{1-d}\frac{\log\langle \tilde \xi\rangle}{|\tilde \xi|}
$$
which shows \eqref{bd:degreeonelinear2}. If $|\tilde \xi|<\delta \epsilon^{-1}$, we bound using $|\gamma+\tilde \xi.\xi^f+\frac it|\leq C(\delta)\epsilon^{-2}$ in the integrand below:
$$
\int_{|\xi^f|\lesssim \epsilon^{-1}} \frac{{\bf 1}(|\tilde \xi|\geq \delta \epsilon^{-1} \mbox{ or }|\gamma+\tilde \xi.\xi^f|\geq \delta \epsilon^{-2})d\xi^f}{|\gamma+\tilde \xi.\xi^f+\frac it|} = \int_{|\xi^f|\lesssim \epsilon^{-1}} \frac{{\bf 1}(|\gamma+\tilde \xi.\xi^f|\geq \delta \epsilon^{-2})d\xi^f}{|\gamma+\tilde \xi.\xi^f+\frac it|} \leq C(\delta) \epsilon^{2-d}.
$$
which also shows \eqref{bd:degreeonelinear2}. Hence \eqref{bd:degreeonelinear2}.

\medskip

\noindent \underline{Proof of \eqref{lagopede}}. To prove \eqref{lagopede} we treat the two cases $|\tilde \xi|<\epsilon^{-1}$ and $|\tilde \xi|\geq \epsilon^{-1}$ separately. If $|\tilde \xi|<\epsilon^{-1}$ we have $m(\epsilon |\tilde \xi|)\lesssim \epsilon |\tilde \xi|$ because $m\in C^1[0,\infty)$ is nonnegative bounded with $m(0)=0$. Using this, \eqref{bd:degreeonelinear1}, $m(\epsilon \tilde \xi)\min(|\tilde \xi|^{-1},\epsilon^{-1}t)\lesssim \epsilon$ and $\log \langle \tilde \xi\rangle \lesssim |\log \epsilon|$, we obtain \eqref{lagopede}. If $|\tilde \xi|>\epsilon^{-1}$ we have $m(\epsilon |\tilde \xi|)\lesssim 1$ as $m$ is nonnegative and bounded. Using this, \eqref{bd:degreeonelinear1}, $\min(|\tilde \xi|^{-1},\frac{t}{\epsilon})  \lesssim |\tilde \xi|^{-1}$ and $ |\log \epsilon|\lesssim \log \langle \tilde \xi\rangle$, we also obtain \eqref{lagopede}. Hence \eqref{lagopede} for all $\tilde \xi \in \mathbb R^d$.

\medskip

\noindent \underline{Proof of \eqref{bd:degreeonequadra}}. We first assume $|\tilde \xi|\lesssim \epsilon^{-1}$. If either $t\leq \epsilon^{2}$ or $|\gamma|\gg \epsilon^{-2}$, then $|\gamma+(\tilde \xi+\xi^f).\xi^f+\frac it|\gtrsim \epsilon^{-2}$, and \eqref{bd:degreeonequadra} is true by simply bounding the integrand by $\epsilon^2$ and the volume of the support of the integral by $\epsilon^{-d}$. We thus now assume $t\geq \epsilon^2$ and $|\gamma|\lesssim \epsilon^{-2}$. We change variables $\eta=\xi^f+\tilde \xi/2$ and notice the identity $\xi^f.(\xi^f+\tilde \xi)=|\eta|^2-|\tilde \xi|^2/4$ so that:
$$
\int_{|\xi^f|\lesssim \epsilon^{-1}} \frac{d\xi^f}{|\gamma+(\tilde \xi+\xi^f).\xi^f+\frac it|} \lesssim \int_{|\eta-\frac{\tilde \xi}{2}|\lesssim \epsilon^{-1}} \frac{d\eta}{|\gamma-\frac{|\tilde \xi|^2}{4}+|\eta|^2+\frac it|}.
$$
We now claim that for all real numbers $|A|\lesssim \epsilon^{-2}$ there holds:
\be \label{bd:claimdyadic}
\left| \left\{ \eta\in \mathbb R^d, \ |A+|\eta|^2|\leq\frac{1}{t}\right\}\right|\lesssim \frac{\epsilon^{2-d}}{t}.
\ee
Assuming the claim holds true we then partition and bound with the help of \eqref{bd:claimdyadic}:
$$
\int_{|\eta|\lesssim \epsilon^{-1}} \frac{d\eta}{|\gamma-\frac{|\tilde \xi|^2}{4}+|\eta|^2+\frac it|}\lesssim \sum_{j\in \mathbb Z, \ |j|\lesssim t\epsilon^{-2}}\frac{t}{\langle j\rangle}  \left| \left\{ \eta\in \mathbb R^d, \ |\gamma-\frac{|\tilde \xi|}{4}+|\eta|^2+\frac{j}{t}|\leq \frac{1}{t} \right\}\right|\lesssim \epsilon^{2-d}|\log \epsilon|
$$
and \eqref{bd:degreeonequadra} is obtained. It now remains to prove \eqref{bd:claimdyadic}. If $|A|\leq t^{-1}$ then we have 
$$
\left| \left\{ \eta\in \mathbb R^d, \ |A+|\eta|^2|\leq\frac{1}{t}\right\}\right|\leq \left| \left\{ \eta\in \mathbb R^d, \ |\eta|^2|\leq\frac{2}{t}\right\}\right|\lesssim t^{\frac d2}\lesssim \frac{\epsilon^{2-d}}{t}
$$
where we used that $t\geq \epsilon^2$. If $|A|\geq t^{-1}$ we change variables $\eta=|A|^{1/2}\tilde \eta$ and estimate:
$$
\left| \left\{ \eta\in \mathbb R^d, \ |A+|\eta|^2|\leq\frac{1}{t}\right\}\right|= |A|^{\frac d2} \left| \left\{ \tilde \eta\in \mathbb R^d, \ |\frac{A}{|A|}+ |\tilde \eta|^2|\leq \frac{1}{|A|t}\right\}\right|\lesssim \frac{|A|^{\frac d2-1}}{t} \lesssim \frac{\epsilon^{2-d}}{t}
$$
where we used $1/(|A|t)\leq 1$ and $|A|\lesssim \epsilon^{-2}$. The two estimates above imply \eqref{bd:claimdyadic}. This in turn shows \eqref{bd:degreeonequadra}.
 
We now assume $|\tilde \xi|\gg \epsilon^{-1} $. By rotational invariance, we can assume that $\tilde \xi=(|\tilde \xi|,0,...,0)$. We introduce $ \xi^{f}_\perp=(\xi^f_2,...,\xi^f_d)$, and then change variables $\alpha^f_1=(|\tilde \xi|+\xi_1^f) \xi^f_1$ which gives 
\begin{align}
\nonumber \int_{|\xi^f|\lesssim \epsilon^{-1}} \frac{d\xi^f}{|\gamma+(\tilde \xi+\xi^f).\xi^f+\frac it|} & = \int_{|\xi^f_\perp|\lesssim \epsilon^{-1}}d\xi^f_\perp\int_{|\xi_1^f|\lesssim \epsilon^{-1}} \frac{d\xi^f_1}{|\gamma+| \xi^f_\perp|^2+(|\tilde \xi|+\xi_1^f)\xi_1^f +\frac it|} \\
\label{technical} & \lesssim \frac{1}{|\tilde \xi|} \int_{|\xi^f_\perp|\lesssim \epsilon^{-1}}d\xi^f_\perp\int_{|\alpha_1^f|\lesssim |\tilde \xi| \epsilon^{-1}} \frac{d\alpha^f_1}{|\gamma+| \xi^f_\perp|^2+\alpha^f_1 +\frac it|}
\end{align}
where we used that $d\alpha_1 / d\xi_1^f \approx |\tilde \xi|$ for $|\xi_f^1|\lesssim \epsilon^{-1}$ since $|\tilde \xi|\gg \epsilon^{-1}$. For any $A\in \mathbb R$, we bound $\int_{|\alpha_1^f|\lesssim |\tilde \xi| \epsilon^{-1}} \frac{d\alpha^f_1}{|A+\alpha^f_1 +\frac it|}\leq \int_{|\alpha_1^f|\lesssim |\tilde \xi| \epsilon^{-1}} \frac{d\alpha^f_1}{|\alpha^f_1 +\frac it|} |\lesssim \log (|\tilde \xi| \epsilon^{-1})\lesssim \log \langle \tilde \xi \rangle$ where we used that $0< t\leq 1$ and $|\tilde \xi|\gg \epsilon^{-1}$. Injecting this inequality in \eqref{technical}, bounding by $\epsilon^{1-d}$ the integration over the remaining $\xi^f_\perp$ variable, this proves \eqref{bd:degreeonequadra}.

Hence \eqref{bd:degreeonequadra} for any $\tilde \xi \in \mathbb R^d$. This ends the proof of the Lemma.

\end{proof}

\begin{lemma}[Weighted integrals]

For $0<\epsilon\leq 1$ there holds for any $\xi'\in \mathbb R^d$, $\gamma\in \mathbb R$ and $0\leq t\leq 1$:
$$
\int_{\alpha \in \mathbb R, \ |\alpha|\lesssim \epsilon^{-2}}\frac{d\alpha}{|\gamma+\alpha+\frac it|}\lesssim |\log \epsilon |,
$$
and for $d\geq 2$:
$$
\int_{\xi \in \mathbb R^d, \ |\xi|\lesssim \epsilon^{-1}}\frac{d\xi}{|\xi'+\xi|}\lesssim \epsilon^{1-d}.
$$
For any $\beta,\beta' \in \mathbb R$, for all $t>0$ there holds:
\be \label{bd:weightedintegral3}
\int_{\alpha\in \mathbb R} \frac{d\alpha }{|\alpha+\frac it||\alpha+\beta+\frac it |} \lesssim \frac{1}{|\beta+\frac it|},
\ee
\be \label{bd:weightedintegral4}
\int_{\alpha\in \mathbb R} \frac{d \alpha }{|\alpha+\frac it||\alpha+\beta+\frac it ||\alpha+\beta'+\frac it |} \lesssim \frac{1}{|\beta+\frac it|}\frac{1}{|\beta'+\frac it|},
\ee
and if $0<t\leq 1$:
\be \label{bd:weightedintegral5}
\int_{\alpha\in \mathbb R} \frac{d \alpha}{|\alpha+i||\alpha+\beta+\frac it |} \lesssim \frac{\langle \ln t\rangle }{|\beta+\frac it|}, \ \ \int_{\alpha\in \mathbb R} \frac{d \alpha}{|\alpha+i||\alpha+\beta+\frac it ||\alpha+\beta'+\frac it |} \lesssim \frac{\langle \ln t \rangle}{|\beta+\frac it|}\frac{1}{|\beta'+\frac it|}.
\ee
\end{lemma}

\begin{proof}

\underline{Proof of \eqref{bd:weightedintegral3}.} By rescaling the integration variable, it suffices to prove the inequality for $t=1$. For $t=1$ and $|\beta|\leq 1$ the integral is $\lesssim \int \langle \alpha\rangle^{-2}d\alpha \lesssim 1$ which proves the result. For $t=1$ and $|\beta|\geq 1$ we estimate first in the zone $|\alpha|\leq 10 |\beta|$ that $|\alpha+i|^{-1} |\alpha+\beta+i |^{-1}\lesssim |\beta|^{-2}$ so that this zone contributes at most to $|\beta|^{-1}$. In the zone $|\alpha|\geq 10|\beta|$ we change variables $\alpha =|\beta|\tilde \alpha$ and bound the contribution of this zone by $|\beta|\int_{|\tilde \alpha|\geq 10} ||\beta|\tilde \alpha+i|^{-1}||\beta|\tilde \alpha+\beta +i|^{-1}\lesssim |\beta|^{-1}$, and \eqref{bd:weightedintegral3} is established.\\

\underline{Proof of \eqref{bd:weightedintegral4}.} By rescaling, it suffices to consider $t=1$. We assume $\beta \beta'\leq 0$, and $\beta \geq 0$, $\beta'\leq 0$ without loss of generality. In the zone $|\alpha|\leq 0$, we upper bound in the integral $|\alpha+\beta'+i|^{-1}\leq |\beta'+i|^{-1}$, apply \eqref{bd:weightedintegral3} to estimate $\int_{\alpha\leq 0}|\alpha+i|^{-1}|\alpha+\beta+i|^{-1}d\alpha$, and obtain the desired upper bound \eqref{bd:weightedintegral4}. In the zone $|\alpha|\geq 0$, we upper bound $|\alpha+\beta+i|^{-1}\leq |\beta+i|^{-1}$, apply \eqref{bd:weightedintegral3} to estimate $\int_{\alpha\leq 0}|\alpha+i|^{-1}|\alpha+\beta'+i|^{-1}d\alpha$, and \eqref{bd:weightedintegral4} is established. The proof if $\beta \beta'\geq 0$ can be done similarly.\\

\underline{Proof of \eqref{bd:weightedintegral5}.} For the first inequality, we bound for $|\alpha|\leq t^{-1}$ that $|\alpha+\beta+\frac it|\approx |\beta+\frac it|$ and for $|\alpha|\geq t^{-1}$ that $|\alpha+i|\approx |\alpha+\frac it|$, and use \eqref{bd:weightedintegral3} to estimate:
$$
\int_{\alpha\in \mathbb R} \frac{d \alpha}{|\alpha+i||\alpha+\beta+\frac it |}\lesssim \frac{1}{|\beta+\frac it|} \int_{|\alpha|\leq t^{-1}}\frac{d\alpha }{|\alpha+i|} + \int_{|\alpha|\geq t^{-1}}\frac{d\alpha }{|\alpha+\frac it||\alpha+\beta+\frac it |}  \lesssim \frac{\langle \ln t\rangle}{|\beta+\frac it|}.
$$
The proof of the second inequality is similar and we omit it.

\end{proof}

\end{document}